\documentclass{amsart}

\usepackage{amsmath, amsfonts, amssymb} 
\usepackage{enumerate}
\usepackage{enumitem}
\usepackage{hyperref}
\hypersetup{
  colorlinks   = true, 
  urlcolor     = blue, 
  linkcolor    = blue, 
  citecolor   = red 
}
\usepackage{graphicx}
\usepackage[all]{xy}
\usepackage{float}
\usepackage{wrapfig}
\usepackage{fancyvrb}
\usepackage[mathscr]{euscript}
\usepackage[T1]{fontenc}
\usepackage{listings}
\usepackage{tikz}
\usepackage{amsthm}
\usepackage{dsfont}
\usepackage{tikz-cd}
\usepackage{adjustbox}
\usepackage{centernot}
\usepackage{mathtools}
\usepackage[noadjust]{cite}
\usepackage{bbm}
\usepackage{stmaryrd}

\usepackage[margin=3cm]{geometry}

\let\oldtocsection=\tocsection

\let\oldtocsubsection=\tocsubsection

\let\oldtocsubsubsection=\tocsubsubsection

\renewcommand{\tocsection}[2]{\hspace{0em}\oldtocsection{#1}{#2}}
\renewcommand{\tocsubsection}[2]{\hspace{1em}\oldtocsubsection{#1}{#2}}
\renewcommand{\tocsubsubsection}[2]{\hspace{2em}\oldtocsubsubsection{#1}{#2}}

\title{Calculating double-leaf coefficients in the Hecke category}
\author{Colton Sandvik}
\date{}
\address{Department of Mathematics, Johns Hopkins University, Baltimore, MD, USA}
\email{csandvi1@jhu.edu}

\newcommand{\Q}{\mathbb{Q}}
\newcommand{\Z}{\mathbb{Z}}
\newcommand{\F}{\mathbb{F}}

\renewcommand{\k}{\mathbbm{k}}
\newcommand{\R}{\mathbb{R}}

\newcommand{\scrF}{\mathcal{F}}

\newcommand{\uw}{\underline{w}}
\newcommand{\ux}{\underline{x}}
\newcommand{\uy}{\underline{y}}

\DeclareMathOperator{\Hom}{Hom}
\newcommand{\grmod}[1]{#1\textnormal{-mod}^{\Z}}

\newcommand{\fr}[1]{\mathfrak{#1}}
\DeclareMathOperator{\op}{op}

\DeclareMathOperator{\id}{id}

\newcommand{\abs}[1]{\lvert #1 \rvert}
\renewcommand{\emptyset}{\varnothing}

\newcommand{\DD}{\mathbb{D}}

\DeclareMathOperator{\Sym}{Sym}

\DeclareMathOperator{\Exp}{Exp}
\DeclareMathOperator{\Mat}{Mat}
\DeclareMathOperator{\Sub}{Sub}
\DeclareMathOperator{\dec}{dec}
\DeclareMathOperator{\rex}{rex}
\DeclareMathOperator{\Fun}{Fun}
\DeclareMathOperator{\rk}{rk}
\DeclareMathOperator{\LL}{LL}
\newcommand{\oLL}{\overline{\LL}}

\DeclareMathOperator{\ev}{ev}
\newcommand{\dLL}{\mathbb{L}\mathbb{L}}
\newcommand{\ue}{\mathbf{e}}
\newcommand{\uf}{\mathbf{f}}
\newcommand{\uc}{\mathbf{c}}
\newcommand{\ur}{\mathbf{r}}
\newcommand{\uh}{\mathbf{h}}
\newcommand{\mcR}{\mathcal{R}}
\newcommand{\qbinom}{\genfrac{[}{]}{0pt}{}}
\newcommand{\BS}{\mathsf{ABS}}
\newcommand{\gBS}{g\mathsf{ABS}}
\newcommand{\EW}{\mathsf{EW}}
\newcommand{\sfC}{\mathsf{C}}
\DeclareMathOperator{\diag}{diag}

\DeclareMathOperator{\Frac}{Frac}
\DeclareMathOperator{\minpoly}{minpoly}
\DeclareMathOperator{\Aut}{Aut}

\newtheorem{theorem}{Theorem}[section]
\newtheorem{corollary}[theorem]{Corollary}
\newtheorem{lemma}[theorem]{Lemma}
\newtheorem{proposition}[theorem]{Proposition}

\theoremstyle{definition}
\newtheorem{definition}[theorem]{Definition}
\newtheorem{example}[theorem]{Example}
\newtheorem{remark}[theorem]{Remark}
\newtheorem{goal}[theorem]{Goal}
\newtheorem{algorithm}[theorem]{Algorithm}
\newtheorem{question}[theorem]{Question}
\newtheorem{assumption}[theorem]{Assumption}

\tikzstyle{row}=[-, draw=black, thick]
\tikzstyle{dashed_row}=[-, draw=black, dashed]

\usepackage{tikzit}

\tikzstyle{EW_s}=[fill=red, draw=red, shape=circle, inner sep=0pt, minimum size=3pt]
\tikzstyle{box}=[fill=white, draw=black, shape=rectangle]
\tikzstyle{circle_region}=[fill={rgb,255: red,255; green,205; blue,155}, draw=black, shape=circle, dashed, minimum width=6em]
\tikzstyle{JW}=[fill=white, draw=black, shape=rectangle, minimum width=1.75cm, minimum height=0.5cm]
\tikzstyle{EW_t}=[fill=blue, draw=blue, shape=circle, inner sep=0pt, minimum size=3pt]
\tikzstyle{black_dot}=[fill=black, draw=black, shape=circle, inner sep=0pt, minimum size=3pt]
\tikzstyle{small_none}=[fill=none, draw=none, shape=circle, scale=0.5]
\tikzstyle{gnode}=[fill=white, draw=black, shape=circle]
\tikzstyle{g2node}=[fill=cyan, draw=black, shape=circle]
\tikzstyle{small_dot_s}=[fill=red, draw=red, shape=circle, scale=0.5]
\tikzstyle{small_dot_t}=[fill=blue, draw=blue, shape=circle, scale=0.5]
\tikzstyle{EW_u}=[fill={rgb,255: red,128; green,0; blue,128}, draw={rgb,255: red,128; green,0; blue,128}, shape=circle, inner sep=0pt, minimum size=3pt]
\tikzstyle{mor_box}=[fill=white, draw=black, shape=rectangle, minimum width=1cm, minimum height=0.3cm]
\tikzstyle{small_mor_box}=[fill=white, draw=black, shape=rectangle, minimum width=0.6cm, minimum height=0.3cm]
\tikzstyle{med_mor_box}=[fill=white, draw=black, shape=rectangle, minimum width=1.5cm, minimum height=0.3cm]
\tikzstyle{large_mor_box}=[fill=white, draw=black, shape=rectangle, minimum width=2.3cm, minimum height=0.3cm]
\tikzstyle{small_white_node}=[fill=white, draw=black, shape=circle, scale=0.3]
\tikzstyle{3_strand_box}=[fill=white, draw=black, shape=rectangle, minimum height=0.3cm, minimum width=0.75cm]
\tikzstyle{4_strand_box}=[fill=white, draw=black, shape=rectangle, minimum height=0.3cm, minimum width=1cm]
\tikzstyle{5_strand_box}=[fill=white, draw=black, shape=rectangle, minimum height=0.3cm, minimum width=1.25cm]
\tikzstyle{6_strand_box}=[fill=white, draw=black, shape=rectangle, minimum height=0.3cm, minimum width=1.5cm]
\tikzstyle{7_strand_box}=[fill=white, draw=black, shape=rectangle, minimum height=0.3cm, minimum width=1.75cm]
\tikzstyle{8_strand_box}=[fill=white, draw=black, shape=rectangle, minimum height=0.3cm, minimum width=2cm]
\tikzstyle{9_strand_box}=[fill=white, draw=black, shape=rectangle, minimum height=0.3cm, minimum width=2.25cm]
\tikzstyle{small_mon_label}=[fill=none, draw=none, shape=circle, scale=0.5]
\tikzstyle{dot_green}=[fill=green, draw=green, shape=circle, inner sep=0pt, minimum size=3pt]
\tikzstyle{big_bracket_6}=[fill=none, draw=none, shape=circle, scale=2.5]
\tikzstyle{thick_dot_ochre}=[fill=teal, draw=teal, shape=circle, inner sep=0pt, minimum size=5pt]
\tikzstyle{thick_violet_dot}=[fill=violet, draw=violet, shape=circle, inner sep=0pt, minimum size=5pt]

\tikzstyle{row}=[-, draw=black, thick]
\tikzstyle{dashed_row}=[-, draw=black, dashed]
\tikzstyle{red_edge}=[-, draw=red, thick]
\tikzstyle{blue edge}=[-, draw=blue, thick]
\tikzstyle{purple_edge}=[-, draw={rgb,255: red,128; green,0; blue,128}, thick]
\tikzstyle{green_edge}=[-, draw=green, thick]
\tikzstyle{fill_orange}=[-, fill={rgb,255: red,255; green,205; blue,155}, draw=none]
\tikzstyle{fill_purple}=[-, fill={rgb,255: red,225; green,137; blue,225}, draw=none]
\tikzstyle{fill_green}=[-, fill={rgb,255: red,155; green,255; blue,155}, draw=none]
\tikzstyle{fill_yellow}=[-, fill={rgb,255: red,255; green,255; blue,155}, draw=none]
\tikzstyle{fill_blue}=[-, fill={rgb,255: red,155; green,255; blue,255}, draw=none]
\tikzstyle{fill_coral}=[-, fill={rgb,255: red,255; green,205; blue,205}, draw=none]
\tikzstyle{fill_brown}=[-, fill={rgb,255: red,120; green,90; blue,20}, draw=none]
\tikzstyle{fill_cyan}=[-, fill={rgb,255: red,0; green,255; blue,255}, draw=none]
\tikzstyle{fill_magenta}=[-, fill={rgb,255: red,255; green,0; blue,255}, draw=none]
\tikzstyle{fill_darkgreen}=[-, fill={rgb,255: red,0; green,165; blue,55}, draw=none]
\tikzstyle{fill_lavender}=[-, fill={rgb,255: red,230; green,170; blue,255}, draw=none]
\tikzstyle{fill_gold}=[-, fill={rgb,255: red,255; green,180; blue,0}, draw=none]
\tikzstyle{fill_dark_blue}=[-, fill=blue, draw=none]
\tikzstyle{fill_red}=[-, fill=red, draw=none]
\tikzstyle{fill_white}=[-, fill=white, draw=none]
\tikzstyle{red_dashed_edge}=[-, fill=none, draw=red, dashed, thick]
\tikzstyle{blue_dashed_edge}=[-, draw=blue, dashed, thick]
\tikzstyle{purple_dashed_edge}=[-, draw={rgb,255: red,128; green,0; blue,128}, thick, dashed]
\tikzstyle{fill_light_gray}=[-, fill={rgb,255: red,198; green,198; blue,198}, draw=none]
\tikzstyle{fill_dark_gray}=[-, fill={rgb,255: red,109; green,109; blue,109}, draw=none]
\tikzstyle{fill_sky}=[-, fill={rgb,255: red,230; green,234; blue,255}, draw=none]
\tikzstyle{fill_bright_green}=[-, fill={rgb,255: red,66; green,255; blue,79}, draw=none]
\tikzstyle{red_tl_edge}=[-, draw={rgb,255: red,255; green,128; blue,128}]
\tikzstyle{blue_tl_edge}=[-, draw={rgb,255: red,128; green,128; blue,255}]
\tikzstyle{red_tl_region}=[-, fill={rgb,255: red,255; green,128; blue,128}, draw=none]
\tikzstyle{blue_tl_region}=[-, draw=none, fill={rgb,255: red,128; green,128; blue,255}]
\tikzstyle{purple_tl_region}=[-, draw=none, fill={rgb,255: red,220; green,135; blue,220}]
\tikzstyle{orange_line}=[-, draw={rgb,255: red,255; green,205; blue,155}]
\tikzstyle{purple_line}=[-, draw={rgb,255: red,225; green,137; blue,225}]
\tikzstyle{implies_arrow}=[double, double equal sign distance, -implies, ->]
\tikzstyle{densely_dashed}=[-, densely dashed]
\tikzstyle{mapsto}=[{|->}]
\tikzstyle{arrow}=[->]
\tikzstyle{thick_purple}=[-, draw=violet, line width=3pt]
\tikzstyle{thick_teal}=[-, draw=teal, line width=3pt]

\begin{document}

        \begin{abstract}
        Given a morphism of Abe--Bott--Samelson bimodules, we construct an algorithm for computing its double-leaves coefficients.
        The key ingredient is a modification of Bruhat upper-triangularity which allows us to use forward substitution to solve for the double-leaf coefficients.
        We provide a complete implementation of the algorithm in the Sage computer algebra system.
        We use this to compute the Zamolodchikov $3$-color relation in type $H_3$ for balanced realizations. 
        We further simplify the $3$-color relation using thick Soergel calculus and express it as a difference of two thick Elias--Williamson diagrams.
        \end{abstract}

        \maketitle

        \section{Introduction}

        Let $(W,S)$ be a Coxeter system.
        A \emph{Hecke category} is a graded additive monoidal category which categorifies the Hecke algebra of $(W,S)$.
        Hecke categories are ubiquitous throughout representation theory. 
        They can be used to show positivity of Kazhdan--Lusztig polynomials \cite{So07, EW14}, discover counterexamples to Lusztig's conjecture on modular representations of reductive groups \cite{torsion}, and 
         develop character formulas for tilting modules in the modular representation theory of reductive groups \cite{AMRW}.
        
        There are now numerous incarnations of the Hecke category. The first construction for general Coxeter systems was \emph{Soergel bimodules} \cite{So07}.
        Starting with a realization $\fr{h}$ of $W$ and $R = \Sym (\fr{h}^*)$, one forms the bimodules $B_s = R \otimes_{R^s} R(1)$.
        The full subcategory of graded $R$-bimodules containing 
        \[B_{(s_1, \ldots, s_k)} \coloneq B_{s_1} \otimes_R \ldots \otimes_R B_{s_k} (n)\]
        for all sequences $(s_1, \ldots, s_k)$ of simple reflections and $n \in \Z$ forms the category of Bott--Samelson bimodules. 
        By taking the idempotent completion of the additive hull of Bott--Samelson bimodules, one forms the category of Soergel bimodules. 
        However, if $\fr{h}$ is not a sufficiently nice (i.e., reflection faithful) realization the category of Soergel bimodules need not categorify the Hecke algebra.
        
        Much more recently, there has been a modification of Soergel bimodules due to Abe that incorporates certain localization data.
        We will review their construction in \S\ref{ss:BS}. 
        The resulting category is a graded monoidal category $\BS (\fr{h}, W)$ which is well-behaved provided $\fr{h}$ satisfies some weak technical constraints.
        Every Bott--Samelson bimodule lifts to an Abe--Bott--Samelson bimodule.

        There is another notable categorification due to Elias and Williamson which we call the \emph{diagrammatic Hecke category} \cite{EW}.
        A general technique in categorification is to try and present your preferred monoidal category by means of generators and relations.
        This often has two key benefits. First, generators and relations give a set of rules for manipulating morphisms similar to a table of derivative rules
        in single-variable calculus. Second, defining monoidal functors from a monoidal category presented by generators and relations to another monoidal category 
        is often quite easy. One only needs to define the images of the generators and then check that the relations are satisfied in the target category.
        The diagrammatic Hecke category is the application of this philosophy to the category of Abe--Bott--Samelson bimodules.\footnote{  We have presented a historically anachronistic perspective for the sake of our narrative, but it is worth noting that the diagrammatic Hecke category actually 
        predates Abe--Bott--Samelson bimodules. In fact, one of Elias and Williamson's principal motivations for introducing the diagrammatic Hecke category was
        to develop a well-behaved Hecke category for a broader class of realizations than the original bimodule construction.}
        The generating diagrams are relatively simple to describe. They are given by a small collection of colored planar graphs which we will not review here.
        However, the relations take considerable effort. Nonetheless, Elias and Williamson managed to compute these relations in all but one case--- there is 
        a missing relation when $W$ has a parabolic subgroup of type $H_3$ called the \emph{Zamolodchikov $3$-color relation}. 
        For some history on this problem and its difficulty to compute see \cite[\S 1.4.3]{EW} and \cite[\S 3.6]{EWloc}.
        
        This project started with a desire to find this relation. 
        The general technique for finding this relation is not new. Let $\varphi : B_{\ux} \to B_{\uy}$ be a homomorphism of Abe--Bott--Samelson bimodules where $\ux$ and $\uy$ are sequences of simple reflections.
        The morphism space 
        \[\Hom_{\BS (\fr{h}, W)} (B_{\ux},B_{\uy})\]
        is known to be a graded free $R$-module. A basis is provided by the collection of \emph{double leaves}. 
        For simplicity we will enumerate these double leaves by
        $\dLL_1, \ldots, \dLL_{N}$. This basis was first introduced for Bott--Samelson bimodules by Libedinsky \cite{Lib}, and has since been extended to
        the diagrammatic category \cite{EW} and Abe--Bott--Samelson bimodules \cite{Abe19}.
        Finding the Zamolodchikov $3$-color relation is a special example of computing the coefficients $c_i \in R$ such that
        \[\varphi = \sum_{i=1}^{N} c_i \dLL_i.\]
        This is by no means an easy process. The existing methods in the literature amounted to constructing a matrix with entries in $\Frac (R)$, denoted $\Lambda (\varphi)$, called 
        the \emph{localization matrix} \cite{EWloc}. The columns of $\Lambda (\varphi)$ are indexed by subexpressions of $\ux$ and the rows by subexpressions of $\uy$.
        To better emphasize the computational difficulties of working with these matrices, for the remainder of the introduction, we will identify the set of subexpressions of $\ux$ (resp. $\uy$) with the integers $1,\ldots, 2^{\ell (\ux)}$ (resp. $1, \ldots, 2^{\ell (\uy)}$).
        In particular, under this convention, we can regard $\Lambda (\varphi)$ as a $2^{\ell (\uy)} \times 2^{\ell (\ux)}$-matrix.
        The localization matrix $\Lambda (\varphi)$ determines $\varphi$, so it suffices to solve the (enormous) system of $Q$-linear equations
        \begin{equation}\label{eq:loc_mat_eqn}
                \Lambda (\varphi) = \sum_{i=1}^{N} c_i \Lambda (\dLL_i).
        \end{equation}

        There are two key computational difficulties associated with solving \eqref{eq:loc_mat_eqn}.
        \begin{enumerate}
                \item The localization matrices themselves are not easy to compute.
                        They often have very few nonzero entries; however, proceeding naively they involve multiplying many large sparse matrices.
                        For example, in the $H_3$ relation, one has that $\ell (\ux) = \ell (\uy) = 15$, and so the localization matrices are $32768 \times 32768$. 
                \item Even after computing the localization matrices appearing in \eqref{eq:loc_mat_eqn}, one still has to use Gaussian elimination to solve this system of equations.
                        Despite the matrices being sparse, the entries being in $\Frac (R)$ makes this difficult for modern computers.
        \end{enumerate}

        To address these difficulties, we perform a study on localization matrices for the double leaves. 
         The main result in this direction is a formula for computing the localization matrix from the localization matrices of the constituent light leaves.

         \begin{proposition}[Proposition \ref{prop:loc_of_dl}]\label{prop:intro}
                For each $i =1,\ldots, N$, we write
                \[\dLL_i = \overline{\LL}_{\uy, \uf^i} \circ \LL_{\ux, \ue^i}\]
                where $\LL_{\uy, \uf^i}$ and $\LL_{\ux, \ue^i}$ are the constituent light leaves of the double leaf $\dLL_i$. Then
                \[\Lambda (\dLL_i)_{r, c} = \frac{1}{\zeta (r)} \sum_{h} \zeta (h) (\Lambda (\LL_{\uy, \uf^i})_{h,r} ) \Lambda (\LL_{\ux, \ue^i})_{h, c},\]
                where $\zeta (r)$ and $\zeta (h)$ are certain explicit products of roots in $R$ that depend on the subexpressions corresponding to $r$ and $h$ and their ambient expressions. 
         \end{proposition}
        
         One can cache the individual localization matrices of the constituent light leaves which greatly speeds up the computation of the localization matrices for the double leaves.
         Moreover, by computing each column of $\Lambda (\dLL_i)$ at a time, we can avoid multiplying any large sparse matrices.

         To address our second computational difficulty, we first define a total order on the double leaves.
         This total order is not defined uniquely, but satisfies some desirable properties. 
         We then define an injective linear map
         \[\lambda : \Hom (B_{\ux}, B_{\uy}) \hookrightarrow \Frac (R)^{N}.\]
         The map $\lambda$ selects $N$ entries of the localization matrix $\Lambda (\varphi)$. These entries determines the morphism.
         More precisely,
         \[\lambda (\varphi) = \begin{bmatrix} \vdots & & \vdots \\  \lambda (\dLL_1) & \ldots & \lambda (\dLL_{N}) \\ \vdots & & \vdots \end{bmatrix} \begin{bmatrix} c_1 \\ \vdots \\ c_{N}\end{bmatrix}.\]
         We are now ready to state the main theorem of the paper.

         \begin{theorem}[Theorem \ref{thm:lower_triang_of_coord_functionals}]\label{thm:intro}
                Under our fixed total order on the double leaves, the $N \times N$-matrix
                \[\begin{bmatrix} \vdots & & \vdots \\  \lambda (\dLL_1) & \ldots & \lambda (\dLL_{N}) \\ \vdots & & \vdots \end{bmatrix}\]
                is lower triangular and all the diagonal entries are nonzero.
         \end{theorem}

        As a result, we can use forward substitution to solve for the coefficients $c_i$. The resulting algorithm is detailed in Algorithm~\ref{algo:coeff_algo}.

        Computer algebra systems are not very efficient at computing in the polynomial ring $R$ or its fraction field.
        To speed up Algorithm~\ref{algo:coeff_algo} we also present a variation, Algorithm~\ref{algo:coeff_algo_interpolate}, which makes use of polynomial interpolation (see Proposition~\ref{prop:interpolation}) to 
        replace symbolic calculations by exact arithmetic in $\k$.

         We use Algorithm~\ref{algo:coeff_algo_interpolate} to find the double leaf expansion of the missing $H_3$ relation.
         By choosing our starting reduced expression of the longest element of $H_3$ carefully, we will see that only 125 double leaves contribute to this expansion.
         Furthermore, by using the thick Soergel calculus of \cite{Elib,EliThick,AbeThick1,AbeThick2}, we can further simplify the $H_3$ relation to be a difference of two thick Elias--Williamson graphs.

         \begin{theorem}\label{thm:thick_h3}
                For any balanced realization of $H_3$, the following relation holds between Abe--Bott--Samelson bimodules,
        \begin{equation*}
                \resizebox{0.95\hsize}{!}{$

  \tikzsetnextfilename{h3_left}
  \tikzstyle{every picture}=[tikzfig]
  \begin{tikzpicture}
	\begin{pgfonlayer}{nodelayer}
		\node [style=none] (0) at (2.25, -3.25) {};
		\node [style=none] (1) at (1.75, -3.25) {};
		\node [style=none] (2) at (1.25, -3.25) {};
		\node [style=none] (3) at (0.75, -3.25) {};
		\node [style=none] (4) at (0.25, -3.25) {};
		\node [style=none] (5) at (-0.25, -3.25) {};
		\node [style=none] (6) at (-0.75, -3.25) {};
		\node [style=none] (7) at (-1.25, -3.25) {};
		\node [style=none] (8) at (-1.75, -3.25) {};
		\node [style=none] (9) at (2.75, -3.25) {};
		\node [style=none] (10) at (3.25, -3.25) {};
		\node [style=none] (11) at (-2.25, -3.25) {};
		\node [style=none] (12) at (-2.75, -3.25) {};
		\node [style=none] (13) at (-3.25, -3.25) {};
		\node [style=none] (14) at (-3.75, -3.25) {};
		\node [style=none] (15) at (2.25, 3.25) {};
		\node [style=none] (16) at (1.75, 3.25) {};
		\node [style=none] (17) at (1.25, 3.25) {};
		\node [style=none] (18) at (0.75, 3.25) {};
		\node [style=none] (19) at (0.25, 3.25) {};
		\node [style=none] (20) at (-0.25, 3.25) {};
		\node [style=none] (21) at (-0.75, 3.25) {};
		\node [style=none] (22) at (-1.25, 3.25) {};
		\node [style=none] (23) at (-1.75, 3.25) {};
		\node [style=none] (24) at (2.75, 3.25) {};
		\node [style=none] (25) at (3.25, 3.25) {};
		\node [style=none] (26) at (-2.25, 3.25) {};
		\node [style=none] (27) at (-2.75, 3.25) {};
		\node [style=none] (28) at (-3.25, 3.25) {};
		\node [style=none] (29) at (-3.75, 3.25) {};
		\node [style=none] (30) at (2.75, -2.75) {};
		\node [style=none] (31) at (1.25, -2.25) {};
		\node [style=none] (32) at (-0.25, -2.75) {};
		\node [style=none] (33) at (-1.25, -1.75) {};
		\node [style=none] (34) at (0.25, -1.25) {};
		\node [style=none] (35) at (-2.25, -1.25) {};
		\node [style=none] (36) at (1.25, -0.75) {};
		\node [style=none] (37) at (-1.25, -0.75) {};
		\node [style=none] (38) at (0.25, -0.25) {};
		\node [style=none] (39) at (-2.25, 0.25) {};
		\node [style=none] (40) at (-0.75, 0.75) {};
		\node [style=none] (41) at (1, 1) {};
		\node [style=none] (42) at (2.75, 0.5) {};
		\node [style=none] (43) at (-0.25, 2.25) {};
		\node [style=none] (44) at (1.25, 2.75) {};
		\node [style=none] (45) at (-1.5, 2.5) {};
	\end{pgfonlayer}
	\begin{pgfonlayer}{edgelayer}
		\draw [style={red_edge}] (31.center) to (3.center);
		\draw [style={red_edge}] (31.center) to (1.center);
		\draw [style=blue edge] (31.center) to (2.center);
		\draw [style=blue edge] (31.center) to (32.center);
		\draw [style=blue edge] (32.center) to (5.center);
		\draw [style=blue edge] (31.center) to (30.center);
		\draw [style=blue edge] (30.center) to (9.center);
		\draw [style={green_edge}] (32.center) to (4.center);
		\draw [style={green_edge}] (32.center) to (6.center);
		\draw [style={green_edge}] (30.center) to (0.center);
		\draw [style={green_edge}] (30.center) to (10.center);
		\draw [style={red_edge}] (33.center) to (31.center);
		\draw [style={red_edge}] (33.center) to (7.center);
		\draw [style=blue edge] (33.center) to (32.center);
		\draw [style=blue edge] (33.center) to (8.center);
		\draw [style=blue edge] (33.center) to (34.center);
		\draw [style=blue edge] (34.center) to (31.center);
		\draw [style={green_edge}] (34.center) to (32.center);
		\draw [style={green_edge}] (35.center) to (11.center);
		\draw [style={red_edge}] (33.center) to (12.center);
		\draw [style=blue edge] (35.center) to (33.center);
		\draw [style=blue edge] (13.center) to (35.center);
		\draw [style={green_edge}] (34.center) to (36.center);
		\draw [style=blue edge] (31.center) to (36.center);
		\draw [style={green_edge}] (36.center) to (30.center);
		\draw [style={green_edge}] (37.center) to (35.center);
		\draw [style={green_edge}] (37.center) to (34.center);
		\draw [style=blue edge] (37.center) to (33.center);
		\draw [style=blue edge] (37.center) to (38.center);
		\draw [style=blue edge] (38.center) to (34.center);
		\draw [style=blue edge] (38.center) to (36.center);
		\draw [style={red_edge}] (38.center) to (31.center);
		\draw [style={red_edge}] (38.center) to (33.center);
		\draw [style={red_edge}] (39.center) to (38.center);
		\draw [style={red_edge}] (39.center) to (14.center);
		\draw [style={green_edge}] (35.center) to (29.center);
		\draw [style=blue edge] (39.center) to (28.center);
		\draw [style=blue edge] (39.center) to (35.center);
		\draw [style=blue edge] (39.center) to (37.center);
		\draw [style={red_edge}] (39.center) to (33.center);
		\draw [style=blue edge] (39.center) to (40.center);
		\draw [style=blue edge] (41.center) to (42.center);
		\draw [style={green_edge}] (40.center) to (41.center);
		\draw [style={red_edge}] (43.center) to (42.center);
		\draw [style=blue edge] (43.center) to (44.center);
		\draw [style=blue edge] (43.center) to (45.center);
		\draw [style={green_edge}] (40.center) to (37.center);
		\draw [style={green_edge}] (45.center) to (40.center);
		\draw [style={green_edge}] (45.center) to (27.center);
		\draw [style={green_edge}] (45.center) to (21.center);
		\draw [style=blue edge] (23.center) to (45.center);
		\draw [style=blue edge] (45.center) to (39.center);
		\draw [style={red_edge}] (26.center) to (39.center);
		\draw [style={red_edge}] (43.center) to (22.center);
		\draw [style=blue edge] (44.center) to (42.center);
		\draw [style=blue edge] (43.center) to (20.center);
		\draw [style={green_edge}] (44.center) to (16.center);
		\draw [style={green_edge}] (44.center) to (19.center);
		\draw [style=blue edge] (17.center) to (44.center);
		\draw [style={red_edge}] (43.center) to (18.center);
		\draw [style=blue edge] (43.center) to (41.center);
		\draw [style=blue edge] (41.center) to (38.center);
		\draw [style={green_edge}] (41.center) to (44.center);
		\draw [style={green_edge}] (41.center) to (36.center);
		\draw [style={red_edge}] (39.center) to (43.center);
		\draw [style={red_edge}] (42.center) to (15.center);
		\draw [style=blue edge] (42.center) to (24.center);
		\draw [style={red_edge}] (42.center) to (25.center);
		\draw [style={red_edge}] (42.center) to (38.center);
		\draw [style=blue edge] (42.center) to (36.center);
		\draw [style=blue edge] (42.center) to (30.center);
		\draw [style={red_edge}] (42.center) to (31.center);
		\draw [style=blue edge] (40.center) to (38.center);
		\draw [style=blue edge] (43.center) to (40.center);
		\draw [style={red_edge}] (43.center) to (38.center);
	\end{pgfonlayer}
\end{tikzpicture}
 - 
  \tikzsetnextfilename{h3_right}
  \tikzstyle{every picture}=[tikzfig]
  \begin{tikzpicture}
	\begin{pgfonlayer}{nodelayer}
		\node [style=none] (0) at (-2, 3.25) {};
		\node [style=none] (1) at (-1.5, 3.25) {};
		\node [style=none] (2) at (-1, 3.25) {};
		\node [style=none] (3) at (-0.5, 3.25) {};
		\node [style=none] (4) at (0, 3.25) {};
		\node [style=none] (5) at (0.5, 3.25) {};
		\node [style=none] (6) at (1, 3.25) {};
		\node [style=none] (7) at (1.5, 3.25) {};
		\node [style=none] (8) at (2, 3.25) {};
		\node [style=none] (9) at (-2.5, 3.25) {};
		\node [style=none] (10) at (-3, 3.25) {};
		\node [style=none] (11) at (2.5, 3.25) {};
		\node [style=none] (12) at (3, 3.25) {};
		\node [style=none] (13) at (3.5, 3.25) {};
		\node [style=none] (14) at (4, 3.25) {};
		\node [style=none] (15) at (-2, -3.25) {};
		\node [style=none] (16) at (-1.5, -3.25) {};
		\node [style=none] (17) at (-1, -3.25) {};
		\node [style=none] (18) at (-0.5, -3.25) {};
		\node [style=none] (19) at (0, -3.25) {};
		\node [style=none] (20) at (0.5, -3.25) {};
		\node [style=none] (21) at (1, -3.25) {};
		\node [style=none] (22) at (1.5, -3.25) {};
		\node [style=none] (23) at (2, -3.25) {};
		\node [style=none] (24) at (-2.5, -3.25) {};
		\node [style=none] (25) at (-3, -3.25) {};
		\node [style=none] (26) at (2.5, -3.25) {};
		\node [style=none] (27) at (3, -3.25) {};
		\node [style=none] (28) at (3.5, -3.25) {};
		\node [style=none] (29) at (4, -3.25) {};
		\node [style=none] (30) at (-2.5, 2.75) {};
		\node [style=none] (31) at (-1, 2.25) {};
		\node [style=none] (32) at (0.5, 2.75) {};
		\node [style=none] (33) at (1.5, 1.75) {};
		\node [style=none] (34) at (0, 1.25) {};
		\node [style=none] (35) at (2.5, 1.25) {};
		\node [style=none] (36) at (-1, 0.75) {};
		\node [style=none] (37) at (1.5, 0.75) {};
		\node [style=none] (38) at (0, 0.25) {};
		\node [style=none] (39) at (2.5, -0.25) {};
		\node [style=none] (40) at (1, -0.75) {};
		\node [style=none] (41) at (-0.75, -1) {};
		\node [style=none] (42) at (-2.5, -0.5) {};
		\node [style=none] (43) at (0.5, -2.25) {};
		\node [style=none] (44) at (-1, -2.75) {};
		\node [style=none] (45) at (1.75, -2.5) {};
	\end{pgfonlayer}
	\begin{pgfonlayer}{edgelayer}
		\draw [style={red_edge}] (31.center) to (3.center);
		\draw [style={red_edge}] (31.center) to (1.center);
		\draw [style=blue edge] (31.center) to (2.center);
		\draw [style=blue edge] (31.center) to (32.center);
		\draw [style=blue edge] (32.center) to (5.center);
		\draw [style=blue edge] (31.center) to (30.center);
		\draw [style=blue edge] (30.center) to (9.center);
		\draw [style={green_edge}] (32.center) to (4.center);
		\draw [style={green_edge}] (32.center) to (6.center);
		\draw [style={green_edge}] (30.center) to (0.center);
		\draw [style={green_edge}] (30.center) to (10.center);
		\draw [style={red_edge}] (33.center) to (31.center);
		\draw [style={red_edge}] (33.center) to (7.center);
		\draw [style=blue edge] (33.center) to (32.center);
		\draw [style=blue edge] (33.center) to (8.center);
		\draw [style=blue edge] (33.center) to (34.center);
		\draw [style=blue edge] (34.center) to (31.center);
		\draw [style={green_edge}] (34.center) to (32.center);
		\draw [style={green_edge}] (35.center) to (11.center);
		\draw [style={red_edge}] (33.center) to (12.center);
		\draw [style=blue edge] (35.center) to (33.center);
		\draw [style=blue edge] (13.center) to (35.center);
		\draw [style={green_edge}] (34.center) to (36.center);
		\draw [style=blue edge] (31.center) to (36.center);
		\draw [style={green_edge}] (36.center) to (30.center);
		\draw [style={green_edge}] (37.center) to (35.center);
		\draw [style={green_edge}] (37.center) to (34.center);
		\draw [style=blue edge] (37.center) to (33.center);
		\draw [style=blue edge] (37.center) to (38.center);
		\draw [style=blue edge] (38.center) to (34.center);
		\draw [style=blue edge] (38.center) to (36.center);
		\draw [style={red_edge}] (38.center) to (31.center);
		\draw [style={red_edge}] (38.center) to (33.center);
		\draw [style={red_edge}] (39.center) to (38.center);
		\draw [style={red_edge}] (39.center) to (14.center);
		\draw [style={green_edge}] (35.center) to (29.center);
		\draw [style=blue edge] (39.center) to (28.center);
		\draw [style=blue edge] (39.center) to (35.center);
		\draw [style=blue edge] (39.center) to (37.center);
		\draw [style={red_edge}] (39.center) to (33.center);
		\draw [style=blue edge] (39.center) to (40.center);
		\draw [style=blue edge] (41.center) to (42.center);
		\draw [style={green_edge}] (40.center) to (41.center);
		\draw [style={red_edge}] (43.center) to (42.center);
		\draw [style=blue edge] (43.center) to (44.center);
		\draw [style=blue edge] (43.center) to (45.center);
		\draw [style={green_edge}] (40.center) to (37.center);
		\draw [style={green_edge}] (45.center) to (40.center);
		\draw [style={green_edge}] (45.center) to (27.center);
		\draw [style={green_edge}] (45.center) to (21.center);
		\draw [style=blue edge] (23.center) to (45.center);
		\draw [style=blue edge] (45.center) to (39.center);
		\draw [style={red_edge}] (26.center) to (39.center);
		\draw [style={red_edge}] (43.center) to (22.center);
		\draw [style=blue edge] (44.center) to (42.center);
		\draw [style=blue edge] (43.center) to (20.center);
		\draw [style={green_edge}] (44.center) to (16.center);
		\draw [style={green_edge}] (44.center) to (19.center);
		\draw [style=blue edge] (17.center) to (44.center);
		\draw [style={red_edge}] (43.center) to (18.center);
		\draw [style=blue edge] (43.center) to (41.center);
		\draw [style=blue edge] (41.center) to (38.center);
		\draw [style={green_edge}] (41.center) to (44.center);
		\draw [style={green_edge}] (41.center) to (36.center);
		\draw [style={red_edge}] (39.center) to (43.center);
		\draw [style={red_edge}] (42.center) to (15.center);
		\draw [style=blue edge] (42.center) to (24.center);
		\draw [style={red_edge}] (42.center) to (25.center);
		\draw [style={red_edge}] (42.center) to (38.center);
		\draw [style=blue edge] (42.center) to (36.center);
		\draw [style=blue edge] (42.center) to (30.center);
		\draw [style={red_edge}] (42.center) to (31.center);
		\draw [style=blue edge] (40.center) to (38.center);
		\draw [style=blue edge] (43.center) to (40.center);
		\draw [style={red_edge}] (43.center) to (38.center);
	\end{pgfonlayer}
\end{tikzpicture}
 = 2 \left( 
  \tikzsetnextfilename{K_rot_G}
  \tikzstyle{every picture}=[tikzfig]
  \begin{tikzpicture}
	\begin{pgfonlayer}{nodelayer}
		\node [style=none] (0) at (-3.25, -3.25) {};
		\node [style=none] (1) at (-2.75, -3.25) {};
		\node [style=none] (2) at (-2.25, -3.25) {};
		\node [style=none] (3) at (-1.75, -3.25) {};
		\node [style=none] (4) at (-1.25, -3.25) {};
		\node [style=none] (5) at (-0.75, -3.25) {};
		\node [style=none] (6) at (-0.25, -3.25) {};
		\node [style=none] (7) at (0.25, -3.25) {};
		\node [style=none] (8) at (0.75, -3.25) {};
		\node [style=none] (9) at (1.25, -3.25) {};
		\node [style=none] (10) at (1.75, -3.25) {};
		\node [style=none] (11) at (2.25, -3.25) {};
		\node [style=none] (12) at (2.75, -3.25) {};
		\node [style=none] (13) at (3.25, -3.25) {};
		\node [style=none] (14) at (3.75, -3.25) {};
		\node [style=none] (47) at (1.75, -1.5) {};
		\node [style={thick_dot_ochre}] (50) at (3.25, -2.75) {};
		\node [style={thick_dot_ochre}] (52) at (0.25, -2.75) {};
		\node [style=none] (53) at (-2, -1.5) {};
		\node [style=none] (57) at (3.25, 3.25) {};
		\node [style=none] (58) at (2.75, 3.25) {};
		\node [style=none] (59) at (2.25, 3.25) {};
		\node [style=none] (60) at (1.75, 3.25) {};
		\node [style=none] (61) at (1.25, 3.25) {};
		\node [style=none] (62) at (0.75, 3.25) {};
		\node [style=none] (63) at (0.25, 3.25) {};
		\node [style=none] (64) at (-0.25, 3.25) {};
		\node [style=none] (65) at (-0.75, 3.25) {};
		\node [style=none] (66) at (-1.25, 3.25) {};
		\node [style=none] (67) at (-1.75, 3.25) {};
		\node [style=none] (68) at (-2.25, 3.25) {};
		\node [style=none] (69) at (-2.75, 3.25) {};
		\node [style=none] (70) at (-3.25, 3.25) {};
		\node [style=none] (71) at (-3.75, 3.25) {};
		\node [style={thick_dot_ochre}] (78) at (-1.75, 2.5) {};
		\node [style=none] (79) at (-1.75, 1.25) {};
		\node [style={thick_dot_ochre}] (81) at (1.25, 2.5) {};
		\node [style=none] (82) at (1.75, 1.25) {};
		\node [style=none] (83) at (0, 0.25) {};
		\node [style=none] (84) at (0, -0.5) {};
		\node [style={thick_dot_ochre}] (85) at (-2.5, -2.25) {};
	\end{pgfonlayer}
	\begin{pgfonlayer}{edgelayer}
		\draw [style=blue edge] (7.center) to (52);
		\draw [style=blue edge, in=-30, out=105, looseness=0.75] (50) to (47.center);
		\draw [style=blue edge] (47.center) to (10.center);
		\draw [style={green_edge}] (8.center) to (52);
		\draw [style={green_edge}] (52) to (6.center);
		\draw [style={red_edge}, in=-45, out=90] (11.center) to (47.center);
		\draw [style={red_edge}, in=90, out=-135] (47.center) to (9.center);
		\draw [style={red_edge}, in=-45, out=90] (5.center) to (53.center);
		\draw [style={red_edge}] (2.center) to (53.center);
		\draw [style=blue edge, in=-75, out=90] (4.center) to (53.center);
		\draw [style={red_edge}, in=90, out=-165] (53.center) to (0.center);
		\draw [style=blue edge] (78) to (79.center);
		\draw [style={red_edge}, in=45, out=-90, looseness=0.75] (66.center) to (79.center);
		\draw [style={red_edge}, in=150, out=-90, looseness=0.75] (68.center) to (79.center);
		\draw [style=blue edge, in=15, out=-90, looseness=0.75] (64.center) to (79.center);
		\draw [style=blue edge, in=165, out=-90] (70.center) to (79.center);
		\draw [style={red_edge}, in=15, out=-90] (57.center) to (82.center);
		\draw [style={red_edge}, in=60, out=-90, looseness=0.75] (59.center) to (82.center);
		\draw [style=blue edge, in=45, out=-90] (58.center) to (82.center);
		\draw [style=blue edge, in=270, out=120] (82.center) to (81);
		\draw [style={red_edge}, in=165, out=-90, looseness=1.25] (62.center) to (82.center);
		\draw [style={thick_purple}, in=45, out=-90, looseness=0.75] (82.center) to (83.center);
		\draw [style={thick_purple}, in=-90, out=120, looseness=0.75] (83.center) to (79.center);
		\draw [style={thick_purple}] (84.center) to (83.center);
		\draw [style={green_edge}, in=90, out=-45] (85) to (3.center);
		\draw [style={green_edge}, in=-90, out=135, looseness=0.50] (85) to (71.center);
		\draw [style=blue edge, in=90, out=-120, looseness=1.25] (85) to (1.center);
		\draw [style=blue edge] (85) to (53.center);
		\draw [style=blue edge] (13.center) to (50);
		\draw [style=blue edge, in=60, out=-150, looseness=0.75] (47.center) to (52);
		\draw [style={green_edge}] (14.center) to (50);
		\draw [style={green_edge}] (50) to (12.center);
		\draw [style={thick_purple}, in=-90, out=90, looseness=0.75] (47.center) to (84.center);
		\draw [style={thick_purple}] (84.center) to (83.center);
		\draw [style={thick_purple}, in=90, out=-105, looseness=0.75] (84.center) to (53.center);
		\draw [style=blue edge] (78) to (67.center);
		\draw [style={green_edge}] (65.center) to (78);
		\draw [style={green_edge}] (78) to (69.center);
		\draw [style={green_edge}] (81) to (63.center);
		\draw [style={green_edge}] (81) to (60.center);
		\draw [style=blue edge] (81) to (61.center);
	\end{pgfonlayer}
\end{tikzpicture}
 - 
  \tikzsetnextfilename{K_H}
  \tikzstyle{every picture}=[tikzfig]
  \begin{tikzpicture}
	\begin{pgfonlayer}{nodelayer}
		\node [style=none] (0) at (3.25, 3.25) {};
		\node [style=none] (1) at (2.75, 3.25) {};
		\node [style=none] (2) at (2.25, 3.25) {};
		\node [style=none] (3) at (1.75, 3.25) {};
		\node [style=none] (4) at (1.25, 3.25) {};
		\node [style=none] (5) at (0.75, 3.25) {};
		\node [style=none] (6) at (0.25, 3.25) {};
		\node [style=none] (7) at (-0.25, 3.25) {};
		\node [style=none] (8) at (-0.75, 3.25) {};
		\node [style=none] (9) at (-1.25, 3.25) {};
		\node [style=none] (10) at (-1.75, 3.25) {};
		\node [style=none] (11) at (-2.25, 3.25) {};
		\node [style=none] (12) at (-2.75, 3.25) {};
		\node [style=none] (13) at (-3.25, 3.25) {};
		\node [style=none] (14) at (-3.75, 3.25) {};
		\node [style=none] (47) at (-1.75, 1.5) {};
		\node [style={thick_dot_ochre}] (50) at (-3.25, 2.75) {};
		\node [style={thick_dot_ochre}] (52) at (-0.25, 2.75) {};
		\node [style=none] (53) at (2, 1.5) {};
		\node [style=none] (57) at (-3.25, -3.25) {};
		\node [style=none] (58) at (-2.75, -3.25) {};
		\node [style=none] (59) at (-2.25, -3.25) {};
		\node [style=none] (60) at (-1.75, -3.25) {};
		\node [style=none] (61) at (-1.25, -3.25) {};
		\node [style=none] (62) at (-0.75, -3.25) {};
		\node [style=none] (63) at (-0.25, -3.25) {};
		\node [style=none] (64) at (0.25, -3.25) {};
		\node [style=none] (65) at (0.75, -3.25) {};
		\node [style=none] (66) at (1.25, -3.25) {};
		\node [style=none] (67) at (1.75, -3.25) {};
		\node [style=none] (68) at (2.25, -3.25) {};
		\node [style=none] (69) at (2.75, -3.25) {};
		\node [style=none] (70) at (3.25, -3.25) {};
		\node [style=none] (71) at (3.75, -3.25) {};
		\node [style={thick_dot_ochre}] (78) at (1.75, -2.5) {};
		\node [style=none] (79) at (1.75, -1.25) {};
		\node [style={thick_dot_ochre}] (81) at (-1.25, -2.5) {};
		\node [style=none] (82) at (-1.75, -1.25) {};
		\node [style=none] (83) at (0, -0.25) {};
		\node [style=none] (84) at (0, 0.5) {};
		\node [style={thick_dot_ochre}] (85) at (2.5, 2.25) {};
	\end{pgfonlayer}
	\begin{pgfonlayer}{edgelayer}
		\draw [style=blue edge] (7.center) to (52);
		\draw [style=blue edge, in=120, out=-75, looseness=0.75] (50) to (47.center);
		\draw [style=blue edge] (47.center) to (10.center);
		\draw [style={green_edge}] (8.center) to (52);
		\draw [style={green_edge}] (52) to (6.center);
		\draw [style={red_edge}, in=135, out=-90] (11.center) to (47.center);
		\draw [style={red_edge}, in=-90, out=45] (47.center) to (9.center);
		\draw [style={red_edge}, in=135, out=-90] (5.center) to (53.center);
		\draw [style={red_edge}] (2.center) to (53.center);
		\draw [style=blue edge, in=105, out=-90] (4.center) to (53.center);
		\draw [style={red_edge}, in=-90, out=15] (53.center) to (0.center);
		\draw [style=blue edge] (78) to (79.center);
		\draw [style={red_edge}, in=-135, out=90, looseness=0.75] (66.center) to (79.center);
		\draw [style={red_edge}, in=-30, out=90, looseness=0.75] (68.center) to (79.center);
		\draw [style=blue edge, in=-165, out=90, looseness=0.75] (64.center) to (79.center);
		\draw [style=blue edge, in=-15, out=90] (70.center) to (79.center);
		\draw [style={red_edge}, in=-165, out=90] (57.center) to (82.center);
		\draw [style={red_edge}, in=-120, out=90, looseness=0.75] (59.center) to (82.center);
		\draw [style=blue edge, in=-135, out=90] (58.center) to (82.center);
		\draw [style=blue edge] (82.center) to (81);
		\draw [style={red_edge}, in=-15, out=90, looseness=1.25] (62.center) to (82.center);
		\draw [style={thick_purple}, in=-135, out=90, looseness=0.75] (82.center) to (83.center);
		\draw [style={thick_purple}, in=90, out=-60, looseness=0.75] (83.center) to (79.center);
		\draw [style={thick_purple}] (84.center) to (83.center);
		\draw [style={green_edge}, in=-90, out=135] (85) to (3.center);
		\draw [style={green_edge}, in=90, out=-45, looseness=0.50] (85) to (71.center);
		\draw [style=blue edge, in=-90, out=60, looseness=1.25] (85) to (1.center);
		\draw [style=blue edge] (85) to (53.center);
		\draw [style=blue edge] (13.center) to (50);
		\draw [style=blue edge, in=-120, out=45, looseness=0.75] (47.center) to (52);
		\draw [style={green_edge}] (14.center) to (50);
		\draw [style={green_edge}] (50) to (12.center);
		\draw [style={thick_purple}, in=90, out=-90, looseness=0.75] (47.center) to (84.center);
		\draw [style={thick_purple}] (84.center) to (83.center);
		\draw [style={thick_purple}, in=-90, out=75, looseness=0.75] (84.center) to (53.center);
		\draw [style=blue edge] (78) to (67.center);
		\draw [style={green_edge}] (65.center) to (78);
		\draw [style={green_edge}] (78) to (69.center);
		\draw [style={green_edge}] (81) to (63.center);
		\draw [style={green_edge}] (81) to (60.center);
		\draw [style=blue edge] (81) to (61.center);
	\end{pgfonlayer}
\end{tikzpicture}
 \right).
                $}
        \end{equation*}
                where the purple thick strand corresponds to the subset $\textcolor{violet}{I} = \{\textcolor{red}{s},\textcolor{blue}{t}\}$, the teal thick strand corresponds to the subset $\textcolor{teal}{J} = \{\textcolor{blue}{t}, \textcolor{green}{u}\}$, and the morphism involving a thick teal dot is given by
                \[
  \tikzsetnextfilename{half_JW}
  \tikzstyle{every picture}=[tikzfig]
  \begin{tikzpicture}
	\begin{pgfonlayer}{nodelayer}
		\node [style=none] (36) at (-3.75, -1) {};
		\node [style=none] (38) at (-2.75, -1) {};
		\node [style=none] (39) at (-3.75, 1) {};
		\node [style=none] (40) at (-2.75, 1) {};
		\node [style={thick_dot_ochre}] (41) at (-3.25, 0) {};
		\node [style=none] (42) at (-2.25, 0) {$\coloneq$};
		\node [style=none] (43) at (-1.75, -1) {};
		\node [style=none] (44) at (-1.25, -1) {};
		\node [style={EW_t}] (45) at (-0.75, -0.75) {};
		\node [style=none] (46) at (-1.25, -0.5) {};
		\node [style=none] (47) at (-1.75, 1) {};
		\node [style=none] (48) at (-1.25, 1) {};
		\node [style={dot_green}] (49) at (-0.75, 0.75) {};
		\node [style=none] (50) at (-1.25, 0.5) {};
		\node [style=none] (53) at (0, 0) {$=$};
		\node [style=none] (54) at (0.75, -1) {};
		\node [style=none] (55) at (0.75, 1) {};
		\node [style=none] (56) at (1.75, -1) {};
		\node [style=none] (57) at (1.75, 1) {};
		\node [style=none] (58) at (2.25, 0) {$+$};
		\node [style=none] (59) at (2.75, -1) {};
		\node [style=none] (60) at (2.75, 1) {};
		\node [style=none] (61) at (3.75, 1) {};
		\node [style=none] (62) at (3.75, -1) {};
		\node [style={dot_green}] (63) at (1.75, -0.25) {};
		\node [style={dot_green}] (64) at (0.75, 0.25) {};
		\node [style={EW_t}] (65) at (3.75, 0.25) {};
		\node [style={EW_t}] (66) at (2.75, -0.25) {};
	\end{pgfonlayer}
	\begin{pgfonlayer}{edgelayer}
		\draw [style=blue edge] (36.center) to (41);
		\draw [style=blue edge] (41) to (40.center);
		\draw [style={green_edge}] (38.center) to (41);
		\draw [style={green_edge}] (41) to (39.center);
		\draw [style=blue edge] (43.center) to (46.center);
		\draw [style=blue edge] (46.center) to (45);
		\draw [style={green_edge}] (50.center) to (49);
		\draw [style={green_edge}] (50.center) to (47.center);
		\draw [style=blue edge] (50.center) to (48.center);
		\draw [style={green_edge}] (46.center) to (44.center);
		\draw [style={thick_teal}] (50.center) to (46.center);
		\draw [style=blue edge] (54.center) to (57.center);
		\draw [style=blue edge] (66) to (59.center);
		\draw [style=blue edge] (65) to (61.center);
		\draw [style={green_edge}] (60.center) to (62.center);
		\draw [style={green_edge}] (55.center) to (64);
		\draw [style={green_edge}] (63) to (56.center);
	\end{pgfonlayer}
\end{tikzpicture}
\]
         \end{theorem}

         \begin{remark}
                In the final stages of preparing this paper, the author was made aware of related work by Nechitailo \cite{Nechitailo} that also found a $3$-color relation in type $H_3$.
                The approaches between the two papers differ substantially with the main similarity being the use of polynomial interpolation.
                In particular, \emph{loc. cit.} solves for the coefficients by computing inverses of blocks of localization matrices evaluating at a fixed point whereas our approach uses forward-substitution using the $\lambda$-functionals
                developed in this paper.

                Nechitailo's relation involved 43,424 nonzero coefficients, and is not further simplified. The better efficiency of the algorithm in the present work made it quite feasible
                to check all possible $3$-color relations for a fixed choice of light-leaf datum. This allowed us to find a minimal relation with 125 nonzero coefficients, and then further simplify it using thick Soergel calculus.
                
                Finally, the code accompanying \emph{loc. cit.} is specialized to $H_3$ and cannot be used to find double-leaf expansions for more general morphisms or Coxeter groups.
                The code accompanying this paper was designed to be a general-purpose Soergel calculus utility.
                However, our code is implemented in Python which tends to be much slower than a modern performant language like Rust used in Nechitailo's implementation.
                A port of the codebase for this paper to a language like Rust would likely yield even higher efficiency gains.   
         \end{remark}

        \subsection{Implementation}

        The algorithms of the paper are implemented using the Sage computer algebra system. The source code containing both the symbolic and interpolation algorithms is available online at \url{https://github.com/colton5007/DLCalc}.
        
        The code can be used to compute localization data and double-leaf expansions for diagramatically-defined morphisms between Bott--Samelson bimodules. Abstract Coxeter groups and realizations described by Cartan matrices are also allowed with some restrictions on the ground field.
         A thorough README file accompanies the repository and contains all the necessary instructions and examples to both verify the work in this paper and to become more comfortable using the codebase.

        \subsection{Statement on AI Usage}

        The author acknowledges the use of generative AI during the preparation of this manuscript.
        The vast majority of the paper, including the formula for computing localization matrices of double leaves, Theorem \ref{thm:intro}, Algorithm \ref{algo:coeff_algo}, and Theorem \ref{thm:thick_h3} was developed
        solely by the author without AI assistance. 

        The first implementation of Algorithm \ref{algo:coeff_algo} was done using GPT-5.5 Sol and OpenAI Codex.
        Instead of implementing Algorithm \ref{algo:coeff_algo} as prompted, ChatGPT actually implemented a polynomial interpolation variant.
        This was done without any specific instructions from the author to find a more efficient algorithm.
        ChatGPT's implementation was essentially Algorithm \ref{algo:coeff_algo_interpolate} in the special case of $W = H_3$. 
        The author subsequently mathematically formulated this algorithm and extended it to handle general Coxeter groups.
        This formulation and extended implementation did not use further AI assistance other than the author's inspection of ChatGPT's $H_3$ implementation. 
        No other math in this paper was created with the assistance of AI.

        The initial implementation of the algorithms in this paper was AI generated; however, the codebase constructed was not to the authors liking. 
        In particular, it was overly specialized to the $H_3$ example, many functions were written solely with optimization in mind rather than human readability/understanding, and the code had essentially no documentation.
        The author subsequently heavily modified the code provided by the LLM to create a utility that mathematicians can utilize for performing general computations around finding double-leaf expansions of Bott--Samelson bimodule morphisms. 
        During these coding revisions, generative AI was used to assist in debugging and construct unit tests.
        The majority of the present codebase is human written, and all of it has been human verified. More detailed comments on the use of AI in the development of the code can be found in the README.md file accompanying it.

        None of the text in this paper is AI generated. In the preparation of the manuscript, LLMs were only used to find typos and review mathematical accuracy.
        The author takes full responsibility for the content of the paper and the implementations of Algorithms \ref{algo:coeff_algo} and \ref{algo:coeff_algo_interpolate}.
         
        \subsection{Acknowledgements}

        The author would like to thank Ben Elias for his interest, helpful comments, and the suggestion to use thick Soergel calculus to simplify the $3$-color relation. The author was supported by Simons Collaboration Award 994328 ``New Structures in Low-Dimensional Topology.''

       \section{Soergel Bimodules and Diagrammatics}

       \subsection{Coxeter systems}

       Let $(W,S)$ be a Coxeter system and let $e \in W$ denote the identity. Throughout this work, we assume that $S$ is finite.
        To each pair of distinct simple reflections $s,t \in S$, we associate a number $m_{st} \in \{2,3,\ldots\} \cup \{\infty\}$ such that 
        \[\underbrace{sts\ldots}_{m_{st}} = \underbrace{tst\ldots}_{m_{st}} \]
        provided $m_{st} < \infty$. The Coxeter structure of $W$ is determined by the $m_{st}$'s.
        The group $W$ is equipped with the Bruhat order $\leq$ and the length function $\ell : W \to \Z_{\geq 0}$.
        The rank of $W$ is the cardinality of $S$. A reflection in $W$ is an element of the form $wsw^{-1}$ for $w \in W$ and $s \in S$.

        We write $\Exp (W)$ for the set of expressions in $W$. These consist of sequences $\uw = (s_1, \ldots, s_k)$ of simple reflections.
        The length of $\uw$ is $\ell (\uw) \coloneq k$. Note that $\ell (\uw ) \geq \ell (w)$ where $w = s_1 \ldots s_k$, and equality holds if  and only if $\uw$ is a reduced expression.
        We adopt the convention of \cite{EW} and shorten ``reduced expression'' to \emph{rex} with the plural form being \emph{rexes}.

        For a fixed $w \in W$, we can consider the rex graph $\Gamma_w$ of $w$.
        This is a directed graph whose vertices are reduced expressions of $w$. There is an edge between two reduced expressions if they differ by a single application of a braid relation.
        By Matsumoto's theorem \cite{Matsumoto}, 
        the rex graph is connected. Moreover, it is crucial to note that $\Gamma_w$ often has cycles. These cycles can be decomposed into elementary cycles of two types (see \cite[\S 4.2]{EW}).
        First, there are the disjoint cycles which come from performing two braid relations on disjoint parts of an expression in either order. 
        The second are the Zamolodchikov cycles coming from rank 3 finite parabolic subgroups.

        A subexpression $\ue$ of an expression $\uw = (s_1, \ldots, s_k)$ is a tuple $(e_1, \ldots, e_k)$ where $e_i \in \{0,1\}$ for all $i=1,\ldots, k$.
        Given such a subexpression, we can consider the evaluation of $\uw$ at $\ue$ which is the element
        \[\uw^{\ue} \coloneq s_1^{e_1} \ldots s_k^{e_k} \in W.\]
        The \emph{Bruhat stroll} associated to $\ue$ is the sequence $\uw_0^{\ue}, \ldots, \uw_k^{\ue}$ in $W$ defined by
        \[\uw_i^{\ue} = s_1^{e_1} \ldots s_i^{e_i} \in W.\]
        Note our convention is that $\uw_0^{\ue} = e$.
        Let $\Sub (\uw)$ denote the set of all subexpressions of $\uw$. We can equip $\Sub (\uw)$ with a partial order $\leq$ called the \emph{path dominance order}.
        This partial order is defined by
        \[\ue \leq \uf \iff \uw_i^{\ue} \leq \uw_i^{\uf} \text{ for all } i=1,\ldots, k.\]
        For some $x \in W$, we can also consider the subset of subexpressions which evaluate to $x$,
        \[\Sub^x (\uw) \coloneq \{ \ue \in \Sub (\uw) \mid \uw^{\ue} = x\}.\]
        There is a distinguished element in $\Sub (\uw)$ called the \emph{all ones subexpression}, $\mathbf{1} = (1,1,\ldots,1)$.

        Let $\uw = (s_1, \ldots, s_k) \in \Exp (W)$ and $\ue \in \Sub (\uw)$.
       For each $i=1,\ldots, k$, we write $\uw_{\leq i}$ (resp. $\ue_{\leq i}$) for the truncation of $\uw$ (resp. $\ue$) to the first $i$ terms.
       Note $\ue_{\leq i}$ is a subexpression of $\uw_{\leq i}$. Let $w_i = \uw_{\leq i}^{\ue_{\leq i}}$.
       We can define the $i$-th \emph{decoration} of $\ue$ by
       \[\dec_i (\uw, \ue) \coloneq \begin{cases} U0 & \text{if } w_{i-1} s_i > w_{i-1} \text{ and } e_i = 0, \\ U1 & \text{if } w_{i-1} s_i > w_{i-1} \text{ and } e_i = 1, \\ D0 & \text{if } w_{i-1} s_i < w_{i-1} \text{ and } e_i = 0, \\ D1 & \text{if } w_{i-1} s_i < w_{i-1} \text{ and } e_i = 1. \end{cases}\]
       The \emph{defect} of $\ue$ is the integer defined by
       \[d (\uw, \ue) \coloneq \# \{i \mid \dec_i (\uw, \ue) = U0\} - \# \{ i \mid \dec_i (\uw, \ue) = D0\}.\]

        \subsection{Realizations}

        Let $\k$ be a noetherian domain.
        Let $A = \Z [x,y]$ be the integral polynomial ring in two variables. Given $a,b \in \k$ and $p(x,y) \in A$, we write $p(a,b)$ for the element in $\k$ obtained by substituting $x=a$ and $y=b$.
        We define a collection of elements in $A$ called the \emph{two-colored quantum numbers}  as follows.
        First, set $[0]_x = 0$, $[0]_y = 0$, $[1]_x = 1$, $[1]_y = 1$, $[2]_x = x$, and $[2]_y = y$.
        Define the remaining quantum numbers recursively by the formulas
        \[ [n+1]_x = [2]_x [n]_y - [n-1]_x \qquad\text{and}\qquad [n+1]_y = [2]_y [n]_x - [n-1]_y.\]
        
        Let $n \in \Z_{\geq 0}$ and $0 \leq k \leq n$. Define the \emph{two-colored quantum binomial coefficients} as
        \[\qbinom{n}{k}_x \coloneq \frac{[n]_x [n-1]_x \ldots [n-k+1]_x}{[1]_x [2]_x \ldots [k]_x} \in A,\]
        and likewise $\qbinom{n}{k}_y$ can be defined by interchanging $x$ with $y$ in the above definition.

        \begin{definition}
                A \emph{pre-realization} of a Coxeter group $(W,S)$ over $\k$ consists of a free, finite rank $\k$-module $\fr{h}$ along with subsets
                \[\{\alpha_s^{\vee} : s \in S\} \subset \fr{h} \qquad\text{and}\qquad \{\alpha_s : s\in S\} \subset \fr{h}^*,\]
                called the simple coroots and simple roots respectively, such that
                \begin{enumerate}
                        \item $\langle \alpha_s^{\vee}, \alpha_s \rangle = 2$ for all $s \in S$;
                        \item the assignment
                                \[s (v) = v - \langle v , \alpha_s \rangle \alpha_s^{\vee}\]
                                for all $s \in S$ and $v \in \fr{h}$ defines a representation of $W$ on $\fr{h}$;
                        \item $\alpha_s : \fr{h} \to \k$ and $\alpha_s^{\vee}: \fr{h}^* \to \k$ are surjective for all $s \in S$.
                \end{enumerate}

                To every pre-realization $\fr{h}$, there is a Cartan matrix $[a_{st}]_{s,t \in S}$ where $a_{st} \coloneq \langle \alpha_s^{\vee}, \alpha_t \rangle$.
                We will further call $\fr{h}$ a \emph{realization} if for all distinct $s,t \in S$ such that $st$ has order $m_{st} < \infty$, we have
                \begin{equation}\label{eq:qbinom_vanishing}
                         \qbinom{m_{st}}{k}_x (-a_{st}, -a_{ts})  = \qbinom{m_{st}}{k}_y  (-a_{st}, -a_{ts})  = 0
                \end{equation}
                for all integers $1 \leq k \leq m_{st} - 1$.
        \end{definition}

        \begin{definition}
                Let $(W,S)$ be a Coxeter group with realization $\fr{h}$ over $\k$. We say that $\fr{h}$ is \emph{balanced} if 
                \begin{equation}\label{eq:balanced}
                        [m_{st} - 1]_x (-a_{st}, -a_{ts}) = 1 = [m_{st} - 1]_y (-a_{st}, -a_{ts})
                \end{equation}
                for all distinct $s,t \in S$ with $m_{st} < \infty$.
        \end{definition}

        Throughout this paper, we will assume that all of our realizations are balanced.
        This assumption is not strictly necessary (cf. \cite{EW, EWloc}); however, it greatly simplifies the quantum combinatorics and diagrammatics.

        It will be useful to reinterpret \eqref{eq:qbinom_vanishing} and \eqref{eq:balanced} in terms of the vanishing of certain minimal polynomials.
        For $m > 2$, we set
        \begin{equation}\label{eq:min_polys}
                \Psi_m (z) \coloneq \minpoly_{\Q} \left(4\cos^2 \frac{\pi}{m}\right) \qquad\text{and}\qquad \Psi_m' (z) \coloneq \minpoly_{\Q} \left( 2 \cos \frac{\pi}{m}\right).
        \end{equation}

        \begin{theorem}[{\cite{Hazi1, Hazi2}}]\label{thm:abe_conditions}
                Let $s,t \in S$ be distinct simple reflections with $m = m_{st} < \infty$ and $m > 2$.
                \begin{enumerate}
                        \item If $m$ is odd, then \eqref{eq:qbinom_vanishing} and \eqref{eq:balanced} hold if and only if $a_{st} = a_{ts}$ and $\Psi_m' (-a_{st}) = 0$.
                        \item If $m$ is even, then \eqref{eq:qbinom_vanishing} and \eqref{eq:balanced} hold if and only if $\Psi_m (a_{st} a_{ts}) = 0$.
                \end{enumerate}
        \end{theorem}
        \begin{proof}
                By \cite[Theorem 3.8]{Hazi1}, one has that \eqref{eq:qbinom_vanishing} holds if and only if $\Psi_m (a_{st} a_{ts}) = 0$.
                This is also stated explicitly in \cite[(15)]{Hazi1}. 
                If $m$ is even, the balancedness equation \eqref{eq:balanced} is automatic under \eqref{eq:qbinom_vanishing} as observed in \cite[Lemma 6.25]{EWloc}.
                As a result, we only need to consider the case of $m$ being odd.

                Assume that $m$ is odd, and that both \eqref{eq:qbinom_vanishing} and \eqref{eq:balanced} hold. 
                By \cite[Theorem 2.3]{Hazi2}, the element $a_{st} a_{ts}$ has a square root $a$ in $\k^{\times}$.
                Moreover, this element satisfies $\Psi_m' (a) = 0$ and 
                \[[m-1]_x (-a_{st}, -a_{ts}) = \frac{-a_{st}}{a} \qquad\text{and}\qquad [m-1]_y (-a_{st}, -a_{ts}) = \frac{-a_{ts}}{a}.\]
                Balancedness \eqref{eq:balanced} then forces $a_{st} = a_{ts} = -a$, and therefore $\Psi_m' (-a_{st}) = 0$.

                Conversely, suppose that $a_{st} = a_{ts}$ and $\Psi_m' (-a_{st}) = 0$. 
                Since $2 \cos (\pi/m)$ is a root of $\Psi_m (z^2)$, its monic minimal polynomial $\Psi_m' (z)$ divides $\Psi_m (z^2)$ in $\Z [z]$.
                As a result, $\Psi_m (a_{st} a_{ts}) = \Psi_m (a_{st}^2) = 0$. By\cite[Theorem 3.8]{Hazi1}, the quantum-binomial vanishing condition holds. 
                Now apply \cite[Corollary 2.5]{Hazi2} to conclude balancedness.
        \end{proof}

        \begin{example}\label{ex:golden_realization}
                The most prominent realization we will encounter in this paper is the \emph{golden realization}.
                Let $(W,S)$ be a Coxeter system for which $m_{st} \in \{2,3,5,\infty\}$. Set $\k = \Q (\phi)$ where $\phi = (1 + \sqrt{5})/2$ denotes the golden ratio.
                Note that $\k$ is isomorphic to the field of fractions for $\Z [\tau]/(\tau^2 - \tau - 1)$ via $\tau \mapsto \phi$.
                We set $\fr{h} = \bigoplus_{s \in S} \k \alpha_s^{\vee}$ and define $\alpha_s \in \fr{h}^*$ via
                \[\langle \alpha_s^{\vee} , \alpha_t \rangle = \begin{cases} 2 & \text{if }s=t, \\ 0 & \text{if }m_{st} = 2, \\ -1 & \text{if }m_{st} = 3, \\ -\phi & \text{if } m_{st} = 5, \\ -2 & \text{if } m_{st} = \infty.\end{cases}\] 
                It is easy to check from its definition that $\fr{h}$ is always a balanced realization.
        \end{example}

        \begin{lemma}[{\cite[Lemma 1.2]{Pat25}}]\label{lem:root_about_refln}
                Let $(W,S)$ be a Coxeter group with a balanced realization $\fr{h}$. Let $r$ be a reflection in $W$.
                Assume that $r = xsx^{-1} = yty^{-1}$ for some $x,y \in W$ and $s,t \in S$ such that $xs > x$ and $yt > y$. Then $x \alpha_s = y \alpha_t$.
        \end{lemma}

        In light of Lemma~\ref{lem:root_about_refln}, we can associate a distinguished root to each reflection $r = wsw^{-1}$ in $W$ where $ws > w$,
        \[\alpha_r = w\alpha_s. \]
        We define the set of roots by
        \[\Phi \coloneq \{ w \alpha_s : w \in W, s \in S\}.\]

       \subsection{The category \texorpdfstring{$\mathsf{C}$}{C}}

       For the remainder of the section, we will review the enhanced theory of Bott--Samelson bimodules due to Abe \cite{Abe19, Abe21}.
       
       Set $R = \Sym (\fr{h}^*)$, assigning degree 2 to each element in $\fr{h}^*$. For an even integer $d$, we write $R^d$ for the collection of homogeneous polynomials in $R$ of degree $d/2$.
       This is the degree $d$ component of $R$.  Consider the localized ring
       \[Q \coloneq R \left[ \frac{1}{\alpha} : \alpha \in \Phi \right].\]
       The roots $\alpha$ are well-defined by Lemma~\ref{lem:root_about_refln}.
       The ring $Q$ naturally inherits a $\Z$-grading from $R$. Moreover, the $W$-action on $\fr{h}^*$ induces actions on $R$ and $Q$ by graded algebra automorphisms.

        Let $\ue$ be a subexpression of $\uw = (s_1, \ldots, s_k) \in \Exp (W)$. 
        For each $i=1,\ldots, k$, define
        \[\beta_i (\uw, \ue) = \uw_{i-1}^{\ue} (\alpha_{s_i}) \in \fr{h}^*, \]
        where we set $\uw_{0}^{\ue} = e$ by convention. Now, we can define
        \[\zeta_{\uw} (\ue) = \prod_{i=1}^k \beta_i (\uw, \ue) \in R,\]
        \[\gamma_{\uw} (\ue) = \left( \prod_{\substack{i \\ \dec_i (\uw, \ue) = U0}} \beta_i (\uw, \ue) \right)  \left( \prod_{\substack{i \\ \dec_i (\uw, \ue) = D1}} (-\beta_i (\uw, \ue)) \right) \in R.\]
        We can now consider a special case of $\zeta$. First, let $s,t$ be distinct simple reflections such that $m = m_{st} < \infty$.
 Define length $m$ alternating expressions in $s$ and $t$,
       \[ {}_s \underline{m} = \underbrace{(s,t,s,\ldots)}_{m \text{ terms}} \qquad\text{and}\qquad {}_t \underline{m} = \underbrace{(t,s,t,\ldots)}_{m \text{ terms}}.\]
        We define
        \begin{equation}\label{eq:pi_st}
                \pi_{s,t} = \zeta_{{}_s \underline{m}} (\mathbf{1}) = \zeta_{{}_t \underline{m}} (\mathbf{1}).
        \end{equation}

       We define a category $\sfC (\fr{h}, W)$ whose objects are triples
       \[(M, (M_Q^w)_{w \in W}, \xi_M)\]
       such that 
       \begin{itemize}
        \item $M$ is a graded $R$-bimodule;
        \item each $M_Q^w$ is a graded $(R,Q)$-bimodule such that $m \cdot r = w (r) \cdot m$ for all $m \in M_Q^w$ and $r \in R$;
        \item $M_Q^w$ is nonzero for finitely many $w \in W$;
        \item $\xi_M : M_Q \coloneq M \otimes_R Q \to \bigoplus_{w \in W} M_Q^w$ is an isomorphism of graded $(R,Q)$-bimodules.
       \end{itemize}
       A morphism in $\sfC (\fr{h}, W)$ from $(M, (M_Q^w)_{w \in W}, \xi_M)$ to $(N, (N_Q^w)_{w \in W}, \xi_N)$ is a morphism of $R$-bimodules\footnote{We allow for morphisms which are not homogeneous of degree $0$.} $\varphi : M \to N$ such that
       \[\xi_N \circ (\varphi \otimes_R Q) \circ \xi_M^{-1} (M_Q^w) \subset N_Q^w\]
       for all $w \in W$. This category has a natural monoidal structure induced by $\otimes_R$ with unit $R$.
       It also admits a shift-of-grading functor $(1)$ defined by $(M(1))^n = M^{n+1}$ and $(M_Q^w (1))^n = (M_Q^w)^{n+1}$.
       To simplify notation, we will often just write $M$ to refer to the whole triple $(M, (M_Q^w)_{w \in W}, \xi_M)$ and $\varphi_Q \coloneq \xi_N \circ (\varphi \otimes_R Q) \circ \xi_M^{-1}$.

       Given $M, N \in \sfC (\fr{h}, W)$ and $d \in \Z$, we will write
       \[\Hom_{\sfC (\fr{h}, W)}^{d} (M,N)\]
       for the $\k$-module of homogeneous morphisms of degree $d$. 

       Let $M \in \sfC (\fr{h}, W)$ such that $M$ is finitely generated as an $R$-bimodule and flat as a right $R$-module.
       Consider the natural map $\pi_M^w : M \to M \otimes_R Q \twoheadrightarrow M_Q^w$. We define $M^w = \pi_M^w (M)$.
       By definition, if $m \in M^w$, then $m \cdot r = w(r) \cdot m$ for all $r \in R$.
       We then have that 
       \begin{equation}\label{eq:Mw_and_embedding}
                M \hookrightarrow \bigoplus_{w \in W} M^w.
       \end{equation} 

       \begin{lemma}\label{lem:detecting_bimodularity}
                Let $M, N \in \sfC (\fr{h}, W)$ such that $M$ and $N$ are finitely generated as $R$-bimodules and flat as right $R$-modules.
                Let $\varphi : M \to N$ be a left $R$-module map. Assume that there are left $R$-module maps $\varphi^w : M^w \to N^w$ satisfying
                \[\pi_N^w \circ \varphi = \varphi^w \circ \pi_M^w\]
                for all $w \in W$.
                Then $\varphi$ is a morphism in $\sfC (\fr{h}, W)$.
       \end{lemma}
       \begin{proof}
                We can first show that $\varphi$ is a morphism of $R$-bimodules.
                Let $m \in M$ and $r \in R$. For all $w \in W$, one has that
                \[\pi_N^w (\varphi (mr) - \varphi (m) r) = \varphi^w (w(r) \pi_M^w (m)) - w(r) \varphi^w (\pi_M^w (m)) = 0.\]
                By \eqref{eq:Mw_and_embedding}, we deduce that $\varphi (mr) = \varphi (m) r$, and so $\varphi$ is a morphism of $R$-bimodules.

                It remains to check that $\varphi$ preserves the localization data. By \cite[Lemma 2.4]{Abe19}, $M^w \otimes_R Q = M_Q^w$ and $N^w \otimes_R Q = N_Q^w$.
                The twisted right actions show that each $\varphi^w$ is also right $R$-linear. Localizing the identities $\pi_N^w \varphi = \varphi^w \pi_M^w$ then shows that $\varphi_Q$ acts componentwise by $\varphi^w \otimes_R Q$.
                Therefore, $\varphi_Q (M_Q^w) \subseteq N_Q^w$ as desired.
       \end{proof}

       \subsection{Abe--Bott--Samelson bimodules}\label{ss:BS}

       Let $s \in S$, and write $R^s$ for the subalgebra of $s$-invariant elements in $R$.
       By Demazure surjectivity, we can choose some $\delta_s \in \fr{h}^*$ such that $\langle \delta_s, \alpha_s^{\vee} \rangle = 1$.
       The graded $R$-bimodule $B_s \coloneq R \otimes_{R^s} R(1)$ can be upgraded to an object in $\sfC (\fr{h}, W)$. In particular, if $w \notin \{e ,s\}$, we set $(B_s)_Q^w = 0$, and otherwise,
       \begin{equation}\label{eq:loc_of_B_s}
        (B_s)_Q^e = Q (\delta_s \otimes 1 - 1 \otimes s(\delta_s)) \qquad\text{and}\qquad (B_s)_Q^s = Q (\delta_s \otimes 1 - 1 \otimes \delta_s).
       \end{equation}

       \begin{definition}
                The category of \emph{Abe--Bott--Samelson bimodules} is the full subcategory $\BS (\fr{h}, W)$ of $\sfC (\fr{h}, W)$ whose objects are the bimodules $B_{\uw} (n)$ for expressions $\uw$ and integers $n$.
                Equivalently, it is generated by $R$ and the $B_s$ for $s \in S$ under $\otimes_R$ and grading shifts. 
       \end{definition}

       Given an expression $\uw = (s_1, \ldots, s_k)$, we can define a graded $R$-bimodule
       \[B_{\uw} \coloneq B_{s_1} \otimes_R B_{s_2} \otimes_R \ldots \otimes_R B_{s_k},\]
       which we view as an object of $\BS (\fr{h}, W)$.

       Let $M \in \sfC (\fr{h}, W)$ such that $M$ is finitely generated as an $R$-bimodule and flat as a right $R$-module. Define
       \[\DD (M) \coloneq \Hom_{\textnormal{mod-}R} (M, R)\qquad\text{and}\qquad \DD (M)_Q^w \coloneq \Hom_{\textnormal{mod-}Q} (M_Q^w, Q). \]
       The $R$-bimodule structure on $\DD (M)$ is given by $(rfs) (m) = f (rm)s$ for $r,s \in R$, $m \in M$, and $f \in \DD (M)$.
       One can check that $\DD (M)$ is an object in $\sfC (\fr{h}, W)$. For each $s \in S$, one can define an isomorphism
       \[B_s=R \otimes_{R^s} R(1) \to \DD (B_s) = \Hom_{\textnormal{mod-}R} (R \otimes_{R^s} R (1), R) , \qquad a \otimes b \mapsto (c \otimes d \mapsto \partial_s (ac) bd).  \] 
       By \cite[Lemma 2.20]{Abe19}, this extends to a canonical isomorphism $\DD (B_{\uw}) \cong B_{\uw}$ for any expression $\uw \in \Exp (W)$ which is compatible with $\otimes_R$.
       More precisely, if we write $\uw = (s_1, \ldots, s_k)$, there is a perfect right $R$-bilinear pairing on $B_{\uw}$ given by
       \begin{equation}\label{eq:pairing}
        \langle f_0 \otimes \ldots \otimes f_k, g_0 \otimes \ldots \otimes g_k \rangle_{\uw} = \partial_{s_k} ( \ldots \partial_{s_1} (f_0 g_0) f_1 g_1 ) f_k g_k,
       \end{equation}
       that defines the desired isomorphism. 
       As a result, $\DD$ restricts to a monoidal functor
       \[\DD : \BS (\fr{h}, W)^{\op} \to \BS (\fr{h}, W).\]
       The argument from \cite[Lemma 2.20]{Abe19} shows that $\DD$ is an involution on Abe--Bott--Samelson bimodules, i.e., there is a natural isomorphism $\DD \circ \DD \cong \id$.

       For a subexpression $\ue = (e_1, \ldots, e_k)$ of $\uw = (s_1, \ldots, s_k)$, define 
       \[t_{\uw, \ue} = 1 \otimes \delta_{s_1}^{e_1} \otimes \ldots \otimes \delta_{s_k}^{e_k} \in R \otimes_{R^{s_1}} R \otimes_{R^{s_2}} \ldots \otimes_{R^{s_k}} R (k) = B_{\uw}.\]
        We also define elements 
        \[ b_{\uw, \ue} = b_1 \otimes \ldots \otimes b_k \in B_{s_1} \otimes_R B_{s_2} \otimes_R \ldots \otimes_R B_{s_k} = B_{\uw}\] 
        where $b_i \in B_{s_i}$ is defined by
        \[ b_i = \begin{cases} 1 \otimes 1 & \text{if } \dec_i (\uw, \ue) = U0 \text{ or } U1, \\ \delta_{s_i} \otimes 1 & \text{if } \dec_i (\uw, \ue) = D0 \text{ or } D1. \end{cases} \]
       For $x \in W$, we set $b_{\uw, \ue}^x = \pi_{B_{\uw}}^x (b_{\uw, \ue}) \in B_{\uw}^x$. 

       \begin{proposition}\label{prop:bases_of_BS}
        Let $\uw \in \Exp (W)$ and $x \in W$.
        \begin{enumerate}
                \item The set $\{ t_{\uw, \ue} \}_{\ue \in \Sub (\uw)}$ forms a left $R$-module basis of $B_{\uw}$.
                \item The set $\{ b_{\uw, \ue}^x\}_{\ue \in \Sub^x (\uw)}$ forms a left $R$-module basis of $B_{\uw}^x$.
        \end{enumerate}
       \end{proposition}
       \begin{proof}
        Statement (1) follows by repeatedly using the $R^s$-module decomposition $R = R^s \oplus \delta_s R^s$ of \cite[Lemma 2.10]{Abe19}.
        Statement (2) is \cite[Theorem 3.12]{Abe19}.
       \end{proof}

       Let $\uw = (s_1, \ldots, s_k) \in \Exp (W)$ and $\ue = (e_1, \ldots, e_k) \in \Sub (\uw)$.
       We can define the $\ue$\emph{-component} of $B_{\uw} \otimes_R Q$, as the direct summand corresponding to the $Q$-bimodule
       \[(B_{\uw})_Q^{\ue} \coloneq Q_{s_1^{e_1}} \otimes_Q Q_{s_2^{e_2}} \otimes_Q \ldots \otimes_Q Q_{s_k^{e_k}} (-\ell (\uw)) \cong Q_{\uw^{\ue}} (-\ell (\uw))\]
       under the isomorphism $B_s \otimes_R Q \cong Q_e (-1) \oplus Q_s (-1)$ of \eqref{eq:loc_of_B_s}. We write $q_{\uw}^{\ue} \in (B_{\uw})_Q^{\ue}$ for the element corresponding to $1 \in Q_{\uw^{\ue}} (-\ell(\uw))$.
       Explicitly, this element is given by the tensor product
       \begin{equation}\label{eq:qwe}
        q_{\uw}^{\ue} = (\delta_{s_1} \otimes 1 - 1 \otimes s_1^{1-e_1} (\delta_{s_1})) \otimes \ldots \otimes  (\delta_{s_k} \otimes 1 - 1 \otimes s_k^{1-e_k} (\delta_{s_k})).
       \end{equation}
       
       \begin{lemma}[{\cite[\S 3.4]{Abe21}}]\label{lem:decomp_into_loc_pieces}
        Let $\uw \in \Exp (W)$. There is a decomposition of graded $Q$-bimodules,
        \[B_{\uw} \otimes_R Q \cong \bigoplus_{\ue \in \Sub (\uw)} B_{\uw, Q}^{\ue} \cong \bigoplus_{\ue \in \Sub (\uw)} Q_{\uw^{\ue}} (-\ell (\uw)).\]
       \end{lemma}

       \subsection{Generating morphisms}\label{ss:gen_mors}

       Let $s \in S$. Define an $R^s$-module map $\partial_s : R \to R^s$ by $\partial_s (f) = (f - s(f))/\alpha_s$.
       We use the identification $B_s \otimes_R B_s = R \otimes_{R^s} R \otimes_{R^s} R (2)$. We then define the \emph{$1$-color morphisms} of $R$-graded bimodules
       \begin{align*}
                \nu^s &: B_s \to B_s \otimes_R B_s , & f \otimes g &\mapsto f \otimes 1 \otimes g, \\
                \epsilon^s &: B_s \to R, & f \otimes g &\mapsto fg, \\
                \mu^s &: B_s \otimes_R B_s \to B_s , & f \otimes g \otimes h &\mapsto \partial_s (g) f \otimes h, \\
                \eta^s &: R \to B_s , & 1 &\mapsto \delta_s \otimes 1 - 1 \otimes s (\delta_s), \\
                \cup^s &: R \to B_s \otimes_R B_s, & 1 &\mapsto \delta_s \otimes 1 \otimes 1 - 1\otimes 1 \otimes s(\delta_s), \\
                \cap^s &: B_s \otimes_R B_s \to R, & f\otimes g \otimes h &\mapsto \partial_s (g) fh.
       \end{align*}
       The maps $\mu^s$ and $\nu^s$ are of degree $-1$, the maps $\epsilon^s$ and $\eta^s$ are of degree $1$, and the maps $\cup^s$ and $\cap^s$ are of degree $0$.
       Note that $\cup^s = \nu^s \circ \eta^s$ and $\cap^s = \epsilon^s \circ \mu^s$. It is straightforward to check that all of the above maps define morphisms in $\BS (\fr{h}, W)$.
       Moreover, an easy but tedious calculation shows that
       \[\DD (\epsilon^s) = \eta^s \qquad\text{and}\qquad \DD (\nu^s) = \mu^s.\]

       For each $\uw = (s_1, \ldots, s_k) \in \Exp (W)$, we define the \emph{$1$-tensor},
       \[u_{\uw} = 1 \otimes 1 \otimes \ldots \otimes 1 \in R \otimes_{R^{s_1}} R \otimes_{R^{s_2}} \ldots \otimes_{R^{s_k}} R (k) = B_{\uw}. \]

       Let $s,t \in S$ be distinct simple reflections such that $m = m_{st} < \infty$. 
       By \cite[Theorem 3.9]{Abe21}, there exists a homogeneous degree $0$ morphism
       \[ \beta_{st} : B_{{}_s \underline{m}} \to B_{{}_t \underline{m}} \]
       such that $\beta_{st} (u_{{}_s \underline{m}}) = u_{{}_t \underline{m}}$.
       Moreover, this morphism is unique since the degree-$0$ morphism space is a free $\k$-module of rank 1 (cf. \cite[Theorem 4.6]{Abe19}).
       This along with the trace pairings above forces that $\DD (\beta_{st}) = \beta_{ts}$. 

       More generally, if $\ux$ and $\uy$ are reduced expressions for some $w \in W$, then each path $p$ in the rex graph for $w$ gives rise to a morphism $\beta_p : B_{\ux} \to B_{\uy}$ composed of the various $\beta_{st}$-maps.
       It is clear from its construction that $\beta_p (u_{\ux}) = u_{\uy}$. Of course, this morphism depends on the choice of the path $p$.

       \subsection{Light leaves}\label{ss:light_leaves}

        We will review the construction of light leaves for Abe--Bott--Samelson bimodules introduced by Libedinsky \cite{Lib} and further studied in \cite{Abe19}.
        Light leaves are not canonical. They depend on the choice of certain rigidification data.
       
       \begin{definition}
                An $\LL$-datum is a triple $(\rex, \{\rex_s\}_{s \in S}, \{p_{\ux, \uy}\}_{\ux, \uy})$ consisting of the following:
                \begin{enumerate}
                        \item a map $\rex : W \to \Exp (W)$ such that for each $x \in W$, $\rex (x)$ is a reduced expression for $x$;
                        \item for each $s \in S$, a map $\rex_s : \mcR_s \to \Exp (W)$ where $\mcR_s = \{w \in W \mid ws < w\}$ and $\rex_s (x)$ is a reduced expression for $x$ ending in $s$;
                        \item for any two rexes $\ux$ and $\uy$ of $x$, a path $p_{\ux, \uy}$ from $\ux$ to $\uy$ in the rex graph of $x$.
                \end{enumerate}
       \end{definition}

       For the remainder of the section, we will fix an $\LL$-datum. 
       For reduced expressions $\ux$ and $\uy$ of $x \in W$, we will write $\beta_{\ux, \uy} : B_{\ux} \to B_{\uy}$ for the morphism in $\BS (\fr{h}, W)$ determined by $p_{\ux, \uy}$ as in \S\ref{ss:gen_mors}.

        Let $\ux=(s_1,\ldots,s_k)$ and $\ue\in\Sub(\ux)$.
        Set $w_0=e$. For each $i=1,\ldots, k$, we write $w_i=\ux_{\leq i}^{\ue_{\leq i}}$.
        We set $\uw = \rex (w_k)$ for the rex of $w_k$ fixed by our $\LL$-datum. 
        We construct a reduced expression $\underline r_i$ for $w_i$ at each stage, beginning with
        $\underline r_0=\emptyset$. The intermediate expression $\underline r_i$ need not equal $\rex(w_i)$.
        Put $\LL^0_{\ux,\ue}=\id_{B_{\ux}}$. For $1\leq i\leq k$, set $s=s_i$ and define $\phi_i : B_{\underline r_{i-1}}\otimes_R B_s \to B_{\underline r_i}$ by the following rules.
        \begin{enumerate}
        \item[$U0$:]  Set $\underline r_i=\underline r_{i-1}$ and
        \[\begin{tikzcd}
         \phi_i : B_{\underline{r}_{i-1}} \otimes_R B_{s} \arrow[rr, "\id \otimes \epsilon^{s}"] &  & B_{\underline{r}_{i-1}} = B_{\underline{r}_{i}}.
         \end{tikzcd}\]
        \item[$U1$:]  Set $\underline{r}_i=\underline{r}_{i-1}s$ and
        \[\begin{tikzcd}
         \phi_i : B_{\underline{r}_{i-1}} \otimes_R B_{s} = B_{\underline{r}_i} \arrow[rr, "\id_{B_{\underline{r}_{i}}}"] &  &  B_{\underline{r}_{i}}.
         \end{tikzcd}\]
        \item[$D0$:] Let $\underline{t} =\rex_s(w_{i-1})=\underline{q} s$. We can write $\underline{t} = \underline{q} s$ for some rex $\underline{q}$.
        Set $\underline{r}_i=\underline t$ and
        \[\begin{tikzcd}
         \phi_i : B_{\underline{r}_{i-1}} \otimes_R B_{s} \arrow[rrr, "\beta_{\underline{r}_{i-1}, \underline{t}} \otimes \id_{B_{s}}"] & & & B_{\underline{q}} \otimes_R B_s \otimes_R B_s \arrow[rr, "\id_{B_{\underline{q}}} \otimes \mu^s"] & & B_{\underline{q}} \otimes_R B_s = B_{\underline{r}_i}.
         \end{tikzcd}\]
        \item[$D1$:] Let $\underline{t} = \rex_s (w_{i-1})$. We can write $\underline{t} = \underline{q} s$ for some rex $\underline{q}$.
        Set $\underline r_i=\underline q$ and
         \[\begin{tikzcd}
         \phi_i : B_{\underline{r}_{i-1}} \otimes_R B_{s} \arrow[rrr, "\beta_{\underline{r}_{i-1}, \underline{t}} \otimes \id_{B_{s}}"] & & & B_{\underline{q}} \otimes_R B_s \otimes_R B_s \arrow[rr, "\id_{B_{\underline{q}}} \otimes \cap^s"] & & B_{\underline{q}} = B_{\underline{r}_i}.
         \end{tikzcd}\]
        \end{enumerate}
        Define
        \[ \LL^i_{\ux,\ue} =(\phi_i\otimes\id_{B_{(s_{i+1},\ldots,s_k)}}) \circ\LL^{i-1}_{\ux,\ue}: B_{\ux}\longrightarrow B_{\underline r_i s_{i+1}\cdots s_k}. \]
        Finally, the light leaf corresponding to $\ue$ is defined by
        \[ \LL_{\ux,\ue} =\beta_{\underline r_k,\uw}\circ\LL^k_{\ux,\ue}: B_{\ux}\longrightarrow B_{\uw}. \]
        This morphism has degree $d(\ux,\ue)$.

        \begin{remark} The construction of light leaves differs slightly from that of \cite{Abe19}. In particular, usually a rex move is applied at the end of each $\phi_i$ so that the target is $B_{\rex (w_i)}$. These rex moves can safely be removed as explained in \cite[Remark 10.34]{EMTW}. \end{remark}

       \subsection{Double leaves}

       Let $\ux \in \Exp (W)$ and $\ue \in \Sub (\ux)$, and set $\uw = \rex (\ux^{\ue})$. The \emph{reversed light leaf} is the map
       \[\overline{\LL}_{\ux, \ue} \coloneq \DD (\LL_{\ux, \ue}) : B_{\uw} \to B_{\ux}.\]
       This morphism is also of degree $d (\ux, \ue)$.
       
       \begin{definition}
        Let $\ux, \uy \in \Exp (W)$. Let $w \in W$, $\ue \in \Sub^w (\ux)$, and $\uf \in \Sub^w (\uy)$. A \emph{double leaf} is a degree $d (\ux, \ue) + d (\uy, \uf)$ map given by a composition of the form
        \[\dLL_{\ue, \uf}^w \coloneq \overline{\LL}_{\uy, \uf} \circ \LL_{\ux, \ue} : B_{\ux} \to B_{\uy}.\]
       \end{definition}

       \begin{theorem}[{\cite[Theorem 5.5]{Abe19}}]\label{thm:dll_basis}
                The set of double leaves $\{ \dLL_{\ue, \uf}^w \mid \ux^{\ue} = w = \uy^{\uf} \}$ is a left graded $R$-module basis of $\Hom_{\BS (\fr{h}, W)} (B_{\ux}, B_{\uy})$.
       \end{theorem}

       \begin{corollary}\label{cor:deg_0_mor_and_basis}
        Let $\varphi : B_{\ux} \to B_{\uy}$ be a degree $d$ morphism. There is a unique expansion
        \[\varphi = \sum_{\ux^{\ue} = w = \uy^{\uf}} c_{\ue, \uf}^w \dLL_{\ue, \uf}^w,\]
        for some coefficients $c_{\ue, \uf}^w \in R$ which are either $0$ or of degree $d -d (\ux, \ue) - d (\uy, \uf)$. All other double-leaf coefficients are zero.
       \end{corollary}

       The rank of the space of homogeneous polynomials of a given degree is given by certain binomial coefficients. Let $I$ be an indexing set for 
       the set of double leaves from $B_{\ux}$ to $B_{\uy}$. For each $i \in I$, we set $d_i$ to be the degree of the associated double leaf. 
       For $d \in \Z$, set
       \[I_d = \{ i \in I \mid d - d_i \geq 0 \text{ and } d - d_i \text{ even} \}.\]
       For $k \in \Z_{\geq 0}$ and $d \in \Z$, we define
       \[I_d^k = \{ i \in I \mid d_i = d - 2k \}.\]
       Note that $I_d = \bigcup_{k \geq 0} I_d^k$.
       As a result of Corollary~\ref{cor:deg_0_mor_and_basis}, we see that
       \begin{align}
        \rk_{\k} \Hom^d (B_{\ux}, B_{\uy}) &= \sum_{i \in I_d } \binom{(d-d_i)/2 + \rk \fr{h} - 1}{\rk \fr{h} -1} \notag \\
        &= \sum_{k \geq 0} (\# I_d^k) \cdot \binom{k + \rk \fr{h} - 1}{\rk \fr{h} -1}. \label{eq:rk_of_homog_homs}
       \end{align}
       
       \begin{example}\label{ex:H3_hom_calc}
                Let $W = H_3$ and $\fr{h}$ be the realization from Example~\ref{ex:golden_realization}. Let
                \[\ux = stsutsutustsutu \qquad\text{and}\qquad \uy = utustsutustusts.\]
                These are reduced expressions for the longest element of $H_3$. Then
                \[
                    \rk_{\k} \Hom^0 (B_{\ux}, B_{\uy}) = \binom{5}{3} + 1031 \binom{4}{2} + 19732\cdot 3 + 168263 = 233655.
                \]
       \end{example}

       \subsection{Functorality of realizations}

       In this subsection, we will work with multiple realizations of the same Coxeter group.
       As a result, it is beneficial to introduce some extra notation to make the choice of realization more explicit.
       For example, if $\fr{h}$ is a realization of $W$, we will sometimes write
       \[R_{\fr{h}} = \Sym (\fr{h}^*) \qquad\text{and}\qquad Q_{\fr{h}} = R_{\fr{h}} \left[ \frac{1}{w (\alpha_s)} : w \in W, s\in S \right].\]
       Likewise, if $\uw$ is an expression, then we write
       \[ B_{\uw}^{\fr{h}} = B_{\uw} \in \BS (\fr{h}, W).\]
       We also adorn the Demazure operators $\partial_s$ by a superscript, such as $\partial_s^{\fr{h}}$.

       \begin{definition}
        Let $(\fr{h}, \{\alpha_s\}_{s \in S}, \{\alpha_s^{\vee}\}_{s \in S})$ be a realization of $W$ over $\k$ and $(\fr{h}', \{\alpha_s'\}_{s \in S}, \{(\alpha_s')^{\vee}\}_{s \in S})$ be a realization of $W$ over $\k'$. 
        A \emph{morphism of realizations} $\fr{h} \to \fr{h}'$ consists of a pair $(\sigma, f)$ where $\sigma : \k \to \k'$ is a ring homomorphism and $f : \fr{h}^* \to (\fr{h}')^*$ is a $\k$-linear map via $\sigma$ 
        such that for all $s \in S$  and $v \in \fr{h}^*$, 
        \[f ( \alpha_s ) = \alpha_s' \qquad\text{and}\qquad \langle (\alpha_s')^{\vee}, f (v) \rangle = \sigma (\langle \alpha_s^{\vee}, v \rangle).\]
       \end{definition}

       \begin{remark}
                It would be more natural to view a realization as the space $\fr{h}^*$ rather than $\fr{h}$ in the above definition so that a morphism of realizations
                arises via a linear map on the ambient spaces rather than their duals.
                This convention on realizations is used, for example, in \cite{Abe19}. 
       \end{remark}

       One can easily check from the definition that if $(\sigma, f) : \fr{h} \to \fr{h}'$ is a morphism of realizations, then $f$ is $W$-equivariant.
       As a result, there are induced maps 
       \[f : R_{\fr{h}} \to R_{\fr{h}'} \qquad \text{and}\qquad f_Q : Q_{\fr{h}} \to Q_{\fr{h}'}.\]
       These maps are graded, $W$-equivariant, and take the simple roots to simple roots. We then define a functor
       \[\grmod{R_{\fr{h}}} \longrightarrow \grmod{R_{\fr{h}'}}, \qquad M \mapsto R_{\fr{h}'} \otimes_{R_{\fr{h}}} M.\]
       Our goal is to upgrade this to a functor
       \begin{equation}\label{eq:functorial_BS}
        \BS_f : \BS (\fr{h}, W) \to \BS (\fr{h}', W).
       \end{equation}

       \begin{lemma}\label{lem:real_and_demazure}
        Let $(\sigma, f) : \fr{h} \to \fr{h}'$ be a morphism of realizations. Then
        \[f \circ \partial_s^{\fr{h}} = \partial_s^{\fr{h}'} \circ f.\]
        As a result, if $\delta_s \in \fr{h}^*$ satisfies $\partial_s^{\fr{h}} (\delta_s) = 1$, then for $\delta_s' \coloneq f (\delta_s)$, one has that $\partial_s^{\fr{h}'} (\delta_s') = 1$.
       \end{lemma}
       \begin{proof}
        Let $r \in R_{\fr{h}}$ and $s \in S$. Then
        \[\alpha_s' f (\partial_s^{\fr{h}} (r)) = f (\alpha_s \partial_s^{\fr{h}} (r)) = f(r - s(r)) = f(r) - s f(r) = \alpha_s' \partial_s^{\fr{h}'} f(r).\]
        Since $R_{\fr{h}'}$ is an integral domain, we conclude that the Demazure operators commute with $f$.
        As a result, $\partial_s^{\fr{h}'} (\delta_s') = f (\partial_s^{\fr{h}} (\delta_s)) = 1$.
       \end{proof}

       For each $s\in S$, we fix a choice of $\delta_s \in \fr{h}^*$ satisfying $\partial_s^{\fr{h}} (\delta_s) = 1$. We then set $\delta_s' \coloneq f(\delta_s) \in (\fr{h}')^*$.
       By Lemma~\ref{lem:real_and_demazure}, we have that $\partial_s^{\fr{h}'} (\delta_s') = 1$ as well. 
       
       By Proposition~\ref{prop:bases_of_BS} there are left $R_{\fr{h}}$-module bases
       \[ \{ t_{\uw, \ue}^{\fr{h}} = t_{\uw, \ue} \}_{\ue \in \Sub (\uw)}  \subset B_{\uw}^{\fr{h}} \qquad\text{and}\qquad \{ b_{\uw, \ue}^{\fr{h}, x} = b_{\uw, \ue}^x \}_{\ue \in \Sub^x (\uw)} \subset B_{\uw}^{\fr{h}, x} \]
        After the obvious changes, we can also define elements $t_{\uw, \ue}^{\fr{h}'} \in B_{\uw}^{\fr{h}'}$ and $b_{\uw, \ue}^{\fr{h}', x} \in B_{\uw}^{\fr{h}', x}$ using $\delta_s'$ instead of $\delta_s$.
       These also form left $R_{\fr{h}'}$-bases for their respective $R$-bimodules.

       \begin{lemma}\label{lem:Bs_left_module_isom}
                Let $\uw = (s_1, \ldots, s_k)$ be an expression in $W$. The map 
                \[
                \jmath_{\uw} : B_{\uw}^{\fr{h}} \to B_{\uw}^{\fr{h}'}, \qquad r_1 \otimes \ldots \otimes r_{k+1} \mapsto f (r_1) \otimes \ldots \otimes f (r_{k+1}) 
                \]
                is well-defined and $R_{\fr{h}}$-linear. Moreover, the induced map 
                \[\iota_{\uw} : R_{\fr{h}'} \otimes_{R_{\fr{h}}} B_{\uw}^{\fr{h}} \to B_{\uw}^{\fr{h}'}\]
                is an isomorphism of graded left $R_{\fr{h}'}$-modules.
       \end{lemma}
       \begin{proof}
                The fact that $\jmath_{\uw}$ is well-defined follows from $f$ being $W$-equivariant, and so, $f (R_{\fr{h}}^s) \subseteq R_{\fr{h}'}^s$.

                From the discussion preceding the lemma, the collection of elements $\{ 1 \otimes t_{\uw, \ue}^{\fr{h}} \}_{\ue \in \Sub (\uw)}$ (resp. $\{ t_{\uw, \ue}^{\fr{h}'} \}_{\ue \in \Sub (\uw)}$)
                forms a left $R_{\fr{h}'}$-basis for $R_{\fr{h}'} \otimes_{R_{\fr{h}}} B_{\uw}^{\fr{h}}$ (resp. $B_{\uw}^{\fr{h'}}$).
                It follows from its construction that $\iota_{\uw} (1 \otimes t_{\uw, \ue}^{\fr{h}}) = t_{\uw, \ue}^{\fr{h}'}$ for all $\ue \in \Sub (\uw)$. As a result, $\iota_{\uw}$ is an isomorphism of left $R_{\fr{h}'}$-modules.
       \end{proof}

        \begin{lemma}\label{lem:Bs_w_left_module_isom}
                Let $\uw = (s_1, \ldots, s_k)$ be an expression in $W$ and $x \in W$. 
                There is a left $R_{\fr{h}}$-linear map
                \[
                \jmath_{\uw}^x : B_{\uw}^{\fr{h}, x} \to B_{\uw}^{\fr{h}', x} 
                \]
                such that 
                \[\jmath_{\uw}^x \circ \pi_{B_{\ux}^{\fr{h}}}^x = \pi_{B_{\uw}^{\fr{h}'}}^x \circ \jmath_{\uw}. \]
                 Moreover, the induced map 
                \[\iota_{\uw}^x : R_{\fr{h}'} \otimes_{R_{\fr{h}}} B_{\uw}^{\fr{h}, x} \to B_{\uw}^{\fr{h}', x}\]
                is an isomorphism of graded left $R_{\fr{h}'}$-modules.
       \end{lemma}
       \begin{proof}
                The map $\jmath_{\uw} : B_{\uw}^{\fr{h}} \to B_{\uw}^{\fr{h}'}$ induces a $Q_{\fr{h}}$-linear map
                \[\jmath_{\uw, Q} : B_{\uw}^{\fr{h}} \otimes_{R_{\fr{h}}} Q_{\fr{h}} \to B_{\uw}^{\fr{h}'} \otimes_{R_{\fr{h}'}} Q_{\fr{h}'}, \qquad m \otimes q \mapsto \jmath_{\uw} (m) \otimes f_Q (q).\]
                We claim that $\jmath_{\uw, Q}$ takes $x$-components to $x$-components.
                Recall from Lemma~\ref{lem:decomp_into_loc_pieces}, that there is a decomposition of $Q_{\fr{h}}$-modules
                \[(B_{\uw}^{\fr{h}})_Q^x = \bigoplus_{\uw^{\ue} = x} q_{\uw}^{\ue, \fr{h}} Q_{\fr{h}},\]
                where $q_{\uw}^{\ue, \fr{h}}$ is defined as in \eqref{eq:qwe}. It is clear that $\jmath_{\uw, Q} (q_{\uw}^{\ue, \fr{h}}) = q_{\uw}^{\ue, \fr{h}'}$.
                As a result, 
                \[ \jmath_{\uw, Q} ((B_{\uw}^{\fr{h}})_Q^x ) \subseteq (B_{\uw}^{\fr{h}'})_Q^x\]
                for all $x \in W$. The formula
                \begin{equation}\label{eq:Bs_w_left_module_isom_1}
                        \jmath_{\uw}^x (\pi_{B_{\uw}^{\fr{h}}}^x (m) ) = \pi_{B_{\uw}^{\fr{h}'}}^x (\jmath_{\uw} (m)) 
                \end{equation}
                then defines a left $R_{\fr{h}}$-linear map
                \[\jmath_{\uw}^x : B_{\uw}^{\fr{h}, x} \to B_{\uw}^{\fr{h}', x}.\]
                This map induces a left $R_{\fr{h}'}$-linear map
                \[\iota_{\uw}^x : R_{\fr{h}'} \otimes_{R_{\fr{h}}} B_{\uw}^{\fr{h}, x} \to B_{\uw}^{\fr{h}', x}.\]
                
                It remains to show that $\iota_{\uw}^x$ is an isomorphism of graded left $R_{\fr{h}'}$-modules.
                By Proposition~\ref{prop:bases_of_BS}, the sets $\{ 1 \otimes b_{\uw, \ue}^{\fr{h}, x} \}_{\ue \in \Sub^x (\uw)} \subset R_{\fr{h}'} \otimes_{R_{\fr{h}}} B_{\uw}^{\fr{h}, x}$ and $\{ b_{\uw, \ue}^{\fr{h}', x} \}_{\ue \in \Sub^x (\uw)} \subset B_{\uw}^{\fr{h}', x}$
                are left $R_{\fr{h}'}$-module bases.
                Moreover, it can be easily checked from the definitions and \eqref{eq:Bs_w_left_module_isom_1} that the basis for $R_{\fr{h}'} \otimes_{R_{\fr{h}}} B_{\uw}^{\fr{h}, x}$ is sent to the basis for $B_{\uw}^{\fr{h}' , x}$.
                As a result, $\iota_{\uw}^x$ is an isomorphism of graded left $R_{\fr{h}'}$-modules.
       \end{proof}

       Let $\varphi : B_{\ux}^{\fr{h}} \to B_{\uy}^{\fr{h}}$ be a morphism of Abe--Bott--Samelson bimodules.
       By Lemma~\ref{lem:Bs_left_module_isom}, the map 
       \[\BS_f (\varphi) \coloneq \iota_{\uy} \circ (1 \otimes \varphi) \circ \iota_{\ux}^{-1} : B_{\ux}^{\fr{h}'} \to B_{\uy}^{\fr{h}'}\]
       is a graded left $R_{\fr{h}'}$-linear map. By Lemma~\ref{lem:Bs_w_left_module_isom}, the map
       \[\BS_f^w (\varphi) \coloneq \iota_{\uy}^w \circ (1 \otimes \varphi^w) \circ (\iota_{\ux}^w)^{-1} : B_{\ux}^{\fr{h}', w} \to B_{\uy}^{\fr{h}', w}\]
       is a graded left $R_{\fr{h}'}$-linear map. 
       
       \begin{proposition}\label{prop:BS_functor}
        Let $\varphi : B_{\ux}^{\fr{h}} \to B_{\uy}^{\fr{h}}$.
        The map $\BS_f (\varphi) : B_{\ux}^{\fr{h}'} \to B_{\uy}^{\fr{h}'}$ is a morphism of Abe--Bott--Samelson bimodules. This defines a monoidal functor
        \[\BS_f : \BS (\fr{h}, W) \to \BS (\fr{h}', W).\]
       \end{proposition}
       \begin{proof}
                Set $\psi = \BS_f (\varphi)$ and $\psi^w = \BS_f^w (\varphi)$.
                By Lemma~\ref{lem:Bs_w_left_module_isom}, one has that
                \[\pi_{B_{\uy}^{\fr{h}'}}^w \circ \psi = \psi^w \circ \pi_{B_{\ux}^{\fr{h}'}}^w \]
                for all $w \in W$.
                It follows from Lemma~\ref{lem:detecting_bimodularity} that $\psi$ is a morphism of Abe--Bott--Samelson bimodules.

                The functorality and monoidality of $\BS_f$ is a routine check and omitted.
       \end{proof}

       \begin{proposition}\label{prop:behavior_of_BS_on_generators}
                \begin{enumerate}
                        \item Let $s \in S$. Recall the one-color maps,
                        \[\nu_{\fr{h}}^s : B_s \to B_s \otimes B_s, \qquad \epsilon_{\fr{h}}^s : B_s \to R, \qquad \mu_{\fr{h}}^s : B_s \otimes_R B_s \to B_s, \qquad \eta_{\fr{h}}^s : R \to B_s . \]
                        Under $\BS_f$, one has that
                        \[\BS_f (\nu_{\fr{h}}^s) = \nu_{\fr{h}'}^s, \qquad \BS_f (\epsilon_{\fr{h}}^s) = \epsilon_{\fr{h}'}^s, \qquad \BS_f (\mu_{\fr{h}}^s) = \mu_{\fr{h}'}^s, \qquad \BS_f (\eta_{\fr{h}}^s) = \eta_{\fr{h}'}^s.\]
                        \item Let $s,t \in S$ be distinct simple reflections with $m_{st} < \infty$. Recall the braid morphism
                        \[\beta_{st}^{\fr{h}} : B_{{}_s \underline{m}}^{\fr{h}} \to B_{{}_t \underline{m}}^{\fr{h}}.\]
                        We have that
                        \[\BS_f (\beta_{st}^{\fr{h}}) = \beta_{st}^{\fr{h}'}.\]
                        \item For a fixed common $\LL$-datum, $\BS_f$ sends each light leaf and double leaf to the corresponding one for $\fr{h}'$. In particular, it identifies the indexed double-leaf bases.
                \end{enumerate}
       \end{proposition}
       \begin{proof}
        Statement (1) can be easily checked from the definitions of these morphisms using Lemma~\ref{lem:real_and_demazure} and the definition of $\BS_f$.
        For (2), the uniqueness of $\beta_{st}$ reduces the work to checking that $\BS_f (\beta_{st}^{\fr{h}})$ takes $1$-tensors to $1$-tensors.
        However, $\BS_f$ is defined by $f$ on each tensor factor, so this is clear as well.
        Finally, (3) is an immediate consequence of (1) and (2).
       \end{proof}

       \begin{corollary}\label{cor:BS_inj_or_surj}
        Let $(\sigma, f) : \fr{h} \to \fr{h}'$ be a morphism of realizations.
        \begin{enumerate}
                \item $f : R_\fr{h} \to R_{\fr{h}'}$ is injective if and only if $\BS_f$ is faithful.
                \item $f : R_{\fr{h}} \to R_{\fr{h}'}$ is surjective if and only if $\BS_f$ is full.
        \end{enumerate}
        \end{corollary}
        \begin{proof}
                We will just prove the first statement as the second follows from the same argument.
                If  $f : R_\fr{h} \to R_{\fr{h}'}$ is injective, then Theorem~\ref{thm:dll_basis} and Proposition~\ref{prop:behavior_of_BS_on_generators} (3) together imply that $\BS_f$ is faithful.
                On the other hand, assume that $\BS_f$ is faithful. For all $r \in R_{\fr{h}}$, one has that $f (r) = \BS_f (r \cdot \id_{R_{\fr{h}}})$. In particular, faithfulness implies that $f$ is injective.
        \end{proof}

       \subsection{The Elias--Williamson diagrammatic category}

       \begin{definition}
       Define a $\k$-linear monoidal category $\EW_{\leq 2} (\fr{h}, W)$ with:
       \begin{itemize}
        \item the objects are symbols $B_{\uw}$ for $\uw$ an expression, with the monoidal product $B_{\ux} \star B_{\uy} = B_{\ux \uy}$;
        \item the morphisms generated (under composition, monoidal product and $\k$-linear combinations) by the following morphisms:
        \begin{enumerate}
                \item for any polynomial $f \in R$, a morphism
                \[
  \tikzsetnextfilename{poly_box}
  \tikzstyle{every picture}=[tikzfig]
  \begin{tikzpicture}
	\begin{pgfonlayer}{nodelayer}
		\node [style=none] (0) at (-1, -1) {};
		\node [style=none] (1) at (1, -1) {};
		\node [style=none] (2) at (-1, 1) {};
		\node [style=none] (3) at (1, 1) {};
		\node [style=box] (4) at (0, 0) {$f$};
	\end{pgfonlayer}
	\begin{pgfonlayer}{edgelayer}
		\draw [style={dashed_row}] (2.center) to (3.center);
		\draw [style={dashed_row}] (0.center) to (1.center);
	\end{pgfonlayer}
\end{tikzpicture}
\]
                from $B_{\emptyset}$ to itself, of degree $\deg (f)$;
                \item for any $s \in S$, ``dot'' morphisms
                \[ 
  \tikzsetnextfilename{upper_dot}
  \tikzstyle{every picture}=[tikzfig]
  \begin{tikzpicture}
	\begin{pgfonlayer}{nodelayer}
		\node [style=none] (0) at (-1, -1) {};
		\node [style=none] (1) at (1, -1) {};
		\node [style=none] (2) at (-1, 1) {};
		\node [style=none] (3) at (1, 1) {};
		\node [style=none] (4) at (0, -1) {};
		\node [style={EW_s}] (5) at (0, 0) {};
	\end{pgfonlayer}
	\begin{pgfonlayer}{edgelayer}
		\draw [style={dashed_row}] (2.center) to (3.center);
		\draw [style={dashed_row}] (0.center) to (1.center);
		\draw [style={red_edge}] (5) to (4.center);
	\end{pgfonlayer}
\end{tikzpicture}
 \qquad\text{and}\qquad 
  \tikzsetnextfilename{lower_dot}
  \tikzstyle{every picture}=[tikzfig]
  \begin{tikzpicture}
	\begin{pgfonlayer}{nodelayer}
		\node [style=none] (0) at (-1, -1) {};
		\node [style=none] (1) at (1, -1) {};
		\node [style=none] (2) at (-1, 1) {};
		\node [style=none] (3) at (1, 1) {};
		\node [style={EW_s}] (5) at (0, 0) {};
		\node [style=none] (6) at (0, 1) {};
	\end{pgfonlayer}
	\begin{pgfonlayer}{edgelayer}
		\draw [style={dashed_row}] (2.center) to (3.center);
		\draw [style={dashed_row}] (0.center) to (1.center);
		\draw [style={red_edge}] (6.center) to (5);
	\end{pgfonlayer}
\end{tikzpicture}
\]
                from $B_s$ to $B_{\emptyset}$ and from $B_{\emptyset}$ to $B_s$, respectively, of degree 1;
                \item for any $s\in S$, trivalent morphisms
                \[ 
  \tikzsetnextfilename{split}
  \tikzstyle{every picture}=[tikzfig]
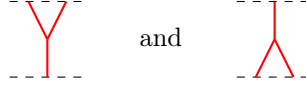
  \begin{tikzpicture}
	\begin{pgfonlayer}{nodelayer}
		\node [style=none] (0) at (-1, -1) {};
		\node [style=none] (1) at (1, -1) {};
		\node [style=none] (2) at (-1, 1) {};
		\node [style=none] (3) at (1, 1) {};
		\node [style=none] (4) at (-0.5, 1) {};
		\node [style=none] (5) at (0.5, 1) {};
		\node [style=none] (6) at (0, 0) {};
		\node [style=none] (7) at (0, -1) {};
	\end{pgfonlayer}
	\begin{pgfonlayer}{edgelayer}
		\draw [style={dashed_row}] (2.center) to (3.center);
		\draw [style={dashed_row}] (0.center) to (1.center);
		\draw [style={red_edge}] (7.center) to (6.center);
		\draw [style={red_edge}] (6.center) to (4.center);
		\draw [style={red_edge}] (6.center) to (5.center);
	\end{pgfonlayer}
\end{tikzpicture}
 \qquad\text{and}\qquad 
  \tikzsetnextfilename{merge}
  \tikzstyle{every picture}=[tikzfig]
  \begin{tikzpicture}
	\begin{pgfonlayer}{nodelayer}
		\node [style=none] (0) at (-1, -1) {};
		\node [style=none] (1) at (1, -1) {};
		\node [style=none] (2) at (-1, 1) {};
		\node [style=none] (3) at (1, 1) {};
		\node [style=none] (4) at (-0.5, -1) {};
		\node [style=none] (5) at (0.5, -1) {};
		\node [style=none] (6) at (0, 0) {};
		\node [style=none] (7) at (0, 1) {};
	\end{pgfonlayer}
	\begin{pgfonlayer}{edgelayer}
		\draw [style={dashed_row}] (2.center) to (3.center);
		\draw [style={dashed_row}] (0.center) to (1.center);
		\draw [style={red_edge}] (7.center) to (6.center);
		\draw [style={red_edge}] (6.center) to (4.center);
		\draw [style={red_edge}] (6.center) to (5.center);
	\end{pgfonlayer}
\end{tikzpicture}
\]
                from $B_s$ to $B_{(s,s)}$ and from $B_{(s,s)}$ to $B_s$, respectively, of degree $-1$;
                \item for any pair $(s,t)$ of distinct simple reflections such that $st$ has finite order $m = m_{st}$ in $W$, a morphism
                \[
  \tikzsetnextfilename{odd_braid}
  \tikzstyle{every picture}=[tikzfig]
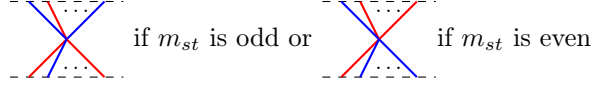
  \begin{tikzpicture}
	\begin{pgfonlayer}{nodelayer}
		\node [style=none] (0) at (-1.5, -1) {};
		\node [style=none] (1) at (1.5, -1) {};
		\node [style=none] (2) at (-1.5, 1) {};
		\node [style=none] (3) at (1.5, 1) {};
		\node [style=none] (4) at (-1, -1) {};
		\node [style=none] (5) at (-1, 1) {};
		\node [style=none] (6) at (-0.5, 1) {};
		\node [style=none] (7) at (-0.5, -1) {};
		\node [style=none] (8) at (0, 0) {};
		\node [style=none] (9) at (1, -1) {};
		\node [style=none] (10) at (1, 1) {};
		\node [style=none] (11) at (0.25, 0.75) {\footnotesize$\ldots$};
		\node [style=none] (13) at (0.25, -0.75) {\footnotesize$\ldots$};
	\end{pgfonlayer}
	\begin{pgfonlayer}{edgelayer}
		\draw [style={dashed_row}] (2.center) to (3.center);
		\draw [style={dashed_row}] (0.center) to (1.center);
		\draw [style={red_edge}] (4.center) to (8.center);
		\draw [style={red_edge}] (8.center) to (6.center);
		\draw [style=blue edge] (8.center) to (5.center);
		\draw [style=blue edge] (8.center) to (10.center);
		\draw [style={red_edge}] (8.center) to (9.center);
		\draw [style=blue edge] (8.center) to (7.center);
	\end{pgfonlayer}
\end{tikzpicture}
 \text{ if } m_{st} \text{ is odd or } 
  \tikzsetnextfilename{even_braid}
  \tikzstyle{every picture}=[tikzfig]
  \begin{tikzpicture}
	\begin{pgfonlayer}{nodelayer}
		\node [style=none] (0) at (-1.5, -1) {};
		\node [style=none] (1) at (1.5, -1) {};
		\node [style=none] (2) at (-1.5, 1) {};
		\node [style=none] (3) at (1.5, 1) {};
		\node [style=none] (4) at (-1, -1) {};
		\node [style=none] (5) at (-1, 1) {};
		\node [style=none] (6) at (-0.5, 1) {};
		\node [style=none] (7) at (-0.5, -1) {};
		\node [style=none] (8) at (0, 0) {};
		\node [style=none] (9) at (1, -1) {};
		\node [style=none] (10) at (1, 1) {};
		\node [style=none] (11) at (0.25, 0.75) {\footnotesize$\ldots$};
		\node [style=none] (13) at (0.25, -0.75) {\footnotesize$\ldots$};
	\end{pgfonlayer}
	\begin{pgfonlayer}{edgelayer}
		\draw [style={dashed_row}] (2.center) to (3.center);
		\draw [style={dashed_row}] (0.center) to (1.center);
		\draw [style={red_edge}] (4.center) to (8.center);
		\draw [style={red_edge}] (8.center) to (6.center);
		\draw [style=blue edge] (8.center) to (5.center);
		\draw [style={red_edge}] (8.center) to (10.center);
		\draw [style=blue edge] (8.center) to (9.center);
		\draw [style=blue edge] (8.center) to (7.center);
	\end{pgfonlayer}
\end{tikzpicture}
  \text{if } m_{st} \text{ is even}\]
                from $B_{{}_s \underline{m}}$ to $B_{{}_t \underline{m}}$, of degree $0$.
        \end{enumerate}
        We set
        \[
  \tikzsetnextfilename{cap}
  \tikzstyle{every picture}=[tikzfig]
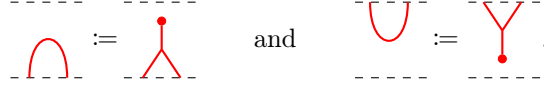
  \begin{tikzpicture}
	\begin{pgfonlayer}{nodelayer}
		\node [style=none] (0) at (-2.5, -1) {};
		\node [style=none] (1) at (-0.5, -1) {};
		\node [style=none] (2) at (-2.5, 1) {};
		\node [style=none] (3) at (-0.5, 1) {};
		\node [style=none] (4) at (-2, -1) {};
		\node [style=none] (5) at (-1, -1) {};
		\node [style=none] (6) at (0.5, -1) {};
		\node [style=none] (7) at (2.5, -1) {};
		\node [style=none] (8) at (0.5, 1) {};
		\node [style=none] (9) at (2.5, 1) {};
		\node [style=none] (10) at (1, -1) {};
		\node [style=none] (11) at (2, -1) {};
		\node [style=none] (12) at (1.5, -0.25) {};
		\node [style={EW_s}] (13) at (1.5, 0.5) {};
		\node [style=none] (14) at (0, 0) {$\coloneq$};
	\end{pgfonlayer}
	\begin{pgfonlayer}{edgelayer}
		\draw [style={dashed_row}] (2.center) to (3.center);
		\draw [style={dashed_row}] (0.center) to (1.center);
		\draw [style={red_edge}, bend left=90, looseness=3.50] (4.center) to (5.center);
		\draw [style={dashed_row}] (8.center) to (9.center);
		\draw [style={dashed_row}] (6.center) to (7.center);
		\draw [style={red_edge}] (10.center) to (12.center);
		\draw [style={red_edge}] (12.center) to (11.center);
		\draw [style={red_edge}] (12.center) to (13);
	\end{pgfonlayer}
\end{tikzpicture}
 \qquad\text{and}\qquad 
  \tikzsetnextfilename{cup}
  \tikzstyle{every picture}=[tikzfig]
  \begin{tikzpicture}
	\begin{pgfonlayer}{nodelayer}
		\node [style=none] (0) at (-0.5, 1) {};
		\node [style=none] (1) at (-2.5, 1) {};
		\node [style=none] (2) at (-0.5, -1) {};
		\node [style=none] (3) at (-2.5, -1) {};
		\node [style=none] (4) at (-1, 1) {};
		\node [style=none] (5) at (-2, 1) {};
		\node [style=none] (6) at (2.5, 1) {};
		\node [style=none] (7) at (0.5, 1) {};
		\node [style=none] (8) at (2.5, -1) {};
		\node [style=none] (9) at (0.5, -1) {};
		\node [style=none] (10) at (2, 1) {};
		\node [style=none] (11) at (1, 1) {};
		\node [style=none] (12) at (1.5, 0.25) {};
		\node [style={EW_s}] (13) at (1.5, -0.5) {};
		\node [style=none] (14) at (0, 0) {$\coloneq$};
	\end{pgfonlayer}
	\begin{pgfonlayer}{edgelayer}
		\draw [style={dashed_row}] (2.center) to (3.center);
		\draw [style={dashed_row}] (0.center) to (1.center);
		\draw [style={red_edge}, bend left=90, looseness=3.50] (4.center) to (5.center);
		\draw [style={dashed_row}] (8.center) to (9.center);
		\draw [style={dashed_row}] (6.center) to (7.center);
		\draw [style={red_edge}] (10.center) to (12.center);
		\draw [style={red_edge}] (12.center) to (11.center);
		\draw [style={red_edge}] (12.center) to (13);
	\end{pgfonlayer}
\end{tikzpicture}
.\]
        These morphisms are subject to the following relations:
        \begin{enumerate}
                \item the boxes add and multiply in the obvious way;
                \item The one-color relations
                \[
  \tikzsetnextfilename{one_color_relations}
  \tikzstyle{every picture}=[tikzfig]
  \input{./figures/one_color_relations.tikz}
\]
                \item The two-color relations
                \[
  \tikzsetnextfilename{two_color_relations}
  \tikzstyle{every picture}=[tikzfig]
  \input{./figures/two_color_relations.tikz}
\]
                where the boxes labeled by $\textnormal{JW}_{s,t}$ are the ``Jones--Wenzl projectors'' (see \cite[\S 5.2]{EW}).
        \end{enumerate}

        The composition of morphisms is induced by vertical concatenation. The monoidal product on morphisms is by horizontal concatenation of diagrams.
       \end{itemize}
       \end{definition}

       The only difference between $\EW_{\leq 2} (\fr{h}, W)$ and the usual Elias--Williamson diagrammatic category $\EW (\fr{h}, W)$ of \cite{EW} is the absence of the 3-color Zamolodchikov relations.
       We are primarily interested in the case where $W = H_3$ in which we wish to calculate the 3-color relation. 
       When the 3-color relation is known, there is an obvious full monoidal functor
       \[\EW_{\leq 2} (\fr{h}, W) \to \EW (\fr{h}, W).\]

       In \cite{EW} and \cite{Abe19}, a monoidal functor
       \[\scrF : \EW_{\leq 2} (\fr{h}, W) \to \BS (\fr{h}, W)\]
       is constructed. The functor is defined on objects by $\scrF (B_{\uw}) = B_{\uw}$ for all $\uw \in \Exp (W)$.
       On generating morphisms $\scrF$ is defined by, 
       \[
  \tikzsetnextfilename{poly_box}
  \tikzstyle{every picture}=[tikzfig]
  \begin{tikzpicture}
	\begin{pgfonlayer}{nodelayer}
		\node [style=none] (0) at (-1, -1) {};
		\node [style=none] (1) at (1, -1) {};
		\node [style=none] (2) at (-1, 1) {};
		\node [style=none] (3) at (1, 1) {};
		\node [style=box] (4) at (0, 0) {$f$};
	\end{pgfonlayer}
	\begin{pgfonlayer}{edgelayer}
		\draw [style={dashed_row}] (2.center) to (3.center);
		\draw [style={dashed_row}] (0.center) to (1.center);
	\end{pgfonlayer}
\end{tikzpicture}
 \mapsto f, \quad 
  \tikzsetnextfilename{upper_dot}
  \tikzstyle{every picture}=[tikzfig]
  \begin{tikzpicture}
	\begin{pgfonlayer}{nodelayer}
		\node [style=none] (0) at (-1, -1) {};
		\node [style=none] (1) at (1, -1) {};
		\node [style=none] (2) at (-1, 1) {};
		\node [style=none] (3) at (1, 1) {};
		\node [style=none] (4) at (0, -1) {};
		\node [style={EW_s}] (5) at (0, 0) {};
	\end{pgfonlayer}
	\begin{pgfonlayer}{edgelayer}
		\draw [style={dashed_row}] (2.center) to (3.center);
		\draw [style={dashed_row}] (0.center) to (1.center);
		\draw [style={red_edge}] (5) to (4.center);
	\end{pgfonlayer}
\end{tikzpicture}
 \mapsto \epsilon^s, \quad 
  \tikzsetnextfilename{lower_dot}
  \tikzstyle{every picture}=[tikzfig]
  \begin{tikzpicture}
	\begin{pgfonlayer}{nodelayer}
		\node [style=none] (0) at (-1, -1) {};
		\node [style=none] (1) at (1, -1) {};
		\node [style=none] (2) at (-1, 1) {};
		\node [style=none] (3) at (1, 1) {};
		\node [style={EW_s}] (5) at (0, 0) {};
		\node [style=none] (6) at (0, 1) {};
	\end{pgfonlayer}
	\begin{pgfonlayer}{edgelayer}
		\draw [style={dashed_row}] (2.center) to (3.center);
		\draw [style={dashed_row}] (0.center) to (1.center);
		\draw [style={red_edge}] (6.center) to (5);
	\end{pgfonlayer}
\end{tikzpicture}
 \mapsto \eta^s, \quad 
  \tikzsetnextfilename{split}
  \tikzstyle{every picture}=[tikzfig]
  \begin{tikzpicture}
	\begin{pgfonlayer}{nodelayer}
		\node [style=none] (0) at (-1, -1) {};
		\node [style=none] (1) at (1, -1) {};
		\node [style=none] (2) at (-1, 1) {};
		\node [style=none] (3) at (1, 1) {};
		\node [style=none] (4) at (-0.5, 1) {};
		\node [style=none] (5) at (0.5, 1) {};
		\node [style=none] (6) at (0, 0) {};
		\node [style=none] (7) at (0, -1) {};
	\end{pgfonlayer}
	\begin{pgfonlayer}{edgelayer}
		\draw [style={dashed_row}] (2.center) to (3.center);
		\draw [style={dashed_row}] (0.center) to (1.center);
		\draw [style={red_edge}] (7.center) to (6.center);
		\draw [style={red_edge}] (6.center) to (4.center);
		\draw [style={red_edge}] (6.center) to (5.center);
	\end{pgfonlayer}
\end{tikzpicture}
 \mapsto \nu^s, \quad 
  \tikzsetnextfilename{merge}
  \tikzstyle{every picture}=[tikzfig]
  \begin{tikzpicture}
	\begin{pgfonlayer}{nodelayer}
		\node [style=none] (0) at (-1, -1) {};
		\node [style=none] (1) at (1, -1) {};
		\node [style=none] (2) at (-1, 1) {};
		\node [style=none] (3) at (1, 1) {};
		\node [style=none] (4) at (-0.5, -1) {};
		\node [style=none] (5) at (0.5, -1) {};
		\node [style=none] (6) at (0, 0) {};
		\node [style=none] (7) at (0, 1) {};
	\end{pgfonlayer}
	\begin{pgfonlayer}{edgelayer}
		\draw [style={dashed_row}] (2.center) to (3.center);
		\draw [style={dashed_row}] (0.center) to (1.center);
		\draw [style={red_edge}] (7.center) to (6.center);
		\draw [style={red_edge}] (6.center) to (4.center);
		\draw [style={red_edge}] (6.center) to (5.center);
	\end{pgfonlayer}
\end{tikzpicture}
 \mapsto \mu^s, \quad 
  \tikzsetnextfilename{amb_braid}
  \tikzstyle{every picture}=[tikzfig]
  \begin{tikzpicture}
	\begin{pgfonlayer}{nodelayer}
		\node [style=none] (0) at (-1.5, -1) {};
		\node [style=none] (1) at (1.5, -1) {};
		\node [style=none] (2) at (-1.5, 1) {};
		\node [style=none] (3) at (1.5, 1) {};
		\node [style=none] (4) at (-1, -1) {};
		\node [style=none] (5) at (-1, 1) {};
		\node [style=none] (6) at (-0.5, 1) {};
		\node [style=none] (7) at (-0.5, -1) {};
		\node [style=none] (8) at (0, 0) {};
		\node [style=none] (9) at (1, -1) {};
		\node [style=none] (10) at (1, 1) {};
		\node [style=none] (11) at (0.25, 0.75) {\footnotesize$\ldots$};
		\node [style=none] (13) at (0.25, -0.75) {\footnotesize$\ldots$};
	\end{pgfonlayer}
	\begin{pgfonlayer}{edgelayer}
		\draw [style={dashed_row}] (2.center) to (3.center);
		\draw [style={dashed_row}] (0.center) to (1.center);
		\draw [style={red_edge}] (4.center) to (8.center);
		\draw [style={red_edge}] (8.center) to (6.center);
		\draw [style=blue edge] (8.center) to (5.center);
		\draw [style={purple_edge}] (8.center) to (10.center);
		\draw [style={purple_edge}] (8.center) to (9.center);
		\draw [style=blue edge] (8.center) to (7.center);
	\end{pgfonlayer}
\end{tikzpicture}
 \mapsto \beta_{st}.\]
       This functor turns out to be well-defined; although, it is quite laborious to show (see \cite[\S 5.3]{EW}).

       \section{Localization}

       \subsection{Localization matrices}

       We review the theory of localization matrices developed in \cite{EWloc,Abe21}. 

       We will construct a monoidal category $\sfC_Q (\fr{h}, W)$ whose objects consist of tuples $(P^w)_{w \in W}$ such that
       \begin{itemize}
        \item $P^w$ is a graded $Q$-bimodule;
        \item $m \cdot q = w (q) \cdot m$ for any $m \in P^w$ and $q \in Q$;
        \item $P^w$ is nonzero for finitely many $w \in W$.
       \end{itemize} 
       A morphism in $\sfC_Q (\fr{h}, W)$ from $(P_1^w)_{w \in W}$ to $(P_2^w)_{w \in W}$ consists of a tuple $(\varphi_w)_{w \in W}$ where each $\varphi_w : P_1^w \to P_2^w$ is a $Q$-bimodule homomorphism.
       The monoidal structure on $\sfC_Q (\fr{h}, W)$ is induced from the tensor product of graded $Q$-bimodules. More precisely,
       \[(P_1 \otimes_Q P_2)^w = \bigoplus_{xy = w} P_1^x \otimes_Q P_2^y.\]
       
       The assignment $M \mapsto (M_Q^w)_{w \in W}$ defines a monoidal functor
       \[\sfC (\fr{h}, W) \to \sfC_Q (\fr{h}, W).\]
       We will focus our attention on its restriction to Abe--Bott--Samelson bimodules,
       \[(-)_Q : \BS (\fr{h}, W) \to \sfC_Q (\fr{h}, W).\]
       which we call the \emph{localization functor}. It is explained in \cite[\S 3.3]{Abe21} that $(-)_Q$ is faithful.

       As a consequence of the Lemma~\ref{lem:decomp_into_loc_pieces}, one has that there is an isomorphism of ungraded left $Q$-modules
       \[\Hom_{\sfC_Q (\fr{h}, W)} ((B_{\ux})_Q, (B_{\uy})_Q) \cong \bigoplus_{\ux^{\ue} = w = \uy^{\uf}} Q. \]
       Let $\varphi : B_{\ux} \to B_{\uy}$ be a morphism in $\BS (\fr{h}, W)$. For subexpressions $\ue$ of $\ux$ and $\uf$ of $\uy$, we can consider the map
       \[(\varphi_Q)_{\uf, \ue} : B_{\ux, Q}^{\ue} \to B_{\uy, Q}^{\uf}\]
       given by the composition
       \[B_{\ux, Q}^{\ue} \hookrightarrow B_{\ux} \otimes_R Q \stackrel{\varphi \otimes 1}{\longrightarrow} B_{\uy} \otimes_R Q \twoheadrightarrow B_{\uy, Q}^{\uf}.\]
       Since both the source and target of this map are free rank-one $Q$-modules with fixed generators $q_{\ux}^{\ue}$ and $q_{\uy}^{\uf}$ respectively, it is determined by a scalar $\Lambda (\varphi)_{\uf, \ue} \in Q$.
       Note that $\Lambda (\varphi)_{\uf, \ue} = 0$ unless $\ux^{\ue} = \uy^{\uf}$. By considering all possible subexpressions of $\ux$ and $\uy$, we form a matrix
       \[\Lambda (\varphi) = (\Lambda (\varphi)_{\uf, \ue})_{\ue \in \Sub (\ux), \uf \in \Sub (\uy)}\]
       whose entries are in $Q$ called the \emph{localization matrix}. Since $(-)_Q$ is faithful, it follows that $\varphi$ is uniquely determined by its localization matrix.
       In particular, $\Lambda (\varphi) = 0$ if and only if $\varphi = 0$.

       \subsection{Localization matrices of products and duals}

        Let $\Mat_{\uy, \ux}$ denote the set of $\Sub (\uy) \times \Sub (\ux)$-matrices with entries in $Q$ such that if $A \in \Mat_{\uy, \ux}$, then $A_{\uf, \ue} \neq 0$ only when $\ux^{\ue} = \uy^{\uf}$.
        Of course, localization matrices are examples of such matrices. In particular, if $\varphi : B_{\ux} \to B_{\uy}$ is a morphism of Abe--Bott--Samelson bimodules, then $\Lambda (\varphi) \in \Mat_{\uy, \ux}$.

         Let $A \in \Mat_{\uy, \ux}$ and $B \in \Mat_{\uy', \ux'}$.
         We define a matrix $A \star B \in \Mat_{\uy \uy', \ux \ux'}$ whose entries are given by
         \begin{equation}\label{eq:tensor_prod_of_Q_mats}
                (A \star B)_{\uf \uf', \ue \ue'} \coloneq A_{\uf, \ue} \ux^{\ue} B_{\uf', \ue'}.
         \end{equation}

       \begin{lemma}\label{lem:prod_of_loc_mats}
                Let $\varphi : B_{\ux} \to B_{\uy}$ and $\psi : B_{\ux'} \to B_{\uy'}$. 
                Then
                \[ \Lambda (\varphi \otimes \psi) = \Lambda (\varphi) \star \Lambda (\psi). \]
       \end{lemma}
       \begin{proof}
        The identity is a simple consequence of the localization functor being monoidal (see \cite[(1.2)]{EWloc}).
       \end{proof}

       Let $A \in \Mat_{\uy, \ux}$.
       Consider the diagonal matrix $G_{\ux} \in \Mat_{\ux, \ux}$ whose $(\ue, \ue)$-entry is given by $\zeta_{\ux} (\ue)$.
       We define a matrix $A^{\dagger} \in \Mat_{\ux, \uy}$ by
       \[A^{\dagger} \coloneq G_{\ux}^{-1} A^{\top} G_{\uy}.\]
       Note that 
       \[
        (A^{\dagger})^{\dagger} = A.
       \]
       Now suppose $B \in \Mat_{\ux, \uw}$.
       It follows from the definitions that
       \begin{equation}\label{eq:dagger_and_compositions}
                B^{\dagger} A^{\dagger} = (AB)^{\dagger}.
       \end{equation}

       \begin{lemma}\label{lem:dagger_and_product}
        Let $A \in \Mat_{\uy, \ux}$ and $B \in \Mat_{\uy', \ux'}$.
        Then
        \[(A \star B)^{\dagger} = A^{\dagger} \star B^{\dagger}.\]
       \end{lemma}
       \begin{proof}
        The definition of $\zeta$ gives
        \[\zeta_{\ux \ux'} (\ue \ue') = \zeta_{\ux} (\ue) \ux^{\ue} (\zeta_{\ux'} (\ue')).\]
        We can use this with definition of $\star$ from \eqref{eq:tensor_prod_of_Q_mats} to compute
        \begin{equation}\label{eq:dagger_and_product_1}
                (A \star B)_{\ue \ue', \uf \uf'}^{\dagger} = \frac{\zeta_{\uy} (\uf) \uy^{\uf} (\zeta_{\uy'} (\uf'))}{\zeta_{\ux} (\ue) \ux^{\ue} (\zeta_{\ux'} (\ue'))} A_{\uf, \ue} \ux^{\ue} (B_{\uf', \ue'}).
        \end{equation}
        If $A_{\uf, \ue} = 0$, then right-hand side vanishes and there is nothing to show. 
        If $A_{\uf, \ue} \neq 0$, then we must have that $\ux^{\ue} = w = \uy^{\uf}$. 
        As a result \eqref{eq:dagger_and_product_1} becomes
        \begin{align*}
                 (A \star B)_{\ue \ue', \uf \uf'}^{\dagger} &= \left( \frac{\zeta_{\uy} (\uf)}{\zeta_{\ux} (\ue)} A_{\uf, \ue} \right) w \left( \frac{\zeta_{\uy'} (\uf') }{\zeta_{\ux'} (\ue')} B_{\uf', \ue'}\right) \\
                 &= A_{\ue, \uf}^{\dagger} \uy^{\uf} (B_{\ue', \uf'}^{\dagger}) \\
                 &= (A^{\dagger} \star B^{\dagger})_{\ue \ue', \uf \uf'}.
        \end{align*}
       \end{proof}

       \begin{lemma}\label{lem:loc_mat_of_dual}
                Let $\varphi : B_{\ux} \to B_{\uy}$, $\ue \in \Sub (\ux)$, and $\uf \in \Sub (\uy)$. 
                Then
                \begin{equation}\label{eq:loc_mat_of_dual_0}
                        \Lambda (\DD (\varphi)) = \Lambda (\varphi)^\dagger .
                \end{equation}
                In particular,
                \[\Lambda (\DD (\varphi))_{\ue, \uf} = \frac{\zeta_{\uy} (\uf)}{\zeta_{\ux} (\ue)} \Lambda (\varphi)_{\uf, \ue}. \]
       \end{lemma}
       \begin{proof}
                As a consequence of \eqref{eq:dagger_and_compositions} and Lemma~\ref{lem:dagger_and_product} and the monoidality of $\DD$, it suffices to check that \eqref{eq:loc_mat_of_dual_0} holds on the generating morphisms from \S\ref{ss:gen_mors}.
                These can be readily computed by hand. 
       \end{proof}

       \subsection{Diagrammatic localization}\label{ss:diag_localization}

       A related monoidal functor
       \[(-)_Q^{\diag} : \EW_{\leq 2} (\fr{h}, W) \to \sfC_Q (\fr{h}, W)\]
       is constructed in \cite{EWloc}. On objects, it is defined by 
       \[(B_{\uw})_Q^{\diag} \coloneq \bigoplus_{\ue \in \Sub (\uw)} Q_{\uw^{\ue}} (-\ell (\uw)).\]
       On morphisms, it suffices to define the localization matrix associated to one of the generating morphisms.
       
       \begin{enumerate}
        \item \emph{Polynomial. } 
        \[
  \tikzsetnextfilename{poly_box}
  \tikzstyle{every picture}=[tikzfig]
  \begin{tikzpicture}
	\begin{pgfonlayer}{nodelayer}
		\node [style=none] (0) at (-1, -1) {};
		\node [style=none] (1) at (1, -1) {};
		\node [style=none] (2) at (-1, 1) {};
		\node [style=none] (3) at (1, 1) {};
		\node [style=box] (4) at (0, 0) {$f$};
	\end{pgfonlayer}
	\begin{pgfonlayer}{edgelayer}
		\draw [style={dashed_row}] (2.center) to (3.center);
		\draw [style={dashed_row}] (0.center) to (1.center);
	\end{pgfonlayer}
\end{tikzpicture}
 \mapsto \begin{tabular}{c | c} & $\emptyset$ \\ \hline $\emptyset$ & $f$ \end{tabular}\]
        \item \emph{Upper dot. } 
        \[
  \tikzsetnextfilename{upper_dot}
  \tikzstyle{every picture}=[tikzfig]
  \begin{tikzpicture}
	\begin{pgfonlayer}{nodelayer}
		\node [style=none] (0) at (-1, -1) {};
		\node [style=none] (1) at (1, -1) {};
		\node [style=none] (2) at (-1, 1) {};
		\node [style=none] (3) at (1, 1) {};
		\node [style=none] (4) at (0, -1) {};
		\node [style={EW_s}] (5) at (0, 0) {};
	\end{pgfonlayer}
	\begin{pgfonlayer}{edgelayer}
		\draw [style={dashed_row}] (2.center) to (3.center);
		\draw [style={dashed_row}] (0.center) to (1.center);
		\draw [style={red_edge}] (5) to (4.center);
	\end{pgfonlayer}
\end{tikzpicture}
 \mapsto \begin{tabular}{c | c c } & 0 & 1 \\ \hline $\emptyset$ & $\alpha_s$ & 0 \end{tabular}\]
        \item \emph{Lower dot. }
        \[
  \tikzsetnextfilename{lower_dot}
  \tikzstyle{every picture}=[tikzfig]
  \begin{tikzpicture}
	\begin{pgfonlayer}{nodelayer}
		\node [style=none] (0) at (-1, -1) {};
		\node [style=none] (1) at (1, -1) {};
		\node [style=none] (2) at (-1, 1) {};
		\node [style=none] (3) at (1, 1) {};
		\node [style={EW_s}] (5) at (0, 0) {};
		\node [style=none] (6) at (0, 1) {};
	\end{pgfonlayer}
	\begin{pgfonlayer}{edgelayer}
		\draw [style={dashed_row}] (2.center) to (3.center);
		\draw [style={dashed_row}] (0.center) to (1.center);
		\draw [style={red_edge}] (6.center) to (5);
	\end{pgfonlayer}
\end{tikzpicture}
 \mapsto \begin{tabular}{c | c } & $\emptyset$ \\ \hline 0 & 1 \\ 1 & 0 \end{tabular}\]
        \item \emph{Split. }
        \[
  \tikzsetnextfilename{split}
  \tikzstyle{every picture}=[tikzfig]
  \begin{tikzpicture}
	\begin{pgfonlayer}{nodelayer}
		\node [style=none] (0) at (-1, -1) {};
		\node [style=none] (1) at (1, -1) {};
		\node [style=none] (2) at (-1, 1) {};
		\node [style=none] (3) at (1, 1) {};
		\node [style=none] (4) at (-0.5, 1) {};
		\node [style=none] (5) at (0.5, 1) {};
		\node [style=none] (6) at (0, 0) {};
		\node [style=none] (7) at (0, -1) {};
	\end{pgfonlayer}
	\begin{pgfonlayer}{edgelayer}
		\draw [style={dashed_row}] (2.center) to (3.center);
		\draw [style={dashed_row}] (0.center) to (1.center);
		\draw [style={red_edge}] (7.center) to (6.center);
		\draw [style={red_edge}] (6.center) to (4.center);
		\draw [style={red_edge}] (6.center) to (5.center);
	\end{pgfonlayer}
\end{tikzpicture}
 \mapsto \begin{tabular}{c | c c } & 0 & 1 \\ \hline 00 & $1/\alpha_s$ & 0 \\ 01 & 0 & $1/\alpha_s$ \\ 10 & 0 & $-1/\alpha_s$ \\ 11 & $-1/\alpha_s$ & 0 \end{tabular}\]
        \item \emph{Merge. }
        \[
  \tikzsetnextfilename{merge}
  \tikzstyle{every picture}=[tikzfig]
  \begin{tikzpicture}
	\begin{pgfonlayer}{nodelayer}
		\node [style=none] (0) at (-1, -1) {};
		\node [style=none] (1) at (1, -1) {};
		\node [style=none] (2) at (-1, 1) {};
		\node [style=none] (3) at (1, 1) {};
		\node [style=none] (4) at (-0.5, -1) {};
		\node [style=none] (5) at (0.5, -1) {};
		\node [style=none] (6) at (0, 0) {};
		\node [style=none] (7) at (0, 1) {};
	\end{pgfonlayer}
	\begin{pgfonlayer}{edgelayer}
		\draw [style={dashed_row}] (2.center) to (3.center);
		\draw [style={dashed_row}] (0.center) to (1.center);
		\draw [style={red_edge}] (7.center) to (6.center);
		\draw [style={red_edge}] (6.center) to (4.center);
		\draw [style={red_edge}] (6.center) to (5.center);
	\end{pgfonlayer}
\end{tikzpicture}
 \mapsto \begin{tabular}{c | c c c c } & 00 & 01 & 10 & 11 \\ \hline 0 & 1 & 0 & 0 & 1 \\ 1 & 0 & 1 & 1 & 0 \end{tabular}\]
        \item \emph{$2m$-valent vertices.} Let $s,t$ be distinct simple reflections such that $m = m_{st} < \infty$. 
        We define
        \[
  \tikzsetnextfilename{amb_braid}
  \tikzstyle{every picture}=[tikzfig]
  \begin{tikzpicture}
	\begin{pgfonlayer}{nodelayer}
		\node [style=none] (0) at (-1.5, -1) {};
		\node [style=none] (1) at (1.5, -1) {};
		\node [style=none] (2) at (-1.5, 1) {};
		\node [style=none] (3) at (1.5, 1) {};
		\node [style=none] (4) at (-1, -1) {};
		\node [style=none] (5) at (-1, 1) {};
		\node [style=none] (6) at (-0.5, 1) {};
		\node [style=none] (7) at (-0.5, -1) {};
		\node [style=none] (8) at (0, 0) {};
		\node [style=none] (9) at (1, -1) {};
		\node [style=none] (10) at (1, 1) {};
		\node [style=none] (11) at (0.25, 0.75) {\footnotesize$\ldots$};
		\node [style=none] (13) at (0.25, -0.75) {\footnotesize$\ldots$};
	\end{pgfonlayer}
	\begin{pgfonlayer}{edgelayer}
		\draw [style={dashed_row}] (2.center) to (3.center);
		\draw [style={dashed_row}] (0.center) to (1.center);
		\draw [style={red_edge}] (4.center) to (8.center);
		\draw [style={red_edge}] (8.center) to (6.center);
		\draw [style=blue edge] (8.center) to (5.center);
		\draw [style={purple_edge}] (8.center) to (10.center);
		\draw [style={purple_edge}] (8.center) to (9.center);
		\draw [style=blue edge] (8.center) to (7.center);
	\end{pgfonlayer}
\end{tikzpicture}
 \mapsto G_{s,t},\]
        where for $\ue \in \Sub ({}_s \underline{m})$ and $\uf \in \Sub ({}_t \underline{m})$, the $(\uf, \ue)$-entry of the localization matrix for the $2m$-valent vertex is
        \[(G_{s,t})_{\uf, \ue} \coloneq \begin{cases} \frac{\pi_{s,t}}{\zeta_{{}_t \underline{m}} (\uf)} & \text{if } {}_s\underline{m}^{\ue} = {}_t \underline{m}^{\uf}, \\ 0 & \text{otherwise}.\end{cases}\]
       \end{enumerate}

       It follows from the construction that the following diagram commutes
       \[\begin{tikzcd}
{\EW_{\leq 2} (\fr{h}, W)} \arrow[rr, "\scrF"] \arrow[rd, "(-)_Q^{\diag}"'] &                      & {\BS (\fr{h}, W)} \arrow[ld, "(-)_Q"] \\
                                                          & {\sfC_Q (\fr{h}, W)}. &                                      
\end{tikzcd}\]
The content of this commutative diagram is that we know what the localization matrices are for the generating morphisms of the diagrammatic category.

        \begin{goal}
                Let $W$ be a finite rank-3 Coxeter group. Let $w_0$ be the longest element in $W$ and pick two reduced expressions $\ux$ and $\uy$ for $w_0$.
                Let $p$ and $q$ be distinct paths in the rex graph of $w_0$ from $\ux$ to $\uy$ around a Zamolodchikov cycle.
                There are corresponding morphisms $\beta_p$ and $\beta_q$ from $B_{\ux}$ to $B_{\uy}$ in $\BS (\fr{h}, W)$. 
                To complete the definition of the Elias--Williamson category, one needs to find a linear combination of double leaves
                \[\beta_p - \beta_q = \sum_{\ux^\ue = w = \uy^\uf} c_{\ue, \uf}^w \dLL_{\ue, \uf}.\]
                In light of the localization functor, it suffices to solve the enormous matrix equations
                \begin{equation}\label{eq:big_eqn_system}
                        \Lambda (\beta_p - \beta_q) =  \sum_{\ux^\ue = w = \uy^\uf} c_{\ue, \uf}^w \Lambda (\dLL_{\ue, \uf}).
                \end{equation}
                When $W$ is of type $A_3$ or type $B_3$, the saving grace is that the $c_{\ue, \uf}^w = 0$ for all subexpressions provided one chooses $\ux$ and $\uy$ in a clever manner.
                For type $A_1 \times I_2 (m)$, one again has that $c_{\ue, \uf}^w = 0$; however, one can show this directly without localization.

                The hard case, and the subject of the paper, is handling the type $H_3$ case. It was already calculated by Elias--Williamson that no choice of rexes will cause the coefficients $c_{\ue, \uf}^w$ to be all $0$.
                Proceeding naively, even writing down this system of equations would involve the multiplication of many sparse matrices with $2^{15} = 32768$ rows.
                Computing the localization matrices is still doable with some effort (although in practice you never multiply these massive sparse matrices).
                However, solving the system of equations efficiently requires some new insight.
        \end{goal}

        Before we continue with developing some general tools to attack this problem, we will first note some reductions/heuristics that we will employ.
        \begin{enumerate}
                \item Since $\beta_p - \beta_q$ is of degree 0, we do not need to consider all double leaves. Instead, we only need to consider those whose degrees are non-positive and have even parity.
                As explained in Example~\ref{ex:H3_hom_calc}, this leaves 189027 double leaves to consider. 
                \item The golden realization from Example~\ref{ex:golden_realization} turns out to be somewhat universal amongst $H_3$ realizations. This allows us to restrict to just a 3-dimensional realization.
                From this, we can write explicit bases for the degree $d$ component of $R$ for the finitely many degrees allowed by the defects of the source and target subexpressions. 
                \item The localization matrices have strong cellular properties. 
                By abusing this, we can construct a recursive algorithm for solving the system of equations \eqref{eq:big_eqn_system}.
        \end{enumerate}

        \subsection{Localizing light leaves}

        \begin{lemma}\label{lem:column_of_loc_matrix_for_LL}
                If a column $\uc$ of $\Lambda (\LL_{\ux, \ue})$ is nonzero, then $\uc \leq \ue$.
        \end{lemma}
        \begin{proof}
                Write $\ux = (s_1, \ldots, s_k)$ and set $\uw = \rex (\ux^{\ue})$.
                Assume that $\Lambda (\LL_{\ux, \ue})_{\bullet, \uc} \neq 0$. By the definition of the path dominance order, we need to show that $\ux_i^{\uc} \leq \ux_i^{\ue}$ for all $i=1,\ldots, k$. 
                Fix $i \in \{1, \ldots, k \}$ and let $\underline{r}_i$ be the reduced expression for $\ux_i^{\ue}$ occurring after the $i$-th step of the light-leaf construction. Recall the map $\LL_{\ux, \ue}^i : B_{\ux} \to B_{\underline{r}_i s_{i+1} \ldots s_k}$.
                By construction, we can write 
                \begin{equation}\label{eq:column_of_loc_matrix_for_LL_1}
                        \LL_{\ux, \ue}^i = F_i \otimes \id_{B_{(s_{i+1}, \ldots, s_k)}}
                \end{equation}
                where $F_i : B_{(s_1, \ldots, s_i)} \to B_{\underline{r}_i}$. Now write $\LL_{\ux, \ue} = G_i \circ \LL_{\ux, \ue}^i$ where $G_i : B_{\underline{r}_i s_{i+1} \ldots s_k} \to B_{\uw}$ is the morphism obtained by finishing the light leaf construction.
                By functorality of localization, we have that $\Lambda (\LL_{\ux, \ue}) = \Lambda (G_i) \Lambda (\LL_{\ux, \ue}^i)$.
                Since the column $\uc$ of $\Lambda (\LL_{\ux, \ue})$ is nonzero, elementary linear algebra says that the column $\uc$ of $\Lambda (\LL_{\ux, \ue}^i)$ is nonzero.
                By Lemma~\ref{lem:prod_of_loc_mats}, the tensor product decomposition of \eqref{eq:column_of_loc_matrix_for_LL_1} implies that the column $\uc_{\leq i}$ of $\Lambda (F_i)$ is nonzero.
                In particular, we can find a subexpression $\uh$ of $\underline{r}_i$ such that $\Lambda (F_i)_{\uh, \uc_{\leq i}} \neq 0$, and hence, $\ux_i^{\uc} = \underline{r}_i^{\uh}$.
                Since $\underline{r}_i$ is a reduced expression for $\ux_i^{\ue}$, we must have that $\underline{r}_i^{\uh} \leq \ux_i^{\ue}$, and so $\ux_i^{\uc} \leq \ux_i^{\ue}$.
        \end{proof}

        \begin{remark}
                Lemma~\ref{lem:column_of_loc_matrix_for_LL} can be thought of as a mild generalization of path dominance upper-triangularity for light leaves (see \cite[Proposition 6.6]{EW}).
                In particular, if $\uc$ is a subexpression of $\ux$, then path dominance upper-triangularity says that if $\Lambda (\LL_{\ux, \ue})_{\mathbf{1}, \uc} \neq 0$, then $\uc \leq \ue$.
        \end{remark}

        \begin{lemma}[{\cite[Proposition 6.6]{EW}}]\label{lem:coeff_of_LL_at_e}
                Let $\ue$ be a subexpression of $\ux \in \Exp (W)$.
                Then
                \[\Lambda (\LL_{\ux, \ue})_{\mathbf{1}, \ue} = \gamma_{\ux} (\ue) \neq 0.\]
                In particular, this entry of the localization matrix is independent of the rigidification datum. 
        \end{lemma}

        \begin{proposition}\label{prop:loc_of_dl}
                Let $\ux, \uy \in \Exp (W)$ and $\ue \in \Sub (\ux)$, $\uf \in \Sub (\uy)$ such that $\ux^{\ue} = w = \uy^{\uf}$. Then
                \[\Lambda (\dLL_{\ue, \uf}^w)_{\ur, \uc} = \frac{1}{\zeta_{\uy} (\ur)} \sum_{\uh \in \Sub (\rex (w))} \zeta_{\rex (w)} (\uh) (\Lambda (\LL_{\uy, \uf})_{\uh, \ur} ) \Lambda (\LL_{\ux, \ue})_{\uh, \uc}.\]
        \end{proposition}
        \begin{proof}
                This is an immediate consequence of Lemma~\ref{lem:loc_mat_of_dual}.
        \end{proof}

        \subsection{Localization coordinates}

        For the remainder of the section, we will fix expressions $\ux$ and $\uy$.

        Define the \emph{height} of a subexpression $\ue$ for $\ux \in \Exp (W)$ by
        \[ h_{\ux} (\ue) = \sum_{i=1}^{\ell (\ux)} \ell (\ux_i^{\ue}). \]
        If $\uc \leq \ue$, then $h_{\ux} (\uc) \leq h_{\ux} (\ue)$. If we further assume that $\uc \neq \ue$, then there exists some $i$ such that $\ux_i^{\uc} < \ux_i^{\ue}$ which implies that $h_{\ux} (\uc) < h_{\ux} (\ue)$.

        For the algorithm that will be presented in the following section, it is convenient to give the set of double leaves from $B_{\ux}$ to $B_{\uy}$ a total order. 
        In particular, enumerate the double leaves $\dLL_i = \dLL_{\ue^i, \uf^i}^{w_i}$ such that
        \begin{enumerate}
                \item for $i < j$, then $\ell (w_j) \leq \ell (w_i)$;
                \item if $i < j$ and $\ell (w_i) = \ell (w_j)$, then $h_{\ux} (\ue^j) + h_{\uy} (\uf^j) \leq h_{\ux} (\ue^i) + h_{\uy} (\uf^i)$.
        \end{enumerate}
        Any enumeration that abides by these requirements is allowed, and there are typically many choices of such an enumeration. Let $N$ be the total number of double leaves from $B_{\ux}$ to $B_{\uy}$, so that our double leaves are ordered $\dLL_1, \dLL_2, \ldots, \dLL_N$.

        For a morphism $\varphi : B_{\ux} \to B_{\uy}$, define 
        \[\lambda_i (\varphi) \coloneq \Lambda (\varphi)_{\uf^i, \ue^i}.\]

        \begin{theorem}\label{thm:lower_triang_of_coord_functionals}
                If $\lambda_i (\dLL_j) \neq 0$, then $j \leq i$. Moreover,
                \begin{equation}\label{eq:lower_triang_of_coord_functionals_0}
                \lambda_i (\dLL_i) = \frac{\zeta_{\rex (w_i)} (\mathbf{1})}{\zeta_{\uy} (\uf^i)} \gamma_{\ux} (\ue^i) \gamma_{\uy} (\uf^i) \in Q^{\times}.
                \end{equation}
        \end{theorem}
        \begin{proof}
                Assume that $\lambda_i (\dLL_j) \neq 0$.
                 By Proposition~\ref{prop:loc_of_dl}, one has that 
                \begin{equation}\label{eq:lower_triang_of_coord_functionals_1}
                        \lambda_i (\dLL_j) = \frac{1}{\zeta_{\uy} (\uf^i)} \sum_{\uh \in \Sub (\rex (w_j))} \zeta_{\rex (w_j)} (\uh) (\Lambda (\LL_{\uy, \uf^j})_{\uh, \uf^i} ) (\Lambda (\LL_{\ux, \ue^j})_{\uh, \ue^i} ).
                \end{equation}
                If \eqref{eq:lower_triang_of_coord_functionals_1} is nonzero, then Lemma~\ref{lem:column_of_loc_matrix_for_LL} implies that $\ue^i \leq \ue^j$ and $\uf^i \leq \uf^j$.
                In particular, $w_i = \ux^{\ue^i} \leq \ux^{\ue^j} = w_j$. If $w_i \neq w_j$, then $\ell (w_i) < \ell (w_j)$, and so $j \leq i$.
                If $w_i = w_j$ but $i \neq j$, then either $\ue^i \neq \ue^j$ or $\uf^i \neq \uf^j$. 
                In particular, $h_{\ux} (\ue^i) + h_{\uy} (\uf^i) < h_{\ux} (\ue^j) + h_{\uy} (\uf^j)$, and so $j \leq i$.

               We will now consider $\lambda_i (\dLL_i)$.
                For $\Lambda (\LL_{\ux, \ue^i})_{\uh, \ue^i} \neq 0$, we must have that $\rex (w_i)^{\uh} = \ux^{\ue^i} = w_i$. This can only happen when $\uh = \mathbf{1}$.
                As a result, \eqref{eq:lower_triang_of_coord_functionals_1} simplifies to
                \begin{align*}
                        \lambda_i (\dLL_i) &= \frac{\zeta_{\rex (w_i)} (\mathbf{1})}{\zeta_{\uy} (\uf^i)}  (\Lambda (\LL_{\uy, \uf^i})_{\mathbf{1}, \uf^i}) (\Lambda (\LL_{\ux, \ue^i})_{\mathbf{1}, \ue^i}) \\
                        &= \frac{\zeta_{\rex (w_i)} (\mathbf{1})}{\zeta_{\uy} (\uf^i)} \gamma_{\ux} (\ue^i) \gamma_{\uy} (\uf^i),
                \end{align*}
                where the second equality follows from Lemma~\ref{lem:coeff_of_LL_at_e}.
                Note that all factors in \eqref{eq:lower_triang_of_coord_functionals_0} are roots or products of roots. In particular, their inverses exist in $Q$, and so, \eqref{eq:lower_triang_of_coord_functionals_0} is invertible in $Q$.
        \end{proof}

        \begin{corollary}\label{cor:recursion_for_coeffs}
                Let $\varphi : B_{\ux} \to B_{\uy}$. We can write
                \begin{equation}\label{eq:recursion_for_coeffs_1}
                        \varphi = \sum_{i=1}^N c_i \dLL_i
                \end{equation}
                for some $c_i \in R$. Then
                \begin{equation}\label{eq:recursion_for_coeffs_2}
                        c_i = \frac{\lambda_i (\varphi) - \sum_{j < i} c_j \lambda_i (\dLL_j)}{\lambda_i (\dLL_i)}.
                \end{equation}
        \end{corollary}
        \begin{proof}
                We can apply $\lambda_i$ to \eqref{eq:recursion_for_coeffs_1} and use Theorem~\ref{thm:lower_triang_of_coord_functionals} to deduce that
                \[\lambda_i (\varphi) = c_i \lambda_i (\dLL_i) + \sum_{j < i } c_j \lambda_i (\dLL_j).\]
                Since $\lambda_i (\dLL_i) \neq 0$ by Theorem~\ref{thm:lower_triang_of_coord_functionals}, we can rearrange the above equation to deduce \eqref{eq:recursion_for_coeffs_2}.
        \end{proof}

        \begin{corollary}[Equality Testing]\label{cor:eq_testing}
                Let $\varphi, \psi : B_{\ux} \to B_{\uy}$ be two morphisms of Abe--Bott--Samelson bimodules.
                Then $\varphi = \psi$ if and only if $\lambda_i (\varphi) = \lambda_i (\psi)$ for all $i=1,\ldots, N$.
        \end{corollary}
        \begin{proof}
                This is an obvious consequence of Theorem~\ref{thm:lower_triang_of_coord_functionals}.
        \end{proof}

        \begin{remark}
                The utility of Corollary~\ref{cor:recursion_for_coeffs} is that it gives a recursive algorithm for finding the coefficients appearing in a double leaves expansion.
                More precisely, the lower triangularity of Theorem~\ref{thm:lower_triang_of_coord_functionals} allows one to perform an efficient forward substitution of the system of equations
                \[\Lambda (\varphi) = \sum_{i=1}^N c_i \Lambda (\dLL_i).\]

                After these localization coefficients are computed, determining the $c_i$'s amounts to simple forward substitution. For example,
                \begin{align*}
                        c_1 &= \frac{\lambda_1 (\varphi)}{\lambda_1 (\dLL_1)} \\
                        c_2 &= \frac{\lambda_2 (\varphi) - c_1 \lambda_2 (\dLL_1)}{\lambda_2 (\dLL_2)} \\
                        c_3 &=  \frac{\lambda_3 (\varphi) - c_1 \lambda_3 (\dLL_1) - c_2 \lambda_3 (\dLL_2)}{\lambda_3 (\dLL_3)}.
                \end{align*}
        \end{remark}

        \begin{remark}\label{rem:generic_to_degree_d}
            The results of this section can be made more useful by filtering based on degree. 
            Let $\varphi : B_{\ux} \to B_{\uy}$ be a degree $d$ morphism.
            For each $i=1,\ldots, N$, set $d_i$ to be the degree of $\dLL_i$. We can consider the set
            \[I_d = \{ i=1,\ldots, N \mid d-d_i \in 2\Z_{\geq 0} \}.\]
            By Corollary~\ref{cor:deg_0_mor_and_basis}, the only double leaves which have a non-zero coefficient in $\varphi$ are those indexed by $I_d$.
            In particular, we can write
            \[\varphi = \sum_{i \in I_d} c_i \dLL_i.\]
             There are variations of Corollaries~\ref{cor:recursion_for_coeffs} and \ref{cor:eq_testing} with this setup by replacing the indexing set $\{1, \ldots, N\}$ by $I_d$.
        \end{remark}

       \section{The Coefficient Algorithm}

       \subsection{Computing entries of the localization matrix}\label{ss:entries_of_loc_mat}

       Let $\varphi : B_{\ux} \to B_{\uy}$ be a morphism of Abe--Bott--Samelson bimodules. We can write
       \[\varphi = \sum_{i=1}^n a_i \varphi_i,\]
       where $a_i \in \k$ and $\varphi_i$ is a morphism corresponding to an Elias--Williamson graph.
       In particular, we can write $\varphi_i = g_{i, r_i} \circ \ldots \circ g_{i, 1}$ where each $g_{i, j}$ is an Elias--Williamson generator tensored on either side with identity morphisms.
       The localization matrices $\Lambda (g_{i, j})$ are very easy to compute using the tables from \S\ref{ss:diag_localization}.
       However, it is computationally expensive to compute the localization matrix of $\varphi$ by taking a product of sparse matrices
       \[\Lambda (\varphi) = \sum_{i=1}^n a_i \Lambda (g_{i, r_i}) \ldots \Lambda (g_{i, 1}).\]
       There is an easy solution to this problem though. If we want to know the $\uc \in \Sub (\ux)$-column of $\Lambda (\varphi_i)$, we can instead compute
       \[\Lambda (\varphi_i)_{\bullet, \uc} = \Lambda (g_{i, r_i}) (\Lambda (g_{i, r_i -1}) ( \ldots (\Lambda (g_{i,1})_{\bullet, \uc}))).\]
       This is much faster than multiplying the whole sparse matrices.

        \subsection{Choice of \texorpdfstring{$\LL$}{LL}-datum}\label{ss:ll_datum}

        The choice of $\LL$-datum can essentially be made arbitrarily, but to ensure that the algorithm is reproducible, 
        we outline the choice that is used in the implemented form of our algorithm.

        To each $w \in W$, we need to assign a reduced word $\rex (w)$ to $w$.
        The collection of reduced expressions $\textnormal{Rex} (w)$ can be equipped with the lexicographic order.
        First, we fix a total ordering on $S$, denoted $<$. Set $k = \ell (w)$.
        Let $\ux = (s_1, \ldots, s_k)$ and $\uy = (t_1, \ldots, t_k)$. 
        Let $i$ be the first place in which $s_i \neq t_i$. We declare $\ux < \uy$ if $s_i < t_i$ and $\ux > \uy$ if $s_i > t_i$.
        This makes $\textnormal{Rex} (w)$ into a totally ordered set. 
        We define $\rex (w)$ to be the minimal element in $\textnormal{Rex} (w)$ under this total order.

        Next, to each $w \in W$ and $s \in S$ such that $ws < w$, we need to assign a reduced word $\rex_s (w)$ of $w$ ending in $s$.
        To do this, we simply define $\rex_s (w) = \rex (ws) s$.

        Finally, between any two reduced expressions $\ux$ and $\uy$ of $w$, we need to assign a path $p_{\ux, \uy} $ from $\ux$ to $\uy$ in the rex graph of $w$.
        Let $P_{\ux, \uy}$ be the set of all paths from $\ux$ to $\uy$. We denote by $P_{\ux, \uy}^{\min}$ the subset of $P_{\ux, \uy}$ consisting of paths whose lengths are minimal and set $K$ for this minimum length.
        We can also equip $P_{\ux, \uy}^{\min}$ with the lexicographic order. In particular, any path $p \in P_{\ux, \uy}^{\min}$ corresponds to a sequence $ \sigma (p) = (\sigma_1 (p), \ldots, \sigma_K (p))$ where $\sigma_i (p)$ denotes the (left-most) position of the braid move applied on the $i$-th edge.
        Let $p,q \in P_{\ux, \uy}^{\min}$ be distinct paths, and take $i$ to be the first place in which $\sigma_i (p) \neq \sigma_i (q)$. 
        We declare $p < q$ if $\sigma_i (p) < \sigma_i (q)$, and $p > q$ otherwise.
        This endows $P_{\ux, \uy}^{\min}$ with a total order. We then define $p_{\ux, \uy}$ to be the minimal element in $P_{\ux, \uy}^{\min}$ under this order.

        These choices are used with the intermediate-expression convention of \S\ref{ss:light_leaves}. In particular, no rex move is inserted at a $U1$ step or after a $D0$ merge, and the final target is rewritten to $\rex (w_k)$ only after the last step of the light-leaf construction. 

        \subsection{The coefficient algorithm}

        \begin{algorithm}[The Coefficient Algorithm]\label{algo:coeff_algo}
                Let $\varphi : B_{\ux} \to B_{\uy}$ be a morphism of Abe--Bott--Samelson bimodules of homogeneous degree $d$.
                
                Consider the collection of double leaves from $B_{\ux}$ to $B_{\uy}$ of degree $\leq d$ and such that the degree has the same parity as $d$.
                Enumerate this collection by $\dLL_1, \dLL_2, \ldots, \dLL_{N_d}$ where $\dLL_i = \dLL_{\ue^i, \uf^i}^{w_i}$ such that 
                \begin{enumerate}[label=(\alph*)]
                        \item for $i < j$, then $\ell (w_j) \leq \ell (w_i)$;
                        \item for $i < j$ and $\ell (w_i) = \ell (w_j)$, then $h_{\ux} (\ue^j) + h_{\uy} (\uf^j) \leq  h_{\ux} (\ue^i) + h_{\uy} (\uf^i)$.
                \end{enumerate}
                Any enumeration satisfying (a) and (b) is allowed.

                Write
                \[\varphi = \sum_{i=1}^{N_d} c_i \dLL_i\]
                for some $c_i \in R$. 
                \begin{enumerate}   
                        \item If $N_d = 0$, then $\varphi = 0$, so we may assume that $N_d > 0$.   
                        \item For $i=1, \ldots, N_d$, we solve  the coefficients $c_i$ as follows:
                        \[c_i = \frac{\lambda_i (\varphi) - \sum_{j < i} c_j \lambda_i (\dLL_j)}{\lambda_i (\dLL_i)}.\]
                        The terms appearing on the right-hand side can be computed as needed:
                        \begin{enumerate}
                                \item Compute the terms $\lambda_i (\varphi) = \Lambda (\varphi)_{\uf^i, \ue^i}$ using the observations from \S\ref{ss:entries_of_loc_mat}.
                                \item For the diagonal entries, 
                                \begin{equation}\label{eq:coeff_algo_1}
                                 \lambda_i (\dLL_i) = \frac{\zeta_{\rex(w_i)} (\mathbf{1})}{\zeta_{\uy} (\uf^i)} \gamma_{\ux} (\ue^i) \gamma_{\uy} (\uf^i).
                                \end{equation}
                                 \item For $j < i$, we only need to compute $\lambda_i (\dLL_j)$ when $c_j \neq 0$. In this case, we compute these values by the equation
                                 \begin{equation}\label{eq:coeff_algo_2}
                                 \lambda_i (\dLL_j) = \frac{1}{\zeta_{\uy} (\uf^i)} \sum_{\uh \in \Sub (\rex (w_j))} \zeta_{\rex (w_j)} (\uh) (\Lambda (\LL_{\uy, \uf^j})_{\uh, \uf^i}) (\Lambda (\LL_{\ux, \ue^j})_{\uh, \ue^i}).
                                 \end{equation}
                                 This amounts to computing a weighted dot product of the $\uf^i$-column vector of $\Lambda (\LL_{\uy, \uf^j})$ and the $\ue^i$-column vector of $\Lambda (\LL_{\ux, \ue^j})$.
                                 These columns of the localization matrix can be computed as needed using the observations from \S\ref{ss:entries_of_loc_mat}. Some extra details on speeding up this computation are given in Remark~\ref{rem:computing_dll_coeffs} below.
                        \end{enumerate}
                        \item After $i$ has been iterated from $1$ to $N_d$, all coefficients $c_1, \ldots, c_{N_d}$ have been calculated.
                \end{enumerate}
        \end{algorithm}
        \begin{proof}
                The validity of the algorithm is a culmination of Proposition~\ref{prop:loc_of_dl}, Theorem~\ref{thm:lower_triang_of_coord_functionals}, and Corollary~\ref{cor:recursion_for_coeffs}.
        \end{proof}

        \begin{remark}\label{rem:computing_dll_coeffs}
                In Algorithm~\ref{algo:coeff_algo}, we compute $\lambda_i (\dLL_j)$ for $j < i$ using Proposition~\ref{prop:loc_of_dl}.
                We could have alternatively computed these using the same method as $\lambda_i (\varphi)$; however, in implementing the algorithm there is a preference towards the former method.
                Namely, one can cache the columns of the localization matrices for the light leaves appearing in Proposition~\ref{prop:loc_of_dl}.
                Since different double leaves often share one of their light leaves, this saves significant computation. Likewise, some columns of $\Lambda (\LL_{\ux, \ue})$ can be skipped from the computation due to the path-dominance ordering of Lemma~\ref{lem:column_of_loc_matrix_for_LL}. 
        \end{remark}

        \begin{remark}
                Let $\varphi, \psi : B_{\ux} \to B_{\uy}$ be degree $d$ morphisms.
                By Corollary~\ref{cor:eq_testing}, it suffices to check that $\lambda_i (\varphi) = \lambda_i (\psi)$ for all $i=1,\ldots, N_d$.
                This gives a means of verifying the output of Algorithm~\ref{algo:coeff_algo}.
                Alternatively, one can compare the localization matrices of both.
        \end{remark}

        \subsection{Polynomial interpolation}

        We now assume that $\k$ is a characteristic 0 field. As a result, we can regard any rational number as an element of $\k$.

        One of the main problems with Algorithm~\ref{algo:coeff_algo} is that much of the arithmetic takes place in $R$ or $Q$.
        In general, computers are much more efficient at computing in the ground field $\k$ instead of a polynomial ring $R$, especially when $\k$ is something like $\Q$ or $\Q (\phi)$.
        In practice, this limitation makes Algorithm~\ref{algo:coeff_algo} quite computationally expensive, including for our example of the $H_3$ Zamolodchikov relation.
        We can fix this by using polynomial interpolation, i.e., evaluating polynomials in $R$ at a finite number of points.

        The use of polynomial interpolation for efficient polynomial arithmetic is far from novel.
        For example, the Fast Fourier Transform uses evaluation and interpolation at all $N$-th roots of unity, for a suitable $N$, to perform fast polynomial multiplication.
        This reduces the time complexity of polynomial multiplication from $O(n^2)$ to $O (n \log n)$. 
        The key observation in using polynomial interpolation for fast polynomial arithmetic is that one needs to choose the evaluation points carefully. 

        Let $A_n = \k [x_1, \ldots, x_n]$. We set $A_n^{\leq d}$ for the $\k$-vector space of polynomials of total degree at most $d$ and $A_n^d$ for the $\k$-vector space of homogeneous degree $d$ polynomials. 
        For a tuple $a = (a_1, \ldots, a_n) \in \k^n$ and $p \in A_n$, we write
        \[p (a) = p(a_1, \ldots, a_n) \]
        for the element of $\k$ obtained by evaluating $x_i = a_i$. 

        Consider tuples $a = (a_1, \ldots, a_n) \in \Z_{\geq 0}^n$. 
        We make $\Z_{\geq 0}^n$ into a poset by declaring $a \leq b$ if and only if $a_i \leq b_i$ for all $i=1,\ldots, n$. 
        We write
        \[x^{a} = \prod_{i=1}^n x_i^{a_i}, \qquad \abs{a} = \sum_{i=1}^n a_i \qquad\text{and}\qquad a! =  \prod_{i=1}^n a_i ! .\]
        For $a \in \Z_{\geq 0}^n$ and $b \in \Z_{\geq 0}^n$, we define the multi-index binomial coefficients
        \[\binom{b}{a} = \prod_{i=1}^n \binom{b_i}{a_i}. \]
        Note that $\binom{b}{a} \neq 0$ implies that $a \leq b$.

        Recall the usual binomial polynomials,
        \[\binom{x}{m} = \frac{1}{m!} x (x-1) \ldots (x-m+1) \in \Q [x]\]
        for any $m \in \Z_{\geq 0}$. We can then generalize these to multi-index binomial polynomials,
        \[B_{a}^{b} (x) = \prod_{i=1}^n \binom{x_i - b_i}{a_i} \in A_n.\]
        By construction, $B_{a}^{b} (x)$ has degree $\abs{a}$, and so, $B_{a}^{b} (x) \in A_n^{\leq \abs{a}}$.
        By expanding $B_{a}^{b} (x)$, we can write 
        \begin{equation}\label{eq:leading_term_of_binom_poly}
                B_{a}^{b} (x) = \frac{x^{a}}{a !} + r (x)
        \end{equation}
        where $r(x)$ is a polynomial of degree $< \abs{a}$.
        
        Fix $d \geq 0$ and define a subset of $\Z_{\geq 0}^n$,
        \[L_d = \left\{ a \in \Z_{\geq 0}^n \mid \abs{a} \leq d \right\}.\]
        It is well-known that $L_d$ consists of exactly $\binom{n+d}{n}$ points, and that this matches the dimension of $A_n^{\leq d}$.
        Fix $\tilde{a} \in \Z_{\geq 0}^n$. We then define
        \[ X_d^{\tilde{a}} = \tilde{a} + L_d.\]
        Note that $X_i^{\tilde{a}} \subset X_{i+1}^{\tilde{a}}$ for all $i$ and that $\# X_i^{\tilde{a}} = \binom{n+i}{n}$.

        The following proposition is a special case of multivariant Newton interpolation on lower sets developed in \cite{Kuntzmann} and \cite[Theorem 1]{DF}. 
        For the sake of completeness, to improve the exposition, and to record the explicit inverse used in the algorithm, we have included a proof.

        \begin{proposition}\label{prop:interpolation}
                Let $\tilde{a} \in \Z_{\geq 0}^n$.
                \begin{enumerate}
                        \item There is an isomorphism of vector spaces
                        \[\ev_d^{\tilde{a}} : A_n^{\leq d} \stackrel{\sim}{\to} \Fun (X_d^{\tilde{a}}, \k), \qquad p \mapsto (a \mapsto p (a)).\]
                        \item For a function $\mu : X_d^{\tilde{a}} \to \k$, we set
                        \[b_{a}^{\tilde{a}} (\mu) = \sum_{b \leq a} (-1)^{\abs{a} - \abs{b}} \binom{a}{b} \mu (\tilde{a} + b),\]
                        Then
                        \begin{equation}\label{eq:interpolation_1}
                        (\ev_d^{\tilde{a}})^{-1} (\mu) (x_1, \ldots, x_n) = \sum_{a \in L_d} b_{a}^{\tilde{a}} (\mu) B_{a}^{\tilde{a}}.
                        \end{equation}
                \end{enumerate}
        \end{proposition}
        \begin{proof}
                By \eqref{eq:leading_term_of_binom_poly}, the set $\{ B_{a}^{\tilde{a}} \}_{a \in L_d}$ is a basis for $A_n^{\leq d}$.
                Moreover, under the standard basis of $A_n^{\leq d}$ by monomials ordered under the graded lexicographic order, the change of basis matrix is triangular.

                At a point $\tilde{a} + b \in X_d^{\tilde{a}}$, we have
                \[B_{a}^{\tilde{a}} (\tilde{a} + b) = \binom{b}{a}.\]
                In particular, $B_{a}^{\tilde{a}} (\tilde{a} + b) = 0$ unless $a \leq b$ and $B_{a}^{\tilde{a}} (\tilde{a} + a) = 1$.
                This implies that the vectors $\ev_d^{\tilde{a}} (B_{a}^{\tilde{a}} (x))$ are linearly independent, and hence $\ev_d^{\tilde{a}}$ is injective.
                Both the domain and codomain have dimension $\binom{n+d}{n}$, and so, $\ev_d^{\tilde{a}}$ is in fact an isomorphism.

                To verify \eqref{eq:interpolation_1}, we evaluate the right-hand side at $\tilde{a} + c$. This gives
                \begin{align*}
                        \left( \sum_{a \in L_d} b_{a}^{\tilde{a}} (\mu) B_{a}^{\tilde{a}} \right) (\tilde{a} + c) &= \sum_{a \in L_d} b_{a}^{\tilde{a}} (\mu) B_{a}^{\tilde{a}} (\tilde{a} + c) \\
                        &= \sum_{a \in L_d} b_{a}^{\tilde{a}} (\mu) \binom{c}{a} \\ 
                        &= \sum_{b \leq a \leq c}  (-1)^{\abs{a} - \abs{b}} \binom{c}{a} \binom{a}{b} \mu (\tilde{a} + b) \\
                        &= \sum_{b \leq a \leq c}  (-1)^{\abs{a} - \abs{b}} \binom{c}{b} \binom{c - b}{a - b} \mu (\tilde{a} + b).
                \end{align*}
                For a fixed $b$, we set $e = a - b$. The condition on the sum becomes $0 \leq e \leq c - b$ and $\abs{e} = \abs{a} - \abs{b}$.
                Therefore, we have that 
                \begin{equation}\label{eq:interpolation_2}
                        \sum_{b \leq a \leq c}  (-1)^{\abs{a} - \abs{b}} \binom{c}{b} \binom{c - b}{a - b} \mu (\tilde{a} + b) = \sum_{b \leq c} \mu ( \tilde{a} + b) \binom{c}{b} \sum_{e \leq c - b} (-1)^{\abs{e}} \binom{c - b}{e} 
                \end{equation}
                The inner sum on the right-hand side of \eqref{eq:interpolation_2} can be computed by a multi-index version of the binomial theorem. 
                Namely, it vanishes unless $b = c$ in which case it is 1. Therefore,
                \[\sum_{b \leq c} \mu ( \tilde{a} + b) \binom{c}{b} \sum_{e \leq c - b} (-1)^{\abs{e}} \binom{c - b}{e} = \mu (\tilde{a} + c).\]
        \end{proof}

        Dehomogenization gives an isomorphism
        \[A_{n+1}^d \stackrel{\sim}{\to} A_n^{\leq d}, \qquad p \mapsto p(x_1, \ldots, x_n, 1).\]
        Its inverse is given by homogenization,
        \[p \mapsto x_{n+1}^d p \left( \frac{x_1}{x_{n+1}}, \ldots, \frac{x_{n}}{x_{n+1}} \right).\]

        For $\tilde{a} \in \Z_{\geq 0}^n$, define
        \[Y_d^{\tilde{a}} = \{  (a, 1) \in \Z_{\geq 0}^{n+1} : a \in X_d^{\tilde{a}} \}.\]

        \begin{corollary}\label{cor:homog_interpolation}\footnote{Corollary \ref{cor:homog_interpolation} was discovered by ChatGPT without proof in the case of $n=2$ 
                and $\tilde{a} = ( 2,4)$ when it attempted to improve upon Algorithm \ref{algo:coeff_algo}. 
                The author believes that ChatGPT's discovery was based on earlier work of Dyn--Floater \cite{DF}; although, ChatGPT did not provide any references, so this is mostly speculation. }
                Let $\tilde{a} \in \Z_{\geq 0}^n$.
                \begin{enumerate}
                        \item There is an isomorphism of vector spaces
                        \[\ev_d^{\tilde{a}} : A_{n+1}^d \stackrel{\sim}{\to} \Fun (Y_d^{\tilde{a}}, \k), \qquad p \mapsto (a \mapsto p(a)).\]
                        \item For a function $\mu : Y_d^{\tilde{a}} \to \k$, we set
                        \[\overline{b}_{a}^{\tilde{a}} (\mu) = \sum_{b \leq a} (-1)^{\abs{a} - \abs{b}} \binom{a}{b} \mu (\tilde{a} + b, 1).\]
                        Then 
                        \[(\ev_d^{\tilde{a}})^{-1} (\mu) (x_1, \ldots, x_{n+1}) = \sum_{a \in L_d} \overline{b}_{a}^{\tilde{a}} (\mu) x_{n+1}^d B_{a}^{\tilde{a}} \left( \frac{x_1}{x_{n+1}}, \ldots, \frac{x_n}{x_{n+1}} \right).\]
                \end{enumerate}
        \end{corollary}
        \begin{proof}
                By construction the following diagram commutes,
                \[\begin{tikzcd}
A_{n+1}^d \arrow[d, "\ev_d^{\tilde{a}}"'] \arrow[r, "\sim"] & A_n^{\leq d} \arrow[d, "\ev_d^{\tilde{a}}"] \\
{\Fun (Y_d^{\tilde{a}}, \k)} \arrow[r, "\sim"]              & {\Fun (X_d^{\tilde{a}}, \k),}               
\end{tikzcd}\]
where the top horizontal map is given by dehomogenization and the bottom horizontal map is induced from the set-theoretic bijection $X_d^{\tilde{a}} \to Y_d^{\tilde{a}}$ taking $a$ to $(a, 1)$.
The corollary then follows from translating the statement of Proposition~\ref{prop:interpolation} along this commutative diagram.
        \end{proof}

        \begin{example}\label{ex:points_for_H3}
                Let $n = 2$ and set $\tilde{a} = (1,1)$. 
                Then
                \begin{align*}
                        Y_0^{(1,1)} &= \{ (1,1,1) \} \\
                        Y_1^{(1,1)} \setminus Y_0^{(1,1)} &= \{ (2,1,1), (1,2,1) \} \\
                        Y_2^{(1,1)} \setminus Y_1^{(1,1)} &= \{ (3,1,1), (2,2,1), (1,3,1) \} \\
                        Y_3^{(1,1)} \setminus Y_2^{(1,1)} &= \{ (4,1,1), (3,2,1), (2,3,1), (1,4,1) \} \\
                        Y_4^{(1,1)} \setminus Y_3^{(1,1)} &= \{ (5,1,1), (4,2,1), (3,3,1), (2,4,1), (1,5,1) \}.
                \end{align*}
        \end{example}

        \begin{remark}
                There are many other choices of interpolation that one can take for homogeneous polynomials of fixed degree $d$.
                For example, one could take the points
                \[\{ a \in \Z_{\geq 0}^{n+1} \mid \abs{a} = d \}.\]
                The advantage of Corollary~\ref{cor:homog_interpolation} is that for $d < d'$, the points we need to evaluate a polynomial at satisfy $Y_{d}^{\tilde{a}} \subset Y_{d'}^{\tilde{a}}$.
        \end{remark}

        \subsection{A coefficient algorithm via interpolation}

        Let $r$ be the dimension of $\fr{h}$ as a $\k$-vector space and assume that $r \geq 1$. 
        If $r = 0$, then $R = \k$ and no polynomial interpolation is needed.
        We fix a basis $x_1,  \ldots, x_r$ of $\fr{h}^*$ which induces an isomorphism
        \begin{equation}\label{eq:triv_of_R}
          R \cong \k [x_1, \ldots, x_r].  
        \end{equation}
        Fix some $\tilde{a} \in \Z_{\geq 0}^{r-1}$. For each $d \geq 0$, we define 
        \[ Y_d^{\tilde{a}} = \{ (\tilde{a} + a, 1) \in \Z_{\geq 0}^r : a \in \Z_{\geq 0}^{r-1}, \abs{a} \leq d \}.\]
        We can identify $R^{2d}$ with homogeneous polynomials of degree $d$ in the variables $x_1, \ldots, x_r$.
        By Corollary~\ref{cor:homog_interpolation}, evaluation defines an isomorphism 
        \begin{equation}
          \ev_d^{\tilde{a}} : R^{2d} \stackrel{\sim}{\to} \Fun (Y_d^{\tilde{a}}, \k).      
        \end{equation}
        Moreover, there is an explicit formula for the inverse (Corollary~\ref{cor:homog_interpolation} (2)).
        
        Let $\varphi : B_{\ux} \to B_{\uy}$ be an Abe--Bott--Samelson morphism of degree $d$. We restrict to the $N_d$ double leaves which could possibly contribute to the double-leaf expansion. If $N_d = 0$, then $\varphi = 0$, and so we will assume that $N_d > 0$.
        We will write $d_i$ for the degree of $\dLL_i$. By assumption, $d - d_i \geq 0$ and even. 
        Define
        \[K = \frac{1}{2} \left( \max_{i=1,\ldots, N_d} ( d-d_i) \right).\]
        We can find coefficients $c_i \in R$ such that 
        \[\varphi = \sum_{i=1}^{N_d} c_i \dLL_i.\]
        Since $\dLL_i$ has degree $d_i$, the coefficient is in $R^{d - d_i}$. 
        In particular, by Corollary~\ref{cor:homog_interpolation}, it suffices to compute $c_i$ at the points $Y_{K}^{\tilde{a}}$ for a fixed $\tilde{a} \in \Z_{\geq 0}^{r-1}$.

        At first glance, this seems to require substantially more computations than Algorithm~\ref{algo:coeff_algo}.
        However, once we compute $c_i$ at the points $Y_{(d-d_i)/2}^{\tilde{a}}$, the other points in $Y_{K}^{\tilde{a}} \setminus Y_{(d-d_i)/2}^{\tilde{a}}$ can be computed by Corollary~\ref{cor:homog_interpolation} (2).

        \begin{example}
                Let $W = H_3$ and $\fr{h}$ be the golden realization. 
                Consider the reduced expressions for the longest element $w_0$ in $W$, 
                \[ \ux = stsutsutustsutu \qquad\text{and}\qquad \uy = utustsutustusts.\]
                The table below shows how many points are involved for each degree involved in the computation of the coefficient expansion for a morphism $\varphi : B_{\ux} \to B_{\uy}$ of degree 0.
                The entries of the second column of the table can be computed using \eqref{eq:rk_of_homog_homs} (see also Example~\ref{ex:H3_hom_calc}).
                
                \begin{center}
                        \begin{tabular}{c | c | c | c}
                                \textbf{Double-leaf} & \textbf{Number of} & $\# Y_{-d_i/2}^{\tilde{a}}$ & \textbf{Number of }$c_i$ \\
                                \textbf{Degree} & \textbf{Double Leaves} & & \textbf{Point Evaluations} \\
                                \hline
                                $0$ & 168,263 & 1 & 168,263 \\
                                $-2$ & 19,732 & 3 & 59,196 \\
                                $-4$ & 1,031 & 6 & 6,186 \\
                                $-6$ & 1 & 10 & 10 \\
                                \hline
                                \textbf{Total} & 189,027 & & 233,655
                        \end{tabular}
                \end{center}

                In this case, Algorithm~\ref{algo:coeff_algo} requires solving 189,027 equations in $Q$; however, each of these computations is expensive due to arithmetic 
                in polynomial rings being comparably difficult to perform in computer-algebra systems. On the other hand, the interpolation algorithm solves $233,655$ equations over $\Q (\phi)$, using the nested subsets of $Y_3^{\tilde{a}}$.
                Once a coefficient is reconstructed, its values at later points are obtained by polynomial evaluation rather than solving another equation which is quite fast. Those inexpensive calculations are not counted in the above table.
        \end{example}

        Before we present the modified version of Algorithm~\ref{algo:coeff_algo} using the points in $Y_K^{\tilde{a}}$, it is worthwhile to mention the role that $\tilde{a}$ plays.
        In Algorithm~\ref{algo:coeff_algo}, we have to divide by 
        \[\lambda_i (\dLL_i) = \frac{\zeta_{\rex (w_i)} (\mathbf{1})}{\zeta_{\uy} (\uf^i)} \gamma_{\ux} (\ue^i) \gamma_{\uy} (\uf^i),\]
        which is a product of roots and their inverses in $Q$.
        In the point-method, we need to make sure that $\lambda_i (\dLL_i) (a) \neq 0$ for all $a \in Y_K^{\tilde{a}}$.
        The untranslated version where $\tilde{a} = (0,\ldots, 0)$ typically does not work.
        For example, suppose the chosen coordinates $x_1, \ldots, x_r$ are the simple roots and $r > 1$.
        Then $(0, \ldots, 0, 1) \in Y_K^{\mathbf{0}}$ lies on the root hyperplanes $x_i = 0$ for every $i < r$.
        
        \begin{definition}
                We say that $\tilde{a} \in \Z_{\geq 0}^{r-1}$ is $K$-\emph{regular} if $\alpha (a) \neq 0$ for all $\alpha \in \Phi$ and $a \in Y_K^{\tilde{a}}$.  
        \end{definition}

        \begin{lemma}\label{lem:existence_of_K_regular}
                Assume that $\Phi$ is finite. There exists some $\tilde{a} = (\tilde{a}_1, \ldots, \tilde{a}_{r-1}) \in \Z_{> 0}^{r-1}$ such that $\tilde{a}$ is $K$-regular.
        \end{lemma}
        \begin{proof}
                The case of $r=1$ is obvious. Assume that $r \geq 2$. Let $\alpha \in \Phi$ and $a = (a_1, \ldots, a_{r-1}) \in L_K$.
                Write
                \[\alpha = \sum_{i=1}^r b_i x_i.\]
                for some $b_i \in \k$. 
                We will evaluate $\alpha$ at the point $(t + a_1,t^2 + a_2, \ldots, t^{r-1} + a_{r-1}, 1)$ where $t$ is a formal variable,
                \[p_{\alpha, a} (t) \coloneq \alpha (t + a_1,t^2 + a_2, \ldots, t^{r-1} + a_{r-1}, 1) = b_r + \sum_{i=1}^{r-1} b_i a_i + \sum_{i=1}^{r-1} b_i t^i \in \k [t].\]
                Note that $p_{\alpha, a} \neq 0$. Indeed, if some $b_i \neq 0$ with $i < r$, it has a nonzero positive-degree coefficient.
                Otherwise, it is the nonzero constant $b_r$.
                
                Note that $p_{\alpha, a}$ has finitely many roots.
                Since $\Phi$ and $L_K$ are finite, the union of the roots of $p_{\alpha, a}$ for all $\alpha \in \Phi$ and $a \in L_K$ is also finite.
                In particular, there must be some positive integer $T$ which is not a root of $p_{\alpha, a}$. We set   
                \[\tilde{a} = (T, T^2, \ldots, T^{r-1}).\]
                This choice of $\tilde{a}$ is $K$-regular.
        \end{proof}

        \begin{example}\label{ex:K_reg}
                \begin{enumerate}
                        \item   Let $(W,S)$ be any Coxeter system of finite rank. Let $\fr{h}^* = \bigoplus_{s \in S} \R \alpha_s$.
                                Define elements $\alpha_s^{\vee} \in \fr{h} = (\fr{h}^*)^*$ by
                                \[\langle \alpha_t^{\vee}, \alpha_s \rangle = -2\cos(\pi/m_{st}),\]
                                where we set $m_{ss} = 1$ and $\pi/\infty = 0$ by convention. The realization $\fr{h}$ is called the \emph{dual geometric realization}.
                                By \cite[Theorem 5.4]{Humphreys}, any root $\alpha \in \Phi$ can be written as
                                \[\alpha = \sum_{s} c_s \alpha_s\]
                                with all $c_s \geq 0$ or all $c_s \leq 0$, with at least one nonzero coefficient.
                                Every point on $Y_K^{(1,\ldots, 1)}$ has strictly positive coordinates with respect to the simple-root basis. Since every root has either all positive or all negative coefficients, it evaluates non-trivially at every such point.
                                Therefore, the shift $\tilde{a} = (1,\ldots, 1) \in \Z_{ > 0}^{r-1}$ is $K$ regular for every $K \geq 0$.
                        \item   Let $W = H_3$ and $\fr{h}$ be the golden realization. The set of simple roots forms a basis for $\fr{h}^*$.
                                There are 30 roots listed below.
                                \[
                                \begin{matrix}
                                        \alpha_s & \alpha_t & \alpha_u \\
                                         -\alpha_s & -\alpha_t & -\alpha_u \\
                                        \alpha_t + \alpha_u & \alpha_s + \phi \alpha_t & \phi \alpha_s + \alpha_t \\ 
                                        -\alpha_t - \alpha_u & -\alpha_s - \phi \alpha_t & -\phi \alpha_s - \alpha_t \\
                                        \phi (\alpha_s + \alpha_t) & \phi \alpha_s + \alpha_t + \alpha_u & \alpha_s + \phi (\alpha_t + \alpha_u) \\
                                         - \phi(\alpha_s + \alpha_t) & -\phi \alpha_s - \alpha_t - \alpha_u & -\alpha_s - \phi ( \alpha_t + \alpha_u) \\
                                        \phi (\alpha_s + \alpha_t + \alpha_u) & \phi (\alpha_s + \alpha_t) + (\alpha_t + \alpha_u) &  \phi (\alpha_s + \alpha_t + \alpha_u) + \alpha_t \\
                                         -\phi (\alpha_s + \alpha_t + \alpha_u) & -\phi (\alpha_s + \alpha_t) - (\alpha_t + \alpha_u) & -\phi (\alpha_s + \alpha_t + \alpha_u) - \alpha_t \\
                                         \phi (\alpha_s + \alpha_t) + (\alpha_s + \alpha_t + \alpha_u) & \phi (\alpha_s + \alpha_t + \alpha_u) + (\alpha_s + \alpha_t) & \phi (\alpha_s + 2 \alpha_t + \alpha_u) +  \alpha_s \\
                                         -\phi (\alpha_s + \alpha_t) - (\alpha_s + \alpha_t + \alpha_u) & -\phi (\alpha_s + \alpha_t + \alpha_u) - (\alpha_s + \alpha_t) & -\phi (\alpha_s + 2 \alpha_t + \alpha_u) -  \alpha_s 
                                \end{matrix}
                                \] 
                                Every root in this list has either positive or negative coefficients under the embedding $\phi = (1 + \sqrt{5})/2 > 0$. The same argument as in (1) shows that the shift $\tilde{a} = (1,1)$ is $K$-regular for every $K \geq 0$.
                        \item Let $W = I_2 (\infty)$ with $\fr{h}^* = \Q e_1 \oplus \Q e_2$ and set $\alpha_s = e_1$ and $\alpha_t = e_2 - e_1$.
                        We can make $\fr{h} = (\fr{h}^*)^*$ into a realization with Cartan matrix
                        \[\begin{bmatrix} 2 & -2 \\ -2 & 2\end{bmatrix}.\]
                        Every sum of the form $e_1 + ne_2$ for $n \in \Z$ is a root for $\fr{h}$.
                        At every point $(a,1)$ with integer $a$, the root $e_1 - ae_2$ vanishes. In particular, $a \in \Z_{\geq 0}$ is never $K$-regular.
                        \end{enumerate}
        \end{example}

         \begin{algorithm}[The Coefficient Algorithm via Interpolation]\label{algo:coeff_algo_interpolate}
                Assume that $\Phi$ is finite.
                Let $\varphi : B_{\ux} \to B_{\uy}$ be a morphism of Abe--Bott--Samelson bimodules of homogeneous degree $d$. 

                Consider the collection of double leaves from $B_{\ux}$ to $B_{\uy}$ of degree $\leq d$ and such that the degree has the same parity as $d$.
                Enumerate this collection by $\dLL_1, \dLL_2, \ldots, \dLL_{N_d}$ where $\dLL_i = \dLL_{\ue^i, \uf^i}^{w_i}$ such that
                \begin{enumerate}[label=(\alph*)]
                        \item for $i < j$, then $\ell (w_j) \leq \ell (w_i)$;
                        \item for $i < j$ and $\ell (w_i) = \ell (w_j)$, then $h_{\ux} (\ue^j) + h_{\uy} (\uf^j) \leq  h_{\ux} (\ue^i) + h_{\uy} (\uf^i)$.
                \end{enumerate}
                Any enumeration satisfying (a) and (b) is allowed. We set $d_i$ to be the degree of $\dLL_i$. Define $k_i = (d-d_i)/2$.
                
                Write
                \[\varphi = \sum_{i=1}^{N_d} c_i  \dLL_i\]
                for some $c_i \in R^{d-d_i}$ which under the trivialization \eqref{eq:triv_of_R} is viewed as a homogeneous polynomial of degree $k_i$ in the variables $x_1, \ldots, x_r$.
                    
                \begin{enumerate}
                        \item If $N_d = 0$, then $\varphi = 0$. Otherwise, define $K = \max_{i=1,\ldots, N_d} k_i$ and pick some $K$-regular shift $\tilde{a}$. By convention, we define $Y_{-1}^{\tilde{a}} = \emptyset$.
                        \item For $k=0, 1, \ldots, K$, perform the following steps:
                        \begin{enumerate}
                            \item For each $a \in Y_k^{\tilde{a}} \setminus Y_{k-1}^{\tilde{a}}$, we induct on $i=1,\ldots, N_d$. At each index $i$, we solve for $c_i (a)$ as follows.
                            \begin{enumerate}
                                \item If $k_i < k$, the polynomial $c_i (x_1, \ldots, x_r)$ is already determined by interpolation. We compute $c_i (a)$ then by evaluating this known polynomial.
                                \item If $k_i \geq k$, we compute $c_i (a)$ by the equation
                                \begin{equation}\label{eq:coeff_algo_interpolate_1}
                                    c_i (a) = \frac{\lambda_i (\varphi) (a) - \sum_{j < i} c_j (a) \lambda_i (\dLL_j) (a)}{\lambda_i (\dLL_i) (a)}.
                                \end{equation}
                                If $c_j (a) = 0$ for $j < i$, we skip computing $\lambda_i (\dLL_j) (a)$. If $c_j (a) \neq 0$, then $\lambda_i (\dLL_j) (a)$ can be computed from \eqref{eq:coeff_algo_2}, but evaluated at the point $a$ (see also Remark~\ref{rem:computing_dll_coeffs}).
                                The columns of the localization matrices are not computed symbolically in $Q$, but instead, we work only with terms in $\k$ after applying the evaluation morphism $R \to \k$ at the point $a$.  
                                The value $\lambda_i (\dLL_i) (a)$ is computed from \eqref{eq:coeff_algo_1}, but likewise at the point $a$ instead of symbolically.
                            \end{enumerate}
                            \item After all points of $Y_k \setminus Y_{k-1}$ are iterated through, we can reconstruct the homogeneous polynomials $c_i$ of degree $k_i = k$ from their values at the points of $Y_{k}^{\tilde{a}}$ using Corollary~\ref{cor:homog_interpolation}.
                            Store this polynomial for later steps.
                        \end{enumerate}
                        \item After $k$ is iterated from $0$ to $K$, all polynomials $c_1, \ldots, c_{N_d}$ have been constructed.
                \end{enumerate}
        \end{algorithm}
        \begin{proof}
            By $K$-regularity of $\tilde{a}$, every value $\lambda_i (\dLL_i) (a) \neq 0$. As a result, \eqref{eq:coeff_algo_interpolate_1} is well-defined.
            The validity of the algorithm then follows from the same reasoning as Algorithm~\ref{algo:coeff_algo} using Corollary~\ref{cor:homog_interpolation}.
        \end{proof}

                We also have the following equality test of morphisms which is a polynomial-interpolation version of Corollary~\ref{cor:eq_testing}.

        \begin{proposition}\label{prop:eq_testing_interpolate}
                Assume that $\Phi$ is finite.
                Let $\varphi, \psi : B_{\ux} \to B_{\uy}$ be two morphisms of degree $d$.
                Enumerate the admissible double leaves from $B_{\ux}$ to $B_{\uy}$ by $\dLL_1, \ldots, \dLL_{N_d}$ as in Algorithm~\ref{algo:coeff_algo_interpolate}.
                Write $d_i$ for the degree of $\dLL_i$, and set $k_i = (d-d_i)/2$. Pick $\tilde{a} \in \Z_{\geq 0}^{r-1}$ which is $k_i$-regular for all $i$.
                Then $\varphi = \psi$ if and only if $\lambda_i (\varphi - \psi) (a) = 0$ for every $i$ and every $a \in Y_{k_i}^{\tilde{a}}$.
        \end{proposition}
        \begin{proof}
                Write $\varphi - \psi = \sum_i c_i \dLL_i$ where $c_i \in R^{2k_i}$.
                Suppose inductively that $c_j = 0$ for $j < i$. Theorem~\ref{thm:lower_triang_of_coord_functionals} gives
                \[\lambda_i (\varphi - \psi) = c_i \lambda_i (\dLL_i).\]
                By assumption, $\lambda_i (\varphi - \psi) (a) = 0$ and $\lambda_i (\dLL_i) (a) \neq 0$ for all $a \in Y_{k_i}^{\tilde{a}}$.
                As a result, $c_i (a) = 0$ for all $a \in Y_{k_i}^{\tilde{a}}$. Corollary~\ref{cor:homog_interpolation} then gives that $c_i = 0$. The converse is immediate. 
        \end{proof}

        \begin{remark}
                For a fixed input, the finiteness assumptions on $\Phi$ can be replaced by a finite denominator condition.
                Choose finite expressions in the generating morphisms for the input and for the light leaves being used.
                Let $\Phi'$ contain every root occurring in a denominator of the localized computations. In particular, include every root occurring in the numerator or denominator of $\lambda_i (\dLL_i)$.

                Only finitely many roots are needed for this fixed computation.
                The same argument as Lemma~\ref{lem:existence_of_K_regular} applied to $\Phi'$ shows that there exists some $\tilde{a} \in \Z_{\geq 0}^{r-1}$ such that $\alpha (a) \neq 0$
                for all $\alpha \in \Phi'$ and $a \in Y_K^{\tilde{a}}$. This is enough to make Algorithm~\ref{algo:coeff_algo_interpolate} and Proposition~\ref{prop:eq_testing_interpolate} work even when $\Phi$ is infinite.
        \end{remark}

       \subsection{Extension to general realizations}\label{ss:gen_realizations}

        While Algorithm~\ref{algo:coeff_algo} works over any noetherian domain, our polynomial interpolation algorithm makes crucial use of our ground ring being a characteristic 0 field.
        Even in characteristic 0, the number of points one has to check, for fixed maximal coefficient degree, has polynomial growth in the rank of $\fr{h}$, so interpolation favors small realizations as well.
        This does not mean that all hope is lost if one wants an efficient way to compute the double-leaf coefficients for general realizations.
        We will present such a strategy in this subsection which can sometimes be used. 

        \begin{definition}
                Let $\fr{h}$ be a balanced realization for $W$ over a noetherian domain $\k$. 
                An \emph{integral model} for $\fr{h}$ consists of the following:
                \begin{enumerate}
                 \item a characteristic-zero integral domain $\mathcal{Z}$ that is finitely generated as a $\Z$-algebra;
                \item a balanced realization $\fr{h}_{\mathcal{Z}}$ over $\mathcal{Z}$;
                \item a morphism of realizations $(\sigma, f) : \fr{h}_{\mathcal{Z}} \to \fr{h}$.
                \end{enumerate}
        \end{definition}

        \begin{proposition}\label{prop:integral_models}
                Every balanced realization admits an integral model.
        \end{proposition}
        \begin{proof}
                Let $\fr{h}$ be a balanced realization of $(W,S)$ over $\k$ with Cartan matrix $(a_{st})$. 

                Let $s,t \in S$ with $s \neq t$ and write $m = m_{st}$.
                We define a ring associated to the unordered pair $\{s,t\}$,
                \[A_{\{s,t\}} \coloneq \begin{cases} \Z[x_{st}, x_{ts}] / (x_{st}, x_{ts}) & m = 2, \\ \Z [x_{st}, x_{ts}]/(\Psi_m' (x_{st}), x_{st} - x_{ts}) & m > 2 \text{ odd}, \\ \Z [x_{st}, x_{ts}] / (\Psi_m (x_{st} x_{ts})) & m > 2 \text{ even}, \\ \Z [x_{st}, x_{ts}] & m=\infty. \end{cases}\]
                Here $\Psi_m$ and $\Psi_m'$ are the minimal polynomials from \eqref{eq:min_polys}.
                We then define a ring
                \[A \coloneq \bigotimes_{\{s,t\} \subset S} A_{\{s,t\}}.\]
                Each factor is free as a $\Z$-module. In the odd case, this follows from division by the monic polynomial $\Psi_m'$. In the even case, writing $d = \deg \Psi_m$, the monomials $x_{st}^i x_{ts}^j$ with $\min (i,j) < d$ form a $\Z$-basis of the quotient. Division by the leading monomial $x_{st}^d x_{ts}^d$ gives existence and uniqueness of the remainder. Hence $A$ is free over $\Z$.
                Define a Cartan matrix $(\tilde{a}_{st})$ associated to $A$ given by $\tilde{a}_{ss} = 2$ for all $s \in S$ and $\tilde{a}_{st} = -x_{st}$ for $s \neq t$.
                By Theorem~\ref{thm:abe_conditions}, there exists a ring homomorphism 
                \begin{equation}\label{eq:integral_models_1}
                  A \to \k, \qquad x_{st} \mapsto -a_{st}.      
                \end{equation}

                Fix $\delta_s \in \fr{h}^*$ for each $s \in S$ such that $\langle \alpha_s^{\vee}, \delta_s \rangle = 1$.
                We can define 
                \[\overline{A} = A [d_{st} : s\neq t \in S].\]
                By convention, we set $d_{ss} = 1$. We can now extend \eqref{eq:integral_models_1} to a morphism
                \[\tau : \overline{A} \to \k, \qquad d_{st} \mapsto \langle \alpha_s^{\vee}, \delta_t \rangle.\]
                Of course, $\overline{A}$ is still a free $\Z$-module. We define a prime ideal $\fr{q} = \ker (\tau)$.
                Let $\fr{p} \subseteq \fr{q}$ be a minimal prime ideal and define
                \[\mathcal{Z} = \overline{A}/\fr{p}.\]
                The ring homomorphism $\tau$ factors through a ring homomorphism
                \[\sigma : \mathcal{Z} \to \k.\]
                Note that since $\mathfrak{p}$ is a minimal prime ideal and $\overline{A}$ is free over $\Z$, we must have that $\mathcal{Z}$ has characteristic $0$.

                Define a free $\mathcal{Z}$-module,
                \[\fr{h}_{\mathcal{Z}}^* \coloneq \bigoplus_{s \in S} \mathcal{Z} \alpha_s \oplus \bigoplus_{t \in S} \mathcal{Z} e_t.\]
                We define the coroots in $\fr{h}_{\mathcal{Z}} \coloneq (\fr{h}_{\mathcal{Z}}^*)^*$ by
                \[\langle \alpha_s^{\vee}, \alpha_t \rangle = \tilde{a}_{st} \qquad\text{and}\qquad \langle \alpha_s^{\vee}, e_t \rangle = d_{st}.\]
                The Cartan entries satisfy the minimal-polynomial conditions of \cite[Definition 1.1]{Hazi2}. By \cite[Proposition 1.4]{Hazi2}, the reflection formulas define a $W$-action on $\fr{h}_{\mathcal{Z}}^*$, and dually on $\fr{h}_{\mathcal{Z}}$. 

                Next, we can verify that Demazure surjectivity holds for $\fr{h}_{\mathcal{Z}}$.
                By definition, each $\alpha_s$ is a basis vector of $\fr{h}_{\mathcal{Z}}^*$, so $\alpha_s : \fr{h}_{\mathcal{Z}} \to \mathcal{Z}$ is surjective.
                On the other hand, $\alpha_s^{\vee} (e_s) = d_{ss} = 1$, and so $\alpha_s^{\vee} : \fr{h}_{\mathcal{Z}}^* \to \mathcal{Z}$ is also surjective.
                Abe's conditions \eqref{eq:qbinom_vanishing} and the balancedness equations \eqref{eq:balanced} are satisfied by our construction using Theorem~\ref{thm:abe_conditions}.
                Therefore, $\fr{h}_{\mathcal{Z}}$ is a balanced realization of $(W,S)$ over $\mathcal{Z}$.

                Finally, we define $f : \fr{h}_{\mathcal{Z}}^* \to \fr{h}^*$ by $\alpha_s \mapsto \alpha_s$ and $e_t \mapsto \delta_t$. 
                The pair $(\sigma, f)$ then defines a morphism of realizations $\fr{h}_{\mathcal{Z}} \to \fr{h}$ which completes the proof.
        \end{proof}

        Unfortunately, the integral model produced in Proposition~\ref{prop:integral_models} is less than ideal. 
        The ring $\mathcal{Z}$ is quite complicated, and the realization $\fr{h}_{\mathcal{Z}}$ has rank $2 \cdot (\# S)$ which is about twice as large as desirable.
        In practice, it might be more feasible to create an integral model using more explicit knowledge of the given realization.
        For example, in type $H_3$, the integral models we use are much more suitable (see Lemma~\ref{lem:universal_H3_realization}).

        Let $\fr{h}_{\mathcal{Z}}$ be an integral model for $\fr{h}$. Consider $\mathcal{Q} = \Frac (\mathcal{Z})$ which is a characteristic 0 field.
        We can then define the base change by
        \[\fr{h}_{\mathcal{Q}}^* \coloneq \mathcal{Q} \otimes_{\mathcal{Z}} \fr{h}_{\mathcal{Z}}^*.\]
        This defines a realization with $\fr{h}_{\mathcal{Q}} = (\fr{h}_{\mathcal{Q}}^*)^*$.
        There is an obvious morphism of realizations 
        \[(\iota, g) : \fr{h}_{\mathcal{Z}} \hookrightarrow \fr{h}_{\mathcal{Q}},\]
        where $g (v) = 1 \otimes v$ and $\iota : \mathcal{Z} \hookrightarrow \mathcal{Q}$ is the embedding. 
        By Proposition~\ref{prop:BS_functor} there are functors
        \[\begin{tikzcd}
        {\BS (\fr{h}, W)} &  & {\BS (\fr{h}_{\mathcal{Z}}, W)} \arrow[ll, "\BS_f"'] \arrow[rr, "\BS_g"] &  & {\BS (\fr{h}_{\mathcal{Q}}, W)},
        \end{tikzcd}\]
        where $\BS_g$ is faithful by Corollary~\ref{cor:BS_inj_or_surj}.

        \begin{proposition}\label{prop:rational_reduction}
                With the same setup as above, let $\varphi : B_{\ux}^{\fr{h}} \to B_{\uy}^{\fr{h}}$ be a morphism of Abe--Bott--Samelson bimodules.
                Assume that $\varphi = \BS_f (\psi)$ for some $\psi : B_{\ux}^{\fr{h}_{\mathcal{Z}}} \to B_{\uy}^{\fr{h}_{\mathcal{Z}}}$. 
                If $\BS_g (\psi) = 0$, then $\varphi = 0$.
        \end{proposition}
        \begin{proof}
                If $\BS_g (\psi) = 0$, then $\psi = 0$ since $\BS_g$ is faithful by Corollary~\ref{cor:BS_inj_or_surj}.
                As a result, $\varphi = \BS_f (\psi) = 0$ as well.
        \end{proof}

        Proposition~\ref{prop:rational_reduction} allows us to sometimes translate an equality-of-morphisms problem in $\BS (\fr{h}, W)$ to an equality-of-morphisms problem in $\BS (\fr{h}_{\mathcal{Q}}, W)$.
        The latter can be approached using interpolation.

       \section{The \texorpdfstring{$3$}{3}-color Relation in Type \texorpdfstring{$H_3$}{H3} }

       \subsection{The rational case}

        Let $W = H_3$.
        We write $S = \{ s,t,u\} $ with $m_{st} = 5$, $m_{tu} =3$ and $m_{su} = 2$. 
        Let $\fr{h}_{\Q (\phi)}$ be the golden realization over $\Q (\phi)$. Consider the reduced expressions for the longest element $w_0$ in $W$,
        \[\ux = stsutsutustsutu \qquad\text{and}\qquad \uy = utustsutustusts.\]
        To solve the missing $3$-color relation, we will compute the double-leaf coefficients for the morphism in $\BS (\fr{h}_{\Q (\phi)}, W)$ corresponding
        to the following difference of Elias--Williamson graphs,
        \begin{equation}\label{eq:h3_sides}
          \varphi = 
  \tikzsetnextfilename{h3_sides}
  \tikzstyle{every picture}=[tikzfig]
  \input{./figures/h3_sides.tikz}
.      
        \end{equation}
        To see why this suffices, observe that $\ux$ and $\uy$ are both reduced expressions. 
        As a result, any double leaf between $B_{\ux}$ and $B_{\uy}$ whose subexpressions are not the all-one subexpressions factors through a reduced expression for some $w < w_0$.
        Given a rex path $p$ from $\ux$ to $\uy$, the corresponding morphism $\beta_{p} : B_{\ux} \to B_{\uy}$ has $\Lambda (\beta_p)_{\mathbf{1}, \mathbf{1}} = 1$ by the 
        defining criterion for $\beta_p$ (see \S\ref{ss:gen_mors}).
        As a result, $\Lambda (\varphi)_{\mathbf{1}, \mathbf{1}} = 0$. 
        This shows that the only double leaves appearing with non-zero coefficients in the expansion of $\varphi$ factor through some element $w < w_0$.

        For the remainder of the section, $\varphi$ will always refer the morphism in \eqref{eq:h3_sides}.

        Using our fixed enumeration of the double leaves from $B_{\ux}$ to $B_{\uy}$ of even degree $\leq 0$, we can then write
        \[\varphi = \sum_{i=1}^{189027} c_i \dLL_i\]
        for some $c_i \in R$. The number of double leaves of a given degree is given in the table below.
        \begin{center}
                 \begin{tabular}{c | c }
                                 \textbf{Degree of $\dLL_i$} & \textbf{Number of double leaves} \\
                                \hline
                                $0$ & 168,263  \\
                                $-2$ & 19,732  \\
                                $-4$ & 1,031  \\
                                $-6$ & 1  \\
                                \hline
                                \textbf{Total} & 189,027
                        \end{tabular}
        \end{center}
        
        Using Algorithm~\ref{algo:coeff_algo_interpolate}, these coefficients have been computed explicitly. 
        The entire calculation took about 30 minutes on the author's personal computer.
        We verified the relation using Proposition~\ref{prop:eq_testing_interpolate}, which took about 10 minutes.
        More detailed specifications of these calculations can be found in the codebase accompanying the paper.
        
        The numbers of nonzero terms in each degree are given in the table below,
        \begin{center}
                \begin{tabular}{c | c }
                         \textbf{Degree of $\dLL_i$} & \textbf{Nonzero coefficients} \\
                         \hline
                         $0$ & 97 \\
                         $-2$ & 20 \\
                         $-4$ & 7 \\
                         $-6$ & 1 \\ 
                         \hline
                         \textbf{Total } & 125
                \end{tabular}
        \end{center}

        We will use an integer mask to write our subexpressions as integers. In particular, for a subexpressions $\ue = (e_1, \ldots, e_k)$, we will write
        \[\operatorname{mask} (\ue) = \sum_{i=1}^k e_i 2^{i-1}.\]
        For the remainder of the section, we will abuse notation and commonly identify a subexpression with its mask. There is no mathematical content here, it is just a convenient way to compact the notation. 
        Under this notational convention, the table of coefficients is given in Figure~\ref{fig:coeffs}.

\begin{figure}
\begin{minipage}{0.21\textwidth}
    \centering
\begin{tabular}{l l l}
 $c_i$ & $\ue^i$ & $\uf^i$ \\ 
 \hline
$-2$ & 20791 & 2234 \\
 $-2$ & 20791 & 3450 \\
 $-2$ & 21455 & 2234 \\
 $-2$ & 16439 & 2234 \\
 $-2$ & 21455 & 3450 \\
 $-2$ & 16439 & 3450 \\
 $-2$ & 20791 & 2749 \\
 $-2$ & 17103 & 2234 \\
 $-2$ & 17103 & 3450 \\
 $-2$ & 20791 & 7549 \\
 $-2$ & 21455 & 2749 \\
 $-2$ & 16439 & 2749 \\
 $-2$ & 21455 & 7549 \\
 $-2$ & 16439 & 7549 \\
 $-2$ & 17103 & 2749 \\
 $-2$ & 17103 & 7549 \\
 $2$ & 21177 & 2491 \\
 $2$ & 697 & 2491 \\
 $2$ & 21177 & 187 \\
 $2$ & 23929 & 2491 \\
 $2$ & 697 & 187 \\
 $2$ & 3449 & 2491 \\
 $2$ & 23929 & 187 \\
 $2$ & 3449 & 187 \\
 $2$ & 21177 & 11215 \\
 $2$ & 21177 & 1743 \\
 $2$ & 697 & 11215 \\
 $2$ & 697 & 1743 \\
 $2$ & 23929 & 11215 \\
 $2$ & 23929 & 1743 \\
 $2$ & 3449 & 11215 \\
 $2$ & 3449 & 1743 \\
 $-2$ & 21177 & 5769 \\
 $-2$ & 21177 & 27529 \\
 $-2$ & 697 & 5769 \\
 $-2$ & 23929 & 5769 \\
 $-2$ & 697 & 27529 \\
 $-2$ & 23929 & 27529 \\
$-2$ & 3449 & 5769 \\
 $-2$ & 3449 & 27529 \\
  $2$ & 20855 & 2490 \\
 $2$ & 20535 & 2490 
\end{tabular}
\end{minipage}\hfill
\begin{minipage}[t]{0.28\textwidth}
    \centering
    \begin{tabular}{l l l}
         $c_i$ & $\ue^i$ & $\uf^i$ \\ 
 \hline
  $2$ & 375 & 2490  \\
 $2$ & 20855 & 186 \\
 $-2$ & 4407 & 186 \\
 $-2$ & 4407 & 1402 \\
 $2$ & 55 & 2490 \\
 $2$ & 20535 & 186 \\
 $2$ & 375 & 186 \\
 $-2$ & 5071 & 186 \\
 $-2$ & 5071 & 1402 \\
 $-2$ & 55 & 1402 \\
 $2$ & 20855 & 11214 \\
 $2$ & 20855 & 1742 \\
 $-2$ & 4407 & 701 \\
 $2$ & 20535 & 11214 \\
 $2$ & 20535 & 1742 \\
 $2$ & 375 & 11214 \\
 $2$ & 375 & 1742 \\
 $-2$ & 719 & 186 \\
 $-2$ & 719 & 1402 \\
 $-2$ & 4407 & 5501 \\
 $-2$ & 5071 & 701 \\
 $2$ & 55 & 11214 \\
 $2$ & 55 & 1742 \\
 $-2$ & 55 & 701 \\
 $-2$ & 5071 & 5501 \\
 $-2$ & 55 & 5501 \\
 $-2$ & 719 & 701 \\
 $-2$ & 719 & 5501 \\
 $-2 \phi$ & 20857 & 2953 \\
 $-2 \phi$ & 20857 & 23433 \\
 $-2 \phi$ & 23417 & 2953 \\
 $-2 \phi - 2$ & 23417 & 23433 \\
 $\left(2\right)\alpha_{\text{u}}$ & 20857 & 24585 \\
 $\left(2\right)\alpha_{\text{u}}$ & 23561 & 2953 \\
 $-2$ & 20857 & 24905 \\
 $\left(2 \phi\right)\alpha_{\text{u}}$ & 23417 & 24585 \\
 $\left(2 \phi\right)\alpha_{\text{u}}$ & 23561 & 23433 \\
 $-2$ & 23617 & 2953 \\
 $-2 \phi$ & 23417 & 24905 \\
 $-2 \phi$ & 23617 & 23433 \\
 $\left(-2 \phi\right)\alpha_{\text{u}}^{2}$ & 23561 & 24585 \\
 $\left(2 \phi\right)\alpha_{\text{u}}$ & 23561 & 24905 
\end{tabular}
\end{minipage}\hfill
\begin{minipage}[t]{0.31\textwidth}
    \centering
    \begin{tabular}{l l l}
         $c_i$ & $\ue^i$ & $\uf^i$ \\ 
 \hline
  $\left(2 \phi\right)\alpha_{\text{u}}$ & 23617 & 24585 \\
$-2 \phi$ & 23617 & 24905 \\
$-2$ & 20849 & 26633 \\
 $-2$ & 23817 & 3017 \\
 $-2 \phi$ & 23409 & 26633 \\
 $-2 \phi$ & 23817 & 23497 \\
 $\left(2 \phi\right)\alpha_{\text{u}}$ & 23817 & 26633 \\
 $-2 \phi$ & 23817 & 26953 \\
 $-2 \phi$ & 23873 & 26633 \\
 $\left(2\right)\alpha_{\text{u}}$ & 1913 & 15241 \\
 $-2$ & 6001 & 15241 \\
 $\left(-2 \phi\right)\alpha_{\text{u}}^{2}$ & 1913 & 9 \\
 $\left(-2 \phi\right)\alpha_{\text{u}}^{2}$ & 9 & 15241 \\
 $\left(2 \phi\right)\alpha_{\text{u}}$ & 1913 & 329 \\
 $\left(2 \phi\right)\alpha_{\text{u}}$ & 65 & 15241 \\
 $\left(2 \phi\right)\alpha_{\text{u}}$ & 4361 & 15241 \\
 $-2 \phi$ & 4417 & 15241 \\
 $\left(2 \phi\right)\alpha_{\text{u}}$ & 6001 & 9 \\
 $-2 \phi$ & 6001 & 329 \\
 $\left(2 \phi + 2\right)\alpha_{\text{u}}^{3}$ & 9 & 9 \\
 $\left(-2 \phi - 2\right)\alpha_{\text{u}}^{2}$ & 9 & 329 \\
 $\left(-2 \phi - 2\right)\alpha_{\text{u}}^{2}$ & 65 & 9 \\
 $\left(-2 \phi - 2\right)\alpha_{\text{u}}^{2}$ & 4361 & 9 \\
 $\left(2 \phi + 2\right)\alpha_{\text{u}}$ & 65 & 329 \\
 $\left(2 \phi + 2\right)\alpha_{\text{u}}$ & 4361 & 329 \\
 $\left(2 \phi + 2\right)\alpha_{\text{u}}$ & 4417 & 9 \\
 $-2 \phi - 2$ & 4417 & 329 \\
 $-2$ & 18297 & 15305 \\
 $\left(2 \phi\right)\alpha_{\text{u}}$ & 18297 & 2057 \\
 $-2 \phi$ & 18297 & 2377 \\
 $\left(2 \phi\right)\alpha_{\text{u}}$ & 16393 & 15305 \\
 $-2 \phi$ & 16449 & 15305 \\
 $-2 \phi$ & 22385 & 2057 \\
 $-2 \phi$ & 20745 & 15305 \\
 $\left(-2 \phi - 2\right)\alpha_{\text{u}}^{2}$ & 16393 & 2057 \\
 $\left(2 \phi + 2\right)\alpha_{\text{u}}$ & 16393 & 2377 \\
 $\left(2 \phi + 2\right)\alpha_{\text{u}}$ & 16449 & 2057 \\
 $\left(2 \phi + 2\right)\alpha_{\text{u}}$ & 20745 & 2057 \\
 $-2 \phi - 2$ & 16449 & 2377 \\
 $-2 \phi - 2$ & 20745 & 2377 \\
 $-2 \phi - 2$ & 20801 & 2057 \\
 & &
\end{tabular}\end{minipage} \caption{The nonzero coefficients for the double-leaf expansion of $\varphi$, with the $\LL$-datum of \S\ref{ss:ll_datum}.} \label{fig:coeffs} \end{figure}

The size of a double-leaf expansion depends both on the chosen rex morphisms and on the light-leaf datum. 
Keeping the datum of \S\ref{ss:ll_datum} fixed, we calculated the $H_3$ relation for the $286$ reduced expressions of $w_0$ and, for each, the chosen pair of paths with target the reversed expression.
Among the $286$ computations, the smallest number of nonzero double-leaf terms was $125$, and this occurred for $16$ possible choices of source expression.
The relation displayed above is one of these choices. Here and below, ``minimal'' will only refer to this fixed family of computations. 
It is possible that a different choice of light-leaf datum could produce fewer terms.

The chosen relation has a useful rotational symmetry. 
In more detail, using cups and caps, rotation through $180^{\circ}$ defines a contravariant functor
\[\rho : \BS (\fr{h}, W)^{\op} \to \BS (\fr{h}, W)\]
that takes $B_{\uw}$ to $B_{\overline{\uw}}$ where $\overline{\uw}$ is the reversed expression. 
In our example $\uy = \overline{\ux}$, and rotation interchanges the left and right rex morphisms. Hence  
\[\rho (\varphi) = -\varphi.\]
Many of the minimal $H_3$ relations are rotationally anti-invariant; however, the author has no real understanding of 
why this is the case. It remains an interesting question to understand the $H_3$ relation conceptually rather than computationally.

Another interesting observation about the minimal $H_3$ relation is that every coefficient in Figure~\ref{fig:coeffs} belongs to $2 \Z [\phi] [\alpha_s, \alpha_t, \alpha_u]$. 
As the next subsection shows, this means that the integral lift $\varphi_{\mathcal{Z}}$ belongs to 
\[2\Hom_{\BS (\fr{h}_{\mathcal{Z}}, W )} (B_{\ux}, B_{\uy})\]
Again, the author has no explanation for this. We will not investigate this oddity (or even-ty?) in the present work, but we will leave it as an interesting question for the future.

\begin{question} Why is the integral lift of $\varphi$ divisible by $2$? Similarly, what importance does $\varphi/2$ have in the theory of Soergel bimodules? \end{question}

\subsection{The case of a general realization}

        The approach of \S\ref{ss:gen_realizations} can be used to deduce the Zamolodchikov relation in type $H_3$ due to the existence of a \emph{universal integral model}.

        \begin{lemma}\label{lem:universal_H3_realization}
                Let $W = H_3$. There exists a balanced realization $\fr{h}_{\mathcal{Z}}$ over $\mathcal{Z} = \Z[\tau]/(\tau^2-\tau-1)$ such that $\fr{h}_{\mathcal{Z}}$ is an integral model for any balanced realization $\fr{h}$ of $W$.
                Moreover, $\Q (\phi) \cong \Frac (\mathcal{Z})$, and as a result, the base change of $\fr{h}_{\mathcal{Z}}$ to $\Q (\phi)$ is isomorphic to the golden realization $\fr{h}_{\Q (\phi)}$. 
        \end{lemma}
        \begin{proof}
               Consider the Cartan matrix of an arbitrary balanced realization $\fr{h}$ over an integral domain $\k$,
                \[\begin{bmatrix}
                        a_{ss} & a_{st} & a_{su} \\
                        a_{ts} & a_{tt} & a_{tu} \\
                        a_{us} & a_{ut} & a_{uu} 
                \end{bmatrix}.\]
                We must have that $a_{ss} = a_{tt} = a_{uu} = 2$ by the definition of a realization.
                For the entries $a_{su}$ and $a_{us}$, the vanishing of quantum binomial coefficients forces
                \[-a_{su} = [2]_x (-a_{su}, -a_{us}) = 0 \qquad\text{and}\qquad -a_{us} = [2]_y (-a_{su}, -a_{us}) = 0.\]
                As a result, $a_{su} = a_{us} = 0$.
                Next, for the entries $a_{tu}$ and $a_{ut}$, balancedness forces
                \[ -a_{tu} = [2]_x (-a_{tu}, -a_{ut}) = 1 \qquad\text{and}\qquad -a_{ut} = [2]_y (-a_{tu}, -a_{ut}) = 1.\]
                As a result, $a_{tu} = a_{ut} = -1$.  Finally, for the entries $a_{st}$ and $a_{ts}$, balancedness forces
                \[[2]_x (-a_{st}, -a_{ts}) = [3]_x (-a_{st}, a_{ts}) = [3]_y (-a_{st}, a_{ts}) = [2]_y (-a_{st}, -a_{ts}).\]
                Evaluating these quantum numbers gives $\theta \coloneq -a_{st} = -a_{ts}$ and $\theta^2 - \theta - 1 = 0$.
                The Cartan matrix then is of the form
                \begin{equation}\label{eq:universal_H3_realization}
                \begin{bmatrix} 2 & -\theta & 0 \\ - \theta & 2 & -1 \\ 0 & -1 & 2 \end{bmatrix}.
                \end{equation}

                Set $\mathcal{Z} = \Z [\tau]/(\tau^2 - \tau - 1)$. We can then define a realization
                \[\fr{h}_{\mathcal{Z}}^* \coloneq \mathcal{Z} \alpha_s \oplus \mathcal{Z} \alpha_t \oplus \mathcal{Z} \alpha_u.\]
                We define $\alpha_r^{\vee} \in \fr{h}_{\mathcal{Z}} \coloneq (\fr{h}_{\mathcal{Z}}^*)^*$ for $r=s,t,u$ by the Cartan matrix \eqref{eq:universal_H3_realization} where $\theta = \tau$.
                It is a tad tedious, but not difficult to check that $\fr{h}_{\mathcal{Z}}$ is a balanced realization of $W$.

                We can then define a morphism of realizations $\fr{h}_{\mathcal{Z}} \to \fr{h}$ as follows. The ring homomorphism $\mathcal{Z} \to \k$ is defined by $\tau \mapsto \theta$.
                The linear map $\fr{h}_{\mathcal{Z}}^* \to \fr{h}^*$ is defined by $\alpha_r \mapsto \alpha_r$ for $r\in S$. The Cartan matrix \eqref{eq:universal_H3_realization} 
                guarantees that this data defines a morphism of realizations. 
                
                Finally, the map $\mathcal{Z} = \Z [\tau]/(\tau^2 - \tau - 1) \to \Q (\phi)$ defined by $\tau \mapsto \phi$ induces an isomorphism $\Frac (\mathcal{Z}) \cong \Q (\phi)$.
                Over $\Q (\phi)$ the Cartan matrix has determinant $2 (2 - \phi) \neq 0$. Thus the simple roots in the golden realization form a basis. Sending our three basis roots to those roots gives an isomorphism between the base change of $\fr{h}_{\mathcal{Z}}$ to $\Q (\phi)$ and the golden realization.
        \end{proof}

        Lemma~\ref{lem:universal_H3_realization} and Proposition~\ref{prop:rational_reduction} give a convenient reduction to the golden realization in proving the Zamolodchikov relation in type $H_3$.
        In particular, we can first solve for the double-leaves coefficients with the golden realization.
        We then checked that all the coefficients $c_i$ in fact elements of $R_{\fr{h}_{\mathcal{Z}}}$ under the inclusion $R_{\fr{h}_{\mathcal{Z}}} \hookrightarrow R_{\fr{h}_{\Q (\phi)}}$.
        As a result, the relation lifts to the integral model.
        Finally, the image from the integral model to an arbitrary balanced realization induces the relation in full generality.

        \begin{remark}
                The displayed integral coefficients all belong to $2 R_{\fr{h}_{\mathcal{Z}}}$. 
                As a result, one can observe that if $\fr{h}$ is a realization of $H_3$ over $\F_4 = \F_2 [\tau] / (\tau^2 - \tau - 1)$, then the ``naive'' $3$-color relation holds, i.e., $\varphi = 0$. 
        \end{remark}

        \subsection{Thick simplifications}

        Even with the minimal 3-color relation, there are still too many terms to perform most computations by hand. The goal of this subsection is to simplify
        the $3$-color relation to something that is genuinely usable. There are two main reasons for this.
        \begin{enumerate}
                \item Part of the reason diagrammatics is so powerful is in its ability to perform complicated computations by hand.
                \item The author still has no conceptual explanation for the lower terms appearing in the $3$-color relation. 
                Likewise, one would ideally hope that there is a way of verifying the $3$-color relation without computer assistance. 
                To fulfil this dream, there needs to be a cleaner way of writing the relation in the first place.
        \end{enumerate}
        Our approach to simplification is to use \emph{thick Soergel calculus}. A review of the graphical calculus in this section can be found in Appendix~\ref{apdx:thick_calculus}. 
        There are only two thick strands that will appear in our simplification corresponding to $I = \{s,t\}$ and $J = \{t,u\}$.
        We will color the $I$-thick strand purple and the $J$-thick strand teal. In order to use thick Soergel calculus, we need to check that Assumption~\ref{ass:gen_bs} holds.
        This assumption can fail for a general realization; however, we do not actually need the full assumption. Instead, we only need to check it for $I$ and $J$ as above.

        \begin{lemma}\label{lem:gen_bs_ass_satisfied_for_H3}
                With the same notation as above,
                \begin{enumerate}
                        \item $W_I \to \Aut (\fr{h}^*)$ and $W_J \to \Aut (\fr{h}^*)$ are injective,
                        \item $\partial_I ((3-2\theta) \alpha_t \alpha_u^4) = 1$ and $\partial_J ((\theta - 2) \alpha_s^2 \alpha_t) = 1$.
                \end{enumerate}
        \end{lemma}
        \begin{proof}
                Recall that there exists some $\theta \in \k^{\times}$ such that the Cartan matrix is of the form \eqref{eq:universal_H3_realization}.
                The following table describes the action of each simple reflection on the simple roots,
                \begin{center}
                        \begin{tabular}{c | c | c | c}
                            & $\alpha_s$ & $\alpha_t$ & $\alpha_u$ \\
                            \hline
                            $s$ & $-\alpha_s$ & $\alpha_t + \theta \alpha_s$ & $\alpha_u$ \\
                            \hline
                            $t$ & $\alpha_s + \theta \alpha_t$ & $-\alpha_t$ & $\alpha_t + \alpha_u$ \\
                            \hline
                            $u$ & $\alpha_s$ & $\alpha_t + \alpha_u$ & $-\alpha_u$    
                        \end{tabular}
                \end{center}
                We first claim that $W_I \to \Aut (\fr{h}^*)$ is injective. First, we can calculate using the above table that $st (\alpha_u) - \alpha_u = \alpha_t + \theta \alpha_s$.
                Note that $\alpha_t + \theta \alpha_s \neq 0$ since $\langle \alpha_u^{\vee}, \alpha_t + \theta \alpha_u \rangle = -1$, and hence $st (\alpha_u) \neq \alpha_u$. Therefore, $st$ acts non-trivially on $\fr{h}^*$.
                Moreover, since $st$ has prime order $3$, $\langle st \rangle$ has index $2$ in $W_I$, we must have that the kernel of $W_I \to \Aut (\fr{h}^*)$ is contained in the center of $W_I$.
                However, a dihedral group of the form $D_{2p}$ for an odd prime $p$ has trivial center, therefore, the kernel must be trivial.

                Next, we claim that $W_J \to \Aut (\fr{h}^*)$ is injective. 
                Again, we can calculate using the above table to get that $tu (\alpha_s) - \alpha_s = \theta \alpha_t$. Since $\langle \alpha_u^{\vee}, \theta \alpha_t \rangle = - \theta \neq 0$, we have that $tu (\alpha_s) \neq \alpha_s$.
                Again since $tu$ has odd prime order $3$, the same argument as for $W_I$ shows that $W_J \to \Aut (\fr{h}^*)$ is injective.

                It is easy to check from the definitions that
                \[\partial_I (\alpha_t \alpha_u^4) = -\theta^3 \qquad\text{and}\qquad \partial_J (\alpha_s^2 \alpha_t) = -\theta^2.\]
                Now, $\theta$ satisfies $\theta^2 - \theta - \theta = 0$, so we can check that $-\theta^{-3} = 3-2\theta$ and $-\theta^{-2} = \theta - 2$.
                The claim then readily follows.
        \end{proof}

        Lemma \ref{lem:gen_bs_ass_satisfied_for_H3} allows us to draw thick Elias--Williamson diagrams unambiguously for an arbitrary balanced realization and interpret them as morphisms of generalized Abe--Bott--Samelson bimodules.

        At this point, we could simply draw the simplified relation now and say that it can be verified by computer, but this would not tell you \emph{how} such a simplification can be found.
        As a result, we will describe the fairly natural steps that one can take to get from our 125-term relation to its simplification. 
        The majority of these steps can be computed, somewhat tediously, by hand. To make the notation simpler, we will label light leaves and double leaves by their integer masks rather than the subexpressions, as we did in Figure~\ref{fig:coeffs}. 

        Consider the matrix $C = (c_{\ue,\uf})$ of coefficients for $\varphi$, i.e., $c_{\ue, \uf}$ is the coefficient of $\dLL_{\ue, \uf}$ in the double leaf expansion of $\varphi$.
        Our general strategy is to look for blocks of $C$ that could correspond to a thick Elias--Williamson diagram.
        More precisely, we can look for subsets $E$ of the possible source subexpressions and $E'$ of the possible target subexpressions such that the $E \times E'$-block
        is a dense matrix. There is a natural way to find such blocks.
        For each $w \in W$, we can pick $E$ and $E'$ such that if $\ue \in E$ and $\uf \in E'$, then $\ux^{\ue} = w = \uy^{\uf}$.
        There are only $7$ possible $w$ that give nonzero blocks. We will start with the easiest block to study corresponding to $w = ststsu$.

        \begin{figure}[h]
                \begin{center}
                        \begin{tabular}{r|*{4}{c}}
                        $\ue \backslash \uf$ & $2234$ & $2749$ & $3450$ & $7549$ \\
                        \hline
                        $16439$ & $-2$ & $-2$ & $-2$ & $-2$ \\
                        $17103$ & $-2$ & $-2$ & $-2$ & $-2$ \\
                        $20791$ & $-2$ & $-2$ & $-2$ & $-2$ \\
                        $21455$ & $-2$ & $-2$ & $-2$ & $-2$ 
                        \end{tabular} 
                \end{center}
                \caption{The $ststsu$-block of the coefficient matrix.}\label{fig:easy_block}
        \end{figure}

        Define $K = K_2 \circ K_1 : B_{\ux} \to B_{\uy}$, where
        \[K_1 = \LL_{\ux, 16439} + \LL_{\ux, 17103} + \LL_{\ux, 20791} + \LL_{\ux, 21455},\qquad K_2 = \oLL_{\uy, 2234} + \oLL_{\uy, 2749} + \oLL_{\uy, 3450}, + \oLL_{\uy, 7549}. \]
        The summand of $\varphi$ corresponding to Figure~\ref{fig:easy_block} is $-2K$.
        We will first combine $K_1$ into a single thick Elias--Williamson diagram. The constituent light leaves are given as follows,
        \[
  \tikzsetnextfilename{K1_1}
  \tikzstyle{every picture}=[tikzfig]
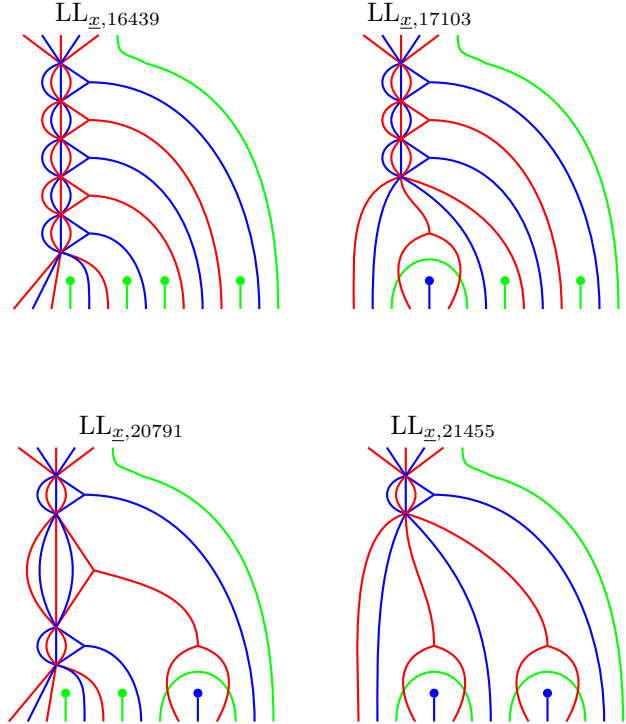
  \begin{tikzpicture}
	\begin{pgfonlayer}{nodelayer}
		\node [style=none] (0) at (-8, -4) {};
		\node [style=none] (1) at (-7.5, -4) {};
		\node [style=none] (2) at (-7, -4) {};
		\node [style=none] (3) at (-6.5, -4) {};
		\node [style=none] (4) at (-6, -4) {};
		\node [style=none] (5) at (-5.5, -4) {};
		\node [style=none] (6) at (-5, -4) {};
		\node [style=none] (7) at (-4.5, -4) {};
		\node [style=none] (8) at (-4, -4) {};
		\node [style=none] (9) at (-3.5, -4) {};
		\node [style=none] (10) at (-3, -4) {};
		\node [style=none] (11) at (-2.5, -4) {};
		\node [style=none] (12) at (-2, -4) {};
		\node [style=none] (13) at (-1.5, -4) {};
		\node [style=none] (14) at (-1, -4) {};
		\node [style={dot_green}] (15) at (-6.5, -3.25) {};
		\node [style={dot_green}] (16) at (-5, -3.25) {};
		\node [style={dot_green}] (17) at (-4, -3.25) {};
		\node [style={dot_green}] (18) at (-2, -3.25) {};
		\node [style=none] (19) at (-6.75, -2.5) {};
		\node [style=none] (20) at (-6.75, -1.5) {};
		\node [style=none] (21) at (-6.75, -0.5) {};
		\node [style=none] (22) at (-6.75, 0.5) {};
		\node [style=none] (23) at (-6.75, 1.5) {};
		\node [style=none] (24) at (-6.75, 2.5) {};
		\node [style=none] (25) at (-6, -2) {};
		\node [style=none] (26) at (-6, 0) {};
		\node [style=none] (27) at (-6, 2) {};
		\node [style=none] (28) at (-6, -1) {};
		\node [style=none] (29) at (-6, 1) {};
		\node [style=none] (30) at (-7.75, 3.25) {};
		\node [style=none] (31) at (-7.25, 3.25) {};
		\node [style=none] (32) at (-6.75, 3.25) {};
		\node [style=none] (33) at (-6.25, 3.25) {};
		\node [style=none] (34) at (-5.75, 3.25) {};
		\node [style=none] (35) at (-5.25, 3.25) {};
		\node [style=none] (36) at (-4.5, 2.5) {};
		\node [style=none] (37) at (1, -4) {};
		\node [style=none] (38) at (1.5, -4) {};
		\node [style=none] (39) at (2, -4) {};
		\node [style=none] (40) at (2.5, -4) {};
		\node [style=none] (41) at (3, -4) {};
		\node [style=none] (42) at (3.5, -4) {};
		\node [style=none] (43) at (4, -4) {};
		\node [style=none] (44) at (4.5, -4) {};
		\node [style=none] (45) at (5, -4) {};
		\node [style=none] (46) at (5.5, -4) {};
		\node [style=none] (47) at (6, -4) {};
		\node [style=none] (48) at (6.5, -4) {};
		\node [style=none] (49) at (7, -4) {};
		\node [style=none] (50) at (7.5, -4) {};
		\node [style=none] (51) at (8, -4) {};
		\node [style={dot_green}] (52) at (5, -3.25) {};
		\node [style={dot_green}] (53) at (7, -3.25) {};
		\node [style=none] (54) at (2.25, -0.5) {};
		\node [style=none] (55) at (2.25, 0.5) {};
		\node [style=none] (56) at (2.25, 1.5) {};
		\node [style=none] (57) at (2.25, 2.5) {};
		\node [style=none] (58) at (3, 0) {};
		\node [style=none] (59) at (3, 2) {};
		\node [style=none] (60) at (2.25, -0.5) {};
		\node [style=none] (61) at (3, 1) {};
		\node [style=none] (62) at (1.25, 3.25) {};
		\node [style=none] (63) at (1.75, 3.25) {};
		\node [style=none] (64) at (2.25, 3.25) {};
		\node [style=none] (65) at (2.75, 3.25) {};
		\node [style=none] (66) at (3.25, 3.25) {};
		\node [style=none] (67) at (3.75, 3.25) {};
		\node [style=none] (68) at (4.5, 2.5) {};
		\node [style={EW_t}] (69) at (3, -3.25) {};
		\node [style=none] (70) at (3, -2) {};
		\node [style=none] (71) at (-5.5, 3.75) {$\textnormal{LL}_{\underline{x}, 16439}$};
		\node [style=none] (72) at (2.75, 3.75) {$\textnormal{LL}_{\underline{x}, 17103}$};
	\end{pgfonlayer}
	\begin{pgfonlayer}{edgelayer}
		\draw [style={red_edge}] (0.center) to (19.center);
		\draw [style={red_edge}] (19.center) to (2.center);
		\draw [style={red_edge}, in=90, out=-15, looseness=1.25] (19.center) to (5.center);
		\draw [style=blue edge] (1.center) to (19.center);
		\draw [style=blue edge, in=90, out=-15, looseness=1.25] (19.center) to (4.center);
		\draw [style={green_edge}] (15) to (3.center);
		\draw [style={green_edge}] (16) to (6.center);
		\draw [style={green_edge}] (17) to (8.center);
		\draw [style={green_edge}] (18) to (12.center);
		\draw [style=blue edge] (19.center) to (20.center);
		\draw [style=blue edge, bend left=75, looseness=1.75] (19.center) to (20.center);
		\draw [style=blue edge] (19.center) to (25.center);
		\draw [style=blue edge] (25.center) to (20.center);
		\draw [style=blue edge, bend left=45, looseness=1.25] (20.center) to (21.center);
		\draw [style=blue edge, bend right=45, looseness=1.25] (20.center) to (21.center);
		\draw [style=blue edge, bend left=75, looseness=1.75] (21.center) to (22.center);
		\draw [style=blue edge] (21.center) to (22.center);
		\draw [style=blue edge] (21.center) to (26.center);
		\draw [style=blue edge] (26.center) to (22.center);
		\draw [style=blue edge, bend left=45, looseness=1.25] (22.center) to (23.center);
		\draw [style=blue edge, bend right=45, looseness=1.25] (22.center) to (23.center);
		\draw [style=blue edge] (23.center) to (24.center);
		\draw [style=blue edge, bend left=75, looseness=1.75] (23.center) to (24.center);
		\draw [style=blue edge] (23.center) to (27.center);
		\draw [style=blue edge] (27.center) to (24.center);
		\draw [style={red_edge}, bend left=45, looseness=1.25] (19.center) to (20.center);
		\draw [style={red_edge}, bend right=45, looseness=1.25] (19.center) to (20.center);
		\draw [style={red_edge}] (20.center) to (21.center);
		\draw [style={red_edge}, bend left=75, looseness=1.75] (20.center) to (21.center);
		\draw [style={red_edge}] (20.center) to (28.center);
		\draw [style={red_edge}] (28.center) to (21.center);
		\draw [style={red_edge}, bend left=45, looseness=1.25] (21.center) to (22.center);
		\draw [style={red_edge}, bend right=45, looseness=1.25] (21.center) to (22.center);
		\draw [style={red_edge}] (22.center) to (23.center);
		\draw [style={red_edge}, bend left=75, looseness=1.75] (22.center) to (23.center);
		\draw [style={red_edge}] (22.center) to (29.center);
		\draw [style={red_edge}] (29.center) to (23.center);
		\draw [style={red_edge}, bend left=45, looseness=1.25] (23.center) to (24.center);
		\draw [style={red_edge}, bend right=45, looseness=1.25] (23.center) to (24.center);
		\draw [style={red_edge}] (30.center) to (24.center);
		\draw [style={red_edge}] (24.center) to (32.center);
		\draw [style={red_edge}] (24.center) to (34.center);
		\draw [style=blue edge] (31.center) to (24.center);
		\draw [style=blue edge] (24.center) to (33.center);
		\draw [style=blue edge, in=0, out=90] (7.center) to (25.center);
		\draw [style=blue edge, in=0, out=90] (10.center) to (26.center);
		\draw [style=blue edge, in=0, out=90] (13.center) to (27.center);
		\draw [style={red_edge}, in=0, out=90] (9.center) to (28.center);
		\draw [style={red_edge}, in=0, out=90] (11.center) to (29.center);
		\draw [style={green_edge}, in=90, out=-15] (36.center) to (14.center);
		\draw [style={green_edge}, in=-90, out=150, looseness=1.50] (36.center) to (35.center);
		\draw [style={green_edge}] (52) to (45.center);
		\draw [style={green_edge}] (53) to (49.center);
		\draw [style=blue edge, bend left=75, looseness=1.75] (54.center) to (55.center);
		\draw [style=blue edge] (54.center) to (55.center);
		\draw [style=blue edge] (54.center) to (58.center);
		\draw [style=blue edge] (58.center) to (55.center);
		\draw [style=blue edge, bend left=45, looseness=1.25] (55.center) to (56.center);
		\draw [style=blue edge, bend right=45, looseness=1.25] (55.center) to (56.center);
		\draw [style=blue edge] (56.center) to (57.center);
		\draw [style=blue edge, bend left=75, looseness=1.75] (56.center) to (57.center);
		\draw [style=blue edge] (56.center) to (59.center);
		\draw [style=blue edge] (59.center) to (57.center);
		\draw [style={red_edge}, bend left=45, looseness=1.25] (54.center) to (55.center);
		\draw [style={red_edge}, bend right=45, looseness=1.25] (54.center) to (55.center);
		\draw [style={red_edge}] (55.center) to (56.center);
		\draw [style={red_edge}, bend left=75, looseness=1.75] (55.center) to (56.center);
		\draw [style={red_edge}] (55.center) to (61.center);
		\draw [style={red_edge}] (61.center) to (56.center);
		\draw [style={red_edge}, bend left=45, looseness=1.25] (56.center) to (57.center);
		\draw [style={red_edge}, bend right=45, looseness=1.25] (56.center) to (57.center);
		\draw [style={red_edge}] (62.center) to (57.center);
		\draw [style={red_edge}] (57.center) to (64.center);
		\draw [style={red_edge}] (57.center) to (66.center);
		\draw [style=blue edge] (63.center) to (57.center);
		\draw [style=blue edge] (57.center) to (65.center);
		\draw [style=blue edge, in=0, out=90] (47.center) to (58.center);
		\draw [style=blue edge, in=0, out=90] (50.center) to (59.center);
		\draw [style={red_edge}, in=-15, out=90] (46.center) to (60.center);
		\draw [style={red_edge}, in=0, out=90] (48.center) to (61.center);
		\draw [style={green_edge}, in=90, out=-15] (68.center) to (51.center);
		\draw [style={green_edge}, in=-90, out=150, looseness=1.50] (68.center) to (67.center);
		\draw [style={red_edge}, in=90, out=-165] (60.center) to (37.center);
		\draw [style=blue edge, in=-135, out=90, looseness=0.75] (38.center) to (60.center);
		\draw [style=blue edge] (69) to (41.center);
		\draw [style={green_edge}, in=90, out=90, looseness=2.25] (39.center) to (43.center);
		\draw [style={red_edge}, in=-165, out=120, looseness=1.25] (40.center) to (70.center);
		\draw [style={red_edge}, in=-15, out=60, looseness=1.25] (42.center) to (70.center);
		\draw [style={red_edge}, in=-90, out=90] (70.center) to (60.center);
		\draw [style=blue edge, in=-30, out=90] (44.center) to (60.center);
	\end{pgfonlayer}
\end{tikzpicture}
\]
        \[
  \tikzsetnextfilename{K1_2}
  \tikzstyle{every picture}=[tikzfig]
  \begin{tikzpicture}
	\begin{pgfonlayer}{nodelayer}
		\node [style=none] (0) at (-8, -4) {};
		\node [style=none] (1) at (-7.5, -4) {};
		\node [style=none] (2) at (-7, -4) {};
		\node [style=none] (3) at (-6.5, -4) {};
		\node [style=none] (4) at (-6, -4) {};
		\node [style=none] (5) at (-5.5, -4) {};
		\node [style=none] (6) at (-5, -4) {};
		\node [style=none] (7) at (-4.5, -4) {};
		\node [style=none] (8) at (-1.5, -4) {};
		\node [style=none] (9) at (-1, -4) {};
		\node [style={dot_green}] (10) at (-6.5, -3.25) {};
		\node [style={dot_green}] (11) at (-5, -3.25) {};
		\node [style=none] (12) at (-6.75, -2.5) {};
		\node [style=none] (13) at (-6.75, -1.5) {};
		\node [style=none] (14) at (-6.75, 1.5) {};
		\node [style=none] (15) at (-6.75, 2.5) {};
		\node [style=none] (16) at (-6, -2) {};
		\node [style=none] (17) at (-6, 2) {};
		\node [style=none] (18) at (-7.75, 3.25) {};
		\node [style=none] (19) at (-7.25, 3.25) {};
		\node [style=none] (20) at (-6.75, 3.25) {};
		\node [style=none] (21) at (-6.25, 3.25) {};
		\node [style=none] (22) at (-5.75, 3.25) {};
		\node [style=none] (23) at (-5.25, 3.25) {};
		\node [style=none] (24) at (-4.5, 2.5) {};
		\node [style=none] (25) at (-5.75, 0) {};
		\node [style=none] (26) at (1.25, -4) {};
		\node [style=none] (27) at (1.75, -4) {};
		\node [style=none] (28) at (4.75, -4) {};
		\node [style=none] (29) at (7.75, -4) {};
		\node [style=none] (30) at (8.25, -4) {};
		\node [style=none] (31) at (2.5, 1.5) {};
		\node [style=none] (32) at (2.5, 2.5) {};
		\node [style=none] (33) at (3.25, 2) {};
		\node [style=none] (34) at (1.5, 3.25) {};
		\node [style=none] (35) at (2, 3.25) {};
		\node [style=none] (36) at (2.5, 3.25) {};
		\node [style=none] (37) at (3, 3.25) {};
		\node [style=none] (38) at (3.5, 3.25) {};
		\node [style=none] (39) at (4, 3.25) {};
		\node [style=none] (40) at (4.75, 2.5) {};
		\node [style=none] (41) at (2.25, -4) {};
		\node [style=none] (42) at (2.75, -4) {};
		\node [style=none] (43) at (3.25, -4) {};
		\node [style=none] (44) at (3.75, -4) {};
		\node [style=none] (45) at (4.25, -4) {};
		\node [style={EW_t}] (46) at (3.25, -3.25) {};
		\node [style=none] (47) at (3.25, -2) {};
		\node [style=none] (48) at (5.25, -4) {};
		\node [style=none] (49) at (5.75, -4) {};
		\node [style=none] (50) at (6.25, -4) {};
		\node [style=none] (51) at (6.75, -4) {};
		\node [style=none] (52) at (7.25, -4) {};
		\node [style={EW_t}] (53) at (6.25, -3.25) {};
		\node [style=none] (54) at (6.25, -2) {};
		\node [style=none] (55) at (-4, -4) {};
		\node [style=none] (56) at (-3.5, -4) {};
		\node [style=none] (57) at (-3, -4) {};
		\node [style=none] (58) at (-2.5, -4) {};
		\node [style=none] (59) at (-2, -4) {};
		\node [style={EW_t}] (60) at (-3, -3.25) {};
		\node [style=none] (61) at (-3, -2) {};
		\node [style=none] (62) at (-4.75, 3.75) {$\textnormal{LL}_{\underline{x}, 20791}$};
		\node [style=none] (63) at (3.5, 3.75) {$\textnormal{LL}_{\underline{x}, 21455}$};
	\end{pgfonlayer}
	\begin{pgfonlayer}{edgelayer}
		\draw [style={red_edge}] (0.center) to (12.center);
		\draw [style={red_edge}] (12.center) to (2.center);
		\draw [style={red_edge}, in=90, out=-15, looseness=1.25] (12.center) to (5.center);
		\draw [style=blue edge] (1.center) to (12.center);
		\draw [style=blue edge, in=90, out=-15, looseness=1.25] (12.center) to (4.center);
		\draw [style={green_edge}] (10) to (3.center);
		\draw [style={green_edge}] (11) to (6.center);
		\draw [style=blue edge] (12.center) to (13.center);
		\draw [style=blue edge, bend left=75, looseness=1.75] (12.center) to (13.center);
		\draw [style=blue edge] (12.center) to (16.center);
		\draw [style=blue edge] (16.center) to (13.center);
		\draw [style=blue edge] (14.center) to (15.center);
		\draw [style=blue edge, bend left=75, looseness=1.75] (14.center) to (15.center);
		\draw [style=blue edge] (14.center) to (17.center);
		\draw [style=blue edge] (17.center) to (15.center);
		\draw [style={red_edge}, bend left=45, looseness=1.25] (12.center) to (13.center);
		\draw [style={red_edge}, bend right=45, looseness=1.25] (12.center) to (13.center);
		\draw [style={red_edge}, bend left=45, looseness=1.25] (14.center) to (15.center);
		\draw [style={red_edge}, bend right=45, looseness=1.25] (14.center) to (15.center);
		\draw [style={red_edge}] (18.center) to (15.center);
		\draw [style={red_edge}] (15.center) to (20.center);
		\draw [style={red_edge}] (15.center) to (22.center);
		\draw [style=blue edge] (19.center) to (15.center);
		\draw [style=blue edge] (15.center) to (21.center);
		\draw [style=blue edge, in=0, out=90] (7.center) to (16.center);
		\draw [style=blue edge, in=0, out=90] (8.center) to (17.center);
		\draw [style={green_edge}, in=90, out=-15] (24.center) to (9.center);
		\draw [style={green_edge}, in=-90, out=150, looseness=1.50] (24.center) to (23.center);
		\draw [style={red_edge}] (14.center) to (25.center);
		\draw [style={red_edge}] (13.center) to (25.center);
		\draw [style={red_edge}] (13.center) to (14.center);
		\draw [style={red_edge}, bend right=45, looseness=1.25] (14.center) to (13.center);
		\draw [style=blue edge, bend right] (14.center) to (13.center);
		\draw [style=blue edge, bend left] (14.center) to (13.center);
		\draw [style=blue edge] (31.center) to (32.center);
		\draw [style=blue edge, bend left=75, looseness=1.75] (31.center) to (32.center);
		\draw [style=blue edge] (31.center) to (33.center);
		\draw [style=blue edge] (33.center) to (32.center);
		\draw [style={red_edge}, bend left=45, looseness=1.25] (31.center) to (32.center);
		\draw [style={red_edge}, bend right=45, looseness=1.25] (31.center) to (32.center);
		\draw [style={red_edge}] (34.center) to (32.center);
		\draw [style={red_edge}] (32.center) to (36.center);
		\draw [style={red_edge}] (32.center) to (38.center);
		\draw [style=blue edge] (35.center) to (32.center);
		\draw [style=blue edge] (32.center) to (37.center);
		\draw [style=blue edge, in=0, out=90] (29.center) to (33.center);
		\draw [style={green_edge}, in=90, out=-15] (40.center) to (30.center);
		\draw [style={green_edge}, in=-90, out=150, looseness=1.50] (40.center) to (39.center);
		\draw [style=blue edge] (46) to (43.center);
		\draw [style={green_edge}, in=90, out=90, looseness=2.25] (41.center) to (45.center);
		\draw [style={red_edge}, in=-165, out=120, looseness=1.25] (42.center) to (47.center);
		\draw [style={red_edge}, in=-15, out=60, looseness=1.25] (44.center) to (47.center);
		\draw [style=blue edge] (53) to (50.center);
		\draw [style={green_edge}, in=90, out=90, looseness=2.25] (48.center) to (52.center);
		\draw [style={red_edge}, in=-165, out=120, looseness=1.25] (49.center) to (54.center);
		\draw [style={red_edge}, in=-15, out=60, looseness=1.25] (51.center) to (54.center);
		\draw [style=blue edge] (60) to (57.center);
		\draw [style={green_edge}, in=90, out=90, looseness=2.25] (55.center) to (59.center);
		\draw [style={red_edge}, in=-165, out=120, looseness=1.25] (56.center) to (61.center);
		\draw [style={red_edge}, in=-15, out=60, looseness=1.25] (58.center) to (61.center);
		\draw [style={red_edge}, in=90, out=-15] (25.center) to (61.center);
		\draw [style={red_edge}, in=-165, out=90, looseness=0.75] (26.center) to (31.center);
		\draw [style=blue edge, in=-120, out=90, looseness=0.75] (27.center) to (31.center);
		\draw [style={red_edge}, in=-90, out=90] (47.center) to (31.center);
		\draw [style={red_edge}, in=90, out=-15, looseness=0.75] (31.center) to (54.center);
		\draw [style=blue edge, in=-45, out=90] (28.center) to (31.center);
	\end{pgfonlayer}
\end{tikzpicture}
\]
        Using Lemmas~\ref{lem:thick_relns} and \ref{lem:thick_jw}, we can simplify these diagrams using thick Soergel calculus as follows,
        \[
  \tikzsetnextfilename{K1_thick_1_1}
  \tikzstyle{every picture}=[tikzfig]
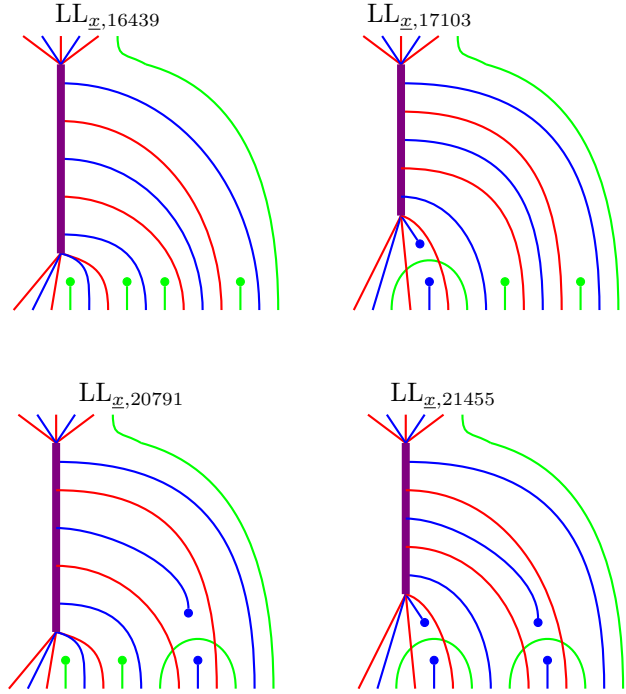
  \begin{tikzpicture}
	\begin{pgfonlayer}{nodelayer}
		\node [style=none] (0) at (-8, 0.5) {};
		\node [style=none] (1) at (-7.5, 0.5) {};
		\node [style=none] (2) at (-7, 0.5) {};
		\node [style=none] (3) at (-6.5, 0.5) {};
		\node [style=none] (4) at (-6, 0.5) {};
		\node [style=none] (5) at (-5.5, 0.5) {};
		\node [style=none] (6) at (-5, 0.5) {};
		\node [style=none] (7) at (-4.5, 0.5) {};
		\node [style=none] (8) at (-4, 0.5) {};
		\node [style=none] (9) at (-3.5, 0.5) {};
		\node [style=none] (10) at (-3, 0.5) {};
		\node [style=none] (11) at (-2.5, 0.5) {};
		\node [style=none] (12) at (-2, 0.5) {};
		\node [style=none] (13) at (-1.5, 0.5) {};
		\node [style=none] (14) at (-1, 0.5) {};
		\node [style={dot_green}] (15) at (-6.5, 1.25) {};
		\node [style={dot_green}] (16) at (-5, 1.25) {};
		\node [style={dot_green}] (17) at (-4, 1.25) {};
		\node [style={dot_green}] (18) at (-2, 1.25) {};
		\node [style=none] (19) at (-6.75, 2) {};
		\node [style=none] (24) at (-6.75, 7) {};
		\node [style=none] (25) at (-6.75, 2.5) {};
		\node [style=none] (26) at (-6.75, 4.5) {};
		\node [style=none] (27) at (-6.75, 6.5) {};
		\node [style=none] (28) at (-6.75, 3.5) {};
		\node [style=none] (29) at (-6.75, 5.5) {};
		\node [style=none] (30) at (-7.75, 7.75) {};
		\node [style=none] (31) at (-7.25, 7.75) {};
		\node [style=none] (32) at (-6.75, 7.75) {};
		\node [style=none] (33) at (-6.25, 7.75) {};
		\node [style=none] (34) at (-5.75, 7.75) {};
		\node [style=none] (35) at (-5.25, 7.75) {};
		\node [style=none] (36) at (-4.5, 7) {};
		\node [style=none] (37) at (1, 0.5) {};
		\node [style=none] (38) at (1.5, 0.5) {};
		\node [style=none] (39) at (2, 0.5) {};
		\node [style=none] (40) at (2.5, 0.5) {};
		\node [style=none] (41) at (3, 0.5) {};
		\node [style=none] (42) at (3.5, 0.5) {};
		\node [style=none] (43) at (4, 0.5) {};
		\node [style=none] (44) at (4.5, 0.5) {};
		\node [style=none] (45) at (5, 0.5) {};
		\node [style=none] (46) at (5.5, 0.5) {};
		\node [style=none] (47) at (6, 0.5) {};
		\node [style=none] (48) at (6.5, 0.5) {};
		\node [style=none] (49) at (7, 0.5) {};
		\node [style=none] (50) at (7.5, 0.5) {};
		\node [style=none] (51) at (8, 0.5) {};
		\node [style={dot_green}] (54) at (5, 1.25) {};
		\node [style={dot_green}] (55) at (7, 1.25) {};
		\node [style=none] (61) at (2.25, 7) {};
		\node [style=none] (63) at (2.25, 5) {};
		\node [style=none] (64) at (2.25, 6.5) {};
		\node [style=none] (66) at (2.25, 5.75) {};
		\node [style=none] (67) at (1.25, 7.75) {};
		\node [style=none] (68) at (1.75, 7.75) {};
		\node [style=none] (69) at (2.25, 7.75) {};
		\node [style=none] (70) at (2.75, 7.75) {};
		\node [style=none] (71) at (3.25, 7.75) {};
		\node [style=none] (72) at (3.75, 7.75) {};
		\node [style=none] (73) at (4.5, 7) {};
		\node [style={EW_t}] (74) at (3, 1.25) {};
		\node [style=none] (170) at (-5.5, 8.25) {$\textnormal{LL}_{\underline{x}, 16439}$};
		\node [style=none] (171) at (2.75, 8.25) {$\textnormal{LL}_{\underline{x}, 17103}$};
		\node [style=none] (174) at (2.25, 3) {};
		\node [style={EW_t}] (175) at (2.75, 2.25) {};
		\node [style=none] (176) at (2.25, 3.5) {};
		\node [style=none] (177) at (2.25, 4.25) {};
	\end{pgfonlayer}
	\begin{pgfonlayer}{edgelayer}
		\draw [style={red_edge}] (0.center) to (19.center);
		\draw [style={red_edge}] (19.center) to (2.center);
		\draw [style={red_edge}, in=90, out=-15, looseness=1.25] (19.center) to (5.center);
		\draw [style=blue edge] (1.center) to (19.center);
		\draw [style=blue edge, in=90, out=-15, looseness=1.25] (19.center) to (4.center);
		\draw [style={green_edge}] (15) to (3.center);
		\draw [style={green_edge}] (16) to (6.center);
		\draw [style={green_edge}] (17) to (8.center);
		\draw [style={green_edge}] (18) to (12.center);
		\draw [style={red_edge}] (30.center) to (24.center);
		\draw [style={red_edge}] (24.center) to (32.center);
		\draw [style={red_edge}] (24.center) to (34.center);
		\draw [style=blue edge] (31.center) to (24.center);
		\draw [style=blue edge] (24.center) to (33.center);
		\draw [style=blue edge, in=0, out=90, looseness=1.25] (7.center) to (25.center);
		\draw [style=blue edge, in=0, out=90] (10.center) to (26.center);
		\draw [style=blue edge, in=0, out=90] (13.center) to (27.center);
		\draw [style={red_edge}, in=0, out=90] (9.center) to (28.center);
		\draw [style={red_edge}, in=0, out=90] (11.center) to (29.center);
		\draw [style={green_edge}, in=90, out=-15] (36.center) to (14.center);
		\draw [style={green_edge}, in=-90, out=150, looseness=1.50] (36.center) to (35.center);
		\draw [style={green_edge}] (54) to (45.center);
		\draw [style={green_edge}] (55) to (49.center);
		\draw [style={red_edge}] (67.center) to (61.center);
		\draw [style={red_edge}] (61.center) to (69.center);
		\draw [style={red_edge}] (61.center) to (71.center);
		\draw [style=blue edge] (68.center) to (61.center);
		\draw [style=blue edge] (61.center) to (70.center);
		\draw [style=blue edge, in=0, out=90, looseness=1.25] (47.center) to (63.center);
		\draw [style=blue edge, in=0, out=90, looseness=1.25] (50.center) to (64.center);
		\draw [style={red_edge}, in=0, out=90, looseness=1.25] (48.center) to (66.center);
		\draw [style={green_edge}, in=90, out=-15] (73.center) to (51.center);
		\draw [style={green_edge}, in=-90, out=150, looseness=1.50] (73.center) to (72.center);
		\draw [style=blue edge] (74) to (41.center);
		\draw [style={green_edge}, in=90, out=90, looseness=2.25] (39.center) to (43.center);
		\draw [style={thick_purple}] (24.center) to (19.center);
		\draw [style={red_edge}] (37.center) to (174.center);
		\draw [style=blue edge] (38.center) to (174.center);
		\draw [style={red_edge}] (40.center) to (174.center);
		\draw [style=blue edge] (174.center) to (175);
		\draw [style={red_edge}, in=-15, out=90, looseness=0.75] (42.center) to (174.center);
		\draw [style={thick_purple}] (61.center) to (174.center);
		\draw [style={red_edge}, in=90, out=0, looseness=1.25] (177.center) to (46.center);
		\draw [style=blue edge, in=90, out=0] (176.center) to (44.center);
	\end{pgfonlayer}
\end{tikzpicture}
\]
        \[
  \tikzsetnextfilename{K1_thick_1_2}
  \tikzstyle{every picture}=[tikzfig]
  \begin{tikzpicture}
	\begin{pgfonlayer}{nodelayer}
		\node [style=none] (76) at (-8.25, -8.25) {};
		\node [style=none] (77) at (-7.75, -8.25) {};
		\node [style=none] (78) at (-7.25, -8.25) {};
		\node [style=none] (79) at (-6.75, -8.25) {};
		\node [style=none] (80) at (-6.25, -8.25) {};
		\node [style=none] (81) at (-5.75, -8.25) {};
		\node [style=none] (82) at (-5.25, -8.25) {};
		\node [style=none] (83) at (-4.75, -8.25) {};
		\node [style=none] (89) at (-1.75, -8.25) {};
		\node [style=none] (90) at (-1.25, -8.25) {};
		\node [style={dot_green}] (91) at (-6.75, -7.5) {};
		\node [style={dot_green}] (92) at (-5.25, -7.5) {};
		\node [style=none] (95) at (-7, -6.75) {};
		\node [style=none] (100) at (-7, -1.75) {};
		\node [style=none] (101) at (-7, -6) {};
		\node [style=none] (103) at (-7, -2.25) {};
		\node [style=none] (106) at (-8, -1) {};
		\node [style=none] (107) at (-7.5, -1) {};
		\node [style=none] (108) at (-7, -1) {};
		\node [style=none] (109) at (-6.5, -1) {};
		\node [style=none] (110) at (-6, -1) {};
		\node [style=none] (111) at (-5.5, -1) {};
		\node [style=none] (112) at (-4.75, -1.75) {};
		\node [style=none] (113) at (-7, -3) {};
		\node [style=none] (129) at (7.5, -8.25) {};
		\node [style=none] (130) at (8, -8.25) {};
		\node [style=none] (136) at (2.25, -1.75) {};
		\node [style=none] (138) at (2.25, -2.25) {};
		\node [style=none] (139) at (1.25, -1) {};
		\node [style=none] (140) at (1.75, -1) {};
		\node [style=none] (141) at (2.25, -1) {};
		\node [style=none] (142) at (2.75, -1) {};
		\node [style=none] (143) at (3.25, -1) {};
		\node [style=none] (144) at (3.75, -1) {};
		\node [style=none] (145) at (4.5, -1.75) {};
		\node [style=none] (156) at (5, -8.25) {};
		\node [style=none] (157) at (5.5, -8.25) {};
		\node [style=none] (158) at (6, -8.25) {};
		\node [style=none] (159) at (6.5, -8.25) {};
		\node [style=none] (160) at (7, -8.25) {};
		\node [style={EW_t}] (161) at (6, -7.5) {};
		\node [style=none] (163) at (-4.25, -8.25) {};
		\node [style=none] (164) at (-3.75, -8.25) {};
		\node [style=none] (165) at (-3.25, -8.25) {};
		\node [style=none] (166) at (-2.75, -8.25) {};
		\node [style=none] (167) at (-2.25, -8.25) {};
		\node [style={EW_t}] (168) at (-3.25, -7.5) {};
		\node [style=none] (172) at (-5, -0.5) {$\textnormal{LL}_{\underline{x}, 20791}$};
		\node [style=none] (173) at (3.25, -0.5) {$\textnormal{LL}_{\underline{x}, 21455}$};
		\node [style=none] (178) at (-7, -4) {};
		\node [style=none] (179) at (-7, -5) {};
		\node [style={EW_t}] (180) at (-3.5, -6.25) {};
		\node [style=none] (181) at (1, -8.25) {};
		\node [style=none] (182) at (1.5, -8.25) {};
		\node [style=none] (183) at (2, -8.25) {};
		\node [style=none] (184) at (2.5, -8.25) {};
		\node [style=none] (185) at (3, -8.25) {};
		\node [style=none] (186) at (3.5, -8.25) {};
		\node [style=none] (187) at (4, -8.25) {};
		\node [style=none] (188) at (4.5, -8.25) {};
		\node [style=none] (189) at (2.25, -3.75) {};
		\node [style=none] (190) at (2.25, -3) {};
		\node [style={EW_t}] (191) at (3, -7.5) {};
		\node [style=none] (192) at (2.25, -5.75) {};
		\node [style={EW_t}] (193) at (2.75, -6.5) {};
		\node [style=none] (194) at (2.25, -5.25) {};
		\node [style=none] (195) at (2.25, -4.5) {};
		\node [style={EW_t}] (196) at (5.75, -6.5) {};
	\end{pgfonlayer}
	\begin{pgfonlayer}{edgelayer}
		\draw [style={red_edge}] (76.center) to (95.center);
		\draw [style={red_edge}] (95.center) to (78.center);
		\draw [style={red_edge}, in=90, out=-15, looseness=1.25] (95.center) to (81.center);
		\draw [style=blue edge] (77.center) to (95.center);
		\draw [style=blue edge, in=90, out=-15, looseness=1.25] (95.center) to (80.center);
		\draw [style={green_edge}] (91) to (79.center);
		\draw [style={green_edge}] (92) to (82.center);
		\draw [style={red_edge}] (106.center) to (100.center);
		\draw [style={red_edge}] (100.center) to (108.center);
		\draw [style={red_edge}] (100.center) to (110.center);
		\draw [style=blue edge] (107.center) to (100.center);
		\draw [style=blue edge] (100.center) to (109.center);
		\draw [style=blue edge, in=0, out=90, looseness=1.25] (83.center) to (101.center);
		\draw [style=blue edge, in=0, out=90, looseness=1.25] (89.center) to (103.center);
		\draw [style={green_edge}, in=90, out=-15] (112.center) to (90.center);
		\draw [style={green_edge}, in=-90, out=150, looseness=1.50] (112.center) to (111.center);
		\draw [style={red_edge}] (139.center) to (136.center);
		\draw [style={red_edge}] (136.center) to (141.center);
		\draw [style={red_edge}] (136.center) to (143.center);
		\draw [style=blue edge] (140.center) to (136.center);
		\draw [style=blue edge] (136.center) to (142.center);
		\draw [style=blue edge, in=0, out=90, looseness=1.25] (129.center) to (138.center);
		\draw [style={green_edge}, in=90, out=-15] (145.center) to (130.center);
		\draw [style={green_edge}, in=-90, out=150, looseness=1.50] (145.center) to (144.center);
		\draw [style=blue edge] (161) to (158.center);
		\draw [style={green_edge}, in=90, out=90, looseness=2.25] (156.center) to (160.center);
		\draw [style=blue edge] (168) to (165.center);
		\draw [style={green_edge}, in=90, out=90, looseness=2.25] (163.center) to (167.center);
		\draw [style={thick_purple}] (95.center) to (100.center);
		\draw [style={red_edge}, in=90, out=0] (179.center) to (164.center);
		\draw [style={red_edge}, in=0, out=90, looseness=1.25] (166.center) to (113.center);
		\draw [style=blue edge, in=90, out=0, looseness=0.75] (178.center) to (180);
		\draw [style=blue edge] (191) to (185.center);
		\draw [style={green_edge}, in=90, out=90, looseness=2.25] (183.center) to (187.center);
		\draw [style={red_edge}] (181.center) to (192.center);
		\draw [style=blue edge] (182.center) to (192.center);
		\draw [style={red_edge}] (184.center) to (192.center);
		\draw [style=blue edge] (192.center) to (193);
		\draw [style={red_edge}, in=-15, out=90, looseness=0.75] (186.center) to (192.center);
		\draw [style=blue edge, in=90, out=0] (194.center) to (188.center);
		\draw [style={thick_purple}] (192.center) to (136.center);
		\draw [style={red_edge}, in=0, out=90] (157.center) to (195.center);
		\draw [style={red_edge}, in=0, out=90] (159.center) to (190.center);
		\draw [style=blue edge, in=90, out=0, looseness=0.75] (189.center) to (196);
	\end{pgfonlayer}
\end{tikzpicture}
\]
        Before simplifying further, we will introduce some new notation,
        \begin{equation}\label{eq:half_JW_reln}
                
  \tikzsetnextfilename{half_JW}
  \tikzstyle{every picture}=[tikzfig]
  \begin{tikzpicture}
	\begin{pgfonlayer}{nodelayer}
		\node [style=none] (36) at (-3.75, -1) {};
		\node [style=none] (38) at (-2.75, -1) {};
		\node [style=none] (39) at (-3.75, 1) {};
		\node [style=none] (40) at (-2.75, 1) {};
		\node [style={thick_dot_ochre}] (41) at (-3.25, 0) {};
		\node [style=none] (42) at (-2.25, 0) {$\coloneq$};
		\node [style=none] (43) at (-1.75, -1) {};
		\node [style=none] (44) at (-1.25, -1) {};
		\node [style={EW_t}] (45) at (-0.75, -0.75) {};
		\node [style=none] (46) at (-1.25, -0.5) {};
		\node [style=none] (47) at (-1.75, 1) {};
		\node [style=none] (48) at (-1.25, 1) {};
		\node [style={dot_green}] (49) at (-0.75, 0.75) {};
		\node [style=none] (50) at (-1.25, 0.5) {};
		\node [style=none] (53) at (0, 0) {$=$};
		\node [style=none] (54) at (0.75, -1) {};
		\node [style=none] (55) at (0.75, 1) {};
		\node [style=none] (56) at (1.75, -1) {};
		\node [style=none] (57) at (1.75, 1) {};
		\node [style=none] (58) at (2.25, 0) {$+$};
		\node [style=none] (59) at (2.75, -1) {};
		\node [style=none] (60) at (2.75, 1) {};
		\node [style=none] (61) at (3.75, 1) {};
		\node [style=none] (62) at (3.75, -1) {};
		\node [style={dot_green}] (63) at (1.75, -0.25) {};
		\node [style={dot_green}] (64) at (0.75, 0.25) {};
		\node [style={EW_t}] (65) at (3.75, 0.25) {};
		\node [style={EW_t}] (66) at (2.75, -0.25) {};
	\end{pgfonlayer}
	\begin{pgfonlayer}{edgelayer}
		\draw [style=blue edge] (36.center) to (41);
		\draw [style=blue edge] (41) to (40.center);
		\draw [style={green_edge}] (38.center) to (41);
		\draw [style={green_edge}] (41) to (39.center);
		\draw [style=blue edge] (43.center) to (46.center);
		\draw [style=blue edge] (46.center) to (45);
		\draw [style={green_edge}] (50.center) to (49);
		\draw [style={green_edge}] (50.center) to (47.center);
		\draw [style=blue edge] (50.center) to (48.center);
		\draw [style={green_edge}] (46.center) to (44.center);
		\draw [style={thick_teal}] (50.center) to (46.center);
		\draw [style=blue edge] (54.center) to (57.center);
		\draw [style=blue edge] (66) to (59.center);
		\draw [style=blue edge] (65) to (61.center);
		\draw [style={green_edge}] (60.center) to (62.center);
		\draw [style={green_edge}] (55.center) to (64);
		\draw [style={green_edge}] (63) to (56.center);
	\end{pgfonlayer}
\end{tikzpicture}

        \end{equation}
        Note that the second equality follows from Lemma~\ref{lem:thick_jw} and the Elias--Jones--Wenzl relation.
        We can then use \eqref{eq:half_JW_reln} and the last relation from Lemma~\ref{lem:thick_relns} to perform our last simplification of $K_1$,
        \begin{equation}
                K_1 = 
  \tikzsetnextfilename{K1_thick_2}
  \tikzstyle{every picture}=[tikzfig]
  \begin{tikzpicture}
	\begin{pgfonlayer}{nodelayer}
		\node [style=none] (0) at (-3.25, -2.25) {};
		\node [style=none] (1) at (-2.75, -2.25) {};
		\node [style=none] (2) at (-2.25, -2.25) {};
		\node [style=none] (3) at (-1.75, -2.25) {};
		\node [style=none] (4) at (-1.25, -2.25) {};
		\node [style=none] (5) at (-0.75, -2.25) {};
		\node [style=none] (6) at (-0.25, -2.25) {};
		\node [style=none] (7) at (0.25, -2.25) {};
		\node [style=none] (8) at (0.75, -2.25) {};
		\node [style=none] (9) at (1.25, -2.25) {};
		\node [style=none] (10) at (1.75, -2.25) {};
		\node [style=none] (11) at (2.25, -2.25) {};
		\node [style=none] (12) at (2.75, -2.25) {};
		\node [style=none] (13) at (3.25, -2.25) {};
		\node [style=none] (14) at (3.75, -2.25) {};
		\node [style=none] (30) at (-1, 2.25) {};
		\node [style=none] (31) at (-0.5, 2.25) {};
		\node [style=none] (32) at (0, 2.25) {};
		\node [style=none] (33) at (0.5, 2.25) {};
		\node [style=none] (34) at (1, 2.25) {};
		\node [style={thick_dot_ochre}] (36) at (1.75, -1.5) {};
		\node [style=none] (37) at (1.75, -0.25) {};
		\node [style={thick_dot_ochre}] (39) at (-1.25, -1.5) {};
		\node [style=none] (40) at (-1.75, -0.25) {};
		\node [style=none] (41) at (0, 0.75) {};
		\node [style=none] (42) at (0, 1.5) {};
		\node [style=none] (43) at (1.5, 2.25) {};
	\end{pgfonlayer}
	\begin{pgfonlayer}{edgelayer}
		\draw [style=blue edge] (36) to (37.center);
		\draw [style={red_edge}, in=-135, out=90, looseness=0.75] (9.center) to (37.center);
		\draw [style={red_edge}, in=-30, out=90, looseness=0.75] (11.center) to (37.center);
		\draw [style=blue edge, in=-165, out=90, looseness=0.75] (7.center) to (37.center);
		\draw [style=blue edge, in=-15, out=90] (13.center) to (37.center);
		\draw [style={red_edge}, in=-165, out=90] (0.center) to (40.center);
		\draw [style={red_edge}, in=-120, out=90, looseness=0.75] (2.center) to (40.center);
		\draw [style=blue edge, in=-135, out=90] (1.center) to (40.center);
		\draw [style=blue edge, bend left=15] (40.center) to (39);
		\draw [style={red_edge}, in=-15, out=90, looseness=1.25] (5.center) to (40.center);
		\draw [style={thick_purple}, in=-135, out=90, looseness=0.75] (40.center) to (41.center);
		\draw [style={thick_purple}, in=90, out=-60, looseness=0.75] (41.center) to (37.center);
		\draw [style={thick_purple}] (42.center) to (41.center);
		\draw [style={red_edge}] (30.center) to (42.center);
		\draw [style={red_edge}] (42.center) to (32.center);
		\draw [style={red_edge}] (42.center) to (34.center);
		\draw [style=blue edge] (31.center) to (42.center);
		\draw [style=blue edge] (42.center) to (33.center);
		\draw [style={green_edge}, in=90, out=-90, looseness=0.50] (43.center) to (14.center);
		\draw [style=blue edge] (36) to (10.center);
		\draw [style={green_edge}] (8.center) to (36);
		\draw [style={green_edge}] (36) to (12.center);
		\draw [style={green_edge}] (39) to (6.center);
		\draw [style={green_edge}] (39) to (3.center);
		\draw [style=blue edge] (39) to (4.center);
	\end{pgfonlayer}
\end{tikzpicture}

        \end{equation}
        The simplification of $K_2$ proceeds in the same manner. The end result is the following thick diagram, 
        \[ K_2 = 
  \tikzsetnextfilename{K2_thick}
  \tikzstyle{every picture}=[tikzfig]
  \begin{tikzpicture}
	\begin{pgfonlayer}{nodelayer}
		\node [style=none] (0) at (3.75, 2.25) {};
		\node [style=none] (1) at (3.25, 2.25) {};
		\node [style=none] (2) at (2.75, 2.25) {};
		\node [style=none] (3) at (2.25, 2.25) {};
		\node [style=none] (4) at (1.75, 2.25) {};
		\node [style=none] (5) at (1.25, 2.25) {};
		\node [style=none] (6) at (0.75, 2.25) {};
		\node [style=none] (7) at (0.25, 2.25) {};
		\node [style=none] (8) at (-0.25, 2.25) {};
		\node [style=none] (9) at (-0.75, 2.25) {};
		\node [style=none] (10) at (-1.25, 2.25) {};
		\node [style=none] (11) at (-1.75, 2.25) {};
		\node [style=none] (12) at (-2.25, 2.25) {};
		\node [style=none] (13) at (-2.75, 2.25) {};
		\node [style=none] (14) at (-3.25, 2.25) {};
		\node [style=none] (30) at (2, -2) {};
		\node [style=none] (31) at (1.5, -2) {};
		\node [style=none] (32) at (1, -2) {};
		\node [style=none] (33) at (0.5, -2) {};
		\node [style=none] (34) at (0, -2) {};
		\node [style=none] (43) at (-0.5, -2) {};
		\node [style=none] (47) at (-1.25, 0.5) {};
		\node [style=none] (48) at (0.5, -0.5) {};
		\node [style=none] (49) at (0.5, -1.25) {};
		\node [style={thick_dot_ochre}] (50) at (-2.75, 1.75) {};
		\node [style={thick_dot_ochre}] (52) at (0.25, 1.75) {};
		\node [style=none] (53) at (2.5, 0.5) {};
		\node [style={EW_t}] (54) at (3, 1) {};
		\node [style={EW_t}] (55) at (3.25, 1.75) {};
		\node [style=none] (56) at (3.75, 0.5) {};
	\end{pgfonlayer}
	\begin{pgfonlayer}{edgelayer}
		\draw [style=blue edge] (7.center) to (52);
		\draw [style=blue edge, in=150, out=-75, looseness=0.75] (50) to (47.center);
		\draw [style=blue edge] (47.center) to (10.center);
		\draw [style={green_edge}] (8.center) to (52);
		\draw [style={green_edge}] (52) to (6.center);
		\draw [style={red_edge}, in=135, out=-90] (11.center) to (47.center);
		\draw [style={red_edge}, in=-90, out=45] (47.center) to (9.center);
		\draw [style={thick_purple}, in=90, out=-90, looseness=0.75] (47.center) to (48.center);
		\draw [style={thick_purple}] (48.center) to (49.center);
		\draw [style={red_edge}] (49.center) to (43.center);
		\draw [style={red_edge}] (49.center) to (33.center);
		\draw [style={red_edge}] (49.center) to (31.center);
		\draw [style=blue edge] (49.center) to (34.center);
		\draw [style=blue edge] (49.center) to (32.center);
		\draw [style={thick_purple}, in=-90, out=75, looseness=0.75] (48.center) to (53.center);
		\draw [style={red_edge}, in=135, out=-90] (5.center) to (53.center);
		\draw [style={red_edge}] (2.center) to (53.center);
		\draw [style=blue edge, in=105, out=-90] (4.center) to (53.center);
		\draw [style=blue edge] (55) to (1.center);
		\draw [style=blue edge] (54) to (53.center);
		\draw [style={red_edge}, in=-90, out=15] (53.center) to (0.center);
		\draw [style={green_edge}, in=90, out=-90] (3.center) to (56.center);
		\draw [style={green_edge}, in=-90, out=90] (30.center) to (56.center);
		\draw [style=blue edge] (13.center) to (50);
		\draw [style=blue edge, in=-105, out=30, looseness=0.75] (47.center) to (52);
		\draw [style={green_edge}] (14.center) to (50);
		\draw [style={green_edge}] (50) to (12.center);
	\end{pgfonlayer}
\end{tikzpicture}
\]
        Finally, we compose these maps and use Lemma~\ref{lem:thick_relns} one more time to get our final simplification of $K$,
        \[ K = 
  \tikzsetnextfilename{K_thick}
  \tikzstyle{every picture}=[tikzfig]
  \begin{tikzpicture}
	\begin{pgfonlayer}{nodelayer}
		\node [style=none] (0) at (3.25, 3.25) {};
		\node [style=none] (1) at (2.75, 3.25) {};
		\node [style=none] (2) at (2.25, 3.25) {};
		\node [style=none] (3) at (1.75, 3.25) {};
		\node [style=none] (4) at (1.25, 3.25) {};
		\node [style=none] (5) at (0.75, 3.25) {};
		\node [style=none] (6) at (0.25, 3.25) {};
		\node [style=none] (7) at (-0.25, 3.25) {};
		\node [style=none] (8) at (-0.75, 3.25) {};
		\node [style=none] (9) at (-1.25, 3.25) {};
		\node [style=none] (10) at (-1.75, 3.25) {};
		\node [style=none] (11) at (-2.25, 3.25) {};
		\node [style=none] (12) at (-2.75, 3.25) {};
		\node [style=none] (13) at (-3.25, 3.25) {};
		\node [style=none] (14) at (-3.75, 3.25) {};
		\node [style=none] (47) at (-1.75, 1.5) {};
		\node [style={thick_dot_ochre}] (50) at (-3.25, 2.75) {};
		\node [style={thick_dot_ochre}] (52) at (-0.25, 2.75) {};
		\node [style=none] (53) at (2, 1.5) {};
		\node [style={EW_t}] (54) at (2.5, 2) {};
		\node [style={EW_t}] (55) at (2.75, 2.75) {};
		\node [style=none] (56) at (3.25, 1.5) {};
		\node [style=none] (57) at (-3.25, -3.25) {};
		\node [style=none] (58) at (-2.75, -3.25) {};
		\node [style=none] (59) at (-2.25, -3.25) {};
		\node [style=none] (60) at (-1.75, -3.25) {};
		\node [style=none] (61) at (-1.25, -3.25) {};
		\node [style=none] (62) at (-0.75, -3.25) {};
		\node [style=none] (63) at (-0.25, -3.25) {};
		\node [style=none] (64) at (0.25, -3.25) {};
		\node [style=none] (65) at (0.75, -3.25) {};
		\node [style=none] (66) at (1.25, -3.25) {};
		\node [style=none] (67) at (1.75, -3.25) {};
		\node [style=none] (68) at (2.25, -3.25) {};
		\node [style=none] (69) at (2.75, -3.25) {};
		\node [style=none] (70) at (3.25, -3.25) {};
		\node [style=none] (71) at (3.75, -3.25) {};
		\node [style={thick_dot_ochre}] (78) at (1.75, -2.5) {};
		\node [style=none] (79) at (1.75, -1.25) {};
		\node [style={thick_dot_ochre}] (81) at (-1.25, -2.5) {};
		\node [style=none] (82) at (-1.75, -1.25) {};
		\node [style=none] (83) at (0, -0.25) {};
		\node [style=none] (84) at (0, 0.5) {};
	\end{pgfonlayer}
	\begin{pgfonlayer}{edgelayer}
		\draw [style=blue edge] (7.center) to (52);
		\draw [style=blue edge, in=150, out=-75, looseness=0.75] (50) to (47.center);
		\draw [style=blue edge] (47.center) to (10.center);
		\draw [style={green_edge}] (8.center) to (52);
		\draw [style={green_edge}] (52) to (6.center);
		\draw [style={red_edge}, in=135, out=-90] (11.center) to (47.center);
		\draw [style={red_edge}, in=-90, out=45] (47.center) to (9.center);
		\draw [style={red_edge}, in=135, out=-90] (5.center) to (53.center);
		\draw [style={red_edge}] (2.center) to (53.center);
		\draw [style=blue edge, in=105, out=-90] (4.center) to (53.center);
		\draw [style=blue edge] (55) to (1.center);
		\draw [style=blue edge] (54) to (53.center);
		\draw [style={red_edge}, in=-90, out=15] (53.center) to (0.center);
		\draw [style={green_edge}, in=90, out=-90] (3.center) to (56.center);
		\draw [style=blue edge] (78) to (79.center);
		\draw [style={red_edge}, in=-135, out=90, looseness=0.75] (66.center) to (79.center);
		\draw [style={red_edge}, in=-30, out=90, looseness=0.75] (68.center) to (79.center);
		\draw [style=blue edge, in=-165, out=90, looseness=0.75] (64.center) to (79.center);
		\draw [style=blue edge, in=-15, out=90] (70.center) to (79.center);
		\draw [style={red_edge}, in=-165, out=90] (57.center) to (82.center);
		\draw [style={red_edge}, in=-120, out=90, looseness=0.75] (59.center) to (82.center);
		\draw [style=blue edge, in=-135, out=90] (58.center) to (82.center);
		\draw [style=blue edge, bend left=15] (82.center) to (81);
		\draw [style={red_edge}, in=-15, out=90, looseness=1.25] (62.center) to (82.center);
		\draw [style={thick_purple}, in=-135, out=90, looseness=0.75] (82.center) to (83.center);
		\draw [style={thick_purple}, in=90, out=-60, looseness=0.75] (83.center) to (79.center);
		\draw [style={thick_purple}] (84.center) to (83.center);
		\draw [style={green_edge}, in=270, out=90] (71.center) to (56.center);
		\draw [style=blue edge] (13.center) to (50);
		\draw [style=blue edge, in=-120, out=30, looseness=0.75] (47.center) to (52);
		\draw [style={green_edge}] (14.center) to (50);
		\draw [style={green_edge}] (50) to (12.center);
		\draw [style={thick_purple}, in=90, out=-90, looseness=0.75] (47.center) to (84.center);
		\draw [style={thick_purple}] (84.center) to (83.center);
		\draw [style={thick_purple}, in=-90, out=75, looseness=0.75] (84.center) to (53.center);
		\draw [style=blue edge] (78) to (67.center);
		\draw [style={green_edge}] (65.center) to (78);
		\draw [style={green_edge}] (78) to (69.center);
		\draw [style={green_edge}] (81) to (63.center);
		\draw [style={green_edge}] (81) to (60.center);
		\draw [style=blue edge] (81) to (61.center);
	\end{pgfonlayer}
\end{tikzpicture}
\]
        We have now simplified 16 of the double leaves appearing in the coefficient expansion of $\varphi$ to $-2K$.

        Next, we will consider the block of $C$ corresponding to $w = wststs$, shown in Figure~\ref{fig:harder_block}. Let $F : B_{\ux} \to B_{\uy}$ be its contribution to $\varphi$.

        \begin{figure}[H]
                \begin{center}
                        \begin{tabular}{r|*{7}{c}}
                        $\ue \backslash \uf$ & $186$ & $1742$ & $2490$ & $11214$ & $701$ & $1402$ & $5501$ \\
                        \hline
                        $55$ & $0$ & $2$ & $2$ & $2$ & $-2$ & $-2$ & $-2$ \\
                        $375$ & $2$ & $2$ & $2$ & $2$ & $0$ & $0$ & $0$ \\
                        $20535$ &  $2$ & $2$ & $2$ & $2$ & $0$ & $0$ & $0$ \\
                        $20855$ &  $2$ & $2$ & $2$ & $2$ & $0$ & $0$ & $0$ \\
                        $719$ & $-2$ & $0$ & $0$ & $0$ & $-2$ & $-2$ & $-2$ \\
                        $4407$ & $-2$ & $0$ & $0$ & $0$ & $-2$ & $-2$ & $-2$ \\
                        $5071$ & $-2$ & $0$ & $0$ & $0$ & $-2$ & $-2$ & $-2$ 
                        \end{tabular} 
                \end{center}
                \caption{The $ststs$-block of the coefficient matrix.}\label{fig:harder_block}
        \end{figure}

        We can write $F = 2 (G-H)$ where
\begin{align*}
G&=(\oLL_{\uy,186}+\oLL_{\uy,1742} +\oLL_{\uy,2490}+\oLL_{\uy,11214}) \circ(\LL_{\ux,55}+\LL_{\ux,375} +\LL_{\ux,20535}+\LL_{\ux,20855}),\\
H&=(\oLL_{\uy,186}+\oLL_{\uy,701} +\oLL_{\uy,1402}+\oLL_{\uy,5501}) \circ(\LL_{\ux,55}+\LL_{\ux,719} +\LL_{\ux,4407}+\LL_{\ux,5071}).
\end{align*}
        We can then simplify $G$ and $H$ into single thick diagrams in an analogous fashion as we did for $K$. 
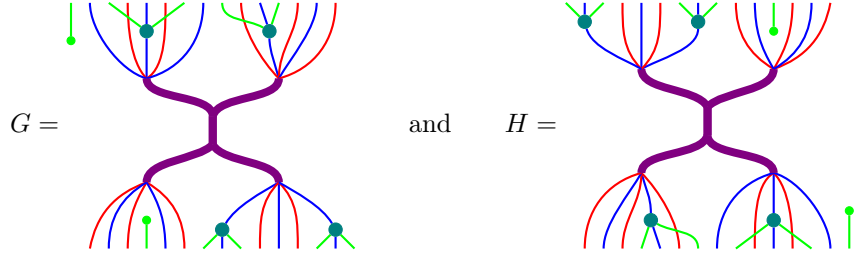
        \[G = 
  \tikzsetnextfilename{P_thick}
  \tikzstyle{every picture}=[tikzfig]
  \begin{tikzpicture}
	\begin{pgfonlayer}{nodelayer}
		\node [style=none] (0) at (3.25, 3.25) {};
		\node [style=none] (14) at (-3.75, 3.25) {};
		\node [style=none] (57) at (-3.25, -3.25) {};
		\node [style=none] (58) at (-2.75, -3.25) {};
		\node [style=none] (59) at (-2.25, -3.25) {};
		\node [style=none] (60) at (-1.75, -3.25) {};
		\node [style=none] (61) at (-1.25, -3.25) {};
		\node [style=none] (62) at (-0.75, -3.25) {};
		\node [style=none] (82) at (-1.75, -1.5) {};
		\node [style=none] (83) at (0, -0.5) {};
		\node [style=none] (84) at (0, 0.25) {};
		\node [style=none] (85) at (3.75, -3.25) {};
		\node [style=none] (86) at (3.25, -3.25) {};
		\node [style=none] (87) at (2.75, -3.25) {};
		\node [style=none] (88) at (2.25, -3.25) {};
		\node [style=none] (89) at (1.75, -3.25) {};
		\node [style=none] (90) at (1.25, -3.25) {};
		\node [style=none] (91) at (0.75, -3.25) {};
		\node [style=none] (92) at (0.25, -3.25) {};
		\node [style=none] (93) at (-0.25, -3.25) {};
		\node [style=none] (95) at (1.75, -1.5) {};
		\node [style={thick_dot_ochre}] (96) at (0.25, -2.75) {};
		\node [style={thick_dot_ochre}] (98) at (3.25, -2.75) {};
		\node [style=none] (99) at (-3.25, 3.25) {};
		\node [style=none] (100) at (-2.75, 3.25) {};
		\node [style=none] (101) at (-2.25, 3.25) {};
		\node [style=none] (102) at (-1.75, 3.25) {};
		\node [style=none] (103) at (-1.25, 3.25) {};
		\node [style=none] (104) at (-0.75, 3.25) {};
		\node [style=none] (105) at (-0.25, 3.25) {};
		\node [style={thick_dot_ochre}] (107) at (-1.75, 2.5) {};
		\node [style=none] (108) at (-1.75, 1.25) {};
		\node [style={dot_green}] (109) at (-3.75, 2.25) {};
		\node [style={dot_green}] (110) at (-1.75, -2.5) {};
		\node [style=none] (118) at (0.25, 3.25) {};
		\node [style=none] (119) at (0.75, 3.25) {};
		\node [style=none] (120) at (1.25, 3.25) {};
		\node [style=none] (121) at (2.25, 3.25) {};
		\node [style=none] (122) at (1.75, 3.25) {};
		\node [style=none] (123) at (2.75, 3.25) {};
		\node [style={thick_dot_ochre}] (125) at (1.5, 2.5) {};
		\node [style=none] (126) at (1.75, 1.25) {};
	\end{pgfonlayer}
	\begin{pgfonlayer}{edgelayer}
		\draw [style={thick_purple}, in=-135, out=90, looseness=0.75] (82.center) to (83.center);
		\draw [style={thick_purple}] (84.center) to (83.center);
		\draw [style=blue edge] (86.center) to (98);
		\draw [style=blue edge, in=-150, out=60, looseness=0.75] (96) to (95.center);
		\draw [style=blue edge] (95.center) to (89.center);
		\draw [style={green_edge}] (87.center) to (98);
		\draw [style={green_edge}] (98) to (85.center);
		\draw [style={red_edge}, in=-135, out=90] (90.center) to (95.center);
		\draw [style={red_edge}, in=90, out=-45] (95.center) to (88.center);
		\draw [style=blue edge] (107) to (108.center);
		\draw [style={red_edge}, in=135, out=-90, looseness=0.75] (101.center) to (108.center);
		\draw [style={red_edge}, in=30, out=-90, looseness=0.75] (103.center) to (108.center);
		\draw [style=blue edge, in=165, out=-90, looseness=0.75] (99.center) to (108.center);
		\draw [style=blue edge, in=15, out=-90] (105.center) to (108.center);
		\draw [style={green_edge}] (14.center) to (109);
		\draw [style={red_edge}, bend left] (57.center) to (82.center);
		\draw [style={red_edge}, bend left=15] (59.center) to (82.center);
		\draw [style={red_edge}, bend left] (82.center) to (62.center);
		\draw [style=blue edge, bend left=15] (82.center) to (61.center);
		\draw [style=blue edge, bend left] (58.center) to (82.center);
		\draw [style={green_edge}] (110) to (60.center);
		\draw [style=blue edge, in=105, out=-90, looseness=0.75] (125) to (126.center);
		\draw [style={red_edge}, in=135, out=-90, looseness=0.75] (119.center) to (126.center);
		\draw [style={red_edge}, in=90, out=-90] (121.center) to (126.center);
		\draw [style=blue edge, in=60, out=-90, looseness=0.75] (123.center) to (126.center);
		\draw [style={red_edge}, in=-90, out=15] (126.center) to (0.center);
		\draw [style={thick_purple}, in=90, out=-90, looseness=0.75] (108.center) to (84.center);
		\draw [style={thick_purple}] (84.center) to (83.center);
		\draw [style={thick_purple}, in=-90, out=75, looseness=0.75] (84.center) to (126.center);
		\draw [style={thick_purple}, in=90, out=-60, looseness=0.75] (83.center) to (95.center);
		\draw [style=blue edge] (92.center) to (96);
		\draw [style=blue edge, in=120, out=-30, looseness=0.75] (95.center) to (98);
		\draw [style={green_edge}] (93.center) to (96);
		\draw [style={green_edge}] (96) to (91.center);
		\draw [style=blue edge] (107) to (102.center);
		\draw [style={green_edge}] (100.center) to (107);
		\draw [style={green_edge}] (107) to (104.center);
		\draw [style=blue edge] (125) to (120.center);
		\draw [style={green_edge}, in=150, out=-90] (118.center) to (125);
		\draw [style={green_edge}] (125) to (122.center);
	\end{pgfonlayer}
\end{tikzpicture}
\qquad\text{and}\qquad H=
  \tikzsetnextfilename{Q_thick}
  \tikzstyle{every picture}=[tikzfig]
  \begin{tikzpicture}
	\begin{pgfonlayer}{nodelayer}
		\node [style=none] (0) at (-3.25, -3.25) {};
		\node [style=none] (14) at (3.75, -3.25) {};
		\node [style=none] (57) at (3.25, 3.25) {};
		\node [style=none] (58) at (2.75, 3.25) {};
		\node [style=none] (59) at (2.25, 3.25) {};
		\node [style=none] (60) at (1.75, 3.25) {};
		\node [style=none] (61) at (1.25, 3.25) {};
		\node [style=none] (62) at (0.75, 3.25) {};
		\node [style=none] (82) at (1.75, 1.5) {};
		\node [style=none] (83) at (0, 0.5) {};
		\node [style=none] (84) at (0, -0.25) {};
		\node [style=none] (85) at (-3.75, 3.25) {};
		\node [style=none] (86) at (-3.25, 3.25) {};
		\node [style=none] (87) at (-2.75, 3.25) {};
		\node [style=none] (88) at (-2.25, 3.25) {};
		\node [style=none] (89) at (-1.75, 3.25) {};
		\node [style=none] (90) at (-1.25, 3.25) {};
		\node [style=none] (91) at (-0.75, 3.25) {};
		\node [style=none] (92) at (-0.25, 3.25) {};
		\node [style=none] (93) at (0.25, 3.25) {};
		\node [style=none] (95) at (-1.75, 1.5) {};
		\node [style={thick_dot_ochre}] (96) at (-0.25, 2.75) {};
		\node [style={thick_dot_ochre}] (98) at (-3.25, 2.75) {};
		\node [style=none] (99) at (3.25, -3.25) {};
		\node [style=none] (100) at (2.75, -3.25) {};
		\node [style=none] (101) at (2.25, -3.25) {};
		\node [style=none] (102) at (1.75, -3.25) {};
		\node [style=none] (103) at (1.25, -3.25) {};
		\node [style=none] (104) at (0.75, -3.25) {};
		\node [style=none] (105) at (0.25, -3.25) {};
		\node [style={thick_dot_ochre}] (107) at (1.75, -2.5) {};
		\node [style=none] (108) at (1.75, -1.25) {};
		\node [style={dot_green}] (109) at (3.75, -2.25) {};
		\node [style={dot_green}] (110) at (1.75, 2.5) {};
		\node [style=none] (118) at (-0.25, -3.25) {};
		\node [style=none] (119) at (-0.75, -3.25) {};
		\node [style=none] (120) at (-1.25, -3.25) {};
		\node [style=none] (121) at (-2.25, -3.25) {};
		\node [style=none] (122) at (-1.75, -3.25) {};
		\node [style=none] (123) at (-2.75, -3.25) {};
		\node [style={thick_dot_ochre}] (125) at (-1.5, -2.5) {};
		\node [style=none] (126) at (-1.75, -1.25) {};
	\end{pgfonlayer}
	\begin{pgfonlayer}{edgelayer}
		\draw [style={thick_purple}, in=45, out=-90, looseness=0.75] (82.center) to (83.center);
		\draw [style={thick_purple}] (84.center) to (83.center);
		\draw [style=blue edge] (86.center) to (98);
		\draw [style=blue edge, in=30, out=-90, looseness=0.75] (96) to (95.center);
		\draw [style=blue edge] (95.center) to (89.center);
		\draw [style={green_edge}] (87.center) to (98);
		\draw [style={green_edge}] (98) to (85.center);
		\draw [style={red_edge}, in=45, out=-90] (90.center) to (95.center);
		\draw [style={red_edge}, in=-90, out=135] (95.center) to (88.center);
		\draw [style=blue edge] (107) to (108.center);
		\draw [style={red_edge}, in=-45, out=90, looseness=0.75] (101.center) to (108.center);
		\draw [style={red_edge}, in=-150, out=90, looseness=0.75] (103.center) to (108.center);
		\draw [style=blue edge, in=-15, out=90, looseness=0.75] (99.center) to (108.center);
		\draw [style=blue edge, in=-165, out=90] (105.center) to (108.center);
		\draw [style={green_edge}] (14.center) to (109);
		\draw [style={red_edge}, bend left] (57.center) to (82.center);
		\draw [style={red_edge}, bend left=15] (59.center) to (82.center);
		\draw [style={red_edge}, bend left] (82.center) to (62.center);
		\draw [style=blue edge, bend left=15] (82.center) to (61.center);
		\draw [style=blue edge, bend left] (58.center) to (82.center);
		\draw [style={green_edge}] (110) to (60.center);
		\draw [style=blue edge, in=-75, out=90] (125) to (126.center);
		\draw [style={red_edge}, in=-45, out=90, looseness=0.75] (119.center) to (126.center);
		\draw [style={red_edge}, in=-90, out=90] (121.center) to (126.center);
		\draw [style=blue edge, in=-120, out=90, looseness=0.75] (123.center) to (126.center);
		\draw [style={red_edge}, in=90, out=-165] (126.center) to (0.center);
		\draw [style={thick_purple}, in=-90, out=90, looseness=0.75] (108.center) to (84.center);
		\draw [style={thick_purple}] (84.center) to (83.center);
		\draw [style={thick_purple}, in=90, out=-105, looseness=0.75] (84.center) to (126.center);
		\draw [style={thick_purple}, in=-90, out=120, looseness=0.75] (83.center) to (95.center);
		\draw [style=blue edge] (92.center) to (96);
		\draw [style=blue edge, in=-75, out=150, looseness=0.75] (95.center) to (98);
		\draw [style={green_edge}] (93.center) to (96);
		\draw [style={green_edge}] (96) to (91.center);
		\draw [style=blue edge] (107) to (102.center);
		\draw [style={green_edge}] (100.center) to (107);
		\draw [style={green_edge}] (107) to (104.center);
		\draw [style=blue edge] (125) to (120.center);
		\draw [style={green_edge}, in=-30, out=90] (118.center) to (125);
		\draw [style={green_edge}] (125) to (122.center);
	\end{pgfonlayer}
\end{tikzpicture}
\]
        This expresses the thirty nonzero terms in the $ststs$-block as a difference of two thick Elias--Williamson diagrams.

        From here we could attempt to simplify the other blocks of $C$ analogously. This is certainly possible, but the computation becomes incredibly tedious. 
        However, a useful observation can save us a lot of effort.
        First, recall that $\varphi$ is rotation anti-invariant, i.e., $\rho (\varphi) = -\varphi$.
        We can see a shadow of this fact for the summand $F$. Indeed, from visual inspection we see that $\rho (G) = H$. 
        We do not currently see what the rotation of $K$ is in the double-leaf expansion. Nonetheless, we could write
        \[\varphi = 2 (\rho(K) - K) + 2 (G-H) + \textnormal{remainder}.\]
        To find what terms correspond to $\rho (K)$, we can use Algorithm~\ref{algo:coeff_algo_interpolate} on $\rho (K)$.
        This is where we really luck out, $2\rho (K)$ is exactly the sum of the remaining 79 nonzero terms. In other words, we get an equality,
        \[\varphi = 2 (\rho(K) - K) + 2 (G-H).\]
        Diagrammatically, we have that
        \begin{align*}
                \varphi &= 2 \left( 
  \tikzsetnextfilename{K_thick_rot}
  \tikzstyle{every picture}=[tikzfig]
  \begin{tikzpicture}
	\begin{pgfonlayer}{nodelayer}
		\node [style=none] (0) at (-3.25, -3.25) {};
		\node [style=none] (1) at (-2.75, -3.25) {};
		\node [style=none] (2) at (-2.25, -3.25) {};
		\node [style=none] (3) at (-1.75, -3.25) {};
		\node [style=none] (4) at (-1.25, -3.25) {};
		\node [style=none] (5) at (-0.75, -3.25) {};
		\node [style=none] (6) at (-0.25, -3.25) {};
		\node [style=none] (7) at (0.25, -3.25) {};
		\node [style=none] (8) at (0.75, -3.25) {};
		\node [style=none] (9) at (1.25, -3.25) {};
		\node [style=none] (10) at (1.75, -3.25) {};
		\node [style=none] (11) at (2.25, -3.25) {};
		\node [style=none] (12) at (2.75, -3.25) {};
		\node [style=none] (13) at (3.25, -3.25) {};
		\node [style=none] (14) at (3.75, -3.25) {};
		\node [style=none] (47) at (1.75, -1.5) {};
		\node [style={thick_dot_ochre}] (50) at (3.25, -2.75) {};
		\node [style={thick_dot_ochre}] (52) at (0.25, -2.75) {};
		\node [style=none] (53) at (-2, -1.5) {};
		\node [style={EW_t}] (54) at (-2.5, -2) {};
		\node [style={EW_t}] (55) at (-2.75, -2.75) {};
		\node [style=none] (56) at (-3.25, -1.5) {};
		\node [style=none] (57) at (3.25, 3.25) {};
		\node [style=none] (58) at (2.75, 3.25) {};
		\node [style=none] (59) at (2.25, 3.25) {};
		\node [style=none] (60) at (1.75, 3.25) {};
		\node [style=none] (61) at (1.25, 3.25) {};
		\node [style=none] (62) at (0.75, 3.25) {};
		\node [style=none] (63) at (0.25, 3.25) {};
		\node [style=none] (64) at (-0.25, 3.25) {};
		\node [style=none] (65) at (-0.75, 3.25) {};
		\node [style=none] (66) at (-1.25, 3.25) {};
		\node [style=none] (67) at (-1.75, 3.25) {};
		\node [style=none] (68) at (-2.25, 3.25) {};
		\node [style=none] (69) at (-2.75, 3.25) {};
		\node [style=none] (70) at (-3.25, 3.25) {};
		\node [style=none] (71) at (-3.75, 3.25) {};
		\node [style={thick_dot_ochre}] (78) at (-1.75, 2.5) {};
		\node [style=none] (79) at (-1.75, 1.25) {};
		\node [style={thick_dot_ochre}] (81) at (1.25, 2.5) {};
		\node [style=none] (82) at (1.75, 1.25) {};
		\node [style=none] (83) at (0, 0.25) {};
		\node [style=none] (84) at (0, -0.5) {};
	\end{pgfonlayer}
	\begin{pgfonlayer}{edgelayer}
		\draw [style=blue edge] (7.center) to (52);
		\draw [style=blue edge, in=-30, out=105, looseness=0.75] (50) to (47.center);
		\draw [style=blue edge] (47.center) to (10.center);
		\draw [style={green_edge}] (8.center) to (52);
		\draw [style={green_edge}] (52) to (6.center);
		\draw [style={red_edge}, in=-45, out=90] (11.center) to (47.center);
		\draw [style={red_edge}, in=90, out=-135] (47.center) to (9.center);
		\draw [style={red_edge}, in=-45, out=90] (5.center) to (53.center);
		\draw [style={red_edge}] (2.center) to (53.center);
		\draw [style=blue edge, in=-75, out=90] (4.center) to (53.center);
		\draw [style=blue edge] (55) to (1.center);
		\draw [style=blue edge] (54) to (53.center);
		\draw [style={red_edge}, in=90, out=-165] (53.center) to (0.center);
		\draw [style={green_edge}, in=-90, out=90] (3.center) to (56.center);
		\draw [style=blue edge] (78) to (79.center);
		\draw [style={red_edge}, in=45, out=-90, looseness=0.75] (66.center) to (79.center);
		\draw [style={red_edge}, in=150, out=-90, looseness=0.75] (68.center) to (79.center);
		\draw [style=blue edge, in=15, out=-90, looseness=0.75] (64.center) to (79.center);
		\draw [style=blue edge, in=165, out=-90] (70.center) to (79.center);
		\draw [style={red_edge}, in=15, out=-90] (57.center) to (82.center);
		\draw [style={red_edge}, in=60, out=-90, looseness=0.75] (59.center) to (82.center);
		\draw [style=blue edge, in=45, out=-90] (58.center) to (82.center);
		\draw [style=blue edge, bend left=15] (82.center) to (81);
		\draw [style={red_edge}, in=165, out=-90, looseness=1.25] (62.center) to (82.center);
		\draw [style={thick_purple}, in=45, out=-90, looseness=0.75] (82.center) to (83.center);
		\draw [style={thick_purple}, in=-90, out=120, looseness=0.75] (83.center) to (79.center);
		\draw [style={thick_purple}] (84.center) to (83.center);
		\draw [style={green_edge}, in=90, out=-90] (71.center) to (56.center);
		\draw [style=blue edge] (13.center) to (50);
		\draw [style=blue edge, in=60, out=-150, looseness=0.75] (47.center) to (52);
		\draw [style={green_edge}] (14.center) to (50);
		\draw [style={green_edge}] (50) to (12.center);
		\draw [style={thick_purple}, in=-90, out=90, looseness=0.75] (47.center) to (84.center);
		\draw [style={thick_purple}] (84.center) to (83.center);
		\draw [style={thick_purple}, in=90, out=-105, looseness=0.75] (84.center) to (53.center);
		\draw [style=blue edge] (78) to (67.center);
		\draw [style={green_edge}] (65.center) to (78);
		\draw [style={green_edge}] (78) to (69.center);
		\draw [style={green_edge}] (81) to (63.center);
		\draw [style={green_edge}] (81) to (60.center);
		\draw [style=blue edge] (81) to (61.center);
	\end{pgfonlayer}
\end{tikzpicture}
 - 
  \tikzsetnextfilename{K_thick}
  \tikzstyle{every picture}=[tikzfig]
  \begin{tikzpicture}
	\begin{pgfonlayer}{nodelayer}
		\node [style=none] (0) at (3.25, 3.25) {};
		\node [style=none] (1) at (2.75, 3.25) {};
		\node [style=none] (2) at (2.25, 3.25) {};
		\node [style=none] (3) at (1.75, 3.25) {};
		\node [style=none] (4) at (1.25, 3.25) {};
		\node [style=none] (5) at (0.75, 3.25) {};
		\node [style=none] (6) at (0.25, 3.25) {};
		\node [style=none] (7) at (-0.25, 3.25) {};
		\node [style=none] (8) at (-0.75, 3.25) {};
		\node [style=none] (9) at (-1.25, 3.25) {};
		\node [style=none] (10) at (-1.75, 3.25) {};
		\node [style=none] (11) at (-2.25, 3.25) {};
		\node [style=none] (12) at (-2.75, 3.25) {};
		\node [style=none] (13) at (-3.25, 3.25) {};
		\node [style=none] (14) at (-3.75, 3.25) {};
		\node [style=none] (47) at (-1.75, 1.5) {};
		\node [style={thick_dot_ochre}] (50) at (-3.25, 2.75) {};
		\node [style={thick_dot_ochre}] (52) at (-0.25, 2.75) {};
		\node [style=none] (53) at (2, 1.5) {};
		\node [style={EW_t}] (54) at (2.5, 2) {};
		\node [style={EW_t}] (55) at (2.75, 2.75) {};
		\node [style=none] (56) at (3.25, 1.5) {};
		\node [style=none] (57) at (-3.25, -3.25) {};
		\node [style=none] (58) at (-2.75, -3.25) {};
		\node [style=none] (59) at (-2.25, -3.25) {};
		\node [style=none] (60) at (-1.75, -3.25) {};
		\node [style=none] (61) at (-1.25, -3.25) {};
		\node [style=none] (62) at (-0.75, -3.25) {};
		\node [style=none] (63) at (-0.25, -3.25) {};
		\node [style=none] (64) at (0.25, -3.25) {};
		\node [style=none] (65) at (0.75, -3.25) {};
		\node [style=none] (66) at (1.25, -3.25) {};
		\node [style=none] (67) at (1.75, -3.25) {};
		\node [style=none] (68) at (2.25, -3.25) {};
		\node [style=none] (69) at (2.75, -3.25) {};
		\node [style=none] (70) at (3.25, -3.25) {};
		\node [style=none] (71) at (3.75, -3.25) {};
		\node [style={thick_dot_ochre}] (78) at (1.75, -2.5) {};
		\node [style=none] (79) at (1.75, -1.25) {};
		\node [style={thick_dot_ochre}] (81) at (-1.25, -2.5) {};
		\node [style=none] (82) at (-1.75, -1.25) {};
		\node [style=none] (83) at (0, -0.25) {};
		\node [style=none] (84) at (0, 0.5) {};
	\end{pgfonlayer}
	\begin{pgfonlayer}{edgelayer}
		\draw [style=blue edge] (7.center) to (52);
		\draw [style=blue edge, in=150, out=-75, looseness=0.75] (50) to (47.center);
		\draw [style=blue edge] (47.center) to (10.center);
		\draw [style={green_edge}] (8.center) to (52);
		\draw [style={green_edge}] (52) to (6.center);
		\draw [style={red_edge}, in=135, out=-90] (11.center) to (47.center);
		\draw [style={red_edge}, in=-90, out=45] (47.center) to (9.center);
		\draw [style={red_edge}, in=135, out=-90] (5.center) to (53.center);
		\draw [style={red_edge}] (2.center) to (53.center);
		\draw [style=blue edge, in=105, out=-90] (4.center) to (53.center);
		\draw [style=blue edge] (55) to (1.center);
		\draw [style=blue edge] (54) to (53.center);
		\draw [style={red_edge}, in=-90, out=15] (53.center) to (0.center);
		\draw [style={green_edge}, in=90, out=-90] (3.center) to (56.center);
		\draw [style=blue edge] (78) to (79.center);
		\draw [style={red_edge}, in=-135, out=90, looseness=0.75] (66.center) to (79.center);
		\draw [style={red_edge}, in=-30, out=90, looseness=0.75] (68.center) to (79.center);
		\draw [style=blue edge, in=-165, out=90, looseness=0.75] (64.center) to (79.center);
		\draw [style=blue edge, in=-15, out=90] (70.center) to (79.center);
		\draw [style={red_edge}, in=-165, out=90] (57.center) to (82.center);
		\draw [style={red_edge}, in=-120, out=90, looseness=0.75] (59.center) to (82.center);
		\draw [style=blue edge, in=-135, out=90] (58.center) to (82.center);
		\draw [style=blue edge, bend left=15] (82.center) to (81);
		\draw [style={red_edge}, in=-15, out=90, looseness=1.25] (62.center) to (82.center);
		\draw [style={thick_purple}, in=-135, out=90, looseness=0.75] (82.center) to (83.center);
		\draw [style={thick_purple}, in=90, out=-60, looseness=0.75] (83.center) to (79.center);
		\draw [style={thick_purple}] (84.center) to (83.center);
		\draw [style={green_edge}, in=270, out=90] (71.center) to (56.center);
		\draw [style=blue edge] (13.center) to (50);
		\draw [style=blue edge, in=-120, out=30, looseness=0.75] (47.center) to (52);
		\draw [style={green_edge}] (14.center) to (50);
		\draw [style={green_edge}] (50) to (12.center);
		\draw [style={thick_purple}, in=90, out=-90, looseness=0.75] (47.center) to (84.center);
		\draw [style={thick_purple}] (84.center) to (83.center);
		\draw [style={thick_purple}, in=-90, out=75, looseness=0.75] (84.center) to (53.center);
		\draw [style=blue edge] (78) to (67.center);
		\draw [style={green_edge}] (65.center) to (78);
		\draw [style={green_edge}] (78) to (69.center);
		\draw [style={green_edge}] (81) to (63.center);
		\draw [style={green_edge}] (81) to (60.center);
		\draw [style=blue edge] (81) to (61.center);
	\end{pgfonlayer}
\end{tikzpicture}
 \right) \\
                &+ 2 \left( 
  \tikzsetnextfilename{P_thick}
  \tikzstyle{every picture}=[tikzfig]
  \begin{tikzpicture}
	\begin{pgfonlayer}{nodelayer}
		\node [style=none] (0) at (3.25, 3.25) {};
		\node [style=none] (14) at (-3.75, 3.25) {};
		\node [style=none] (57) at (-3.25, -3.25) {};
		\node [style=none] (58) at (-2.75, -3.25) {};
		\node [style=none] (59) at (-2.25, -3.25) {};
		\node [style=none] (60) at (-1.75, -3.25) {};
		\node [style=none] (61) at (-1.25, -3.25) {};
		\node [style=none] (62) at (-0.75, -3.25) {};
		\node [style=none] (82) at (-1.75, -1.5) {};
		\node [style=none] (83) at (0, -0.5) {};
		\node [style=none] (84) at (0, 0.25) {};
		\node [style=none] (85) at (3.75, -3.25) {};
		\node [style=none] (86) at (3.25, -3.25) {};
		\node [style=none] (87) at (2.75, -3.25) {};
		\node [style=none] (88) at (2.25, -3.25) {};
		\node [style=none] (89) at (1.75, -3.25) {};
		\node [style=none] (90) at (1.25, -3.25) {};
		\node [style=none] (91) at (0.75, -3.25) {};
		\node [style=none] (92) at (0.25, -3.25) {};
		\node [style=none] (93) at (-0.25, -3.25) {};
		\node [style=none] (95) at (1.75, -1.5) {};
		\node [style={thick_dot_ochre}] (96) at (0.25, -2.75) {};
		\node [style={thick_dot_ochre}] (98) at (3.25, -2.75) {};
		\node [style=none] (99) at (-3.25, 3.25) {};
		\node [style=none] (100) at (-2.75, 3.25) {};
		\node [style=none] (101) at (-2.25, 3.25) {};
		\node [style=none] (102) at (-1.75, 3.25) {};
		\node [style=none] (103) at (-1.25, 3.25) {};
		\node [style=none] (104) at (-0.75, 3.25) {};
		\node [style=none] (105) at (-0.25, 3.25) {};
		\node [style={thick_dot_ochre}] (107) at (-1.75, 2.5) {};
		\node [style=none] (108) at (-1.75, 1.25) {};
		\node [style={dot_green}] (109) at (-3.75, 2.25) {};
		\node [style={dot_green}] (110) at (-1.75, -2.5) {};
		\node [style=none] (118) at (0.25, 3.25) {};
		\node [style=none] (119) at (0.75, 3.25) {};
		\node [style=none] (120) at (1.25, 3.25) {};
		\node [style=none] (121) at (2.25, 3.25) {};
		\node [style=none] (122) at (1.75, 3.25) {};
		\node [style=none] (123) at (2.75, 3.25) {};
		\node [style={thick_dot_ochre}] (125) at (1.5, 2.5) {};
		\node [style=none] (126) at (1.75, 1.25) {};
	\end{pgfonlayer}
	\begin{pgfonlayer}{edgelayer}
		\draw [style={thick_purple}, in=-135, out=90, looseness=0.75] (82.center) to (83.center);
		\draw [style={thick_purple}] (84.center) to (83.center);
		\draw [style=blue edge] (86.center) to (98);
		\draw [style=blue edge, in=-150, out=60, looseness=0.75] (96) to (95.center);
		\draw [style=blue edge] (95.center) to (89.center);
		\draw [style={green_edge}] (87.center) to (98);
		\draw [style={green_edge}] (98) to (85.center);
		\draw [style={red_edge}, in=-135, out=90] (90.center) to (95.center);
		\draw [style={red_edge}, in=90, out=-45] (95.center) to (88.center);
		\draw [style=blue edge] (107) to (108.center);
		\draw [style={red_edge}, in=135, out=-90, looseness=0.75] (101.center) to (108.center);
		\draw [style={red_edge}, in=30, out=-90, looseness=0.75] (103.center) to (108.center);
		\draw [style=blue edge, in=165, out=-90, looseness=0.75] (99.center) to (108.center);
		\draw [style=blue edge, in=15, out=-90] (105.center) to (108.center);
		\draw [style={green_edge}] (14.center) to (109);
		\draw [style={red_edge}, bend left] (57.center) to (82.center);
		\draw [style={red_edge}, bend left=15] (59.center) to (82.center);
		\draw [style={red_edge}, bend left] (82.center) to (62.center);
		\draw [style=blue edge, bend left=15] (82.center) to (61.center);
		\draw [style=blue edge, bend left] (58.center) to (82.center);
		\draw [style={green_edge}] (110) to (60.center);
		\draw [style=blue edge, in=105, out=-90, looseness=0.75] (125) to (126.center);
		\draw [style={red_edge}, in=135, out=-90, looseness=0.75] (119.center) to (126.center);
		\draw [style={red_edge}, in=90, out=-90] (121.center) to (126.center);
		\draw [style=blue edge, in=60, out=-90, looseness=0.75] (123.center) to (126.center);
		\draw [style={red_edge}, in=-90, out=15] (126.center) to (0.center);
		\draw [style={thick_purple}, in=90, out=-90, looseness=0.75] (108.center) to (84.center);
		\draw [style={thick_purple}] (84.center) to (83.center);
		\draw [style={thick_purple}, in=-90, out=75, looseness=0.75] (84.center) to (126.center);
		\draw [style={thick_purple}, in=90, out=-60, looseness=0.75] (83.center) to (95.center);
		\draw [style=blue edge] (92.center) to (96);
		\draw [style=blue edge, in=120, out=-30, looseness=0.75] (95.center) to (98);
		\draw [style={green_edge}] (93.center) to (96);
		\draw [style={green_edge}] (96) to (91.center);
		\draw [style=blue edge] (107) to (102.center);
		\draw [style={green_edge}] (100.center) to (107);
		\draw [style={green_edge}] (107) to (104.center);
		\draw [style=blue edge] (125) to (120.center);
		\draw [style={green_edge}, in=150, out=-90] (118.center) to (125);
		\draw [style={green_edge}] (125) to (122.center);
	\end{pgfonlayer}
\end{tikzpicture}
 - 
  \tikzsetnextfilename{Q_thick}
  \tikzstyle{every picture}=[tikzfig]
  \begin{tikzpicture}
	\begin{pgfonlayer}{nodelayer}
		\node [style=none] (0) at (-3.25, -3.25) {};
		\node [style=none] (14) at (3.75, -3.25) {};
		\node [style=none] (57) at (3.25, 3.25) {};
		\node [style=none] (58) at (2.75, 3.25) {};
		\node [style=none] (59) at (2.25, 3.25) {};
		\node [style=none] (60) at (1.75, 3.25) {};
		\node [style=none] (61) at (1.25, 3.25) {};
		\node [style=none] (62) at (0.75, 3.25) {};
		\node [style=none] (82) at (1.75, 1.5) {};
		\node [style=none] (83) at (0, 0.5) {};
		\node [style=none] (84) at (0, -0.25) {};
		\node [style=none] (85) at (-3.75, 3.25) {};
		\node [style=none] (86) at (-3.25, 3.25) {};
		\node [style=none] (87) at (-2.75, 3.25) {};
		\node [style=none] (88) at (-2.25, 3.25) {};
		\node [style=none] (89) at (-1.75, 3.25) {};
		\node [style=none] (90) at (-1.25, 3.25) {};
		\node [style=none] (91) at (-0.75, 3.25) {};
		\node [style=none] (92) at (-0.25, 3.25) {};
		\node [style=none] (93) at (0.25, 3.25) {};
		\node [style=none] (95) at (-1.75, 1.5) {};
		\node [style={thick_dot_ochre}] (96) at (-0.25, 2.75) {};
		\node [style={thick_dot_ochre}] (98) at (-3.25, 2.75) {};
		\node [style=none] (99) at (3.25, -3.25) {};
		\node [style=none] (100) at (2.75, -3.25) {};
		\node [style=none] (101) at (2.25, -3.25) {};
		\node [style=none] (102) at (1.75, -3.25) {};
		\node [style=none] (103) at (1.25, -3.25) {};
		\node [style=none] (104) at (0.75, -3.25) {};
		\node [style=none] (105) at (0.25, -3.25) {};
		\node [style={thick_dot_ochre}] (107) at (1.75, -2.5) {};
		\node [style=none] (108) at (1.75, -1.25) {};
		\node [style={dot_green}] (109) at (3.75, -2.25) {};
		\node [style={dot_green}] (110) at (1.75, 2.5) {};
		\node [style=none] (118) at (-0.25, -3.25) {};
		\node [style=none] (119) at (-0.75, -3.25) {};
		\node [style=none] (120) at (-1.25, -3.25) {};
		\node [style=none] (121) at (-2.25, -3.25) {};
		\node [style=none] (122) at (-1.75, -3.25) {};
		\node [style=none] (123) at (-2.75, -3.25) {};
		\node [style={thick_dot_ochre}] (125) at (-1.5, -2.5) {};
		\node [style=none] (126) at (-1.75, -1.25) {};
	\end{pgfonlayer}
	\begin{pgfonlayer}{edgelayer}
		\draw [style={thick_purple}, in=45, out=-90, looseness=0.75] (82.center) to (83.center);
		\draw [style={thick_purple}] (84.center) to (83.center);
		\draw [style=blue edge] (86.center) to (98);
		\draw [style=blue edge, in=30, out=-90, looseness=0.75] (96) to (95.center);
		\draw [style=blue edge] (95.center) to (89.center);
		\draw [style={green_edge}] (87.center) to (98);
		\draw [style={green_edge}] (98) to (85.center);
		\draw [style={red_edge}, in=45, out=-90] (90.center) to (95.center);
		\draw [style={red_edge}, in=-90, out=135] (95.center) to (88.center);
		\draw [style=blue edge] (107) to (108.center);
		\draw [style={red_edge}, in=-45, out=90, looseness=0.75] (101.center) to (108.center);
		\draw [style={red_edge}, in=-150, out=90, looseness=0.75] (103.center) to (108.center);
		\draw [style=blue edge, in=-15, out=90, looseness=0.75] (99.center) to (108.center);
		\draw [style=blue edge, in=-165, out=90] (105.center) to (108.center);
		\draw [style={green_edge}] (14.center) to (109);
		\draw [style={red_edge}, bend left] (57.center) to (82.center);
		\draw [style={red_edge}, bend left=15] (59.center) to (82.center);
		\draw [style={red_edge}, bend left] (82.center) to (62.center);
		\draw [style=blue edge, bend left=15] (82.center) to (61.center);
		\draw [style=blue edge, bend left] (58.center) to (82.center);
		\draw [style={green_edge}] (110) to (60.center);
		\draw [style=blue edge, in=-75, out=90] (125) to (126.center);
		\draw [style={red_edge}, in=-45, out=90, looseness=0.75] (119.center) to (126.center);
		\draw [style={red_edge}, in=-90, out=90] (121.center) to (126.center);
		\draw [style=blue edge, in=-120, out=90, looseness=0.75] (123.center) to (126.center);
		\draw [style={red_edge}, in=90, out=-165] (126.center) to (0.center);
		\draw [style={thick_purple}, in=-90, out=90, looseness=0.75] (108.center) to (84.center);
		\draw [style={thick_purple}] (84.center) to (83.center);
		\draw [style={thick_purple}, in=90, out=-105, looseness=0.75] (84.center) to (126.center);
		\draw [style={thick_purple}, in=-90, out=120, looseness=0.75] (83.center) to (95.center);
		\draw [style=blue edge] (92.center) to (96);
		\draw [style=blue edge, in=-75, out=150, looseness=0.75] (95.center) to (98);
		\draw [style={green_edge}] (93.center) to (96);
		\draw [style={green_edge}] (96) to (91.center);
		\draw [style=blue edge] (107) to (102.center);
		\draw [style={green_edge}] (100.center) to (107);
		\draw [style={green_edge}] (107) to (104.center);
		\draw [style=blue edge] (125) to (120.center);
		\draw [style={green_edge}, in=-30, out=90] (118.center) to (125);
		\draw [style={green_edge}] (125) to (122.center);
	\end{pgfonlayer}
\end{tikzpicture}
\right).
        \end{align*}

        There is one last simplification we can perform.
        We can use \eqref{eq:half_JW_reln} to combine $\rho (K)$ and $G$ into a single thick diagram and likewise $K$ and $H$.
        This gives us our final diagrammatic description of the $3$-color relation for type $H_3$.

        \begin{theorem}
        For any balanced realization $\fr{h}$ over $H_3$, we have the following equality in $\gBS (\fr{h}, H_3)$,
        \begin{equation*}
                \resizebox{0.95\hsize}{!}{$

  \tikzsetnextfilename{h3_left}
  \tikzstyle{every picture}=[tikzfig]
  \begin{tikzpicture}
	\begin{pgfonlayer}{nodelayer}
		\node [style=none] (0) at (2.25, -3.25) {};
		\node [style=none] (1) at (1.75, -3.25) {};
		\node [style=none] (2) at (1.25, -3.25) {};
		\node [style=none] (3) at (0.75, -3.25) {};
		\node [style=none] (4) at (0.25, -3.25) {};
		\node [style=none] (5) at (-0.25, -3.25) {};
		\node [style=none] (6) at (-0.75, -3.25) {};
		\node [style=none] (7) at (-1.25, -3.25) {};
		\node [style=none] (8) at (-1.75, -3.25) {};
		\node [style=none] (9) at (2.75, -3.25) {};
		\node [style=none] (10) at (3.25, -3.25) {};
		\node [style=none] (11) at (-2.25, -3.25) {};
		\node [style=none] (12) at (-2.75, -3.25) {};
		\node [style=none] (13) at (-3.25, -3.25) {};
		\node [style=none] (14) at (-3.75, -3.25) {};
		\node [style=none] (15) at (2.25, 3.25) {};
		\node [style=none] (16) at (1.75, 3.25) {};
		\node [style=none] (17) at (1.25, 3.25) {};
		\node [style=none] (18) at (0.75, 3.25) {};
		\node [style=none] (19) at (0.25, 3.25) {};
		\node [style=none] (20) at (-0.25, 3.25) {};
		\node [style=none] (21) at (-0.75, 3.25) {};
		\node [style=none] (22) at (-1.25, 3.25) {};
		\node [style=none] (23) at (-1.75, 3.25) {};
		\node [style=none] (24) at (2.75, 3.25) {};
		\node [style=none] (25) at (3.25, 3.25) {};
		\node [style=none] (26) at (-2.25, 3.25) {};
		\node [style=none] (27) at (-2.75, 3.25) {};
		\node [style=none] (28) at (-3.25, 3.25) {};
		\node [style=none] (29) at (-3.75, 3.25) {};
		\node [style=none] (30) at (2.75, -2.75) {};
		\node [style=none] (31) at (1.25, -2.25) {};
		\node [style=none] (32) at (-0.25, -2.75) {};
		\node [style=none] (33) at (-1.25, -1.75) {};
		\node [style=none] (34) at (0.25, -1.25) {};
		\node [style=none] (35) at (-2.25, -1.25) {};
		\node [style=none] (36) at (1.25, -0.75) {};
		\node [style=none] (37) at (-1.25, -0.75) {};
		\node [style=none] (38) at (0.25, -0.25) {};
		\node [style=none] (39) at (-2.25, 0.25) {};
		\node [style=none] (40) at (-0.75, 0.75) {};
		\node [style=none] (41) at (1, 1) {};
		\node [style=none] (42) at (2.75, 0.5) {};
		\node [style=none] (43) at (-0.25, 2.25) {};
		\node [style=none] (44) at (1.25, 2.75) {};
		\node [style=none] (45) at (-1.5, 2.5) {};
	\end{pgfonlayer}
	\begin{pgfonlayer}{edgelayer}
		\draw [style={red_edge}] (31.center) to (3.center);
		\draw [style={red_edge}] (31.center) to (1.center);
		\draw [style=blue edge] (31.center) to (2.center);
		\draw [style=blue edge] (31.center) to (32.center);
		\draw [style=blue edge] (32.center) to (5.center);
		\draw [style=blue edge] (31.center) to (30.center);
		\draw [style=blue edge] (30.center) to (9.center);
		\draw [style={green_edge}] (32.center) to (4.center);
		\draw [style={green_edge}] (32.center) to (6.center);
		\draw [style={green_edge}] (30.center) to (0.center);
		\draw [style={green_edge}] (30.center) to (10.center);
		\draw [style={red_edge}] (33.center) to (31.center);
		\draw [style={red_edge}] (33.center) to (7.center);
		\draw [style=blue edge] (33.center) to (32.center);
		\draw [style=blue edge] (33.center) to (8.center);
		\draw [style=blue edge] (33.center) to (34.center);
		\draw [style=blue edge] (34.center) to (31.center);
		\draw [style={green_edge}] (34.center) to (32.center);
		\draw [style={green_edge}] (35.center) to (11.center);
		\draw [style={red_edge}] (33.center) to (12.center);
		\draw [style=blue edge] (35.center) to (33.center);
		\draw [style=blue edge] (13.center) to (35.center);
		\draw [style={green_edge}] (34.center) to (36.center);
		\draw [style=blue edge] (31.center) to (36.center);
		\draw [style={green_edge}] (36.center) to (30.center);
		\draw [style={green_edge}] (37.center) to (35.center);
		\draw [style={green_edge}] (37.center) to (34.center);
		\draw [style=blue edge] (37.center) to (33.center);
		\draw [style=blue edge] (37.center) to (38.center);
		\draw [style=blue edge] (38.center) to (34.center);
		\draw [style=blue edge] (38.center) to (36.center);
		\draw [style={red_edge}] (38.center) to (31.center);
		\draw [style={red_edge}] (38.center) to (33.center);
		\draw [style={red_edge}] (39.center) to (38.center);
		\draw [style={red_edge}] (39.center) to (14.center);
		\draw [style={green_edge}] (35.center) to (29.center);
		\draw [style=blue edge] (39.center) to (28.center);
		\draw [style=blue edge] (39.center) to (35.center);
		\draw [style=blue edge] (39.center) to (37.center);
		\draw [style={red_edge}] (39.center) to (33.center);
		\draw [style=blue edge] (39.center) to (40.center);
		\draw [style=blue edge] (41.center) to (42.center);
		\draw [style={green_edge}] (40.center) to (41.center);
		\draw [style={red_edge}] (43.center) to (42.center);
		\draw [style=blue edge] (43.center) to (44.center);
		\draw [style=blue edge] (43.center) to (45.center);
		\draw [style={green_edge}] (40.center) to (37.center);
		\draw [style={green_edge}] (45.center) to (40.center);
		\draw [style={green_edge}] (45.center) to (27.center);
		\draw [style={green_edge}] (45.center) to (21.center);
		\draw [style=blue edge] (23.center) to (45.center);
		\draw [style=blue edge] (45.center) to (39.center);
		\draw [style={red_edge}] (26.center) to (39.center);
		\draw [style={red_edge}] (43.center) to (22.center);
		\draw [style=blue edge] (44.center) to (42.center);
		\draw [style=blue edge] (43.center) to (20.center);
		\draw [style={green_edge}] (44.center) to (16.center);
		\draw [style={green_edge}] (44.center) to (19.center);
		\draw [style=blue edge] (17.center) to (44.center);
		\draw [style={red_edge}] (43.center) to (18.center);
		\draw [style=blue edge] (43.center) to (41.center);
		\draw [style=blue edge] (41.center) to (38.center);
		\draw [style={green_edge}] (41.center) to (44.center);
		\draw [style={green_edge}] (41.center) to (36.center);
		\draw [style={red_edge}] (39.center) to (43.center);
		\draw [style={red_edge}] (42.center) to (15.center);
		\draw [style=blue edge] (42.center) to (24.center);
		\draw [style={red_edge}] (42.center) to (25.center);
		\draw [style={red_edge}] (42.center) to (38.center);
		\draw [style=blue edge] (42.center) to (36.center);
		\draw [style=blue edge] (42.center) to (30.center);
		\draw [style={red_edge}] (42.center) to (31.center);
		\draw [style=blue edge] (40.center) to (38.center);
		\draw [style=blue edge] (43.center) to (40.center);
		\draw [style={red_edge}] (43.center) to (38.center);
	\end{pgfonlayer}
\end{tikzpicture}
 - 
  \tikzsetnextfilename{h3_right}
  \tikzstyle{every picture}=[tikzfig]
  \begin{tikzpicture}
	\begin{pgfonlayer}{nodelayer}
		\node [style=none] (0) at (-2, 3.25) {};
		\node [style=none] (1) at (-1.5, 3.25) {};
		\node [style=none] (2) at (-1, 3.25) {};
		\node [style=none] (3) at (-0.5, 3.25) {};
		\node [style=none] (4) at (0, 3.25) {};
		\node [style=none] (5) at (0.5, 3.25) {};
		\node [style=none] (6) at (1, 3.25) {};
		\node [style=none] (7) at (1.5, 3.25) {};
		\node [style=none] (8) at (2, 3.25) {};
		\node [style=none] (9) at (-2.5, 3.25) {};
		\node [style=none] (10) at (-3, 3.25) {};
		\node [style=none] (11) at (2.5, 3.25) {};
		\node [style=none] (12) at (3, 3.25) {};
		\node [style=none] (13) at (3.5, 3.25) {};
		\node [style=none] (14) at (4, 3.25) {};
		\node [style=none] (15) at (-2, -3.25) {};
		\node [style=none] (16) at (-1.5, -3.25) {};
		\node [style=none] (17) at (-1, -3.25) {};
		\node [style=none] (18) at (-0.5, -3.25) {};
		\node [style=none] (19) at (0, -3.25) {};
		\node [style=none] (20) at (0.5, -3.25) {};
		\node [style=none] (21) at (1, -3.25) {};
		\node [style=none] (22) at (1.5, -3.25) {};
		\node [style=none] (23) at (2, -3.25) {};
		\node [style=none] (24) at (-2.5, -3.25) {};
		\node [style=none] (25) at (-3, -3.25) {};
		\node [style=none] (26) at (2.5, -3.25) {};
		\node [style=none] (27) at (3, -3.25) {};
		\node [style=none] (28) at (3.5, -3.25) {};
		\node [style=none] (29) at (4, -3.25) {};
		\node [style=none] (30) at (-2.5, 2.75) {};
		\node [style=none] (31) at (-1, 2.25) {};
		\node [style=none] (32) at (0.5, 2.75) {};
		\node [style=none] (33) at (1.5, 1.75) {};
		\node [style=none] (34) at (0, 1.25) {};
		\node [style=none] (35) at (2.5, 1.25) {};
		\node [style=none] (36) at (-1, 0.75) {};
		\node [style=none] (37) at (1.5, 0.75) {};
		\node [style=none] (38) at (0, 0.25) {};
		\node [style=none] (39) at (2.5, -0.25) {};
		\node [style=none] (40) at (1, -0.75) {};
		\node [style=none] (41) at (-0.75, -1) {};
		\node [style=none] (42) at (-2.5, -0.5) {};
		\node [style=none] (43) at (0.5, -2.25) {};
		\node [style=none] (44) at (-1, -2.75) {};
		\node [style=none] (45) at (1.75, -2.5) {};
	\end{pgfonlayer}
	\begin{pgfonlayer}{edgelayer}
		\draw [style={red_edge}] (31.center) to (3.center);
		\draw [style={red_edge}] (31.center) to (1.center);
		\draw [style=blue edge] (31.center) to (2.center);
		\draw [style=blue edge] (31.center) to (32.center);
		\draw [style=blue edge] (32.center) to (5.center);
		\draw [style=blue edge] (31.center) to (30.center);
		\draw [style=blue edge] (30.center) to (9.center);
		\draw [style={green_edge}] (32.center) to (4.center);
		\draw [style={green_edge}] (32.center) to (6.center);
		\draw [style={green_edge}] (30.center) to (0.center);
		\draw [style={green_edge}] (30.center) to (10.center);
		\draw [style={red_edge}] (33.center) to (31.center);
		\draw [style={red_edge}] (33.center) to (7.center);
		\draw [style=blue edge] (33.center) to (32.center);
		\draw [style=blue edge] (33.center) to (8.center);
		\draw [style=blue edge] (33.center) to (34.center);
		\draw [style=blue edge] (34.center) to (31.center);
		\draw [style={green_edge}] (34.center) to (32.center);
		\draw [style={green_edge}] (35.center) to (11.center);
		\draw [style={red_edge}] (33.center) to (12.center);
		\draw [style=blue edge] (35.center) to (33.center);
		\draw [style=blue edge] (13.center) to (35.center);
		\draw [style={green_edge}] (34.center) to (36.center);
		\draw [style=blue edge] (31.center) to (36.center);
		\draw [style={green_edge}] (36.center) to (30.center);
		\draw [style={green_edge}] (37.center) to (35.center);
		\draw [style={green_edge}] (37.center) to (34.center);
		\draw [style=blue edge] (37.center) to (33.center);
		\draw [style=blue edge] (37.center) to (38.center);
		\draw [style=blue edge] (38.center) to (34.center);
		\draw [style=blue edge] (38.center) to (36.center);
		\draw [style={red_edge}] (38.center) to (31.center);
		\draw [style={red_edge}] (38.center) to (33.center);
		\draw [style={red_edge}] (39.center) to (38.center);
		\draw [style={red_edge}] (39.center) to (14.center);
		\draw [style={green_edge}] (35.center) to (29.center);
		\draw [style=blue edge] (39.center) to (28.center);
		\draw [style=blue edge] (39.center) to (35.center);
		\draw [style=blue edge] (39.center) to (37.center);
		\draw [style={red_edge}] (39.center) to (33.center);
		\draw [style=blue edge] (39.center) to (40.center);
		\draw [style=blue edge] (41.center) to (42.center);
		\draw [style={green_edge}] (40.center) to (41.center);
		\draw [style={red_edge}] (43.center) to (42.center);
		\draw [style=blue edge] (43.center) to (44.center);
		\draw [style=blue edge] (43.center) to (45.center);
		\draw [style={green_edge}] (40.center) to (37.center);
		\draw [style={green_edge}] (45.center) to (40.center);
		\draw [style={green_edge}] (45.center) to (27.center);
		\draw [style={green_edge}] (45.center) to (21.center);
		\draw [style=blue edge] (23.center) to (45.center);
		\draw [style=blue edge] (45.center) to (39.center);
		\draw [style={red_edge}] (26.center) to (39.center);
		\draw [style={red_edge}] (43.center) to (22.center);
		\draw [style=blue edge] (44.center) to (42.center);
		\draw [style=blue edge] (43.center) to (20.center);
		\draw [style={green_edge}] (44.center) to (16.center);
		\draw [style={green_edge}] (44.center) to (19.center);
		\draw [style=blue edge] (17.center) to (44.center);
		\draw [style={red_edge}] (43.center) to (18.center);
		\draw [style=blue edge] (43.center) to (41.center);
		\draw [style=blue edge] (41.center) to (38.center);
		\draw [style={green_edge}] (41.center) to (44.center);
		\draw [style={green_edge}] (41.center) to (36.center);
		\draw [style={red_edge}] (39.center) to (43.center);
		\draw [style={red_edge}] (42.center) to (15.center);
		\draw [style=blue edge] (42.center) to (24.center);
		\draw [style={red_edge}] (42.center) to (25.center);
		\draw [style={red_edge}] (42.center) to (38.center);
		\draw [style=blue edge] (42.center) to (36.center);
		\draw [style=blue edge] (42.center) to (30.center);
		\draw [style={red_edge}] (42.center) to (31.center);
		\draw [style=blue edge] (40.center) to (38.center);
		\draw [style=blue edge] (43.center) to (40.center);
		\draw [style={red_edge}] (43.center) to (38.center);
	\end{pgfonlayer}
\end{tikzpicture}
 = 2 \left( 
  \tikzsetnextfilename{K_rot_G}
  \tikzstyle{every picture}=[tikzfig]
  \begin{tikzpicture}
	\begin{pgfonlayer}{nodelayer}
		\node [style=none] (0) at (-3.25, -3.25) {};
		\node [style=none] (1) at (-2.75, -3.25) {};
		\node [style=none] (2) at (-2.25, -3.25) {};
		\node [style=none] (3) at (-1.75, -3.25) {};
		\node [style=none] (4) at (-1.25, -3.25) {};
		\node [style=none] (5) at (-0.75, -3.25) {};
		\node [style=none] (6) at (-0.25, -3.25) {};
		\node [style=none] (7) at (0.25, -3.25) {};
		\node [style=none] (8) at (0.75, -3.25) {};
		\node [style=none] (9) at (1.25, -3.25) {};
		\node [style=none] (10) at (1.75, -3.25) {};
		\node [style=none] (11) at (2.25, -3.25) {};
		\node [style=none] (12) at (2.75, -3.25) {};
		\node [style=none] (13) at (3.25, -3.25) {};
		\node [style=none] (14) at (3.75, -3.25) {};
		\node [style=none] (47) at (1.75, -1.5) {};
		\node [style={thick_dot_ochre}] (50) at (3.25, -2.75) {};
		\node [style={thick_dot_ochre}] (52) at (0.25, -2.75) {};
		\node [style=none] (53) at (-2, -1.5) {};
		\node [style=none] (57) at (3.25, 3.25) {};
		\node [style=none] (58) at (2.75, 3.25) {};
		\node [style=none] (59) at (2.25, 3.25) {};
		\node [style=none] (60) at (1.75, 3.25) {};
		\node [style=none] (61) at (1.25, 3.25) {};
		\node [style=none] (62) at (0.75, 3.25) {};
		\node [style=none] (63) at (0.25, 3.25) {};
		\node [style=none] (64) at (-0.25, 3.25) {};
		\node [style=none] (65) at (-0.75, 3.25) {};
		\node [style=none] (66) at (-1.25, 3.25) {};
		\node [style=none] (67) at (-1.75, 3.25) {};
		\node [style=none] (68) at (-2.25, 3.25) {};
		\node [style=none] (69) at (-2.75, 3.25) {};
		\node [style=none] (70) at (-3.25, 3.25) {};
		\node [style=none] (71) at (-3.75, 3.25) {};
		\node [style={thick_dot_ochre}] (78) at (-1.75, 2.5) {};
		\node [style=none] (79) at (-1.75, 1.25) {};
		\node [style={thick_dot_ochre}] (81) at (1.25, 2.5) {};
		\node [style=none] (82) at (1.75, 1.25) {};
		\node [style=none] (83) at (0, 0.25) {};
		\node [style=none] (84) at (0, -0.5) {};
		\node [style={thick_dot_ochre}] (85) at (-2.5, -2.25) {};
	\end{pgfonlayer}
	\begin{pgfonlayer}{edgelayer}
		\draw [style=blue edge] (7.center) to (52);
		\draw [style=blue edge, in=-30, out=105, looseness=0.75] (50) to (47.center);
		\draw [style=blue edge] (47.center) to (10.center);
		\draw [style={green_edge}] (8.center) to (52);
		\draw [style={green_edge}] (52) to (6.center);
		\draw [style={red_edge}, in=-45, out=90] (11.center) to (47.center);
		\draw [style={red_edge}, in=90, out=-135] (47.center) to (9.center);
		\draw [style={red_edge}, in=-45, out=90] (5.center) to (53.center);
		\draw [style={red_edge}] (2.center) to (53.center);
		\draw [style=blue edge, in=-75, out=90] (4.center) to (53.center);
		\draw [style={red_edge}, in=90, out=-165] (53.center) to (0.center);
		\draw [style=blue edge] (78) to (79.center);
		\draw [style={red_edge}, in=45, out=-90, looseness=0.75] (66.center) to (79.center);
		\draw [style={red_edge}, in=150, out=-90, looseness=0.75] (68.center) to (79.center);
		\draw [style=blue edge, in=15, out=-90, looseness=0.75] (64.center) to (79.center);
		\draw [style=blue edge, in=165, out=-90] (70.center) to (79.center);
		\draw [style={red_edge}, in=15, out=-90] (57.center) to (82.center);
		\draw [style={red_edge}, in=60, out=-90, looseness=0.75] (59.center) to (82.center);
		\draw [style=blue edge, in=45, out=-90] (58.center) to (82.center);
		\draw [style=blue edge, in=270, out=120] (82.center) to (81);
		\draw [style={red_edge}, in=165, out=-90, looseness=1.25] (62.center) to (82.center);
		\draw [style={thick_purple}, in=45, out=-90, looseness=0.75] (82.center) to (83.center);
		\draw [style={thick_purple}, in=-90, out=120, looseness=0.75] (83.center) to (79.center);
		\draw [style={thick_purple}] (84.center) to (83.center);
		\draw [style={green_edge}, in=90, out=-45] (85) to (3.center);
		\draw [style={green_edge}, in=-90, out=135, looseness=0.50] (85) to (71.center);
		\draw [style=blue edge, in=90, out=-120, looseness=1.25] (85) to (1.center);
		\draw [style=blue edge] (85) to (53.center);
		\draw [style=blue edge] (13.center) to (50);
		\draw [style=blue edge, in=60, out=-150, looseness=0.75] (47.center) to (52);
		\draw [style={green_edge}] (14.center) to (50);
		\draw [style={green_edge}] (50) to (12.center);
		\draw [style={thick_purple}, in=-90, out=90, looseness=0.75] (47.center) to (84.center);
		\draw [style={thick_purple}] (84.center) to (83.center);
		\draw [style={thick_purple}, in=90, out=-105, looseness=0.75] (84.center) to (53.center);
		\draw [style=blue edge] (78) to (67.center);
		\draw [style={green_edge}] (65.center) to (78);
		\draw [style={green_edge}] (78) to (69.center);
		\draw [style={green_edge}] (81) to (63.center);
		\draw [style={green_edge}] (81) to (60.center);
		\draw [style=blue edge] (81) to (61.center);
	\end{pgfonlayer}
\end{tikzpicture}
 - 
  \tikzsetnextfilename{K_H}
  \tikzstyle{every picture}=[tikzfig]
  \begin{tikzpicture}
	\begin{pgfonlayer}{nodelayer}
		\node [style=none] (0) at (3.25, 3.25) {};
		\node [style=none] (1) at (2.75, 3.25) {};
		\node [style=none] (2) at (2.25, 3.25) {};
		\node [style=none] (3) at (1.75, 3.25) {};
		\node [style=none] (4) at (1.25, 3.25) {};
		\node [style=none] (5) at (0.75, 3.25) {};
		\node [style=none] (6) at (0.25, 3.25) {};
		\node [style=none] (7) at (-0.25, 3.25) {};
		\node [style=none] (8) at (-0.75, 3.25) {};
		\node [style=none] (9) at (-1.25, 3.25) {};
		\node [style=none] (10) at (-1.75, 3.25) {};
		\node [style=none] (11) at (-2.25, 3.25) {};
		\node [style=none] (12) at (-2.75, 3.25) {};
		\node [style=none] (13) at (-3.25, 3.25) {};
		\node [style=none] (14) at (-3.75, 3.25) {};
		\node [style=none] (47) at (-1.75, 1.5) {};
		\node [style={thick_dot_ochre}] (50) at (-3.25, 2.75) {};
		\node [style={thick_dot_ochre}] (52) at (-0.25, 2.75) {};
		\node [style=none] (53) at (2, 1.5) {};
		\node [style=none] (57) at (-3.25, -3.25) {};
		\node [style=none] (58) at (-2.75, -3.25) {};
		\node [style=none] (59) at (-2.25, -3.25) {};
		\node [style=none] (60) at (-1.75, -3.25) {};
		\node [style=none] (61) at (-1.25, -3.25) {};
		\node [style=none] (62) at (-0.75, -3.25) {};
		\node [style=none] (63) at (-0.25, -3.25) {};
		\node [style=none] (64) at (0.25, -3.25) {};
		\node [style=none] (65) at (0.75, -3.25) {};
		\node [style=none] (66) at (1.25, -3.25) {};
		\node [style=none] (67) at (1.75, -3.25) {};
		\node [style=none] (68) at (2.25, -3.25) {};
		\node [style=none] (69) at (2.75, -3.25) {};
		\node [style=none] (70) at (3.25, -3.25) {};
		\node [style=none] (71) at (3.75, -3.25) {};
		\node [style={thick_dot_ochre}] (78) at (1.75, -2.5) {};
		\node [style=none] (79) at (1.75, -1.25) {};
		\node [style={thick_dot_ochre}] (81) at (-1.25, -2.5) {};
		\node [style=none] (82) at (-1.75, -1.25) {};
		\node [style=none] (83) at (0, -0.25) {};
		\node [style=none] (84) at (0, 0.5) {};
		\node [style={thick_dot_ochre}] (85) at (2.5, 2.25) {};
	\end{pgfonlayer}
	\begin{pgfonlayer}{edgelayer}
		\draw [style=blue edge] (7.center) to (52);
		\draw [style=blue edge, in=120, out=-75, looseness=0.75] (50) to (47.center);
		\draw [style=blue edge] (47.center) to (10.center);
		\draw [style={green_edge}] (8.center) to (52);
		\draw [style={green_edge}] (52) to (6.center);
		\draw [style={red_edge}, in=135, out=-90] (11.center) to (47.center);
		\draw [style={red_edge}, in=-90, out=45] (47.center) to (9.center);
		\draw [style={red_edge}, in=135, out=-90] (5.center) to (53.center);
		\draw [style={red_edge}] (2.center) to (53.center);
		\draw [style=blue edge, in=105, out=-90] (4.center) to (53.center);
		\draw [style={red_edge}, in=-90, out=15] (53.center) to (0.center);
		\draw [style=blue edge] (78) to (79.center);
		\draw [style={red_edge}, in=-135, out=90, looseness=0.75] (66.center) to (79.center);
		\draw [style={red_edge}, in=-30, out=90, looseness=0.75] (68.center) to (79.center);
		\draw [style=blue edge, in=-165, out=90, looseness=0.75] (64.center) to (79.center);
		\draw [style=blue edge, in=-15, out=90] (70.center) to (79.center);
		\draw [style={red_edge}, in=-165, out=90] (57.center) to (82.center);
		\draw [style={red_edge}, in=-120, out=90, looseness=0.75] (59.center) to (82.center);
		\draw [style=blue edge, in=-135, out=90] (58.center) to (82.center);
		\draw [style=blue edge] (82.center) to (81);
		\draw [style={red_edge}, in=-15, out=90, looseness=1.25] (62.center) to (82.center);
		\draw [style={thick_purple}, in=-135, out=90, looseness=0.75] (82.center) to (83.center);
		\draw [style={thick_purple}, in=90, out=-60, looseness=0.75] (83.center) to (79.center);
		\draw [style={thick_purple}] (84.center) to (83.center);
		\draw [style={green_edge}, in=-90, out=135] (85) to (3.center);
		\draw [style={green_edge}, in=90, out=-45, looseness=0.50] (85) to (71.center);
		\draw [style=blue edge, in=-90, out=60, looseness=1.25] (85) to (1.center);
		\draw [style=blue edge] (85) to (53.center);
		\draw [style=blue edge] (13.center) to (50);
		\draw [style=blue edge, in=-120, out=45, looseness=0.75] (47.center) to (52);
		\draw [style={green_edge}] (14.center) to (50);
		\draw [style={green_edge}] (50) to (12.center);
		\draw [style={thick_purple}, in=90, out=-90, looseness=0.75] (47.center) to (84.center);
		\draw [style={thick_purple}] (84.center) to (83.center);
		\draw [style={thick_purple}, in=-90, out=75, looseness=0.75] (84.center) to (53.center);
		\draw [style=blue edge] (78) to (67.center);
		\draw [style={green_edge}] (65.center) to (78);
		\draw [style={green_edge}] (78) to (69.center);
		\draw [style={green_edge}] (81) to (63.center);
		\draw [style={green_edge}] (81) to (60.center);
		\draw [style=blue edge] (81) to (61.center);
	\end{pgfonlayer}
\end{tikzpicture}
 \right).
                $}
        \end{equation*}
        \end{theorem}

        \appendix
        \section{Thick Soergel Calculus}\label{apdx:thick_calculus}

        We review the theory of generalized Abe--Bott--Samelson bimodules and thick Soergel calculus. 
        More detailed treatments of the relevant material can be found in \cite{EliThick, Elib, AbeThick1, AbeThick2}.
        This material is somewhat tangential to the discussion in the main body of the paper.
        We do not claim any new ideas in the theory of thick Soergel calculus, we simply find it a convenient way to write some of the morphisms appearing in the 
        $3$-color relation. In particular, we neither claim nor give a full diagrammatic presentation of generalized Abe--Bott--Samelson bimodules.
        Nonetheless, we will write thick Elias--Williamson diagrams and interpret them as morphisms between generalized Abe--Bott--Samelson bimodules.

        \subsection{Generalized Abe--Bott--Samelson bimodules}

        The basic idea of thick Soergel calculus is to add in objects corresponding to the indecomposable Soergel bimodules indexed by longest words of parabolic subgroups.
        This produces the category of \emph{generalized Abe--Bott--Samelson bimodules}.
        We give a somewhat ad-hoc treatment of the theory, since to the author's knowledge, there is no complete diagrammatic presentation of generalized Abe--Bott--Samelson bimodules. 

        For $w \in W$, define $\partial_w = \partial_{s_1} \partial_{s_2} \ldots \partial_{s_k}$ for any reduced expression $\uw = (s_1, \ldots, s_k)$ of $w$.
        The balancedness assumption on $\fr{h}$ ensures that the braid relations hold for the Demazure operators, so this definition is independent of the reduced expression.

        Let $I \subset S$ be finitary, that is, $W_I$ is finite. Write $w_I$ for the longest element of $W_I$, and set $\ell_I = \ell (w_I)$.
         To ease the notation, we will write $R^I = R^{W_I}$ and $\partial_I = \partial_{w_I}$. Observe that $\partial_I : R \to \R^I$ is $R^I$ linear of degree $-2 \ell_I$.
        Indeed, $\partial_s \partial_I = 0$ for every $s \in I$, so the image is $W_I$-invariant.
         To ensure a well-behaved theory, we make the following assumption.

        \begin{assumption}\label{ass:gen_bs}
        For every finitary $I \subset S$:
        \begin{enumerate}
                \item the homomorphism $W_I \to \Aut (\fr{h}^*)$ is injective;
                \item there exists a homogeneous $q_I \in R$ of degree $2\ell_I$ such that $\partial_{I} (q_I) = 1$.
        \end{enumerate}
        \end{assumption}
        
        Assumption~\ref{ass:gen_bs} (2) is known as \emph{parabolic Demazure surjectivity}.
        By \cite[Proposition 2.12]{AbeThick2}, the collection $\{\delta_w (q_I)\}_{w \in W_I}$ forms an a graded $R^I$-module basis of $R$. Moreover, by \cite[Proposition 2.14]{AbeThick2}, the pairing
        \[R \times R \to R^I, \qquad (f,g) \mapsto \partial_I (fg)\]
        is a perfect. As a consequence, $R^I \subset R$ is a graded Frobenius extension with trace $\partial_I$.

        For $I \subset S$ finitary, consider the graded $R$-bimodule,
        \[B_I \coloneq R \otimes_{R^I} R (\ell_I).\]
        We wish to upgrade this bimodule to an object of $\sfC (\fr{h}, W)$.
        To do this, there is a canonical isomorphism of $Q$-bimodules,
        \[B_I \otimes_R Q \stackrel{\sim}{\to} \bigoplus_{w \in W_I} Q_w (\ell_I), \qquad f \otimes g \otimes 1 \mapsto (f w(g))_{w \in W_I}.\]
        This is the localization isomorphism of \cite[Lemma 2.20]{AbeThick1}.
        This canonical isomorphism defines the localization data for $B_I$ and makes it an object of $\sfC (\fr{h}, W)$.

        \begin{definition}
                The category of \emph{generalized Abe--Bott--Samelson bimodules} is the smallest full subcategory $\gBS (\fr{h}, W)$ of $\sfC (\fr{h}, W)$ 
                containing the monoidal unit $R$ and the objects $B_I$ for finitary $I \subset S$, and is closed under the monoidal product $\otimes_R$ and the shift functor $(1)$.
        \end{definition}

        There is a perfect right $R$-bilinear pairing on $B_I$ given by
        \begin{equation}\label{eq:I_pairing}
                \langle f \otimes g, f' \otimes g' \rangle_I \coloneq \partial_I (ff') gg'.
        \end{equation}
        The same argument from \S\ref{ss:BS} shows that $B_I$ identifies with $\DD (B_I)$. As a result, duality preserves the category of generalized Abe--Bott--Samelson bimodules.

        For a reduced expression $\uw = (s_1, \ldots, s_k)$ of $w_I$, we define
        \[\iota_{\uw} : B_I \to B_{\uw}, \qquad f \otimes g \mapsto f \otimes 1 \otimes \ldots \otimes 1 \otimes g.\]
        This map is well-defined since $R^I \subset R^{s_i}$ for every $i=1,\ldots, k$. It is then easy to check that it indeed defines a morphism of generalized Abe--Bott--Samelson bimodules.
        We also define 
        \[ p_{\uw} \coloneq \DD (\iota_{\uw}) : B_{\uw} \to B_I.\]
        Both $\iota_{\uw}$ and $p_{\uw}$ are of degree $0$.

        \begin{lemma}\label{lem:gen_proj}
                Let $\uw$ be a reduced expression for $w_I$. Then $p_{\uw} \iota_{\uw} = \id_{B_I}$. As a consequence, $e_{\uw} \coloneq \iota_{\uw} p_{\uw}$ is a self-adjoint
                idempotent with image isomorphic to $B_I$.
        \end{lemma}
        \begin{proof}
                We can compute,
                \begin{align*}
                        \langle \iota_{\uw} (f \otimes g) , \iota_{\uw} (f' \otimes g') \rangle_{\uw} &= (\partial_{s_k} \ldots \partial_{s_1}) (ff') gg' \\
                        &= \partial_I (ff') gg' \\
                        &= \langle f \otimes g, f' \otimes g' \rangle_I.
                \end{align*}
                The pairing on the left is the canonical perfect pairing on $B_{\uw}$ given by \eqref{eq:pairing}. The second equality uses the fact that $w_I^{-1} = w_I$ and the well-definedness of $\partial_I$.
                Adjointness and nondegeneracy now give that $p_{\uw} \iota_{\uw} = \id_{B_I}$ as desired. The idempotence and self-adjointness of $e_{\uw}$ immediately follow. 
        \end{proof}

        As a consequence of Lemma~\ref{lem:gen_proj}, we see that $B_I$ is a direct summand of an ordinary Bott--Samelson bimodule.
        Every $R$-bimodule endomorphism of $B_I$ is determined by the image of $1 \otimes 1$, since this element generates $B_I$ as an $R$-bimodule. Its degree $-\ell_I$ component is $\k (1 \otimes 1$), so
        \[\Hom_{\sf C(\fr{h}, W)}^0 (B_I, B_I) = \k \id_{B_I}.\]
        In particular, $B_I$ is indecomposable, and in fact, is the indecomposable Soergel bimodules corresponding to $w_I$.

        Choose homogeneous dual $R^I$-bases $\{x_a\}$ and $\{y_a\}$ of $R$ under the trace such that $\partial_I (x_a y_b) = \delta_{a,b}$.
        The existence of such dual bases is guaranteed by parabolic Demazure surjectivity. We then have that
        \begin{equation}\label{eq:dual_basis_identity_1}
                f = \sum_a x_a \partial_I (y_a f) = \sum_a \partial_I (f x_a) y_a
        \end{equation}
        for any $f \in R$. 
        The tensor $\Delta_I = \sum_a x_a \otimes y_a \in B_I$ is invariant under the choice of dual bases. Namely, under the isomorphism
        \[R \otimes_{R^I} R \stackrel{\sim}{\to} \textnormal{End}_{R^I} (R), \qquad f\otimes g \mapsto (\varphi \mapsto f \partial_I (g\varphi)),\]
        $\Delta_I$ corresponds to the identity endomorphism. The same description also gives that 
        \begin{equation}\label{eq:dual_basis_identity_2}
                (f \otimes 1) \Delta_I = \Delta_I (1 \otimes f)
        \end{equation}
        for $f \in R$.
        Under the identity $B_I \otimes_R B_I = R \otimes_{R^I} R \otimes_{R^I} R (2\ell_I)$, we can make $B_I$ into a graded Frobenius algebra as follows,
        \begin{align*}
                \nu^I &: B_I \to B_I \otimes_R B_I , & f \otimes g &\mapsto f \otimes 1 \otimes g, \\
                \epsilon^I &: B_I \to R, & f \otimes g &\mapsto fg, \\
                \mu^I &: B_I \otimes_R B_I \to B_I , & f \otimes g \otimes h &\mapsto \partial_I (g) f \otimes h, \\
                \eta^I &: R \to B_I , & 1 &\mapsto \Delta_I,
        \end{align*}
        The maps $\epsilon^I$ and $\eta^I$ have degree $\ell_I$, and $\mu^I$ and $\nu^I$ have degree $-\ell_I$.
        It is easy to check from \eqref{eq:dual_basis_identity_1} and \eqref{eq:dual_basis_identity_2} that these make $B_I$ into a graded Frobenius algebra. 
        Moreover, the same equations imply that duality interchanges $\epsilon^I$ with $\eta^I$ and $\mu^I$ with $\nu^I$. 
        When $I = \{s\}$, these recover the usual 1-color generating morphisms (see \S\ref{ss:gen_mors}).

        Let $s \in I$. Using the identifications, 
        \[ B_I \otimes_R B_s = R \otimes_{R^I} R \otimes_{R^s} R (\ell_I + 1) \qquad\text{and}\qquad B_s \otimes_R B_I = R \otimes_{R^s} R \otimes_{R^I} R (\ell_I + 1),\] 
        we can define the right and left splitting and absorption maps,
         \begin{align*}
                \nu^{I,s} &: B_I \to B_I \otimes_R B_s , & f \otimes g &\mapsto f \otimes 1 \otimes g, \\
                \mu^{I,s} &: B_I \otimes_R B_s \to B_I , & f \otimes g \otimes h &\mapsto  f \otimes \partial_s (g) h, \\
                 \nu^{s,I} &: B_I \to B_s \otimes_R B_I , & f \otimes g &\mapsto f \otimes 1 \otimes g, \\
                \mu^{s,I} &: B_s \otimes_R B_I \to B_I , & f \otimes g \otimes h &\mapsto \partial_s (g) f \otimes h.
        \end{align*}
        These maps have degree $-1$. Their well-definedness follows from $R^I \subset R^s$ and the $R^s$-linearity of $\partial_s$. 
        It is also not hard to check from similar arguments as before that these preserve localization data, and so are morphisms in $\sfC (\fr{h}, W)$.
        Again, one has that duality interchanges the corresponding maps.

        \subsection{Thick Soergel calculus}

        Let $I \subset S$ be finitary, and fix a reduced expression $\uw$ of $w_I$. 
        We will give a diagrammatic meaning to $B_I$ and the morphisms $p_{\uw}$ and $\iota_{\uw}$.
        To do so, draw $\id_{B_I}$ as a thick strand ``colored'' by $I$. 
        The morphisms $p_{\uw}$ and $\iota_{\uw}$ are drawn by fans:
        \[
  \tikzsetnextfilename{gen_projs}
  \tikzstyle{every picture}=[tikzfig]
  \begin{tikzpicture}
	\begin{pgfonlayer}{nodelayer}
		\node [style=none] (0) at (-5, -1) {};
		\node [style=none] (1) at (-2, -1) {};
		\node [style=none] (2) at (-5, 2) {};
		\node [style=none] (3) at (-2, 2) {};
		\node [style=none] (4) at (-3.5, 0.5) {};
		\node [style=none] (5) at (-5, -1) {};
		\node [style=none] (6) at (-4.5, -1) {};
		\node [style=none] (7) at (-2, -1) {};
		\node [style=none] (8) at (-3.5, 2) {};
		\node [style=none] (9) at (-3, -0.75) {$\ldots$};
		\node [style=none] (10) at (2, 2) {};
		\node [style=none] (11) at (5, 2) {};
		\node [style=none] (12) at (2, -1) {};
		\node [style=none] (13) at (5, -1) {};
		\node [style=none] (14) at (3.5, 0.5) {};
		\node [style=none] (15) at (2, 2) {};
		\node [style=none] (16) at (2.5, 2) {};
		\node [style=none] (17) at (5, 2) {};
		\node [style=none] (18) at (3.5, -1) {};
		\node [style=none] (19) at (4, 1.75) {$\ldots$};
		\node [style=none] (20) at (-3.5, -1.75) {$p_{\uw} : B_{\uw} \to B_{I}$};
		\node [style=none] (21) at (3.5, -1.75) {$\iota_{\uw} : B_I \to B_{\uw} $};
		\node [style=none] (23) at (-4, -1) {};
		\node [style=none] (24) at (3, 2) {};
	\end{pgfonlayer}
	\begin{pgfonlayer}{edgelayer}
		\draw [style=row] (5.center) to (4.center);
		\draw [style=row] (6.center) to (4.center);
		\draw [style=row] (7.center) to (4.center);
		\draw [style={thick_purple}] (4.center) to (8.center);
		\draw [style=row] (15.center) to (14.center);
		\draw [style=row] (16.center) to (14.center);
		\draw [style=row] (17.center) to (14.center);
		\draw [style={thick_purple}] (14.center) to (18.center);
		\draw [style=row] (23.center) to (4.center);
		\draw [style=row] (24.center) to (14.center);
	\end{pgfonlayer}
\end{tikzpicture}
.\]
        The thin strands in the fan carry the colors of $\uw$, we are just using black here to avoid specifying colors.
        
        We can draw the Frobenius structure maps for $B_I$ the same as one does in any monoidal category:
        \[
  \tikzsetnextfilename{frob_thick}
  \tikzstyle{every picture}=[tikzfig]
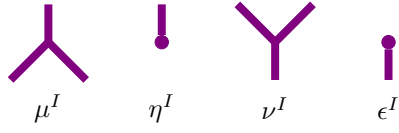
  \begin{tikzpicture}
	\begin{pgfonlayer}{nodelayer}
		\node [style=none] (0) at (-4, -1) {};
		\node [style=none] (1) at (-2, -1) {};
		\node [style=none] (2) at (-3, 1) {};
		\node [style=none] (3) at (0, 1) {};
		\node [style={thick_violet_dot}] (4) at (0, 0) {};
		\node [style=none] (5) at (2, 1) {};
		\node [style=none] (6) at (3, 0) {};
		\node [style=none] (7) at (3, -1) {};
		\node [style=none] (8) at (4, 1) {};
		\node [style={thick_violet_dot}] (10) at (6, 0) {};
		\node [style=none] (12) at (6, -1) {};
		\node [style=none] (13) at (-3, 0) {};
		\node [style=none] (14) at (-3, -1.75) {$\mu^I$};
		\node [style=none] (15) at (0, -1.75) {$\eta^I$};
		\node [style=none] (16) at (3, -1.75) {$\nu^I$};
		\node [style=none] (17) at (6, -1.75) {$\epsilon^I$};
	\end{pgfonlayer}
	\begin{pgfonlayer}{edgelayer}
		\draw [style={thick_purple}] (0.center) to (13.center);
		\draw [style={thick_purple}] (13.center) to (1.center);
		\draw [style={thick_purple}] (13.center) to (2.center);
		\draw [style={thick_purple}] (3.center) to (4);
		\draw [style={thick_purple}] (5.center) to (6.center);
		\draw [style={thick_purple}] (6.center) to (8.center);
		\draw [style={thick_purple}] (6.center) to (7.center);
		\draw [style={thick_purple}] (12.center) to (10);
	\end{pgfonlayer}
\end{tikzpicture}
.\]

        For $s \in I$, we draw the absorption and splitting maps by replacing the corresponding thick strand by a thin strand colored by $s$:
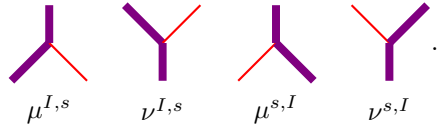
        \[
  \tikzsetnextfilename{abs_split}
  \tikzstyle{every picture}=[tikzfig]
  \begin{tikzpicture}
	\begin{pgfonlayer}{nodelayer}
		\node [style=none] (0) at (-4, -1) {};
		\node [style=none] (1) at (-2, -1) {};
		\node [style=none] (2) at (-3, 1) {};
		\node [style=none] (5) at (-1, 1) {};
		\node [style=none] (6) at (0, 0) {};
		\node [style=none] (7) at (0, -1) {};
		\node [style=none] (8) at (1, 1) {};
		\node [style=none] (13) at (-3, 0) {};
		\node [style=none] (14) at (-3, -1.75) {$\mu^{I,s}$};
		\node [style=none] (16) at (0, -1.75) {$\nu^{I,s}$};
		\node [style=none] (17) at (2, -1) {};
		\node [style=none] (18) at (4, -1) {};
		\node [style=none] (19) at (3, 1) {};
		\node [style=none] (20) at (3, 0) {};
		\node [style=none] (21) at (3, -1.75) {$\mu^{s,I}$};
		\node [style=none] (22) at (5, 1) {};
		\node [style=none] (23) at (6, 0) {};
		\node [style=none] (24) at (6, -1) {};
		\node [style=none] (25) at (7, 1) {};
		\node [style=none] (26) at (6, -1.75) {$\nu^{s,I}$};
	\end{pgfonlayer}
	\begin{pgfonlayer}{edgelayer}
		\draw [style={thick_purple}] (0.center) to (13.center);
		\draw [style={red_edge}] (13.center) to (1.center);
		\draw [style={thick_purple}] (13.center) to (2.center);
		\draw [style={thick_purple}] (5.center) to (6.center);
		\draw [style={red_edge}] (6.center) to (8.center);
		\draw [style={thick_purple}] (6.center) to (7.center);
		\draw [style={red_edge}] (17.center) to (20.center);
		\draw [style={thick_purple}] (20.center) to (18.center);
		\draw [style={thick_purple}] (20.center) to (19.center);
		\draw [style={red_edge}] (22.center) to (23.center);
		\draw [style={thick_purple}] (23.center) to (25.center);
		\draw [style={thick_purple}] (23.center) to (24.center);
	\end{pgfonlayer}
\end{tikzpicture}
.\]
        The red strand here is for the color of $s$.

        \begin{lemma}\label{lem:thick_relns}
                We have the following equalities
                \begin{equation}\label{eq:thick_relns}
                
  \tikzsetnextfilename{thick_relns}
  \tikzstyle{every picture}=[tikzfig]
  \begin{tikzpicture}
	\begin{pgfonlayer}{nodelayer}
		\node [style=none] (2) at (-2.75, 4.5) {};
		\node [style=none] (4) at (-1.75, 3.5) {};
		\node [style=none] (8) at (-1.75, 4.5) {};
		\node [style=none] (12) at (-2.75, 0.5) {};
		\node [style=none] (13) at (-0.75, 0.5) {};
		\node [style=none] (14) at (-1.75, 1.5) {};
		\node [style=none] (18) at (-1.75, 0.5) {};
		\node [style=none] (25) at (-0.75, 4.5) {};
		\node [style=none] (34) at (1.25, 4.5) {};
		\node [style=none] (39) at (1.25, 4.5) {};
		\node [style=none] (42) at (1.25, 2.5) {};
		\node [style=none] (43) at (1.25, 0.5) {};
		\node [style=none] (46) at (2, 4.5) {};
		\node [style=none] (47) at (-1.75, 2.5) {$\ldots$};
		\node [style=none] (48) at (-1.25, 2.5) {};
		\node [style=none] (49) at (0.25, 2.5) {$=$};
		\node [style=none] (50) at (-0.75, -0.75) {};
		\node [style=none] (51) at (-2.75, -4.75) {};
		\node [style=none] (52) at (-2.75, -2.75) {};
		\node [style=none] (53) at (-3.75, -0.75) {};
		\node [style=none] (54) at (-2.25, -1.75) {};
		\node [style=none] (55) at (-3.25, -0.75) {};
		\node [style=none] (56) at (-2.75, -0.75) {};
		\node [style=none] (57) at (-2, -1) {$\ldots$};
		\node [style=none] (59) at (3.5, -0.75) {};
		\node [style=none] (60) at (2.5, -4.75) {};
		\node [style=none] (62) at (1.25, -0.75) {};
		\node [style=none] (64) at (1.75, -0.75) {};
		\node [style=none] (65) at (2.25, -0.75) {};
		\node [style=none] (66) at (2.75, -1) {$\ldots$};
		\node [style=none] (68) at (1.5, -1.5) {};
		\node [style=none] (69) at (1.75, -2.25) {};
		\node [style=none] (71) at (2.25, -4) {};
		\node [style=none] (72) at (0.25, -2.75) {$=$};
		\node [style=none] (75) at (-8.25, 3.5) {};
		\node [style=none] (76) at (-8.25, 4.5) {};
		\node [style=none] (79) at (-8.25, 1.5) {};
		\node [style=none] (80) at (-8.25, 0.5) {};
		\node [style=none] (82) at (-5.25, 4.5) {};
		\node [style=none] (83) at (-5.25, 4.5) {};
		\node [style=none] (84) at (-5.25, 2.5) {};
		\node [style=none] (85) at (-5.25, 0.5) {};
		\node [style=none] (87) at (-8, 2.5) {$\ldots$};
		\node [style=none] (89) at (-6.25, 2.5) {$=$};
		\node [style=none] (92) at (5.5, 3.5) {};
		\node [style=none] (93) at (5.5, 4.5) {};
		\node [style=none] (96) at (5.5, 1.5) {};
		\node [style=none] (97) at (5.5, 0.5) {};
		\node [style=none] (98) at (6.5, 0.5) {};
		\node [style=none] (99) at (8.5, 4.5) {};
		\node [style=none] (100) at (8.5, 4.5) {};
		\node [style=none] (101) at (8.5, 2.5) {};
		\node [style=none] (102) at (8.5, 0.5) {};
		\node [style=none] (103) at (9.25, 0.5) {};
		\node [style=none] (104) at (5.5, 2.5) {$\ldots$};
		\node [style=none] (105) at (6, 2.5) {};
		\node [style=none] (106) at (7.5, 2.5) {$=$};
		\node [style=none] (107) at (8.75, -4.75) {};
		\node [style=none] (108) at (6.75, -0.75) {};
		\node [style=none] (109) at (6.75, -2.75) {};
		\node [style=none] (110) at (5.75, -4.75) {};
		\node [style=none] (111) at (7.25, -3.75) {};
		\node [style=none] (112) at (6.25, -4.75) {};
		\node [style=none] (113) at (6.75, -4.75) {};
		\node [style=none] (114) at (7.5, -4.5) {$\ldots$};
		\node [style=none] (115) at (13, -4.75) {};
		\node [style=none] (116) at (12, -0.75) {};
		\node [style=none] (117) at (10.75, -4.75) {};
		\node [style=none] (118) at (11.25, -4.75) {};
		\node [style=none] (119) at (11.75, -4.75) {};
		\node [style=none] (120) at (12.25, -4.5) {$\ldots$};
		\node [style=none] (121) at (11, -4) {};
		\node [style=none] (122) at (11.25, -3.25) {};
		\node [style=none] (123) at (11.75, -1.5) {};
		\node [style=none] (124) at (9.75, -2.75) {$=$};
		\node [style=none] (125) at (-7.25, 0.5) {};
		\node [style=none] (126) at (-7.25, 4.5) {};
		\node [style=none] (127) at (-9.25, 4.5) {};
		\node [style=none] (128) at (-9.25, 0.5) {};
		\node [style=none] (129) at (2, 0.5) {};
		\node [style=none] (130) at (4.5, 4.5) {};
		\node [style=none] (131) at (4.5, 0.5) {};
		\node [style=none] (132) at (6.5, 4.5) {};
		\node [style=none] (133) at (-11.25, -0.75) {};
		\node [style=none] (134) at (-11.25, -0.75) {};
		\node [style=none] (135) at (-11.25, -2.75) {};
		\node [style=none] (136) at (-11.25, -4.75) {};
		\node [style={small_dot_s}] (137) at (-10.5, -1.5) {};
		\node [style=none] (138) at (-10.5, -4.75) {};
		\node [style=none] (139) at (-8.5, -0.75) {};
		\node [style=none] (140) at (-8.5, -0.75) {};
		\node [style=none] (141) at (-8.5, -2.75) {};
		\node [style=none] (142) at (-8.5, -4.75) {};
		\node [style=none] (143) at (-9.5, -2.75) {$=$};
		\node [style=none] (144) at (-6.5, -0.75) {};
		\node [style=none] (145) at (-6.5, -0.75) {};
		\node [style=none] (146) at (-6.5, -2.75) {};
		\node [style=none] (147) at (-6.5, -4.75) {};
		\node [style={small_dot_s}] (148) at (-5.75, -4) {};
		\node [style=none] (149) at (-5.75, -4.75) {};
		\node [style=none] (150) at (-7.5, -2.75) {$=$};
	\end{pgfonlayer}
	\begin{pgfonlayer}{edgelayer}
		\draw [style={thick_purple}] (4.center) to (8.center);
		\draw [style={thick_purple}] (14.center) to (18.center);
		\draw [style={red_edge}] (42.center) to (46.center);
		\draw [style=row, bend right=75, looseness=1.75] (4.center) to (14.center);
		\draw [style=row, bend right=60, looseness=1.50] (4.center) to (14.center);
		\draw [style=row, bend right=45, looseness=1.25] (4.center) to (14.center);
		\draw [style={red_edge}] (48.center) to (25.center);
		\draw [style={thick_purple}] (53.center) to (52.center);
		\draw [style={thick_purple}] (52.center) to (51.center);
		\draw [style={thick_purple}] (52.center) to (54.center);
		\draw [style=row] (55.center) to (54.center);
		\draw [style=row] (54.center) to (56.center);
		\draw [style=row] (54.center) to (50.center);
		\draw [style=row] (68.center) to (64.center);
		\draw [style=row] (69.center) to (65.center);
		\draw [style=row] (71.center) to (59.center);
		\draw [style={thick_purple}] (62.center) to (60.center);
		\draw [style={thick_purple}] (39.center) to (43.center);
		\draw [style={thick_purple}] (75.center) to (76.center);
		\draw [style={thick_purple}] (79.center) to (80.center);
		\draw [style=row, bend right=75, looseness=1.75] (75.center) to (79.center);
		\draw [style=row, bend right=60, looseness=1.50] (75.center) to (79.center);
		\draw [style=row, bend right=45, looseness=1.25] (75.center) to (79.center);
		\draw [style=row, bend right=75, looseness=1.75] (79.center) to (75.center);
		\draw [style={thick_purple}] (83.center) to (85.center);
		\draw [style={red_edge}] (14.center) to (48.center);
		\draw [style={red_edge}] (48.center) to (4.center);
		\draw [style={thick_purple}] (92.center) to (93.center);
		\draw [style={thick_purple}] (96.center) to (97.center);
		\draw [style={red_edge}] (101.center) to (103.center);
		\draw [style=row, bend right=75, looseness=1.75] (92.center) to (96.center);
		\draw [style=row, bend right=60, looseness=1.50] (92.center) to (96.center);
		\draw [style=row, bend right=45, looseness=1.25] (92.center) to (96.center);
		\draw [style={red_edge}] (105.center) to (98.center);
		\draw [style={thick_purple}] (100.center) to (102.center);
		\draw [style={red_edge}] (96.center) to (105.center);
		\draw [style={red_edge}] (105.center) to (92.center);
		\draw [style={thick_purple}] (110.center) to (109.center);
		\draw [style={thick_purple}] (109.center) to (108.center);
		\draw [style={thick_purple}] (109.center) to (111.center);
		\draw [style=row] (112.center) to (111.center);
		\draw [style=row] (111.center) to (113.center);
		\draw [style=row] (111.center) to (107.center);
		\draw [style=row] (121.center) to (118.center);
		\draw [style=row] (122.center) to (119.center);
		\draw [style=row] (123.center) to (115.center);
		\draw [style={thick_purple}] (117.center) to (116.center);
		\draw [style={red_edge}] (135.center) to (137);
		\draw [style={thick_purple}] (134.center) to (136.center);
		\draw [style={thick_purple}] (140.center) to (142.center);
		\draw [style={red_edge}] (146.center) to (148);
		\draw [style={thick_purple}] (145.center) to (147.center);
	\end{pgfonlayer}
\end{tikzpicture}

                \end{equation}
        \end{lemma}
        \begin{proof}
                The first relation means that $p_{\uw} \iota_{\uw} = \id_{B_I}$ which is proved in Lemma~\ref{lem:gen_proj}.

                For the second relation, we write $\uw = \ux s$ for some $s$ corresponding to the red strand. The relation can be translated algebraically to mean
                \begin{equation}\label{eq:thick_relns_1}
                        (p_{\uw} \otimes \id_{B_s}) (\id_{B_{\ux}} \otimes \nu^s) \iota_{\uw} = \nu^{I, s}.
                \end{equation}
                It is easy to check from the definitions that
                \[(\id_{B_{\ux}} \otimes \nu^s) \iota_{\uw} = (\iota_{\uw} \otimes \id_{B_s}) \nu^{I, s}.\]
                We can then compose with $p_{\uw} \otimes \id_{B_s}$ and use the first relation to deduce,
                \[ (p_{\uw} \otimes \id_{B_s}) (\id_{B_{\ux}} \otimes \nu^s) \iota_{\uw} = (p_{\uw} \otimes \id_{B_s}) (\iota_{\uw} \otimes \id_{B_s}) \nu^{I, s} = \nu^{I,s}. \]

                The third relation follows from the second by duality.

                The fourth relation is an easy consequence of the first three.

                For the fifth relation, we first define a map $A_{\uw} : B_I \to B_I \otimes_R B_{\uw}$ as the map on the right-hand side. 
                In more detail, we can define this map recursively. Set $A_{\emptyset} = \id_{B_I}$, and then
                \[A_{\ux s} = (A_{\ux} \otimes \id_{B_s}) \nu^{I, s} .\]
                By induction, we have that 
                \[A_{\uw} (f \otimes g) = f \otimes 1 \otimes \ldots \otimes 1 \otimes g = (\id_{B_I} \otimes \iota_{\uw}) \nu^I (f \otimes g).\]
                Note that $(\id_{B_I} \otimes \iota_{\uw}) \nu^I$ is the left-hand side of the relation. 

                Finally, the sixth relation follows from the fifth by duality.
        \end{proof}

        For $I = \{s,t\}$ for $s \neq t$, the three first relation and the split/absorption identities in Lemma~\ref{lem:thick_relns} are the relations drawn in \cite[(6.17)--(6.18)]{Elib}.

        The last relation we use is specific to the $2$-color setting. 

        \begin{lemma}\label{lem:thick_jw}
                Let $I = \{s,t\}$, with $s \neq t$ and $m = m_{st} < \infty$.
                We have the following equality,
                 \begin{equation}\label{eq:thick_jw}
                
  \tikzsetnextfilename{thick_jw}
  \tikzstyle{every picture}=[tikzfig]
  \begin{tikzpicture}
	\begin{pgfonlayer}{nodelayer}
		\node [style=none] (73) at (-9, 2) {};
		\node [style=none] (74) at (-5.75, 2) {};
		\node [style=none] (77) at (-9, -2) {};
		\node [style=none] (78) at (-5.75, -2) {};
		\node [style=none] (89) at (-5.25, 0) {$=$};
		\node [style=none] (90) at (-8.25, -2) {};
		\node [style=none] (91) at (-7.75, -2) {};
		\node [style=none] (92) at (-6.25, -2) {};
		\node [style=none] (93) at (-8.25, 2) {};
		\node [style=none] (94) at (-7.75, 2) {};
		\node [style=none] (95) at (-6.25, 2) {};
		\node [style=none] (96) at (-7.25, -0.75) {};
		\node [style=none] (97) at (-7.25, 0.75) {};
		\node [style=none] (98) at (-7, 1.75) {$\ldots$};
		\node [style=none] (99) at (-7, -1.75) {$\ldots$};
		\node [style=none] (116) at (-4.75, 2) {};
		\node [style=none] (117) at (-1.5, 2) {};
		\node [style=none] (118) at (-4.75, -2) {};
		\node [style=none] (119) at (-1.5, -2) {};
		\node [style=none] (120) at (-4, -2) {};
		\node [style=none] (121) at (-3.5, -2) {};
		\node [style=none] (122) at (-2, -2) {};
		\node [style=none] (123) at (-4, 2) {};
		\node [style=none] (124) at (-3.5, 2) {};
		\node [style=none] (125) at (-2, 2) {};
		\node [style=none] (126) at (-3, -0.75) {};
		\node [style=none] (127) at (-3, 0.75) {};
		\node [style=none] (128) at (-2.75, 1.75) {$\ldots$};
		\node [style=none] (129) at (-2.75, -1.75) {$\ldots$};
		\node [style=none] (130) at (-3, 0) {$\ldots$};
		\node [style=none] (131) at (1.5, 2) {};
		\node [style=none] (132) at (4.75, 2) {};
		\node [style=none] (133) at (1.5, -2) {};
		\node [style=none] (134) at (4.75, -2) {};
		\node [style=none] (135) at (5.25, 0) {$=$};
		\node [style=none] (136) at (2.25, -2) {};
		\node [style=none] (137) at (2.75, -2) {};
		\node [style=none] (138) at (4.25, -2) {};
		\node [style=none] (139) at (2.25, 2) {};
		\node [style=none] (140) at (2.75, 2) {};
		\node [style=none] (141) at (4.25, 2) {};
		\node [style=none] (142) at (3.25, -0.75) {};
		\node [style=none] (143) at (3.25, 0.75) {};
		\node [style=none] (144) at (3.5, 1.75) {$\ldots$};
		\node [style=none] (145) at (3.5, -1.75) {$\ldots$};
		\node [style=none] (160) at (5.75, 2) {};
		\node [style=none] (161) at (9, 2) {};
		\node [style=none] (162) at (5.75, -2) {};
		\node [style=none] (163) at (9, -2) {};
		\node [style=none] (164) at (6.5, -2) {};
		\node [style=none] (165) at (7, -2) {};
		\node [style=none] (166) at (8.5, -2) {};
		\node [style=none] (167) at (6.5, 2) {};
		\node [style=none] (168) at (7, 2) {};
		\node [style=none] (169) at (8.5, 2) {};
		\node [style=none] (170) at (7.5, -0.75) {};
		\node [style=none] (171) at (7.5, 0.75) {};
		\node [style=none] (172) at (7.75, 1.75) {$\ldots$};
		\node [style=none] (173) at (7.75, -1.75) {$\ldots$};
		\node [style=none] (174) at (7.5, 0) {$\ldots$};
		\node [style=none] (175) at (-5.25, 2.75) {$m$ odd};
		\node [style=none] (176) at (5.25, 2.75) {$m$ even};
	\end{pgfonlayer}
	\begin{pgfonlayer}{edgelayer}
		\draw [style={thick_purple}] (97.center) to (96.center);
		\draw [style={red_edge}] (93.center) to (97.center);
		\draw [style={red_edge}] (90.center) to (96.center);
		\draw [style=blue edge] (94.center) to (97.center);
		\draw [style=blue edge] (96.center) to (91.center);
		\draw [style={red_edge}] (96.center) to (92.center);
		\draw [style={red_edge}] (97.center) to (95.center);
		\draw [style={red_edge}] (123.center) to (127.center);
		\draw [style={red_edge}] (120.center) to (126.center);
		\draw [style=blue edge] (124.center) to (127.center);
		\draw [style=blue edge] (126.center) to (121.center);
		\draw [style={red_edge}] (126.center) to (122.center);
		\draw [style={red_edge}] (127.center) to (125.center);
		\draw [style={red_edge}, bend left=60, looseness=1.50] (126.center) to (127.center);
		\draw [style=blue edge, bend left=60, looseness=2.00] (126.center) to (127.center);
		\draw [style=blue edge, bend right=60, looseness=1.75] (126.center) to (127.center);
		\draw [style={thick_purple}] (143.center) to (142.center);
		\draw [style={red_edge}] (139.center) to (143.center);
		\draw [style={red_edge}] (136.center) to (142.center);
		\draw [style=blue edge] (140.center) to (143.center);
		\draw [style=blue edge] (142.center) to (137.center);
		\draw [style=blue edge] (142.center) to (138.center);
		\draw [style=blue edge] (143.center) to (141.center);
		\draw [style={red_edge}] (167.center) to (171.center);
		\draw [style={red_edge}] (164.center) to (170.center);
		\draw [style=blue edge] (168.center) to (171.center);
		\draw [style=blue edge] (170.center) to (165.center);
		\draw [style=blue edge] (170.center) to (166.center);
		\draw [style=blue edge] (171.center) to (169.center);
		\draw [style={red_edge}, bend left=60, looseness=1.50] (170.center) to (171.center);
		\draw [style=blue edge, bend left=60, looseness=2.00] (170.center) to (171.center);
		\draw [style={red_edge}, bend right=60, looseness=1.75] (170.center) to (171.center);
	\end{pgfonlayer}
\end{tikzpicture}
.
        \end{equation}
        \end{lemma}
        \begin{proof}
                The relation can be found in \cite[(6.17)]{Elib}. 
                It is equivalent to the algebraic relation,
                \[\iota_{{}_s \underline{m}} p_{{}_s \underline{m}} = \beta_{ts} \beta_{st}.\]
                We will prove this by showing that
                \begin{equation}\label{eq:thick_jw_1}
                        \beta_{st} =  \iota_{{}_t \underline{m}} p_{{}_s \underline{m}}. 
                \end{equation}
                To see why this is sufficient, we can apply Lemma~\ref{lem:gen_proj} (1) to get
                \[ \iota_{{}_s \underline{m}} p_{{}_s \underline{m}}  =  \iota_{{}_s \underline{m}} p_{{}_t \underline{m}} \iota_{{}_t \underline{m}} p_{{}_s \underline{m}}  \stackrel{\eqref{eq:thick_jw_1}}{=} \beta_{ts} \beta_{st} .  \]
        
                It remains to prove \eqref{eq:thick_jw_1}. By Theorem~\ref{thm:dll_basis}, the degree-0 morphism space between $B_{{}_s \underline{m}}$ and $B_{{}_t \underline{m}}$ is a free rank-1 $\k$-module.
                The morphism $\beta_{st}$ is uniquely characterized in taking the $1$-tensor to the $1$-tensor. 
                It is easy to see from the definitions that $\iota_{{}_s \underline{m}} p_{{}_t \underline{m}}$ also takes the 1-tensor to the 1-tensor. Therefore, \eqref{eq:thick_jw_1} holds.
        \end{proof}

\end{document}